\documentclass{amsart}
\usepackage{graphicx}
\usepackage{amssymb}
\usepackage{amsmath, amscd}
\usepackage{mathtools}
\usepackage{pinlabel}
\usepackage{bm}
\usepackage{tikz}
\usepackage{enumitem}
\usepackage[left=23mm, right=23mm, top=25mm, bottom=25mm]{geometry}

\newtheorem{thm}{Theorem}[section]

\newtheorem{prop}[thm]{Proposition}
\newtheorem{cor}[thm]{Corollary}
\newtheorem{lemma}[thm]{Lemma}
\theoremstyle{definition}
\newtheorem{example}[thm]{Example}
\newtheorem{defi}[thm]{Definition}
\newtheorem{rmk}[thm]{Remark}

\newcommand{\lR}{{\mathbb{R}}}
\newcommand{\lC}{{\mathbb{C}}}

\newcommand{\lQ}{{\mathbb{Q}}}
\newcommand{\lD}{{\mathbb{D}}}
\newcommand{\lN}{{\mathbb{N}}}
\newcommand{\lZ}{{\mathbb{Z}}}

\newcommand{\lS}{{\mathbb{S}}}
\newcommand{\cM}{{\mathcal{M}}}
\newcommand{\cR}{{\mathcal{R}}}

\newcommand{\cA}{{\mathcal{A}}}
\newcommand{\cC}{{\mathcal{C}}}

\newcommand{\cI}{{\mathcal{I}}}
\newcommand{\cL}{{\mathcal{L}}}

\newcommand{\ind}{\operatorname{ind}}
\newcommand{\vir}{\operatorname{vir}}
\newcommand{\cyc}{\operatorname{cyc}}
\newcommand{\str}{\operatorname{str}}
\newcommand{\orb}{\operatorname{orb}}
\newcommand{\chord}{\operatorname{chord}}

\begin{document}

\title{Extension of Chekanov--Eliashberg algebra using annuli}
\author{Milica Dukic}

\begin{abstract}
We define an SFT-type invariant for Legendrian knots in the standard contact $\lR^3$. The invariant is a deformation of the Chekanov--Eliashberg differential graded algebra. The differential consists of a part that counts index zero $J$-holomorphic disks with up to two positive punctures, annuli with one positive puncture, and a string topological part. We describe the invariant and demonstrate its invariance combinatorially from the Lagrangian knot projection, and compute some simple examples where the deformation is non-vanishing.
\end{abstract}

\maketitle

\tableofcontents
\allowdisplaybreaks

\section{Introduction}\label{Sec:intro}
The purpose of this paper is to define an invariant of Legendrian knots in $\lR^3$ with the standard contact structure $\xi=\ker(dz-ydx)$ using ideas from symplectic field theory (SFT). More precisely, we extend the Chekanov--Eliashberg algebra \cite{Chekanov02} and the rational SFT \cite{Ng_rLSFT} for Legendrian knots to include pseudoholomorphic annuli with one positive boundary puncture. Furthermore, we explain how to compute the invariant combinatorially from the Lagrangian knot projection and demonstrate its invariance combinatorially.

\subsection*{Background}
Let $\Lambda$ be a Legendrian knot in $\lR^3$, i.e., a smooth knot everywhere tangent to the contact hyperplane distribution $\xi=\ker\lambda,\lambda=dz-ydx$. A Legendrian knot isotopy between Legendrian knots $\Lambda_0$ and $\Lambda_1$ is a smooth path of Legendrian knots $\Lambda_s,s\in[0,1]$ between them. We are interested in the question of distinguishing Legendrian knots up to Legendrian knot isotopy. There are two classical (homotopy theoretic) invariants of Legendrian knots, the Thurston--Bennequin number $\operatorname{tb}(\Lambda)$ and the rotation number $\operatorname{rot}(\Lambda)$. The Thurston--Bennequin number $\operatorname{tb}(\Lambda)$ is the writhe of the Lagrangian projection $\pi_{xy}(\Lambda)\subset\lR^2$ of $\Lambda$. It can also be seen as the linking number between $\Lambda$ and its push-off with respect to a normal framing everywhere tangent to the contact structure. The rotation number $\operatorname{rot}(\Lambda)$ of $\Lambda$ is equal to the rotation number of a parameterization of the Lagrangian projection of the knot in the $xy$-plane. 

The symplectization of a contact manifold $(M,\xi)$ with a contact form $\lambda$ is the symplectic manifold $(\lR\times M,d(e^r\lambda))$, where $r$ is the cylindrical $\lR$-coordinate. To a Legendrian knot $\Lambda\subset\lR^4$, we associate a Lagrangian cylinder $\lR\times\Lambda\subset\lR\times\lR^3$ in the symplectization. Following ideas coming from SFT \cite{intro}, we use pseudoholomorphic curves in $\lR\times\lR^3$ with boundary on $\lR\times\Lambda$ (with boundary punctures asymptotic to Reeb chords and of finite Hofer energy) to study the contact geometry of $\Lambda$ in $\lR^3$. The Reeb vector field on $\lR^3$ with respect to the standard contact form $\lambda=dz-ydx$ is given by $\partial_z$. A Reeb chord on $\Lambda$ is a map $\gamma:[0,l]\to\lR^3$ such that $\dot\gamma=\partial_z$ and $\gamma(0),\gamma(l)\in\Lambda$. 

The simplest Legendrian knot invariant coming from the SFT framework is the Chekanov--Eliashberg differential graded algebra (dga) \cite{Chekanov02}. The Chekanov--Eliashberg dga is freely generated by the Reeb chords on $\Lambda$. The differential is obtained by counting pseudoholomorphic disks in $\lR^4$ with boundary on $\lR\times\Lambda$, one positive and arbitrarily many negative boundary punctures asymptotic to Reeb chords. 

The Chekanov--Eliashberg dga was later extended to Legendrian rational SFT \cite{Ng_rLSFT} that includes pseudoholomorphic disks with arbitrarily many positive punctures. In addition to the usual vertical breaking into SFT buildings, top dimensional boundary strata of the moduli space of disks with more than one positive puncture contain curves with boundary nodes. In particular, for the 1-dimensional moduli space, we can have trivial strip bubbling which prevents the map defined by counting index zero pseudoholomorphic disks (with arbitrarily many positive punctures) from defining a differential. To overcome this, the differential includes a string topological operation that takes loop product with trivial strips over Reeb chords on $\Lambda$. More recently, this invariant was used to define an $L^\infty$ algebra structure associated to a Legendrian knot \cite{ng_linf}.

An important feature of these Legendrian knot invariants is that they are described combinatorially from the knot diagram and are easy to compute. This also leads to a combinatorial proof of invariance.

Here we take the next step in Euler characteristic and introduce an SFT invariant for Legendrian knots that also includes pseudoholomorphic annuli. More precisely, we include pseudoholomorphic annuli in $\lR^4$ with boundary on $\lR\times\Lambda$ and one positive boundary puncture. The boundary of the 1-dimensional moduli space of annuli contains nodal annuli in addition to the SFT buildings. To deal with the nodal breaking, we introduce what we call a corrected loop coproduct for strings. This approach is in part inspired by \cite{Fuk06} and \cite{CieLat}, but allows us to avoid taking the quotient by constant loops. The resulting invariant admits a combinatorial description and an associated invariance proof. The approach to boundary bubbling taken here should extend to more general curves and settings. We explore this in future work.

\subsection*{Main results}
To state the main results, we first introduce the notion of a second-order dga. Let $\Lambda$ be a Legendrian knot. Self-intersections of $\pi_{xy}(\Lambda)$ are in 1-1 correspondence with Reeb chords on $\Lambda$. Assume the Lagrangian projection $\pi_{xy}(\Lambda)$ is in general position and denote the Reeb chords by $\gamma_1,\dots,\gamma_n$. Fix an orientation on $\Lambda$ and a base point $T\in\Lambda$ different from the Reeb chord endpoints. For every $i\in\{1,\dots,n\}$, we introduce a variable $q_i$. Denote by $\widetilde\cA(\Lambda)=\widetilde\cA$ the tensor algebra over $\lQ$ generated by $t^+,t^-,q_i,i\in\{1,\dots,n\}$ with relation $t^+t^-=1=t^-t^+$. The space $\widetilde\cA$ can be viewed as the vector space of words in $q_i,t^\pm$. We have a grading on $\widetilde\cA$ given by $|q_i|=\mu_{CZ}(\gamma_i)$ and $|t^\pm|=\mp 2\operatorname{rot}(\Lambda)$, where $\mu_{CZ}(\gamma)$ is the Conley--Zehnder index of $\gamma_i$ \cite{ENS02}. Denote by $\widetilde\cA^{\cyc}(\Lambda)=\widetilde\cA^{\cyc}$ the corresponding vector space of cyclic words, i.e. the quotient space $\widetilde\cA/\cI$ for $\cI$ the vector subspace generated by 
$$\{vw-(-1)^{|v||w|}wv\,|\,v,w\in\widetilde\cA\text{ words}\}.$$ 
We consider the graded vector space 
$$\cA(\Lambda)=\cA=\widetilde\cA\oplus\hbar\, (\widetilde\cA\otimes\widetilde\cA^{\cyc}),$$
where $\hbar$ is a formal variable such that $|\hbar|=-1$. Elements in $\cA$ are denoted by $u+\hbar\, w$ for $u\in\widetilde\cA,w\in\widetilde\cA\otimes\widetilde\cA^{\cyc}$. The algebra structure on $\cA$ is given by
\begin{align*}
&w\cdot \hbar(v_1\otimes v_2)=(-1)^{|w|(|v_2|+1)}\hbar(wv_1\otimes v_2),\\
&\hbar(v_1\otimes v_2)\cdot w=\hbar(v_1 w\otimes v_2),\\
&\hbar(v_1\otimes v_2)\cdot\hbar(w_1\otimes w_2)=0,
\end{align*}
and concatenation for words in $\widetilde\cA$. 

\vspace{3.1mm}
We introduce the notion of \textit{second-order differential graded algebra} structure on $\cA$ (similar to quantum Batalin--Vilkovisky algebra). Consider the algebra structure on $\widetilde\cA\otimes\widetilde\cA$ given by 
\begin{align*}
(v_1\otimes v_2)\cdot(w_1\otimes w_2)=(-1)^{|v_1||w_2|}(v_1 w_1\otimes v_2 w_2).
\end{align*}
An \textit{antibracket} on $\widetilde\cA$ is a degree $0$ bilinear map $\{\cdot,\cdot\}:\widetilde\cA\times \widetilde\cA\to\widetilde\cA\otimes \widetilde\cA$ such that
\begin{align*}
&\{v,w_1w_2\}=\{v,w_1\}\cdot(w_2\otimes 1)+(-1)^{|v||w_1|}(1\otimes w_1)\cdot\{v,w_2\},\\
&\{v_1v_2,w\}=(v_1\otimes 1)\cdot\{v_2,w\}+(-1)^{|v_2||w|}\{v_1,w\}\cdot(1\otimes v_2),
\end{align*}
for all words $v,v_1,v_2,w,w_1,w_2$ in $\widetilde\cA$. An antibracket induces a degree $-1$ linear map $\{\cdot,\cdot\}_\hbar:\cA\otimes\cA\to\cA$ given by 
$$\{v,w\}_\hbar=\hbar\,\pi_{\cyc}\{\pi_{\widetilde\cA}v,\pi_{\widetilde\cA}w\},$$ 
where $\pi_{\cyc}:\widetilde\cA\otimes\widetilde\cA\to\widetilde\cA\otimes\widetilde\cA^{\cyc}$ is induced by the cyclic quotient.

We say a degree $-1$ linear map $d:\cA\to\cA$ is a \textit{second-order derivation} with respect to an antibracket $\{\cdot,\cdot\}$ on $\widetilde\cA$ if
\begin{align*}
&d(vw)=d(v)w+(-1)^{|v|}v d(w)+\{v,w\}_\hbar,\\
& d(\hbar(v\otimes w))=(-1)^{|w|+1}\hbar(d_0v\otimes w)-\hbar(v\otimes d_0^{\cyc}w),
\end{align*}
for all generators $v,w\in\cA$, where $d_0\coloneq\pi_{\widetilde\cA}\circ d\circ \iota_{\widetilde\cA}$ and $d_0^{\cyc}:\widetilde\cA^{\cyc}\to\widetilde\cA^{\cyc}$ is the linear map induced by $d_0$ on the cyclic quotient. Furthermore, we say $d:\cA\to\cA$ is a \textit{strong} second-order derivation with respect to $\{\cdot,\cdot\}$ if $d$ is additionally a derivation with respect to $\{\cdot,\cdot\}$, i.e. if
$$(d_0\otimes 1+1\otimes d_0)\{v,w\}=\{d_0v,w\}+(-1)^{|v|}\{v,d_0w\}\in\widetilde\cA\otimes\widetilde\cA$$
for all words $v,w\in\widetilde\cA$. Here we define $f\otimes g:\widetilde\cA\otimes\widetilde\cA\to
\widetilde\cA'\otimes\widetilde\cA'$
\begin{align*}
(f\otimes g)(v_1\otimes v_2)=(-1)^{|f||v_2|}f(v_1)\otimes g(v_2)
\end{align*}
for $f,g:\widetilde \cA\to\widetilde\cA'$ graded linear maps.
\begin{defi}
A \textit{second-order differential graded algebra} structure $(\cA,d,\{\cdot,\cdot\})$ on $\cA$ consists of an antibracket $\{\cdot,\cdot\}$ on $\widetilde\cA$ and a strong second-order derivation $d:\cA\to\cA$ with respect to $\{\cdot,\cdot\}$ such that $d^2=0$. 
\end{defi} 
Our main result is a construction of a second-order dga structure $(\cA(\Lambda),d_\Lambda,\{\cdot,\cdot\}_{d_\Lambda})$ on $\cA(\Lambda)$ for any Legendrian knot $\Lambda$, invariant under Legendrian knot isotopy up to stable tame equivalence (see Section \ref{Sec:stabilizationI} for the definition of stable tame equivalence). The differential counts index zero pseudoholomorphic annuli with one positive puncture, pseudoholomorphic disks with up to two positive punctures, and has a string topological component that cancels out the contribution of nodal curves.

\begin{thm}\label{Thm:invariance_intro}
Let $\Lambda_0$ and $\Lambda_1$ be front resolutions of two Legendrian isotopic knots, then the second-order dg algebras $(\cA(\Lambda_0),d_{\Lambda_0},\{\cdot,\cdot\}_{d_{\Lambda_0}}),(\cA(\Lambda_1),d_{\Lambda_1},\{\cdot,\cdot\}_{d_{\Lambda_1}})$ associated to $\Lambda_0,\Lambda_1$ are stable tame equivalent. In particular, their homology groups are isomorphic
\begin{align*}
H_*(\cA(\Lambda_0),d_{\Lambda_0})\cong H_*(\cA(\Lambda_1),d_{\Lambda_1}).
\end{align*}
\end{thm}

Our next results allow us to describe the invariant combinatorially from the Lagrangian projection $\pi_{xy}(\Lambda)$ of $\Lambda$. Let $J$ be the almost complex structure on $\lR^4$ given by 
\begin{align*}
\hspace{40mm}&J\partial_x=\partial_y+y\partial_r,&&J\partial_y=-\partial_x-y\partial_z,\hspace{40mm}\\
&J\partial_z=-\partial_r,&&J\partial_r=\partial_z.
\end{align*}
There is a well-known bijection between holomorphic disks in $\lC$ with boundary on $\pi_{xy}(\Lambda)\subset\lC$ and corners at the self-intersection of $\pi_{xy}(\Lambda)\subset\lC$ and $J$-holomorphic disks on $\lR\times\Lambda$. Denote by $\cM^\pi_2$ the moduli space of holomorphic annuli in $\lC$ with boundary on $\pi_{xy}(\Lambda)$ and one positive corner (arbitrarily many negative corners), and let $\cM_{2,k}^\pi,k\in\lN_0$ be its $k$-dimensional part. The compactification $\overline\cM^\pi_{2,1}$ of $\cM^\pi_{2,1}$ is a 1-dimensional manifold with boundary. Its boundary points can be of two types, which we call \textit{split} and \textit{non-split}. A non-split boundary point consists of an index zero holomorphic disk on $\pi_{xy}(\Lambda)$ with two distinguished corners, one positive and one negative, at some self-intersection of $\pi_{xy}(\Lambda)$. A split boundary point consists of a holomorphic annulus in $\cM_{2,0}^\pi$ (the \textit{annular part}) and an index zero holomorphic disk attached to it at some positive or negative corner. 

\begin{prop}\label{Prop:intro_obstruction}
There exists a smooth section $\Omega:\cM^\pi_2\to\lR$ such that an annulus $u_0\in\cM_{2}^\pi$ can be lifted to a $J$-holomorphic annulus in $\lR^4$ with boundary on $\lR\times\Lambda$ if and only if $\Omega(u_0)=0$. Furthermore, there is an extension of $\Omega|_{\cM^\pi_{2,1}}$ to a continuous map $\Omega:\overline\cM^\pi_{2,1}\to\lR\cup\{+\infty,-\infty\}$ such that
\begin{itemize}
\item for $u$ a non-split boundary point, $\Omega(u)=\pm\infty$ (the distinction between $+\infty$ and $-\infty$ is described in Section \ref{Sec:Elliptic_hyperbolic}),
\item for $u$ a split boundary point, $\Omega(u)=\Omega(u_0)$, where $u_0\in\cM^\pi_{2,0}$ is the annular part of $u$.
\end{itemize}
\end{prop}
The map $\Omega$ from Proposition \ref{Prop:intro_obstruction} is constructed in Section \ref{Sec:obstruction_section} and is referred to as the \textit{obstruction section}. The count of zeros of $\Omega:\overline\cM^\pi_{2,1}\to\lR$, and therefore the count of index zero $J$-holomorphic annuli on $\lR\times\Lambda$, is uniquely determined by the values of $\Omega$ at the boundary whenever $\Omega\pitchfork 0$ and $\Omega|_{\partial\overline\cM_{2,1}^\pi}\subset[-\infty,0)\cup(0,+\infty]$ (which holds for $\Lambda$ generic). For a knot $\Lambda$ with split boundary points in $\partial\overline\cM^\pi_{2,1}$, calculating the values of $\Omega|_{\partial\overline\cM^\pi_{2,1}}$ is not easy. This can be avoided by introducing an object that we call a combinatorial obstruction section, whose zeros can, by our next result, be used in place of the $J$-holomorphic annuli to compute the invariant.

\begin{defi}
A smooth map $\Omega^{\vir}:\cM^\pi_{2,0}\sqcup\overline\cM^\pi_{2,1}\to\lR\cup\{+\infty,-\infty\}$ is called a \textit{combinatorial obstruction section} if it satisfies the following properties
\begin{itemize}
\item for every non-split boundary point $u\in\partial\overline\cM_{2,1}^\pi$ we have $\Omega^{\vir}(u)=\Omega(u)$, 
\item for every split boundary point $u\in\partial\overline\cM_{2,1}^\pi$ we have $\Omega^{\vir}(u)=\Omega^{\vir}(u_0)$, where $u_0\in\cM_{2,0}^\pi$ is the annular part of $u$,
\item $\Omega(\cM_{2,0}^\pi)\subset\lR\backslash\{0\}$,
\item $\Omega\pitchfork 0$.
\end{itemize}
\end{defi}
In the definition of the second-order dga $(\cA(\Lambda),d_{\Lambda},\{\cdot,\cdot\}_{d_\Lambda})$, instead of using the count of index zero $J$-holomorphic annuli on $\lR\times\Lambda$, we can use the count of zeros of $\Omega^{\vir}:\overline\cM^\pi_{2,1}\to\lR\cup\{+\infty,-\infty\}$ for any combinatorial obstruction section $\Omega^{\vir}$. We denote this differential by $d_{\Lambda,\Omega^{\vir}}$.

\begin{prop}\label{Prop:intro_combinatorial_count}
For any combinatorial obstruction section $\Omega^{\vir}:\cM^\pi_{2,0}\sqcup\overline\cM^\pi_{2,1}\to\lR\cup\{+\infty,-\infty\}$ on $\Lambda$, the second-order dg algebras $(\cA(\Lambda),d_{\Lambda,\Omega^{\vir}},\{\cdot,\cdot\}_{d_\Lambda})$ and $(\cA(\Lambda),d_{\Lambda},\{\cdot,\cdot\}_{d_\Lambda})$ are isomorphic.
\end{prop}
This makes it possible to compute the second-order dga structure on $\cA(\Lambda)$ combinatorially from the Lagrangian knot projection.

\vspace{2.1mm}
The Chekanov--Eliashberg dg algebra has a natural interpretation in terms of Legendrian surgery, it is isomorphic to the wrapped Floer homology of the co-core disk after Lagrangian handle attachment \cite{bee}. The extended invariant studied here can also be understood from this perspective. We expect that the second-order dg algebra is related to the coproduct on the linearized contact homology/wrapped Floer homology after Legendrian surgery along $\Lambda$. More precisely, let $Y$ denote the contact manifold after the surgery and $\Gamma\subset Y$ the Legendrian boundary of the handle co-core. Consider the vector space $C^{\chord}_{Y,\Gamma}=\lQ\langle\cR(\Gamma)\rangle$ generated by the set of Reeb chords $\cR(\Gamma)$ on $\Gamma$, the vector space $C^{\orb}_{Y,\Gamma}=\lQ\langle\cR(Y)\rangle$ generated by the set of Reeb orbits $\cR(Y)$ on $Y$ and the complex $C=C_{Y,\Gamma}^{\chord}\oplus C_{Y,\Gamma}^{\orb}$. The differential on this complex counts pseudoholomorphic cylinders, strips, and disks with one negative interior and one positive boundary puncture. The complex has a coproduct that consists of a part $C_{Y,\Gamma}^{\orb}\to C_{Y,\Gamma}^{\orb}\otimes C_{Y,\Gamma}^{\orb}$ that counts pairs of pants, a part $C_{Y,\Gamma}^{\chord}\to C_{Y,\Gamma}^{\chord}\otimes  C_{Y,\Gamma}^{\chord}$ that counts three-punctured disks and a part $C_{Y,\Gamma}^{\chord}\to C_{Y,\Gamma}^{\chord}\otimes C_{Y,\Gamma}^{\orb}$ that counts strips with a negative interior puncture. The Legendrian surgery description of the component $C_{Y,\Gamma}^{\chord}\to C_{Y,\Gamma}^{\chord}\otimes C_{Y,\Gamma}^{\chord}$ corresponds to a point deformation of the dg algebra \cite{bee}. The part $C_{Y,\Gamma}^{\chord}\to C_{Y,\Gamma}^{\chord}\otimes C_{Y,\Gamma}^{\orb}$ corresponds to the annulus part of the invariant defined here, while the part $C_{Y,\Gamma}^{\orb}\to C_{Y,\Gamma}^{\orb}\otimes C_{Y,\Gamma}^{\orb}$ corresponds to the cyclic version of the invariant. The model before surgery is given as follows. Let $(\cA,d,\{\cdot,\cdot\}_d)$ be the second-order dga associated to a Legendrian knot $\Lambda$. Consider the quadratic complex $C=C^{\chord}\oplus(C^{\chord}\otimes C^{\orb})$, where $C^{\chord}=\widetilde\cA,C^{\orb}=\widetilde\cA^{\cyc}$. We consider the map $d_\lD:C\to C$ given by the part of the differential $d$ that comes from the usual Chekanov--Eliashberg differential
$$\pi_{\widetilde\cA}\circ d\circ \pi_{\widetilde\cA}+\pi_{\hbar\,(\widetilde\cA\otimes\widetilde\cA^{\cyc})}\circ d\circ \pi_{\hbar\,(\widetilde\cA\otimes\widetilde\cA^{\cyc})},$$ 
and the map $d_\hbar:C\to C$ given by the diagonal term
$$\pi_{\hbar\,(\widetilde\cA\otimes\widetilde\cA^{\cyc})}\circ d\circ \pi_{\widetilde\cA}.$$ 
More precisely, for $d_{CE}:C^{\chord}\to C^{\chord}$ the Chekanov--Eliashberg differential, we have
\begin{align*}
&d_\lD(x)=d_{CE}(x),\\
&d_{\lD}(x\otimes y)=(-1)^{|y|+1} d_{CE}x\otimes y-x\otimes d_{CE}^{\cyc} y,\\
&d_\hbar(x)=d(x)-d_\lD(x),
\end{align*}
for $x\in C^{\chord},y\in C^{\orb}$ words. Then 
$$d_\lD\circ d_\lD=0$$
and
\begin{align}
\label{Eq:diff_spectr}
d_\hbar\circ d_\lD+d_\lD\circ d_\hbar=d\circ d=0.
\end{align}
We denote by
$$H^{CE}_*(\Lambda)\coloneq H^{CE}_*\coloneq H_*(C,d_\lD)$$ 
the first page of the corresponding spectral sequence. We define
\begin{align*}
& D:H_*^{CE}\to H_*^{CE},\\
& D[x]=[d_\hbar x],\\
& D[z]=0,
\end{align*}
for $x\in C^{\chord},z\in C^{\chord}\otimes C^{\orb}$ such that $d_\lD(x)=0,d_{\lD}(z)=0$. From (\ref{Eq:diff_spectr}) we conclude that $D$ is well defined. Moreover, $D\circ D=0$ by definition. The second (and the final) page of the spectral sequence
\begin{equation}
\label{Eq:P2}
\begin{aligned}
H^\hbar_*(\Lambda)\coloneq H^\hbar_*\coloneq H_*(H^{CE}_*,D)
\end{aligned}
\end{equation}
is isomorphic to $H_*(\cA,d)$ and is an invariant of $\Lambda$ up to Legendrian knot isotopy.

\subsection*{Organization of the paper}
In Section \ref{Sec:moduli_spaces_all} we introduce the moduli spaces of curves used to define the invariant. Theorem \ref{Prop:intro_obstruction} is proven in Section \ref{Sec:counting_curves}. We define coherent orientations on our moduli spaces and discuss the signs in Section \ref{Section:Orientations}. In Section \ref{Sec:Invariant_definition}, we introduce the space of strings and string pairs on $\Lambda$ and give the first definition of the chain complex without introducing any algebraic structure. Corrected loop coproduct is introduced in Section \ref{Sec:loop_coproduct}. In Section \ref{Section:Algebraic_definition}, we introduce the algebraic structure and give the second definition of the invariant. More precisely, we define a second-order dga structure on $\cA(\Lambda)$. This definition is more suitable for computations. Theorem \ref{Thm:invariance_intro} is proven in Section \ref{Sec:Invariance}. Proposition \ref{Prop:intro_combinatorial_count} follows from Section \ref{Section:Invar_IV_degeneration}, where we prove invariance under an isotopy which passes through a Legendrian knot with an index $-1$ $J$-holomorphic annulus obtained by lifting an annulus in $\cM_{2,0}^\pi$. In Section \ref{Sec:Augmentations} we introduce the notion of a second-order augmentation and $\hbar$-linearization, and discuss how a second-order augmentation is obtained from a Lagrangian filling. In Section \ref{Sec:Examples} we compute some simple examples.

\subsection*{Acknowledgments}
The author is grateful to Tobias Ekholm for many valuable comments and discussions.

\section{Pseudoholomorphic disks and annuli in $\lR^4$}\label{Sec:moduli_spaces_all}

We introduce the moduli spaces of disks and annuli that are used to define the Legendrian knot invariant in Section \ref{Sec:Moduli_spaces} and give a combinatorial way to count them in Section \ref{Sec:counting_curves}. We discuss the generic asymptotic behavior of pseudoholomorphic curves in Section \ref{Sec:generic_asymptotic_behavior_admissible_disks}, which comes into play when we define the loop product and the corrected loop coproduct later on. Additionally, we define coherent orientations on the moduli spaces of disks and annuli in Section \ref{Section:Orientations}.

Let $\Lambda\subset\lR^3$ be a generic Legendrian knot, $\{\gamma_1,\dots,\gamma_n\}$  the set of Reeb chords on $\Lambda$ and $i^-,i^+\in\Lambda$ the starting point and the endpoint of $\gamma_i$. Fix additionally an orientation on $\Lambda$ and a base point $T\in\Lambda$ different from all Reeb chord endpoints. For $\beta:[a,b]\to\lR^2$ an immersed path, let $\operatorname{rot}(\beta)\in\lR$ be the rotation number of the unit tangent vector $\dot\beta/\|\dot\beta\|$ along $\beta$. Denote by $\beta_i:[0,1]\to\lR^2$ the (unique up to reparameterization) immersed path on $\pi_{xy}(\Lambda)$ starting at the overcrossing arc at $i$ and ending at the undercrossing arc that does not pass through $\pi_{xy}(T)$. Then the Conley--Zehnder index of $\gamma_i$ is defined as
\begin{align*}
\mu_{CZ}(\gamma_i)=\lfloor 2\operatorname{rot}(\beta_i)\rfloor,
\end{align*}
where $\lfloor\cdot\rfloor$ is the floor function. Denote by $L=\lR\times\Lambda$ the Lagrangian cylinder corresponding to $\Lambda$ in the symplectization $\lR\times\lR^3$. 

\subsection{Moduli spaces of pseudoholomorphic curves} \label{Sec:Moduli_spaces}
In this section, we introduce moduli spaces of disks and annuli that are used later to define the Legendrian knot invariant. Let $J$ be the almost complex structure on $\lR\times\lR^3$ given by 
\begin{equation}\label{Eq:intro_J_formula}
\begin{aligned}
&J\partial_x=\partial_y+y\partial_r,\,\,\,&&J\partial_y=-\partial_x-y\partial_z,\\
&J\partial_z=-\partial_r,\,\,\,&&J\partial_r=\partial_z,
\end{aligned}
\end{equation}
for $(r,x,y,z)\in\lR\times\lR^3$. This almost complex structure is compatible with the symplectic structure in the SFT sense (as defined in \cite{intro}). Denote the punctured Riemann surface by $\mathring\Sigma=\Sigma\backslash\{t_1,\dots,t_k\}$, for $(\Sigma,j)$ a Riemann surface with boundary, complex structure $j$ and $k$ distinct points $t_1,\dots,t_k\in\partial\Sigma$. A smooth map $u:\mathring\Sigma\to\lR^4$ is $J$-holomorphic (or pseudoholomorphic) if $J\circ du=du\circ j$. We are interested in pseudoholomorphic maps $u$ whose boundary is mapped to $L$ and that have finite Hofer energy, i.e., such that $u$ is positively or negatively asymptotic to some Reeb chord on $\Lambda$ at each puncture $t_i$. More precisely, we say $u$ is positively asymptotic to a Reeb chord $\gamma:[0,l]\to\lR^3$ at $t_i$ if for a holomorphic parameterization $\varphi:[0,+\infty)\times[0,1]\to\mathring\Sigma$ of a neighborhood of $t_i$ in $\mathring\Sigma$ we have 
\begin{align*}
&\lim_{s\to+\infty}\pi_{xyz}\circ u\circ\varphi(s,t)=\gamma(lt),\\
&\lim_{s\to+\infty}\pi_{r}\circ u\circ\varphi(s,t)=+\infty,
\end{align*}
where $\pi_{xyz}:\lR\times\lR^3\to\lR^3,\pi_r:\lR\times\lR^3\to\lR$ are projections (see also for example \cite{ENS02}). Similarly, we say $u$ is negatively asymptotic to $\gamma$ at $t_i$ if 
\begin{align*}
&\lim_{s\to+\infty}\pi_{xyz}\circ u\circ\varphi(s,t)=\gamma(l(1-t)),\\
&\lim_{s\to+\infty}\pi_{r}\circ u\circ\varphi(s,t)=-\infty.
\end{align*}
 
Let $\bm{\gamma}=(\gamma_{i_1},\dots,\gamma_{i_k})$ be a tuple of Reeb chords on $\Lambda$ together with signatures $\epsilon_i\in\{-1,1\},i=1,\dots,k$ and $a=(a_1,\dots,a_k)\in\lZ^k$. We denote by $\cM_1(\bm{\gamma},a)$ the moduli space of equivalence classes of pseudoholomorphic disks $u:(\lD\backslash\{t_1,\dots,t_{k}\},j)\to(\lR^4,J)$ with boundary puncture $t_j$ asymptotic to the Reeb chord $\gamma_{i_j}$ (positively if $\epsilon_j=1$ and negatively if $\epsilon_j=-1$) and boundary mapped to $\lR\times\Lambda$, such that $\pi_{xyz}\circ u|_{(t_j,t_{j+1})}$ passes through the base point $T$ $a_j$ times (counted with signs) if $\pi_{xyz}\circ u$ is transverse to $T$ (which holds for generic such curve). We say that two curves are equivalent if one can be obtained from the other by taking a holomorphic reparameterization of the domain preserving the data (punctures are preserved up to cyclic ordering) and by $\lR$-translation in the cylindrical direction.

Similarly, for $\bm{\gamma}=(\gamma_{i_1},\dots,\gamma_{i_{k_1}}),\bm{\beta}=(\beta_{i_1},\dots,\beta_{i_{k_2}}),k_1,k_2\in\lN$ tuples of Reeb chords together with signatures $\epsilon^1_i,\epsilon^2_j$ and $a=(a_1,\dots,a_{k_1})\in\lZ^{k_1},b=(b_1,\dots,b_{k_2})\in\lZ^{k_2}$, we denote by $\cM_2(\bm{\gamma},\bm{\beta},a,b)$ the moduli space of equivalence classes of pseudoholomorphic annuli $u:(\mathring\Sigma,j)\to(\lR^4,J),\Sigma=\{z\in\lC\,|\,1\leq\|z\|\leq r\}$ for some $1<r$, with boundary mapped to $\lR\times\Lambda$ and punctures $t_j$ and $t_l'$ on the two boundary components asymptotic to the corresponding Reeb chords in $\bm{\gamma},\bm{\beta}$ as above, such that $\pi_{xyz}\circ u|_{(t_j,t_{j+1})}$ ($\pi_{xyz}\circ u|_{(t_l',t_{l+1}')}$) passes through $T$ $a_j$ ($b_l$) times if the appropriate transversality condition is satisfied. We additionally allow $k_2=0$ and $\bm{\beta}=\emptyset$, in this case we take $b=b_0\in\lZ$. This determines the homology class of the corresponding boundary component on $\Lambda$. We say two such curves are equivalent if one can be obtained from the other by taking a holomorphic reparameterization of the domain preserving the data and by $\lR$-translation in the cylindrical direction.

The index of the moduli space $\cM_1(\bm{\gamma},a)$ is defined as
\begin{align*}
\operatorname{ind}(\bm{\gamma},a)\coloneqq k_++\mu_L(a)+\sum_{i=1}^{k}\epsilon_i\mu_{CZ}(\gamma_{i})-2,
\end{align*}
where $k_+$ is the number of positive punctures,  $\mu_{CZ}$ is the Conley--Zehnder index, and 
$$\mu_L(a)=2\sum a_i\operatorname{rot}(\Lambda)$$ 
is the Maslov number. Similarly, the index of the moduli space $\cM_2(\bm{\gamma},\bm{\beta},a,b)$ is defined as 
\begin{align*}\operatorname{ind}(\bm{\gamma},\bm{\beta},a,b)\coloneqq k_+^1+k_+^2+\mu_L(a,b)+\sum_{i=1}^{k_1}\epsilon_{\gamma,i}\mu_{CZ}(\gamma_{i})+\sum_{i=1}^{k_2}\epsilon_{\beta,i}\mu_{CZ}(\beta_{i})-1,
\end{align*}
where $k_+^1$ and $k_+^2$ are the numbers of positive punctures on two boundary components and 
$$\mu_L(a,b)=2\left(\sum a_i+\sum b_j\right)\operatorname{rot}(\Lambda).$$ 
For $u$ a pseudoholomorphic disk or an annulus, we denote by $\operatorname{ind}(u)$ the index of the corresponding moduli space. We say that a pseudoholomorphic curve $u$ is of index zero if $\operatorname{ind}(u)=0$.

Next, we describe the Gromov compactification of these moduli spaces. First, we need to introduce nodal pseudoholomorphic curves and SFT buildings.

A closed nodal Riemann surface is a union of closed Riemann surfaces $(\Sigma_i,j_i),i\in\{1,\dots,m\}$ together with finitely many distinct points in $\bigsqcup\Sigma_i$ subdivided into pairs $(z_j^1,z_j^2),j\in\cI$, called nodal pairs. An automorphism of a nodal Riemann surface is a biholomorphic map $\phi:\bigsqcup\Sigma_i\to\bigsqcup\Sigma_i$ that preserves the set of nodal pairs. A nodal Riemann surface is stable if its automorphism group is discrete. Stable nodal Riemann surfaces appear in the compactification of the moduli space of stable Riemann surfaces. Similar can be done for Riemann surfaces with boundary. In this case, points in a nodal pair are either both on the boundary or in the interior, and we distinguish boundary (hyperbolic) and interior nodal pairs. In addition to that, we have finitely many distinguished interior marked points $z^e_i,i\in\cI'$, which are called elliptic nodes, and appear when a boundary component shrinks to a point. 

For $(M,\omega)$ a symplectic manifold and $L\subset M$ a Lagrangian submanifold, a nodal pseudoholomorphic curve on $L$ consists of a nodal Riemann surface $\left(\{(\Sigma_i,j_i)\}_i,\{(z_j^1,z_j^2)\}_j,\{z_k^e\}_k\right)$ and a $J$-holomorphic map $u:\bigsqcup\Sigma_i\to M$ with boundary mapped to $L$, such that $u(z_j^1)=u(z_j^2)$ for all nodal pairs and $u(z^e_k)\in L$ for all elliptic nodes. Nodal pseudoholomorphic maps appear in the compactification of the moduli space of pseudoholomorphic curves with boundary on $L$ in a compact symplectic manifold $M$. When working with symplectizations of contact manifolds, we additionally need to consider breaking into pseudoholomorphic SFT buildings. For us, it is enough to define pseudoholomorphic 2-buildings in $\lR^4$ with no nodes and with one or two gluing pairs. Let $u_1,u_2$ be either two $J$-holomorphic disks or a disk and an annulus. Assume $u_1$ has a negative boundary puncture $t$ asymptotic to a Reeb chord $\gamma$ and $u_2$ a positive boundary puncture $t'$ asymptotic to $\gamma$. Then $(u_1,u_2)$, together with $(t,t')$, forms a pseudoholomorphic 2-building with one gluing pair. Similarly, assume $u_1,u_2$ are pseudoholomorphic disks such that $u_1$ has negative boundary punctures $t_1,t_2$ asymptotic to Reeb chords $\gamma,\gamma'$ and $u_2$ positive boundary punctures $t_1',t_2'$ asymptotic to $\gamma,\gamma'$. Then $(u_1,u_2)$, together with puncture pairs $(t_1,t_1'),(t_2,t_2')$, forms a pseudoholomorphic 2-building with two gluing pairs. The topological class of a 2-building (a disk or an annulus) and the cyclic ordering of the punctures are determined after gluing (topologically) at the gluing pair/pairs.

The following propositions are the main ingredient behind the definition of the algebraic invariant we introduce.

\begin{prop}\label{Prop:moduliI}
For a generic Legendrian knot $\Lambda$ and a Reeb chord tuple $\bm{\gamma}=(\gamma_1,\dots,\gamma_k)$ together with signatures $\epsilon_i$ and $a\in\lZ^k$ such that $\operatorname{ind}(\bm{\gamma},a)=0$, the moduli space $\cM_1(\bm{\gamma},a)$ is a compact manifold of dimension 0. 
\end{prop}

\begin{prop}\label{Prop:moduliII}
For a generic Legendrian knot $\Lambda$ and Reeb chord tuples $\bm{\gamma}=(\gamma_1,\dots,\gamma_{k_1}),\bm{\beta}=(\beta_1,\dots,\beta_{k_2})$ together with signatures $\epsilon^1,\epsilon^2$ and $a\in\lZ^{k_1}, b\in\lZ^{k_2}$ such that $\operatorname{ind}(\bm{\gamma},\bm{\beta},a,b)=0$, the moduli space $\cM_2(\bm{\gamma},\bm{\beta},a,b)$ is a compact manifold of dimension 0. 
\end{prop}
\begin{prop}\label{Prop:compactness_disk}
For a generic Legendrian knot $\Lambda$ and a Reeb chord tuple $\bm{\gamma}=(\gamma_1,\dots,\gamma_k)$ together with signatures $\epsilon_i$ and $a\in\lZ^k$ such that $\operatorname{ind}(\bm{\gamma},a)=1$, the moduli space $\cM_1(\bm{\gamma},a)$ is a 1-dimensional manifold.  It has a natural compactification $\overline\cM_1(\bm{\gamma},a)$, which is a compact 1-dimensional manifold with boundary, obtained by adding the following boundary points (see Figure \ref{Fig:degeneration_disks})

\begin{itemize}
\item pseudoholomorphic two buildings $(u_1,u_2)$ with one gluing pair $(z_1,z_2)$ consisting of index zero pseudoholomorphic disks $u_1,u_2$, such that the order of the Reeb chords at the boundary punctures and the number of crossings of the arcs over the base point after gluing is equivalent to $\bm{\gamma},a$;
\item nodal curves with one (hyperbolic) boundary node consisting of an index zero pseudoholomorphic disk $u$ and a trivial strip over some Reeb chord, such that the order of the Reeb chords at the boundary punctures and the number of crossings of the arcs over the base point after resolving the node is equivalent to $\bm{\gamma},a$.
\end{itemize}
\end{prop}

\begin{figure}
\def\svgwidth{111mm}
\begingroup%
  \makeatletter%
  \providecommand\rotatebox[2]{#2}%
  \newcommand*\fsize{\dimexpr\f@size pt\relax}%
  \newcommand*\lineheight[1]{\fontsize{\fsize}{#1\fsize}\selectfont}%
  \ifx\svgwidth\undefined%
    \setlength{\unitlength}{573.2132966bp}%
    \ifx\svgscale\undefined%
      \relax%
    \else%
      \setlength{\unitlength}{\unitlength * \real{\svgscale}}%
    \fi%
  \else%
    \setlength{\unitlength}{\svgwidth}%
  \fi%
  \global\let\svgwidth\undefined%
  \global\let\svgscale\undefined%
  \makeatother%
  \begin{picture}(1,0.3315578)%
    \lineheight{1}%
    \setlength\tabcolsep{0pt}%
    \put(0,0){\includegraphics[width=\unitlength,page=1]{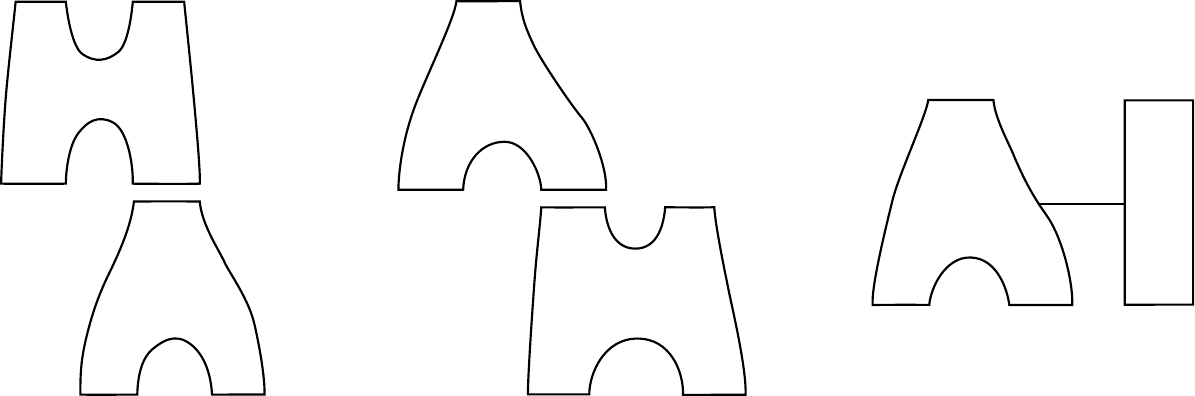}}%
  \end{picture}%
\endgroup%
\caption{Breaking for index one disks with two positive punctures.}
\label{Fig:degeneration_disks}
\end{figure}
\begin{prop}\label{Prop:compactness_annulus}
For a generic Legendrian knot $\Lambda$ and Reeb chord tuples $\bm{\gamma}=(\gamma_1,\dots,\gamma_{k_1}),\bm{\beta}=(\beta_1,\dots,\beta_{k_2})$ together with signatures $\epsilon^1,\epsilon^2$ with precisely one positive puncture and $a\in\lZ^{k_1},b\in\lZ^{k_2}$ such that $\operatorname{ind}(\bm{\gamma},\bm{\beta},a,b)=1$, the moduli space $\cM_2(\bm{\gamma},\bm{\beta},a,b)$ is a 1-dimensional manifold. It has a natural compactification $\overline\cM_2(\bm{\gamma},\bm{\beta},a,b)$, which is a compact 1-dimensional manifold with boundary, obtained by adding the following boundary points (see Figure \ref{Fig:degeneration_annuli})
\begin{itemize}
\item pseudoholomorphic 2-buildings $(u_1,u_2)$ with one gluing pair, where one of $u_1,u_2$ is an index zero disk and the other an index zero annulus, such that the order of the Reeb chords at the punctures and the crossings over the base point on the two boundary components after gluing is equivalent to $\bm{\gamma},\bm{\beta}$; 
\item pseudoholomorphic 2-buildings $(u_1,u_2)$ with two gluing pairs, where $u_1,u_2$ are pseudoholomorphic disks of index zero, such that the order of the Reeb chords at the punctures and the crossings over the base point on the two boundary components after gluing is equivalent to $\bm{\gamma},\bm{\beta},a,b$;
\item nodal curves with one (hyperbolic) boundary node consisting of an index zero pseudoholomorphic disk $u$ with one positive puncture and a nodal pair coming from a boundary self-intersection of $u$, such that the order of the Reeb chords at the punctures and the crossings over the base point on the two boundary components after resolving the node is equivalent to $\bm{\gamma},\bm{\beta},a,b$;
\item if $\beta=\emptyset$ and $b_0=0$, nodal curves consisting of an index zero pseudoholomorphic disk $u\in\cM(\bm{\gamma},a)$ together with one elliptic node coming from an interior intersection of $u$ with $\lR\times\Lambda$.
\end{itemize}
\begin{figure}
\def\svgwidth{139mm}
\begingroup%
  \makeatletter%
  \providecommand\rotatebox[2]{#2}%
  \newcommand*\fsize{\dimexpr\f@size pt\relax}%
  \newcommand*\lineheight[1]{\fontsize{\fsize}{#1\fsize}\selectfont}%
  \ifx\svgwidth\undefined%
    \setlength{\unitlength}{832.22965561bp}%
    \ifx\svgscale\undefined%
      \relax%
    \else%
      \setlength{\unitlength}{\unitlength * \real{\svgscale}}%
    \fi%
  \else%
    \setlength{\unitlength}{\svgwidth}%
  \fi%
  \global\let\svgwidth\undefined%
  \global\let\svgscale\undefined%
  \makeatother%
  \begin{picture}(1,0.26429562)%
    \lineheight{1}%
    \setlength\tabcolsep{0pt}%
    \put(0,0){\includegraphics[width=\unitlength,page=1]{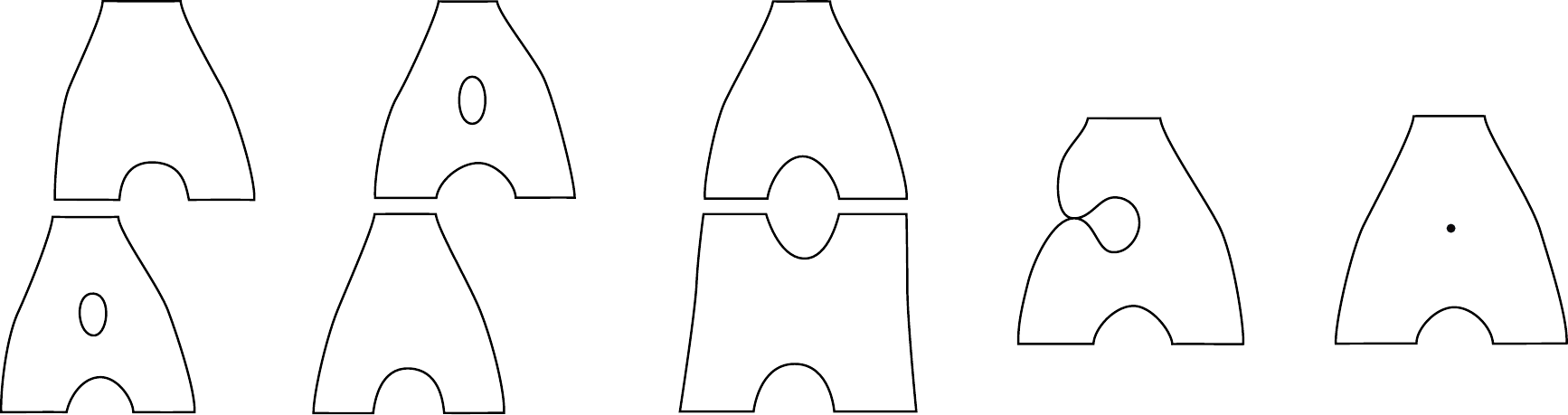}}%
  \end{picture}%
\endgroup%
\caption{Breaking for index one annuli with one positive puncture.}
\label{Fig:degeneration_annuli}
\end{figure}
\end{prop}

The proof of Gromov compactness in the closed case (Riemann surfaces without boundary) can be found in \cite{SFTcomp}, the relative case (Riemann surfaces with boundary) goes similarly and has been treated in different settings, see for example \cite{EES05I,CELN17} for the proof of compactness and gluing in the case of pseudoholomorphic disks. The case of pseudoholomorphic annuli goes analogously \cite{Abbas_book}. We prove regularity for $J$-holomorphic curves for generic knot $\Lambda$ in Section \ref{Sec:hamiltonian_def}. For any Legendrian knot $\Lambda$, the propositions above also hold for any generic compatible almost complex structure on $\lR^4$.

\subsection{Counting pseudoholomorphic curves}
\label{Sec:counting_curves}

Let $\Lambda\subset\lR^3$ be a Legendrian knot and $J$ be the almost complex structure on $\lR\times\lR^3$ given by (\ref{Eq:intro_J_formula}). The main goal of this section is to describe $J$-holomorphic disks and annuli on $\lR\times\Lambda$ using holomorphic curves in $\lC$ with boundary on $\pi_{xy}(\Lambda)$. For pseudoholomorphic annuli, we define an obstruction section $\Omega:\cM^\pi_{2,1}\to\lR$, where $\cM^\pi_{2,1}$ is the 1-dimensional moduli space of holomorphic annuli on $\pi_{xy}(\Lambda)$, such that $u_0\in\cM_{2,1}^\pi$ lifts to a $J$-holomorphic annulus on $\lR\times\Lambda$ if and only if $\Omega(u_0)=0$. Furthermore, we define the notion of a combinatorial obstruction section over $\cM^\pi_{2,1}$, whose zeros can be used in place of $J$-holomorphic annuli when defining the invariant, and can be counted purely combinatorially.

Denote by $\cM^\pi(\Sigma)$ the moduli space of holomorphic maps $u:\mathring\Sigma\to\lC$ with boundary on $\pi_{xy}(\Lambda)$ and corners at the boundary punctures at the self-intersections of $\pi_{xy}(\Lambda)$. We define an obstruction section $\Omega:\cM^\pi(\Sigma)\to\lR^{b_1(\Sigma)}$, such that its zero set is homeomorphic to the moduli space of $J$-holomorphic curves. 

We say a holomorphic map $u:(R,\infty)\times[0,1]\to\lC$ with boundary mapped to $\pi_{xy}(\Lambda)$ has a positive (negative) corner at a self-intersection $\gamma$ of $\pi_{xy}(\Lambda)$ if $\lim_{s\to\infty} u(s,t)=\gamma$, $(R,\infty)\times\{0\}$ is mapped to the undercrossing (overcrossing) arc at $\gamma$ and $(R,\infty)\times\{1\}$ is mapped to the overcrossing (undercrossing) arc. A holomorphic map $u:\mathring\Sigma=\Sigma\backslash\{t_1,\dots,t_k\}\to\lC$ with boundary on $\pi_{xy}(\Lambda)$ has a positive (negative) corner at $t_i$ if the above holds for $u\circ\phi_i$, where $\phi_i:(R,\infty)\times[0,1]\to\mathring\Sigma$ is a holomorphic parameterization of a neighborhood of $t_i$. 

Let $\cM^\pi_1(\Lambda)=\cM^\pi_1$ denote the moduli space of holomorphic disks on $\pi_{xy}(\Lambda)$ with up to two positive corners and $\cM^\pi_2(\Lambda)=\cM^\pi_2$ the moduli space of holomorphic annuli on $\pi_{xy}(\Lambda)$ with one positive corner. Additionally, denote by $\cM^\pi_{2,k}\subset\cM^\pi_2,k\in\lN_0$ its $k$-dimensional component. In particular, $\cM^\pi_{2,1}$ is the union of the connected components that contain holomorphic annuli with precisely one boundary branch point or one non-convex corner, see Figure \ref{Fig:projection_moduli_space}. A holomorphic annulus in a 0-dimensional connected component $\cM^\pi_{2,0}$, i.e. an annulus with no branch points (see Figure \ref{Fig:rigid_annuli}, left), can generically not be lifted to $(\lR^4,\lR\times\Lambda)$. An index zero $J$-holomorphic annulus on $\lR\times\Lambda$ projected to the $xy$-plane is a holomorphic annulus in $\cM^\pi_{2,1}$ and generically has one boundary branch point (see Figure \ref{Fig:rigid_annuli}, right).
\begin{figure}
\def\svgwidth{110mm}
\begingroup%
  \makeatletter%
  \providecommand\rotatebox[2]{#2}%
  \newcommand*\fsize{\dimexpr\f@size pt\relax}%
  \newcommand*\lineheight[1]{\fontsize{\fsize}{#1\fsize}\selectfont}%
  \ifx\svgwidth\undefined%
    \setlength{\unitlength}{594.2601554bp}%
    \ifx\svgscale\undefined%
      \relax%
    \else%
      \setlength{\unitlength}{\unitlength * \real{\svgscale}}%
    \fi%
  \else%
    \setlength{\unitlength}{\svgwidth}%
  \fi%
  \global\let\svgwidth\undefined%
  \global\let\svgscale\undefined%
  \makeatother%
  \begin{picture}(1,0.30145936)%
    \lineheight{1}%
    \setlength\tabcolsep{0pt}%
    \put(0,0){\includegraphics[width=\unitlength,page=1]{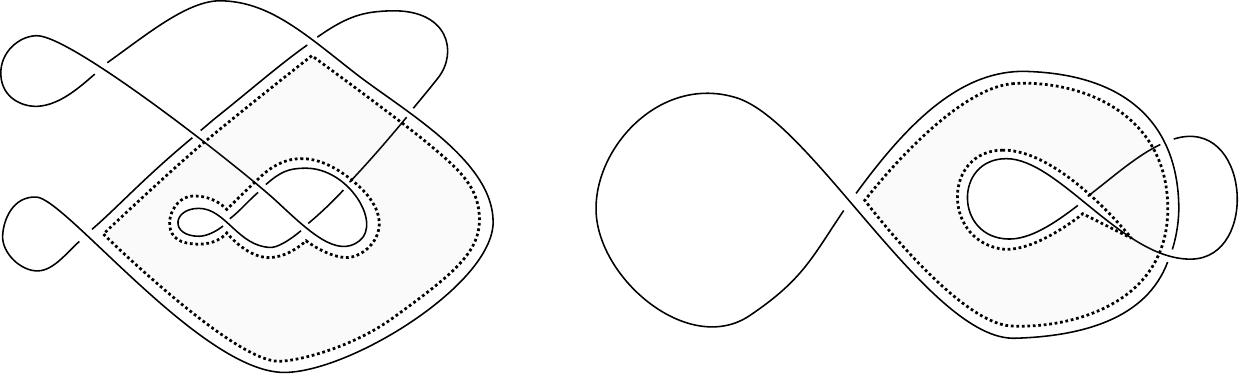}}%
  \end{picture}%
\endgroup%
\caption{Rigid annulus in the projection (left) and a projection of an index zero annulus on $\lR\times\Lambda$ (right).}
\label{Fig:rigid_annuli}
\end{figure}

Let $\widetilde u=(\theta,\overline u):\Sigma\backslash\{t_1,\dots,t_k\}\to\lR\times\lR^3$ be a $J$-holomorphic curve, where $J$ is the almost complex structure given by (\ref{Eq:intro_J_formula}), and $u=\pi_x\circ\overline u,v=\pi_y\circ\overline u,w=\pi_z\circ\overline u$. This is equivalent to the following system of partial differential equations 
\begin{equation}
\label{Eq:equiv_eq}
\begin{aligned}
&\partial_s u-\partial_t v=0,\\
&\partial_sv+\partial_t u=0,\\
&\partial_{ss}w+\partial_{tt}w=0,\\
&\partial_s\theta=\lambda(\partial_t\overline u)=\partial_t w-v(s,t)\partial_t u,\\
&\partial_t\theta=-\lambda(\partial_s\overline u)=-\partial_s w+v(s,t)\partial_s u,
\end{aligned}
\end{equation}
where $(s,t)$ are local holomorphic coordinates on $\Sigma$.  

The first two equations in (\ref{Eq:equiv_eq}) give us the Cauchy--Riemann equation for $(u,v)=\pi_{xy}\widetilde u$. The third equation implies that $w=\pi_z\widetilde u$ is a harmonic function. The last two equations are equivalent to
\begin{align*}
d\theta=\lambda\circ d\overline u\circ j.
\end{align*} 
\begin{lemma}
For $(\mathring\Sigma,j)$ a punctured Riemann surface and $\overline u=(u,v,w):\mathring\Sigma\to\lR^3$ a smooth map such that the first three equations in (\ref{Eq:equiv_eq}) are satisfied, the 1-form $\lambda\circ d \overline u\circ j$ on $\mathring\Sigma$ is closed.
\end{lemma}
\begin{proof}
We have
\begin{align*}
\lambda\circ d\overline u\circ j=(\partial_t w-v\partial_t u)ds+(-\partial_s w+v\partial_s u)dt
\end{align*}
and
\begin{align*}
\partial_t(\partial_t w-v\partial_t u)-\partial_s(-\partial_s w+v\partial_s u)=\\
=\Delta w-v\Delta u-\partial_tv\partial_t u-\partial_s v\partial_s u=0.
\end{align*}
Therefore,
$$d(\lambda\circ d\overline u\circ j)=0.$$
\end{proof}

\subsubsection{Obstruction section}\label{Sec:obstruction_section}

We can now define the obstruction section for lifting holomorphic curves in the Lagrangian projection to $J$-holomorphic curves in the symplectization. We denote the form $\lambda\circ d\overline u\circ j$ on $(\Sigma,j)$ by $\beta(\Sigma,j,\overline u)$ (or simply $\beta(\overline u)$). 
\begin{defi}
The boundary component of a $J$-holomorphic annulus that contains the positive puncture is called the \textit{outer} boundary component, while the boundary component with no positive puncture is called \textit{inner}.
\end{defi}
As a corollary of the conclusions above, we get the following lemmas.
\begin{lemma}[\cite{Russell21}]
There exists a smooth map $\Omega:\cM^\pi_2\to\lR$ such that an annulus $u_0\in\cM^\pi_2$ can be lifted to a $J$-holomorphic annulus in $\lR^4$ with boundary on $\lR\times\Lambda$ if and only if $\Omega(u_0)=0$.
\end{lemma}
\begin{proof}
Take an arbitrary annulus $u_0\in\cM^\pi_{2},u_0:(\mathring\Sigma,j)\to\lC$. Let $l\subset\operatorname{int}(\Sigma)$ be the generator of $H_1(\Sigma)$ oriented as the outer boundary component. Map $u_0|_\partial$ has a unique lift $\overline u_0|_\partial$ to $\Lambda\subset\lR^3$. Let $w:\mathring\Sigma\to\lR$ be the (unique) harmonic extension of $\pi_z\circ\overline u_0|_\partial$ to the interior, and denote $\overline u_0=(u_0,w)$. Note that $\overline u_0$ satisfies the first three equations in (\ref{Eq:equiv_eq}). We define 
$$\Omega(u_0)=\int_l\beta(\Sigma,j,\overline u_0).$$ 
Then we have $\Omega(u_0)=0$ if and only if $\beta(\Sigma,j,\overline u_0)$ is exact. This holds if and only if there exists $\theta:\mathring\Sigma\to\lR$ such that $d\theta=\lambda\circ d\overline u\circ j$, i.e. if and only if there exists a $J$-holomorphic map $(\theta,\overline u_0)$ in the symplectization with boundary on $\lR\times\Lambda$, which finishes the proof.
\end{proof} 

The map $\Omega$ defined above is called the \textit{obstruction section}. Similarly, using $H_{dR}^1(\lD^2)=0$, we find that every holomorphic disk with corners and boundary on $\pi_{xy}(\Lambda)\subset\lC$ can be lifted to a $J$-holomorphic disk on $\lR\times\Lambda$ in the symplectization. See also \cite{ENS02,lifting_Rizell}.

\begin{lemma}
Let $u_0$ be a holomorphic disk in $\lC$ with boundary on $\pi_{xy}(\Lambda)$ and corners at the self-intersections of $\pi_{xy}(\Lambda)$. Then there exists a $J$-holomorphic disk $(\theta,u)$ in the symplectization with boundary on $\lR\times\Lambda$ such that $\pi_{xy}\circ u=u_0$.
\end{lemma}

The fact that the lifts have finite Hofer energy for generic $\Lambda$ follows from a more subtle analysis of the asymptotic behavior, see \cite{RobinSalamon,ENS02,lifting_Rizell}

This can easily be generalized to a definition of an obstruction section $\Omega:\cM^\pi(\Sigma)\to\lR^{b_1(\Sigma)}$ for any $\Sigma$ and $b_1(\Sigma)$ the first Betti number, such that $u\in\cM^\pi(\Sigma)$ lifts to a $J$-holomorphic curve on $\lR\times\Lambda$ if and only if $\Omega(u)=0$.

\subsubsection{Extension of the obstruction section to the boundary}\label{Sec:Elliptic_hyperbolic}

We describe compactification $\overline\cM^\pi_{2,1}$ of $\cM^\pi_{2,1}$ and extend the obstruction section to a map $\Omega:\overline \cM^\pi_{2,1}\to\lR\cup\{+\infty,-\infty\}$.

A generic point in $\cM^\pi_{2,1}$ has convex corners and one boundary branch point. A codimension 1 subset of $\cM^\pi_{2,1}$ consists of immersed holomorphic annuli with one non-convex corner and no branch points (see Figure \ref{Fig:projection_moduli_space}). Connected components of $\cM^\pi_{2,1}$ are parameterized by the image of the boundary branch point or the image of the non-convex corner. A boundary (limit) point of $\cM^\pi_{2,1}$ can be seen as a degenerate annulus in the Lagrangian projection where the branch point meets the boundary of the annulus. We distinguish two cases. First, the branch point meets the boundary component it does not lie on, and we say the degenerate curve is \textit{non-split}. Here, the degenerate annulus can be seen as an index zero holomorphic disk with two additional corners, one positive and one negative, at some self-intersection of $\pi_{xy}(\Lambda)$ (see Figure \ref{Fig:projection_moduli_space}, left). Second, the branch point meets the boundary component it lies on, and we say the degenerate curve is \textit{split}. Here, the degenerate map can be seen as a building consisting of a holomorphic annulus in $\cM_{2,0}^\pi$ (the annular part) and an index zero holomorphic disk attached to it at some positive or negative corner (see Figure \ref{Fig:projection_moduli_space}, right). We denote the moduli space compactified this way by $\overline{\cM}^\pi_{2,1}$.

\begin{figure}
\def\svgwidth{151mm}
\begingroup%
  \makeatletter%
  \providecommand\rotatebox[2]{#2}%
  \newcommand*\fsize{\dimexpr\f@size pt\relax}%
  \newcommand*\lineheight[1]{\fontsize{\fsize}{#1\fsize}\selectfont}%
  \ifx\svgwidth\undefined%
    \setlength{\unitlength}{797.23331351bp}%
    \ifx\svgscale\undefined%
      \relax%
    \else%
      \setlength{\unitlength}{\unitlength * \real{\svgscale}}%
    \fi%
  \else%
    \setlength{\unitlength}{\svgwidth}%
  \fi%
  \global\let\svgwidth\undefined%
  \global\let\svgscale\undefined%
  \makeatother%
  \begin{picture}(1,0.22497675)%
    \lineheight{1}%
    \setlength\tabcolsep{0pt}%
    \put(0,0){\includegraphics[width=\unitlength,page=1]{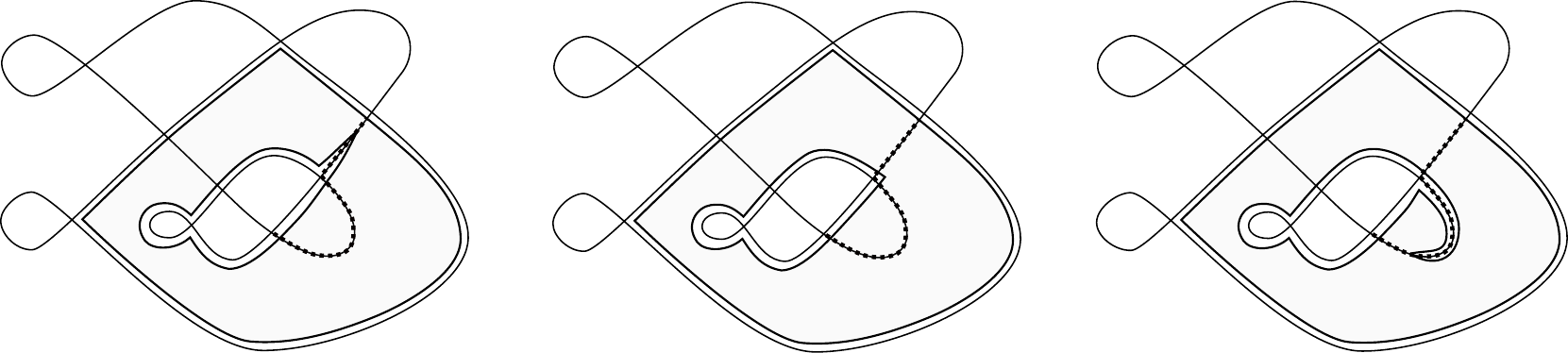}}%
  \end{picture}%
\endgroup%
\caption{Connected component of the moduli space $\overline{\cM}^\pi_{2,1}$ with a non-split and a split boundary point.}
\label{Fig:projection_moduli_space}
\end{figure}

The obstruction section $\Omega:\cM^\pi_{2,1}\to\lR$ can be extended continuously to a section $\Omega:\overline{\cM}^\pi_{2,1}\to\lR\cup\{+\infty,-\infty\}$ as follows. If $\overline u\in\partial\overline{\cM}^\pi_{2,1}$ is a non-split degenerate point, we define $\Omega(\overline u)=+\infty$ if the outer boundary component passes through the undercrossing arc near the self-intersection of $\pi_{xy}(\Lambda)$ at the branch point of $\overline u$, and $\Omega(\overline u)=-\infty$ otherwise (see for example Figure \ref{Fig:hyperbolic_boundary_obstruction}). If $\overline u\in\partial\overline{\cM}^\pi_{2,1}$ is a split degenerate point, we define $\Omega(\overline u)=\Omega(\overline u_0)$, where $\overline u_0\in\cM^\pi_{2,0}$ is the annular part of the building $\overline u$. We show that the extension defined this way is continuous at the split (Lemma \ref{Lemma:OSection_elliptic}) and non-split (Lemma \ref{Lemma:OSection_hyperbolic}) boundary points. For this, we need the following lemma.

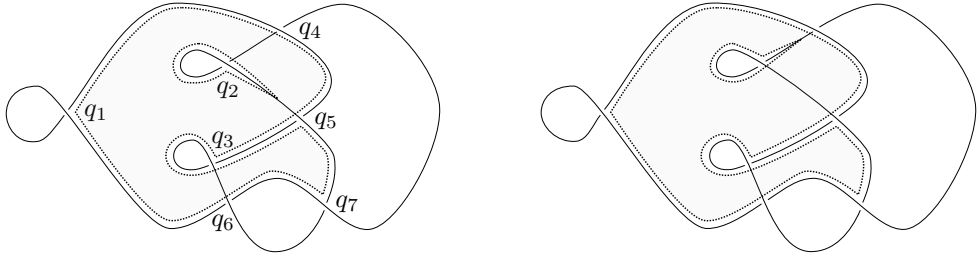
\begin{figure}
\def\svgwidth{129mm}
\begingroup%
  \makeatletter%
  \providecommand\rotatebox[2]{#2}%
  \newcommand*\fsize{\dimexpr\f@size pt\relax}%
  \newcommand*\lineheight[1]{\fontsize{\fsize}{#1\fsize}\selectfont}%
  \ifx\svgwidth\undefined%
    \setlength{\unitlength}{636.99027828bp}%
    \ifx\svgscale\undefined%
      \relax%
    \else%
      \setlength{\unitlength}{\unitlength * \real{\svgscale}}%
    \fi%
  \else%
    \setlength{\unitlength}{\svgwidth}%
  \fi%
  \global\let\svgwidth\undefined%
  \global\let\svgscale\undefined%
  \makeatother%
  \begin{picture}(1,0.25727306)%
    \lineheight{1}%
    \setlength\tabcolsep{0pt}%
    \put(0.08080794,0.14009945){\makebox(0,0)[lt]{\lineheight{1.25}\smash{\begin{tabular}[t]{l}$q_1$\end{tabular}}}}%
    \put(0.2170555,0.164301){\makebox(0,0)[lt]{\lineheight{1.25}\smash{\begin{tabular}[t]{l}$q_2$\end{tabular}}}}%
    \put(0.21195459,0.11256271){\makebox(0,0)[lt]{\lineheight{1.25}\smash{\begin{tabular}[t]{l}$q_3$\end{tabular}}}}%
    \put(0.30195002,0.22466237){\makebox(0,0)[lt]{\lineheight{1.25}\smash{\begin{tabular}[t]{l}$q_4$\end{tabular}}}}%
    \put(0.31798174,0.13235929){\makebox(0,0)[lt]{\lineheight{1.25}\smash{\begin{tabular}[t]{l}$q_5$\end{tabular}}}}%
    \put(0.21183314,0.02645358){\makebox(0,0)[lt]{\lineheight{1.25}\smash{\begin{tabular}[t]{l}$q_6$\end{tabular}}}}%
    \put(0.33984284,0.04661455){\makebox(0,0)[lt]{\lineheight{1.25}\smash{\begin{tabular}[t]{l}$q_7$\end{tabular}}}}%
    \put(0,0){\includegraphics[width=\unitlength,page=1]{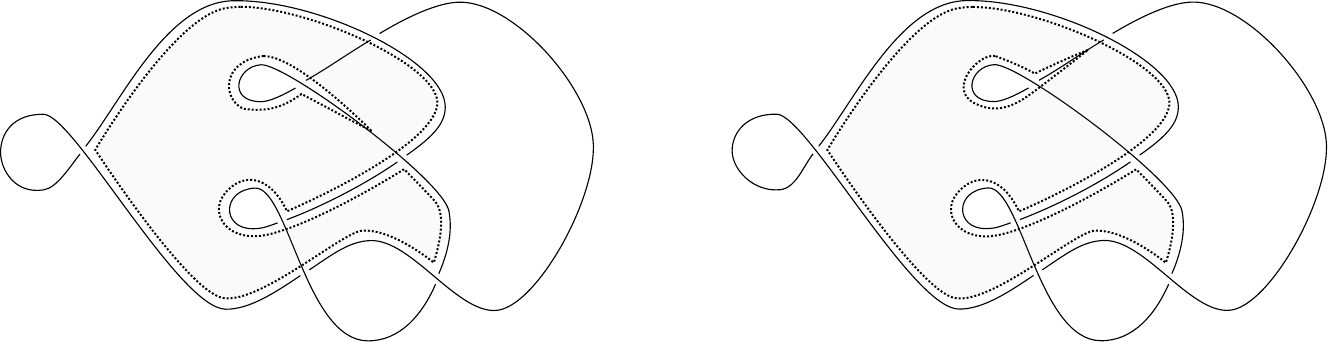}}%
  \end{picture}%
\endgroup%
\caption{Obstruction section $\Omega$ equal to $+\infty$ and $-\infty$ at the two non-split boundary points.}
\label{Fig:hyperbolic_boundary_obstruction}
\end{figure}

\begin{lemma}\label{Lemma:harmonic_strip}
Let $w_n:\lR\times[0,1]\to \lR$ be a sequence of harmonic maps (possibly with isolated boundary discontinuities) such that $w_n$ are uniformly bounded and 
\begin{align*}
\max_{(-2n,2n)\times\{0,1\}}|w_n|\overset{n\to\infty}{\longrightarrow}0.
\end{align*}
Then
\begin{align*}
\max_{(-n,n)\times[0,1]}|w_n|\overset{n\to\infty}{\longrightarrow}0.
\end{align*}
If additionally 
\begin{align*}
\|w_n|_{(-2n,2n)\times\{0,1\}}\|_{C^2}\overset{n\to\infty}{\longrightarrow}0,
\end{align*}
and $\|w_n|_{\lR\times\{0,1\}}\|_{C^2}$ are uniformly bounded, then
\begin{align*}
\|w_n|_{(-n,n)\times[0,1]}\|_{C^1}\overset{n\to\infty}{\longrightarrow}0.
\end{align*}
\end{lemma}
\begin{proof}
For $w_+:\lR\to\lR$ and $w_-:\lR\to\lR$ piecewise smooth functions, the unique harmonic function $w:\lR\times[0,1]\to\lR$ such that $w(s,0)=w_-(s),w(s,1)=w_+(s)$ is given by (see \cite{Widder61})
\begin{align*}
w(s,t)&=\frac{1}{2\pi}\int_{-\infty}^{+\infty}\left(P(\sigma-s,t)w_-(\sigma)+P(\sigma-s,1-t)w_+(\sigma)\right)d\sigma=\\
&=\frac{1}{2\pi}\int_{-\infty}^{+\infty}\left(P(\sigma,t)w_-(\sigma+s)+P(\sigma,1-t)w_+(\sigma+s)\right)d\sigma
\end{align*}
for $(s,t)\in\lR\times(0,1)$, where
\begin{equation}
\label{Eq:Pkernel}
\begin{aligned}
P(s,t)=\pi\frac{\sin \pi t}{\cosh \pi s-\cos \pi t}.
\end{aligned}
\end{equation}
Then for any $s\in(-n,n),t\in(0,1)$ we have
\begin{align*}
2\pi|w_n(s,t)|\leq&\int_{(-2n,2n)}P(\sigma-s,t)|w_n(\sigma,0)|d\sigma+\\
+&\int_{(-2n,2n)} P(\sigma-s,1-t)|w_n(\sigma,1)|d\sigma+\\
+&\int_{(-\infty,-2n)\cup(2n,+\infty)}P(\sigma-s,t)|w_n(\sigma,0)|d\sigma+\\
+&\int_{(-\infty,-2n)\cup(2n,+\infty)}P(\sigma-s,1-t)|w_n(\sigma,1)|d\sigma\leq\\
\leq&\, c\left(\max_{(-2n,2n)\times\{0\}\cup(-2n,2n)\times\{1\}}|w_n|+e^{-n}\right),
\end{align*}
where $c$ is a uniform constant. This follows from the fact that $w_n$ are uniformly bounded, $0< P(\sigma,t)<2\pi e^{-|\sigma|}$ for all $|\sigma|>1,t\in(0,1)$, and $\int_{-\infty}^\infty |P(\sigma,t)|d\sigma=\int_{-\infty}^\infty P(\sigma,t)d\sigma=2(\pi-\pi t)\leq 2\pi$ for $t\in(0,1)$. Therefore,
\begin{align*}
\max_{(-n,n)\times[0,1]}|w_n|\overset{n\to\infty}{\longrightarrow}0.
\end{align*}
Assume now that
\begin{align*}
\|w_n|_{(-2n,2n)\times\{0,1\}}\|_{C^2}\overset{n\to\infty}{\longrightarrow}0.
\end{align*}
For all $s\in\lR,t\in(0,1)$, we have
\begin{align*}
&\partial_sw_n(s,t)=\frac{1}{2\pi}\int_{-\infty}^{+\infty}\left(P(\sigma-s,t)\partial_s w_n(\sigma,0)+P(\sigma-s,1-t)\partial_s w_n(\sigma,1)\right)d\sigma,\\
&\partial_{ss}w_n(s,t)=\frac{1}{2\pi}\int_{-\infty}^{+\infty}\left(P(\sigma-s,t)\partial_{ss} w_n(\sigma,0)+P(\sigma-s,1-t)\partial_{ss} w_n(\sigma,1)\right)d\sigma,\\
&\partial_t w_n(s,1/2)=\frac{1}{2\pi}\int_{-\infty}^{+\infty}\left(\partial_t P(\sigma-s,1/2)w_n(\sigma,0)-\partial_t P(\sigma-s,1/2) w_n(\sigma,1)\right)d\sigma.
\end{align*}
Similarly, we get
\begin{align*}
&\max_{(-n,n)\times[0,1]}|\partial_{s}w_n|\overset{n\to\infty}{\longrightarrow}0,\\
&\max_{(-n,n)\times[0,1]}|\partial_{tt}w_n|=\max_{(-n,n)\times[0,1]}|\partial_{ss}w_n|\overset{n\to\infty}{\longrightarrow}0.
\end{align*}
Additionally,
\begin{align*}
\partial_t P(s,1/2)=-\pi^2\frac{1}{(\cosh\pi s)^2}
\end{align*}
similarly implies
\begin{align*}
\max_{s\in (-n,n)}|\partial_tw_n(s,1/2)|\overset{n\to\infty}{\longrightarrow}0.
\end{align*}
Now, using
\begin{align*}
\partial_tw_n(s,t)=\int_{\frac{1}{2}}^t\partial_{tt}w_n(s,t) dt+\partial_tw_n(s,1/2),
\end{align*}
we finally get
\begin{align*}
\max_{(-n,n)\times[0,1]}|\partial_{t}w_n|\overset{n\to\infty}{\longrightarrow}0.
\end{align*}
This finishes the proof that
\begin{align*}
\|w_n\|_{C^1((-n,n)\times[0,1])}\overset{n\to\infty}{\longrightarrow}0.
\end{align*}
\end{proof}

\begin{cor}\label{Cor:thin_necks}
Let $\Lambda_n,\Lambda$ be Legendrian knots such that $\Lambda_n\overset{C^\infty}{\to}\Lambda$, $u_n:\mathring\Sigma_n\to\lC$ a sequence of holomorphic maps with corners and boundary on $\pi_{xy}(\Lambda_n)$, and $\varphi_n:(-8n,8n)\times[0,1]\to\Sigma_n$ holomorphic embeddings such that $\|u_n\circ\varphi_n-\Gamma\|_{C^2((-8n,8n)\times\{0,1\})}\overset{n\to\infty}{\longrightarrow}0$ for $\Gamma\in\lC$ a self-intersection of $\pi_{xy}(\Lambda)$. Then, for $w_n:\mathring\Sigma_n\to\lR$ the unique harmonic lifts such that $(u_n,w_n)(\partial\mathring\Sigma_n)\subset\Lambda_n$, we have
$$\|w_n\circ\varphi_n-w\|_{C^1((-n,n)\times[0,1])}\overset{n\to\infty}{\longrightarrow}0,$$
where $w(s,t)=\pi_z\circ\gamma(t)$ or $w(s,t)=\pi_z\circ\gamma(1-t)$ for $\gamma:[0,1]\to\lR^3$ the Reeb chord on $\Lambda$ corresponding to $\Gamma$.
\end{cor}
\begin{proof}
There exist holomorphic embeddings $\psi_n:\lR\times[0,1]\to\Sigma_n$ such that $\varphi_n((-n,n)\times[0,1])\subset \psi_n((-2n,2n)\times[0,1])$ and $\psi_n((-4n,4n)\times[0,1])\subset \varphi_n((-8n,8n)\times[0,1])$, and such that the change of coordinates maps between $\varphi_n$ and $\psi_n$ have uniformly bounded derivatives. For disks, these can be obtained by taking strip-like holomorphic parameterization of the domain (with added punctures). Otherwise, we first cut the surface into a disk, where the cuts are chosen away from $\Gamma$. For annuli, similar can also be obtained by taking the holomorphic universal covering by a strip. If $u_n\circ\varphi_n(s,0)$ passes through the undercrossing arc near $\Gamma$ we take $w(s,t)=\pi_z\circ\gamma(t)$, otherwise we take $w(s,t)=\pi_z\circ\gamma(1-t)$. Map $w$ is obviously harmonic. Since $\|u_n\circ\psi_n-\Gamma\|_{C^2((-4n,4n)\times\{0,1\})}\overset{n\to\infty}{\longrightarrow}0$ and $\Lambda,\Lambda_n$ are smooth, we conclude $\|w_n\circ\psi_n-w\|_{C^2((-4n,4n)\times\{0,1\})}\overset{n\to\infty}{\longrightarrow}0$. Additionally, $w_n$ are uniformly bounded since $\Lambda_n,\Lambda$ are. Then, using Lemma \ref{Lemma:harmonic_strip}, we get $\|w_n\circ\varphi_n-w\|_{C^1((-n,n)\times[0,1])}\overset{n\to\infty}{\longrightarrow}0$.
\end{proof}
We give a short overview of some well-known elliptic bootstrapping results that will be used later.
\begin{thm}\cite[Theorem 6.3.1.2 and Theorem 6.3.2.5]{evansPDE}
Let $U'\subset\lC$ be a domain with smooth boundary, $U\subset U'$ a subdomain and $u: U\to\lR$ be a smooth solution of the Dirichlet problem
\begin{align*}
&\Delta u=\eta,\\
&u|_{\partial U'\cap U}=f,
\end{align*}
where $\eta: U\to\lR,f:\partial U'\cap U\to\lR$ are smooth functions. For all compact subsets $K\subset U$ and all $k\in\lN_0$, there exists a constant $c_k\in\lR$ (independent of $u,\eta$ and $f$) such that
\begin{align*}
\|u\|_{W^{k+2,2}(K)}\leq c_k\left(\|\eta\|_{W^{k,2}( U)} +\|f\|_{W^{k+2,2}(\partial U'\cap U)} + \|u\|_{L^2( U)}\right),
\end{align*}
where $\|\cdot\|_{W^{k,p}}$ denotes the Sobolev norm.
\end{thm}

\begin{cor}\label{Corollary:harmonic_bootstrapping}
Let $ U'\subset\lC$ be a domain with smooth boundary, $ U\subset U'$ a bounded 
subdomain and $u: U\to\lR$ a smooth solution of the Dirichlet problem
\begin{align*}
&\Delta u=0,\\
&u|_{\partial U'\cap U}=f,
\end{align*}
where $f:\partial U'\cap U\to\lR$ is a smooth function. For all $K\subset U$ compact, there exists a constant $c\in\lR$ (independent of $u$ and $f$) such that
$$\|u\|_{C^1(K)}\leq c\left(\|f\|_{C^3(\partial U'\cap U)}+\|u\|_{C^0( U)}\right).$$
\end{cor}
\begin{proof}
Using the theorem above, we get
\begin{align*}
\|u\|_{W^{3,2}(K)}\leq c_1\left(\|f\|_{W^{3,2}(\partial U'\cap U)} + \|u\|_{L^2( U)}\right)\leq c_2\left(\|f\|_{C^3(\partial U'\cap U)}+\|u\|_{C^0( U)}\right)
\end{align*} 
for some constants $c_1,c_2\in\lR$. The statement now follows from the fact that we have a natural continuous Sobolev embedding $W^{3,2}(K)\hookrightarrow C^{1}(K)$.
\end{proof}

To prove that the obstruction section extends continuously as described before, we show some special Gromov compactness type results.

\begin{lemma}\label{Lemma:OSection_elliptic}
Let $\overline u\in\partial\overline{\cM}^\pi_{2,1}$ be a split boundary point and $u_n\in\cM^\pi_{2,1},n\in\lN$ a sequence of annuli converging to $\overline u$. Then $\lim_{n\to\infty} \Omega(u_n)=\Omega(\overline u)$.
\end{lemma}
\begin{proof}
Let $u_n\in\cM^\pi_{2,1},u_n:(\Sigma_n,j_n)\to(\lC,i),n\in\lN$ be a sequence of holomorphic annuli such that $\lim_{n\to\infty} u_n=\overline u$ and $w_n$ the corresponding harmonic lifts. Denote the boundary branch point of $u_n$ by $S_n\in\pi_{xy}(\Lambda)$, by $\Gamma$ the self-intersection of $\pi_{xy}(\Lambda)$ such that $S_n\overset{n\to\infty}{\longrightarrow} \Gamma$ and $\gamma:[0,1]\to\lR^3$ the corresponding Reeb chord on $\Lambda$. Let $\varphi_n:(-8R_n,8R_n)\times[0,1]\to\Sigma_n$ be a sequence of conformal embeddings parameterizing the thin necks on $u_n$ that are forming near $S_n$ such that $\|u_n\circ \varphi_n-\Gamma\|_{C^2((-8R_n,8R_n)\times\{0,1\})}\overset{n\to\infty}{\longrightarrow} 0$ and $R_n\to\infty$. 
Let $\alpha_n$ and $\beta_n$ denote the paths $t\to \varphi_n(-R_n+1,t)$ and $t\to \varphi_n(R_n-1,t),t\in [0,1]$. We can assume $\alpha_n$ is in the annular part (see Figure \ref{Fig:projection_lemma_1}). Denote $\overline\gamma(t)=\gamma(1-t)$ and define harmonic functions $w^+,w^-:\lR\times[0,1]\to\lR$, $w^+(s,t)=\pi_z\gamma(t),w^-(s,t)=\pi_z\overline\gamma(t)$. By Corollary \ref{Cor:thin_necks}, we have
\begin{align*}
\|w_n\circ \varphi_n-w^+\|_{C^1((-R_n,R_n)\times[0,1])}\overset{n\to\infty}{\longrightarrow}0,
\end{align*}
or
\begin{align*}
\|w_n\circ \varphi_n-w^-\|_{C^1((-R_n,R_n)\times[0,1])}\overset{n\to\infty}{\longrightarrow}0,
\end{align*}
which implies
$$w_n\circ\alpha_n\overset{n\to\infty}{\longrightarrow}\pi_z\circ\gamma$$
or
$$w_n\circ\alpha_n\overset{n\to\infty}{\longrightarrow}\pi_z\circ\overline\gamma$$
in $C^1([0,1])$. Denote the annular part of the degenerate annulus $\overline u$ by $\overline u_0:(\Sigma,j)\to\lC$ and by $w_0$ the corresponding harmonic lift. We can see $\alpha_n$ as paths in $\Sigma$ converging to the puncture at $\Gamma$. Similar as above, by Corollary \ref{Cor:thin_necks} we get $w_0\circ\alpha_n\overset{C^1([0,1])}{\longrightarrow}\pi_z\circ\gamma,n\to\infty$ or $w_0\circ\alpha_n\overset{C^1([0,1])}{\longrightarrow}\pi_z\circ\overline\gamma,n\to\infty$. Now, denote by $\widetilde\Sigma_n$ the subdomain of $\Sigma$ (can also be seen as a subdomain of $\Sigma_n$) obtained by cutting off the corner at the path $\alpha_n$. Then $(w_n-w_0)|_{\widetilde\Sigma_n}$ is a solution of the Dirichlet problem
\begin{align*}
&\Delta v=0,\\
&v|_{\partial\widetilde\Sigma_n}=\overline f_n,
\end{align*}
on $\widetilde\Sigma_n$ such that $\max_{\partial\widetilde\Sigma_n}|\overline f_n|\overset{n\to\infty}{\longrightarrow} 0$. This follows from $w_n|_{\partial\widetilde\Sigma_n\cap\Sigma_n}=w_0|_{\partial\widetilde\Sigma_n\cap\Sigma}$ and $w_n\circ\alpha_n(t)-w_0\circ\alpha_n(t)\overset{n\to\infty}{\longrightarrow} 0$ uniformly for $t\in[0,1]$ using the observations above. Using the maximum principle we get 
\begin{align}
\label{Eq:max0}
\max_{\widetilde\Sigma_n}|w_n-w_0|\overset{n\to\infty}{\longrightarrow} 0.
\end{align} 
Let $b$ be the boundary component of $\Sigma_n$ that does not contain the branch point, $K\subset\widetilde\Sigma_n$ a compact neighborhood of $b$ and $\tilde l\subset\operatorname{int}(K)$ a loop that generates $H_1(\widetilde \Sigma_n)$ oriented as the outer boundary component. We notice that $w_n|_{b}=w_0|_{b}$ because of the uniqueness of the lift at the boundary. Then, using (\ref{Eq:max0}) and Corollary \ref{Corollary:harmonic_bootstrapping} we get $\partial_t w_n\to\partial_t w_0$ and $\partial_s w_n\to\partial_s w_0,n\to\infty$ uniformly on $\tilde l$. Additionally, we have $u_n|_{K}=\overline u_0|_K$. This gives us
\begin{align*}
\Omega(u_n)=\int_{\tilde l}\lambda\circ d(u_n,w_n)\circ j_n\overset{n\to\infty}{\longrightarrow} \int_{\tilde l}\lambda\circ d(\overline u_0,w_0)\circ j=\Omega(\overline u),
\end{align*}
which finishes the proof.
\end{proof}

\begin{figure}
\def\svgwidth{63mm}
\begingroup%
  \makeatletter%
  \providecommand\rotatebox[2]{#2}%
  \newcommand*\fsize{\dimexpr\f@size pt\relax}%
  \newcommand*\lineheight[1]{\fontsize{\fsize}{#1\fsize}\selectfont}%
  \ifx\svgwidth\undefined%
    \setlength{\unitlength}{299.29435542bp}%
    \ifx\svgscale\undefined%
      \relax%
    \else%
      \setlength{\unitlength}{\unitlength * \real{\svgscale}}%
    \fi%
  \else%
    \setlength{\unitlength}{\svgwidth}%
  \fi%
  \global\let\svgwidth\undefined%
  \global\let\svgscale\undefined%
  \makeatother%
  \begin{picture}(1,0.59855651)%
    \lineheight{1}%
    \setlength\tabcolsep{0pt}%
    \put(0,0){\includegraphics[width=\unitlength,page=1]{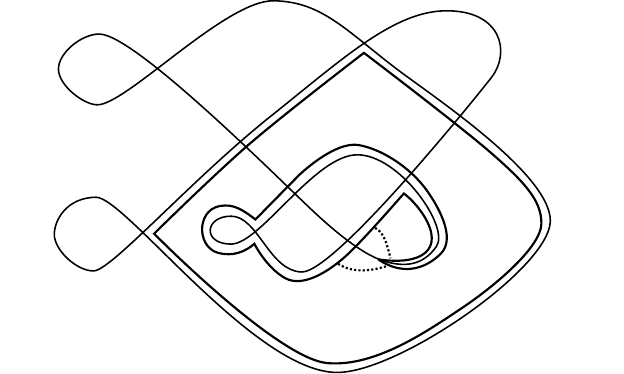}}%
    \put(0.82634528,0.22087993){\makebox(0,0)[lt]{\lineheight{1.25}\smash{\begin{tabular}[t]{l}$\tilde l$\end{tabular}}}}%
    \put(0.5510557,0.12708795){\makebox(0,0)[lt]{\lineheight{1.25}\smash{\begin{tabular}[t]{l}$\alpha_n$\end{tabular}}}}%
    \put(0.62301065,0.21121438){\makebox(0,0)[lt]{\lineheight{1.25}\smash{\begin{tabular}[t]{l}$\beta_n$\end{tabular}}}}%
    \put(0.54747586,0.41705566){\makebox(0,0)[lt]{\lineheight{1.25}\smash{\begin{tabular}[t]{l}$\widetilde\Sigma_n$\end{tabular}}}}%
    \put(-0.00238895,0.56418744){\makebox(0,0)[lt]{\lineheight{1.25}\smash{\begin{tabular}[t]{l}\textcolor{white}{.}\end{tabular}}}}%
  \end{picture}%
\endgroup%
\caption{Split boundary point of the 1-dimensional moduli space of holomorphic annuli on $\pi_{xy}(\Lambda)$ with corners.}
\label{Fig:projection_lemma_1}
\end{figure}

\begin{lemma}\label{Lemma:OSection_hyperbolic}
Let $\overline u\in\partial\overline{\cM}^\pi_{2,1}$ be a non-split boundary point and $u_n\in\cM^\pi_{2,1},n\in\lN$ a sequence of annuli converging to $\overline u$. Then $\lim_{n\to\infty} \Omega(u_n)=\Omega(\overline u)$.
\end{lemma}
\begin{proof}
The proof follows a similar approach as the previous lemma. Let $u_n\in\cM^\pi_{2,1},u_n:(\Sigma_n,j_n)\to(\lC,i),n\in\lN$ be a sequence of holomorphic annuli such that $\lim_{n\to\infty} u_n=\overline u$ and $w_n$ the corresponding harmonic lifts. Denote by $S_n\in\pi_{xy}(\Lambda)$ the boundary branch point of $u_n$ and by $\Gamma$ the self-intersection of $\pi_{xy}(\Lambda)$ such that $S_n\overset{n\to\infty}{\longrightarrow} \Gamma$. As before, we have a sequence of conformal embeddings $\varphi_n:(-8R_n,8R_n)\times[0,1]\to\Sigma_n$ such that $\|u_n\circ \varphi_n-\Gamma\|_{C^2((-8R_n,8R_n)\times\{0,1\})}\overset{n\to\infty}{\longrightarrow} 0$ and $R_n\overset{n\to\infty}{\longrightarrow}\infty$. Let $\alpha_n$ and $\beta_n$ denote the paths $t\to \varphi_n(-R_n+2,t)$ and $t\to \varphi_n(R_n-2,t),t\in [0,1]$ (see Figure \ref{Fig:projection_lemma_2}). As in the previous lemma, we have 
\begin{align*}
\|w_n\circ \varphi_n-w^+\|_{C^1((-R_n,R_n)\times[0,1])}\overset{n\to\infty}{\longrightarrow}0
\end{align*}
or
\begin{align*}
\|w_n\circ \varphi_n-w^-\|_{C^1((-R_n,R_n)\times[0,1])}\overset{n\to\infty}{\longrightarrow}0
\end{align*}
by Corollary \ref{Cor:thin_necks}, where $w^+(s,t)=\pi_z\gamma(t),w^-(s,t)=\pi_z\overline\gamma(t)$ for $\gamma$ the Reeb chord corresponding to $\Gamma$ and $\overline\gamma(t)=\gamma(1-t)$. From this we conclude $ w_n\circ\alpha_n,w_n\circ\beta_n\overset{C^1([0,1])}{\longrightarrow}\pi_z\circ\gamma$ or $ w_n\circ\alpha_n,w_n\circ\beta_n\overset{C^1([0,1])}{\longrightarrow}\pi_z\circ\overline\gamma,n\to\infty$. 

Denote $\widetilde\Sigma_n=\Sigma_n\backslash \varphi_n((-R_n+2,R_n-2)\times[0,1])$. The degenerate annulus $\overline u$ can be seen as a disk $\overline u_0:\Sigma\to\lC$ in the projection with an additional positive and negative corner at $\Gamma$. Denote the harmonic lift of $\overline u_0$ by $w_0$. The paths $\alpha_n,\beta_n$ can be seen as paths in $\Sigma$ converging to the punctures at $\Gamma$. Additionally, we have $ w_0\circ\alpha_n,w_0\circ\beta_n\overset{C^1([0,1])}{\longrightarrow}\pi_z\circ\gamma$ or $ w_0\circ\alpha_n,w_0\circ\beta_n\overset{C^1([0,1])}{\longrightarrow}\pi_z\circ\overline\gamma,n\to\infty$ as before using Corollary \ref{Cor:thin_necks}. This can also be seen using the fact that $\overline u_0$ is asymptotic to $\gamma$ at the two punctures. Then, $(w_n-w_0)|_{\widetilde\Sigma_n}$ is a solution of the Dirichlet problem
\begin{align*}
&\Delta v=0,\\
&v|_{\partial\widetilde\Sigma_n}=\overline f_n,
\end{align*}
where $\max_{\partial\widetilde\Sigma_n}|\overline f_n|\overset{n\to\infty}{\longrightarrow} 0$. This implies
$$\max_{\widetilde\Sigma_n}|w_n-w_0|\overset{n\to\infty}{\longrightarrow} 0.$$

Let $l_n\subset\Sigma_n$ be the loops obtained by slightly pushing the outer boundary component of $\Sigma_n$ into the interior at the punctures. Fix a small enough neighborhood $U$ of $\Gamma$ and denote $l_1^n=l_n\cap(\Sigma\backslash U),l_2^n=l_n\cap U$. As before, we get that $\|w_n-w_0\|_{C^1(K)}\to 0$ for $K\subset\Sigma\backslash U$ a compact neighborhood of $l_1^n$. Then
\begin{align*}
\left|\int_{l_1^n}\lambda\circ d(u_n,w_n)\circ j_n\right|
\end{align*}
are uniformly bounded, and for $U$ small enough
\begin{align*}
\left|\int_{l^n_2}\lambda\circ d(u_n,w_n)\circ j_n\right|=\int_{l^n_2}|\lambda\circ d(u_n,w_n)\circ j_n|>\\
>\frac{1}{2}\int_{-R_n}^{R_n} |\lambda(\partial_t w_n)|ds> \frac{R_n}{2}l(\gamma)\overset{n\to\infty}{\longrightarrow}\infty.
\end{align*}
The inequality follows as before using Lemma \ref{Lemma:harmonic_strip}. More precisely,
\begin{align*}
&\Omega(u_n)=\int_{l_n}\lambda\circ d(u_n,w_n)\circ j_n\overset{n\to\infty}{\longrightarrow}-\infty
\end{align*}
if the outer boundary passes through the overcrossing arc near $\Gamma$, and
\begin{align*}
&\Omega(u_n)=\int_{l_n}\lambda\circ d(u_n,w_n)\circ j_n\overset{n\to\infty}{\longrightarrow}+\infty
\end{align*}
if it passes through the undercrossing arc.
\end{proof}

\begin{figure}
\def\svgwidth{49mm}
\begingroup%
  \makeatletter%
  \providecommand\rotatebox[2]{#2}%
  \newcommand*\fsize{\dimexpr\f@size pt\relax}%
  \newcommand*\lineheight[1]{\fontsize{\fsize}{#1\fsize}\selectfont}%
  \ifx\svgwidth\undefined%
    \setlength{\unitlength}{230.72927092bp}%
    \ifx\svgscale\undefined%
      \relax%
    \else%
      \setlength{\unitlength}{\unitlength * \real{\svgscale}}%
    \fi%
  \else%
    \setlength{\unitlength}{\svgwidth}%
  \fi%
  \global\let\svgwidth\undefined%
  \global\let\svgscale\undefined%
  \makeatother%
  \begin{picture}(1,0.64214521)%
    \lineheight{1}%
    \setlength\tabcolsep{0pt}%
    \put(0,0){\includegraphics[width=\unitlength,page=1]{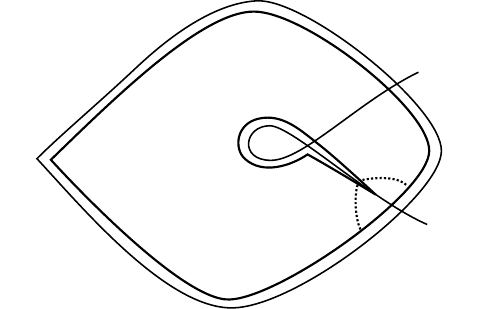}}%
    \put(0.6451562,0.18892397){\makebox(0,0)[lt]{\lineheight{1.25}\smash{\begin{tabular}[t]{l}$\alpha_n$\end{tabular}}}}%
    \put(0.76217639,0.28829838){\makebox(0,0)[lt]{\lineheight{1.25}\smash{\begin{tabular}[t]{l}$\beta_n$\end{tabular}}}}%
    \put(0.21468864,0.28086851){\makebox(0,0)[lt]{\lineheight{1.25}\smash{\begin{tabular}[t]{l}$\widetilde\Sigma_n$\end{tabular}}}}%
    \put(-0.00309886,0.59988799){\makebox(0,0)[lt]{\lineheight{1.25}\smash{\begin{tabular}[t]{l}\textcolor{white}{.}\end{tabular}}}}%
  \end{picture}%
\endgroup%
\caption{Non-split boundary point of the 1-dimensional moduli space of holomorphic annuli on $\pi_{xy}(\Lambda)$ with corners.}
\label{Fig:projection_lemma_2}
\end{figure}

\begin{cor}
The extended obstruction section $\Omega:\overline{\cM}^\pi_{2,1}\to\lR\cup\{+\infty,-\infty\}$ is continuous.
\end{cor}

The values of the obstruction section $\Omega:\overline{\cM}^\pi_{2,1}\to\lR\cup\{+\infty,-\infty\}$ at the boundary points determine the count of index zero annuli on $\lR\times\Lambda$. If all boundary points in $\partial\cM^\pi_{2,1}$ are non-split, we can easily get the (algebraic) count. 

\begin{example}\label{Example:counting_annuli1}
Let $\Lambda$ be the figure-8 knot shown in Figure \ref{Fig:hyperbolic_boundary_obstruction}. There are three 1-parameter families of annuli with one positive corner in the projection, all of them have two non-split boundary points. The boundary values of the obstruction section for two of these families are $+\infty,-\infty$, and $+\infty,+\infty$ for the third one. From this we conclude that, counting with orientation signs introduced in Section \ref{Section:Orientations}, we get two index zero annuli $q_7q_5q_3p_1\otimes q_2$ and $ q_7q_2q_5p_1\otimes q_3$ on $\lR\times\Lambda$ with one positive puncture at $\gamma_1$. These annuli can also be seen by looking at the boundary of the 1-dimensional moduli space of annuli with two positive punctures. For example for the first one, we glue disks $q_7q_5q_3p_5q_2q_2p_4p_1$ and $p_2q_5q_4$ at the punctures at $\gamma_4,\gamma_5$. The second boundary point of the corresponding connected component has to be the nodal annulus consisting of $q_7q_5q_3p_1\otimes q_2$ and a trivial strip bubble at $\gamma_2$. Similarly for the second annulus.
\end{example}

Generic Legendrian knot isotopy can be seen in the Lagrangian projection as a sequence of Reidemeister II (Figure \ref{Figure:ReidII}) and Reidemeister III (Figure \ref{Figure:ReidIII}) moves.  The count of annuli can also change when the Legendrian knot isotopy passes through a Legendrian knot with a degenerate annulus of index $-1$, which we say is a degenerate knot of \textit{type IV}. More precisely, a Legendrian knot $\Lambda$ is degenerate of type IV if the obstruction section $\Omega:\overline{\cM}^\pi_{2,1}\sqcup\cM^\pi_{2,0} \to\lR\cup\{\pm\infty\}$ maps some point in $\cM^\pi_{2,0}$ to zero, or equivalently, some boundary point in $\partial\overline{\cM}^\pi_{2,1}$ to zero.

If there are split boundary points in $\partial\overline{\cM}^\pi_{2,1}$, it is more difficult to count index zero $J$-holomorphic annuli for a given knot diagram due to the fact that type IV degenerate Legendrian knots are difficult to recognize, unlike degenerate knots in Reidemeister II and III move. Using the previous two lemmas, we can understand how the count of annuli changes when the knot isotopy passes through a degenerate knot of type IV (see also Section \ref{Section:Invar_IV_degeneration}). 

\begin{cor}\label{Corollary:counting_annuli_from_O}
Let $\Lambda_s,s\in[0,1]$ be a generic Legendrian knot isotopy without Reidemeister moves such that $\Omega(u_0)=0$ for a rigid holomorphic annulus $u_0\in\cM^\pi_{2,0}(\Lambda_{s_0})$ on $\pi_{xy}(\Lambda_{s_0})$ for some $s_0\in(0,1)$. Then the difference between the count of annuli on $\Lambda_0$ and annuli on $\Lambda_1$ is equal to the number of ways $u_0$ can be glued to some rigid disk on $\lR\times\Lambda_0$ with one positive puncture.
\end{cor}

Using the previous corollary, we show in Section \ref{Section:Invar_IV_degeneration} that our invariant remains the same under isotopy passing through a degenerate knot of type IV. Due to the combinatorial nature of the proof, this allows us to compute the invariant using a "virtual" count of annuli, which can be computed easily combinatorially, instead of the actual count.
\begin{defi}\label{Def:virtual_obstruction_section}
A smooth map $\Omega^{\vir}:\cM^\pi_{2,0}\sqcup\overline\cM^\pi_{2,1}\to\lR\cup\{+\infty,-\infty\}$ is called a \textit{combinatorial obstruction section} if it satisfies the following properties
\begin{itemize}
\item for every non-split boundary point $u\in\partial\overline\cM_{2,1}^\pi$ we have $\Omega^{\vir}(u)=\Omega(u)$, 
\item for every split boundary point $u\in\partial\overline\cM_{2,1}^\pi$ we have $\Omega^{\vir}(u)=\Omega^{\vir}(u_0)$, where $u_0\in\cM_{2,0}^\pi$ is the annular part of $u$,
\item $\Omega^{\vir}(\cM^\pi_{2,0})\subset\lR\backslash\{0\}$,
\item $\Omega^{\vir}\pitchfork 0$.
\end{itemize}
\end{defi}
The \textit{virtual count} of annuli on $\Lambda$ with respect to a combinatorial obstruction section $\Omega^{\vir}$ is defined as the algebraic count of zeros of $\Omega^{\vir}$. To make this precise, we need to fix orientations on $\cM^{\pi}_{2,1}$. This is done in Section \ref{Section:Orientations}. The virtual count is determined by the values of $\Omega^{\vir}$ on $\cM^\pi_{2,0}$. In the definition of the invariant in Section \ref{Sec:Invariant_definition} and Section \ref{Section:Algebraic_definition}, instead of using the count of $J$-holomorphic annuli on $\lR\times\Lambda$, we can use the count of zeros of any combinatorial obstruction section. This allows us to compute the invariant combinatorially. 

\vspace{2.1mm}
Next, we discuss lemmas similar to Lemma \ref{Lemma:OSection_hyperbolic} and Lemma \ref{Lemma:OSection_elliptic} that will be used to show invariance under Reidemeister II move. First, we need the following.
\begin{lemma}
\label{Lemma:ReidII-local_form}
There exists a local model for Reidemeister II move consisting of Legendrian submanifolds $\Lambda_R,R\in\lR_{>0}\cup\{\infty\}$ with $\Lambda_R\overset{C^\infty}{\longrightarrow}\Lambda_\infty,R\to\infty$, such that there exist holomorphic embeddings $\varphi_R:(-R,R)\times[0,1]\to \lC$ with boundary on $\pi_{xy}(\Lambda_R)$ that satisfy 
$$\varphi_R\overset{C^\infty}{\to} \Gamma,R\to\infty,$$
where $\Gamma$ is the degenerate self-intersection of $\Lambda_\infty$.
\end{lemma}
\begin{proof}
For $R\in\lR,R>1$, let $l_1^R,l_2^R$ be the pair of concentric circles with centers at $-R^2i\in\lC$ and radii $R^2$ and $R^2-1$. Denote by $\psi_R:(-R,R)\times[0,1]\to\lC$ the holomorphic strip given by
\begin{align*}
\psi_R(s,t)=-R^2i\left(\left(\frac{R^2}{R^2-1}\right)^{i(s+it)}+1\right),
\end{align*}
with boundary on $l_1^R\cup l_2^R$. After conformal transformation $F:z\to\frac{1}{z}$, $l_1^R,l_2^R$ have the form as shown in Figure \ref{Fig:ReidII_model} near the point $\Gamma=(0,0)$, and $\varphi_R\coloneq F\circ\psi_R$ is given by
\begin{align*}
\varphi_R(s,t)=\frac{i}{R^2}\left(\left(\frac{R^2}{R^2-1}\right)^{i(s+it)}+1\right)^{-1}.
\end{align*} 
Now, it is not difficult to check $\varphi_R\overset{C^\infty}{\to} 0,R\to\infty$, i.e., $\varphi_R$ and all its partial derivatives uniformly converge to zero. Note that $l^R_1$ converges to the $x$-axis $l^\infty_1$ and $l^R_2$ to the circle $l^\infty_2$ with center at $\frac{i}{2}$ and radius $\frac{1}{2}$. Legendrians $\Lambda_R,\Lambda_\infty$ are obtained by lifting $l_1^R\sqcup l_2^R,l_1^\infty\sqcup l_2^\infty$ in a neighborhood of $\Gamma$.
\end{proof}
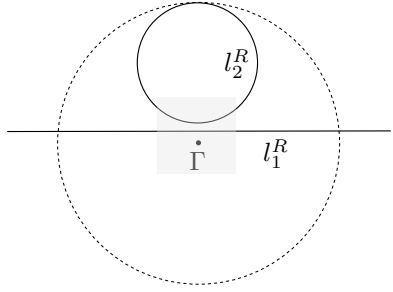
\begin{figure}
\def\svgwidth{51mm}
\begingroup%
  \makeatletter%
  \providecommand\rotatebox[2]{#2}%
  \newcommand*\fsize{\dimexpr\f@size pt\relax}%
  \newcommand*\lineheight[1]{\fontsize{\fsize}{#1\fsize}\selectfont}%
  \ifx\svgwidth\undefined%
    \setlength{\unitlength}{258.75124791bp}%
    \ifx\svgscale\undefined%
      \relax%
    \else%
      \setlength{\unitlength}{\unitlength * \real{\svgscale}}%
    \fi%
  \else%
    \setlength{\unitlength}{\svgwidth}%
  \fi%
  \global\let\svgwidth\undefined%
  \global\let\svgscale\undefined%
  \makeatother%
  \begin{picture}(1,0.73782721)%
    \lineheight{1}%
    \setlength\tabcolsep{0pt}%
    \put(0.66748998,0.32176753){\makebox(0,0)[lt]{\lineheight{1.25}\smash{\begin{tabular}[t]{l}$l_1^R$\end{tabular}}}}%
    \put(0.56546326,0.55641926){\makebox(0,0)[lt]{\lineheight{1.25}\smash{\begin{tabular}[t]{l}$l_2^R$\end{tabular}}}}%
    \put(0.47535851,0.29899319){\makebox(0,0)[lt]{\lineheight{1.25}\smash{\begin{tabular}[t]{l}$\Gamma$\end{tabular}}}}%
    \put(0,0){\includegraphics[width=\unitlength,page=1]{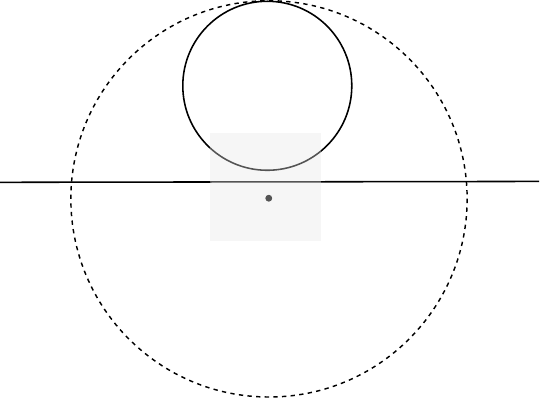}}%
  \end{picture}%
\endgroup%
\caption{Local model for Reidemeister II move.}
\label{Fig:ReidII_model}
\end{figure}
\begin{figure}
\def\svgwidth{106mm}
\begingroup%
  \makeatletter%
  \providecommand\rotatebox[2]{#2}%
  \newcommand*\fsize{\dimexpr\f@size pt\relax}%
  \newcommand*\lineheight[1]{\fontsize{\fsize}{#1\fsize}\selectfont}%
  \ifx\svgwidth\undefined%
    \setlength{\unitlength}{448.95957948bp}%
    \ifx\svgscale\undefined%
      \relax%
    \else%
      \setlength{\unitlength}{\unitlength * \real{\svgscale}}%
    \fi%
  \else%
    \setlength{\unitlength}{\svgwidth}%
  \fi%
  \global\let\svgwidth\undefined%
  \global\let\svgscale\undefined%
  \makeatother%
  \begin{picture}(1,0.27369349)%
    \lineheight{1}%
    \setlength\tabcolsep{0pt}%
    \put(0,0){\includegraphics[width=\unitlength,page=1]{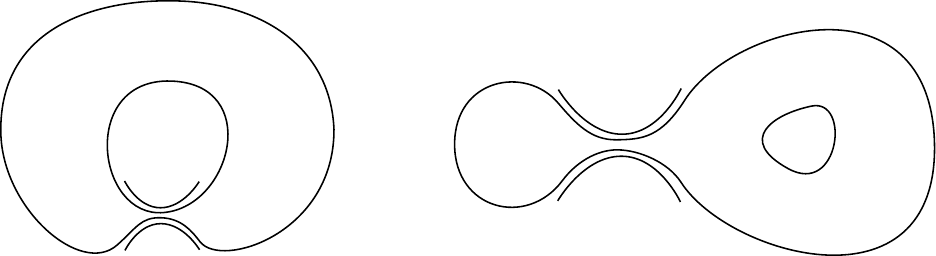}}%
  \end{picture}%
\endgroup%
\caption{Annulus pinching during Reidemeister II move.}
\label{Fig:annulus_pinch_reidII}
\end{figure}
\begin{lemma}\label{Lemma:OSection_ReidII}
Let $\Lambda_s,s\in[0,1]$ be a Legendrian knot isotopy that has the form given in Lemma \ref{Lemma:ReidII-local_form} in a neighborhood of $\Gamma\in\lC$ when $s\to1$, $s_n\in[0,1],n\in\lN$ a sequence such that $s_n\overset{n\to\infty}{\longrightarrow} 1$, and $u_n\in\cM^\pi_{2,0}(\Lambda_{s_n})$ be a sequence of holomorphic annuli with boundary on $\pi_{xy}(\Lambda_{s_n})$ that gets pinched as shown in Figure \ref{Fig:annulus_pinch_reidII}, left. Then
$$\lim_{n\to\infty}\Omega(u_n)=\pm\infty,$$
where the limit is $-\infty$ if the outer boundary passes through the overcrossing arc at the degenerate Reeb chord on $\Lambda_1$ at $\Gamma$, and $+\infty$ if it passes through the undercrossing arc.
\end{lemma}
\begin{proof}
The proof is similar to the proof of Lemma \ref{Lemma:OSection_hyperbolic}. Let $u_n:(\Sigma_n,j_n)\to(\lC,i),n\in\lN$ be a sequence of holomorphic annuli in $\cM^\pi_{2,0}(\Lambda_{s_n})$ as above and $\varphi_n:(-8R_n,8R_n)\times[0,1]\to\Sigma_n$ a sequence of conformal embeddings such that $\|u_n\circ \varphi_n-\Gamma\|_{C^2((-8R_n,8R_n)\times\{0,1\})}\overset{n\to\infty}{\longrightarrow} 0$ and $R_n\overset{n\to\infty}{\rightarrow}\infty$. Let $w_n$ denote the corresponding harmonic lifts and $\alpha_n,\beta_n$ the paths $t\to \varphi_n(-R_n+2,t),t\to \varphi_n(R_n-2,t),t\in [0,1]$. As before, we get 
\begin{align*}
\|w_n\circ \varphi_n-w^+\|_{C^1((-R_n,R_n)\times[0,1])}\overset{n\to\infty}{\longrightarrow}0
\end{align*}
or
\begin{align*}
\|w_n\circ \varphi_n-w^-\|_{C^1((-R_n,R_n)\times[0,1])}\overset{n\to\infty}{\longrightarrow}0
\end{align*}
for $w^+(s,t)=\pi_z\gamma(t),w^-(s,t)=\pi_z\overline\gamma(t)$, where $\gamma$ is the Reeb chord on $\Lambda_1$ corresponding to $\Gamma$ and $\overline\gamma(t)=\gamma(1-t)$. We denote $\widetilde\Sigma_n=\Sigma_n\backslash \varphi_n((-R_n+2,R_n-2)\times[0,1])$. The family of annuli degenerates into a disk $\overline u_0:\Sigma\to\lC$ with boundary on $\pi_{xy}(\Lambda_1)$ and additional positive and negative corner at $\Gamma$. Denote the harmonic lift of $\overline u_0$ by $w_0$. Similar as above, we get $ w_0\circ\alpha_n,w_0\circ\beta_n\overset{C^1([0,1])}{\longrightarrow}\pi_z\circ\gamma$ or $ w_0\circ\alpha_n,w_0\circ\beta_n\overset{C^1([0,1])}{\longrightarrow}\pi_z\circ\overline\gamma,n\to\infty$. Then, $(w_n-w_0)|_{\widetilde\Sigma_n}$ is a solution of the Dirichlet problem
\begin{align*}
&\Delta v=0,\\
&v|_{\partial\widetilde\Sigma_n}=\overline f_n,
\end{align*}
where $\max_{\partial\widetilde\Sigma_n}|\overline f_n|\overset{n\to\infty}{\longrightarrow} 0$, so the maximum principle implies $\max_{\widetilde\Sigma_n}|w_n-w_0|\overset{n\to\infty}{\longrightarrow} 0$.

Let $l_n\subset\Sigma_n$ be loops obtained by slightly pushing the outer boundary component of $\Sigma_n$ into the interior at the punctures. Fix a small enough neighborhood $U\subset\lC$ of $\Gamma$ and denote $l_1^n=l_n\cap(\Sigma_n\backslash U),l_2^n=l_n\cap U$. Similar as in Lemma \ref{Lemma:OSection_hyperbolic}, we get that
\begin{align*}
\left|\int_{l_1^n}\lambda\circ d(u_n,w_n)\circ j_n\right|
\end{align*}
are uniformly bounded,
\begin{align*}
\int_{l_2^n}\lambda\circ d(u_n,w_n)\circ j_n\overset{n\to\infty}{\longrightarrow}-\infty
\end{align*}
if the outer boundary passes through the overcrossing arc near $\Gamma$, and
\begin{align*}
\int_{l_2^n}\lambda\circ d(u_n,w_n)\circ j_n\overset{n\to\infty}{\longrightarrow}+\infty
\end{align*}
if the outer boundary passes through the undercrossing arc.
\end{proof}

\begin{lemma}\label{Lemma:OSection_ReidIIb}
Let $\Lambda_s,s\in[0,1]$ be a Legendrian knot isotopy that has the form given in Lemma \ref{Lemma:ReidII-local_form} in a neighborhood of $\Gamma\in\lC$ when $s\to 1$, $s_n\in[0,1],n\in\lN$ a sequence such that $s_n\overset{n\to\infty}{\longrightarrow} 1$, and $u_n\in\cM_{2,0}^\pi(\Lambda_{s_n})$ be a sequence of holomorphic annuli with boundary on $\pi_{xy}(\Lambda_{s_n})$ that gets pinched as shown in Figure \ref{Fig:annulus_pinch_reidII}, right. Denote the annular part of the degenerate annulus in the limit by $\overline u_0$. Then
$$\lim_{n\to\infty}\Omega(u_n)=\Omega(\overline u_0).$$
\end{lemma}
The proof goes similar to the proof of the previous lemma and Lemma \ref{Lemma:OSection_elliptic}.

\vspace{2.1mm}
Next, we show $J$-holomorphic annuli for generic $\Lambda$ are regular, i.e. transversally cut out by the Cauchy--Riemann operator (see for example \cite{mcduffsal04}).

\begin{lemma}\label{Lemma:obstruction_transversalityI}Let $\Lambda$ be a generic Legendrian knot. The obstruction section $\Omega:\cM^\pi_{2}\to\lR$ is transverse to the zero section if and only if all $J$-holomorphic annuli on $\lR\times\Lambda$ are regular.
\end{lemma}
\begin{proof}
We consider index zero annuli $\cM^\pi_{2,1}$, the proof goes the same in higher dimensions. Assume $\Omega$ is not transverse to the zero section at $\pi_{xy}\circ u_0\in\cM^\pi_{2,1}$ for some index zero annulus $u_0$. Denote by $v_t\in\cM^\pi_{2,1},t\in(-\varepsilon,\varepsilon)$ the holomorphic annuli in the neighborhood of $v_0=\pi_{xy}\circ u_0$. Fix a smooth family of embedded paths $\gamma_t,t\in(-\varepsilon,\varepsilon)$ on $v_t$ such that $\gamma_t$ starts on the outer and ends on the inner boundary component of $v_t$. Denote by $\widetilde v_t$ the holomorphic disks obtained by cutting $v_t$ along $\gamma_t$. We lift $\widetilde v_t$ to a smooth family of $J$-holomorphic disks $\widetilde u_t$ in $\lR^4$ with part of the boundary lifted from $\pi_{xy}(\Lambda)$ to $\lR\times\Lambda$ and the difference between the two branches lifted from $\gamma_t$ equal to $\Omega(v_t)\partial_r$. Since $\Omega'(v_0)=0$ and $\widetilde u_t$ are $J$-holomorphic, linearization of this family of disks gives us a well-defined smooth section along $u_0$ that is in the kernel of the linearized Cauchy--Riemann operator $D_{u_0}$ and is transverse to the $\lR$-translation direction in the kernel. The proof of the other direction goes similarly assuming that the moduli space $\cM^\pi_{2,1}$ is regular, which holds for generic $\Lambda$ similar as in \cite{EES07}.
\end{proof}
The following lemma now implies that index zero and index one $J$-holomorphic annuli on $\lR\times\Lambda$ with one positive puncture are regular for generic $\Lambda$. \begin{lemma}\label{Lemma:obstruction_transversalityII}
The obstruction section $\Omega:\cM^\pi_{2,k}\to\lR,k\in\{0,1,2\}$ is transverse to the zero section for generic $\Lambda$.
\end{lemma}
\begin{proof}
We consider $k=1$, the proof for $k\in\{0,2\}$ goes the same. Assume $\Lambda$ is in general position and there exists $u_0\in\cM^\pi_{2,1}$ such that $\Omega$ is not transverse to 0 at $u_0$. For $\Lambda_\tau,\tau\in[0,\varepsilon)$ a perturbation of $\Lambda$, we naturally identify $\cM^{\pi}_{2,1}(\Lambda_\tau)\cong\cM^{\pi}_{2,1}(\Lambda_0)=:\cM^{\pi}_{2,1}$ and consider the obstruction sections $\Omega_\tau:\cM_{2,1}^{\pi}\to\lR$. It is enough to find a perturbation $\Lambda_\tau$ such that $\frac{d}{d\tau}\Omega_\tau(u_0)|_{\tau=0}\neq 0$. Then, using Sard's theorem, we have $\Omega_{\tau_0}\pitchfork 0$ for generic $\tau_0\in[0,\varepsilon)$ in a small neighborhood of $u_0$ independent of $\tau_0$. Indeed, since $\frac{d}{d\tau}\Omega_\tau(u_0)|_{\tau=0}\neq 0$, $\Omega_\tau$ gives us a smooth foliation by sections on $U\times\lR$ for a small neighborhood $U$ of $u_0$. By looking at the local coordinates where the foliation is trivial, we see that a generic section in the foliation $\Omega_\tau$ is transverse to the zero section. 

The condition $\frac{d}{d\tau}\Omega_{\tau}(u_0)|_{\tau=0}\neq 0$ can be achieved as follows. Let $A_1,A_2\in\Lambda$ be two points that do not lie on $\pi_{xyz}u_0|_\partial$ such that the segment of $\Lambda$ from $A_1$ to $A_2$ intersects only one of the boundary components of $u_0$ (without loss of generality, the outer boundary component). Without this condition, the argument below still works similarly for generic $\Lambda$ for $A_1,A_2\not\in\pi_{xyz}u_0|_\partial$. We denote by $\widetilde S$ the segment of $\Lambda\cong\lR/\lZ$ from $A_1-\eta$ to $A_2+\eta$ for some $\eta>0$ small, and by $S$ segments of the outer boundary component of $u_0$ that map to $\widetilde S$. We construct a perturbation $\Lambda_\tau,\tau\in[0,\varepsilon)$ of $\Lambda_0=\Lambda$ by creating two "bumps" in the Lagrangian projection in small neighborhoods of $A_1,A_2$, such that $\Lambda_\tau$ is equal to $\Lambda$ away from $\widetilde S$ and $\Lambda_\tau=\Lambda+f_\tau\partial_z$ on $S$, where $f_\tau\in\lR$ is a constant. Since $\pi_{xyz}u_0|_\partial$ does not pass through $A_1,A_2$, holomorphic annulus $\pi_{xy}u_0$ can be seen as a holomorphic annulus on $\pi_{xy}(\Lambda_\tau)$, which we denote by $v_\tau$. Let $\widetilde u$ be the holomorphic universal cover of the domain of $v_\tau$ by a strip. Denote by $w_\tau$ the harmonic lift of $v_\tau\circ\widetilde u$ with boundary on $\Lambda_\tau$
\begin{align*}
w_\tau(s,t)=\frac{1}{2\pi}\int_{-\infty}^{+\infty}\left(P(\sigma-s,t)w_\tau(s,0)+P(\sigma-s,1-t)w_\tau(s,1)\right)d\sigma,s\in[0,l),t\in(0,1),
\end{align*}
where $P$ is given by (\ref{Eq:Pkernel}). Note that $w_\tau(s,0)=w_0(s,0)$ (inner boundary) since $\Lambda$ is not perturbed away from $\widetilde S$, and $w_\tau(s,1)=w_0(s,1)+\widetilde f_\tau(s)$, for $\widetilde f_\tau(s)$ a piecewise constant function (equal to $f_\tau$ on $S$ and zero otherwise). Then we have
\begin{align*}
\frac{d}{dt}w_\tau(s,1/2)&=\frac{d}{dt}w_0(s,1/2)-\frac{1}{2\pi}\int_{-\infty}^{+\infty}\partial_t P(\sigma-s,1/2)\widetilde f_\tau(\sigma)d\sigma=\\
&=\frac{d}{dt}w_0(s,1/2)+\frac{\pi}{2}\int_{-\infty}^{+\infty}\frac{1}{\left(\cosh\pi(\sigma-s)\right)^2}\widetilde f_\tau(\sigma)d\sigma,
\end{align*}
and
\begin{align*}
\frac{d}{d\tau}\Omega_\tau(v_\tau)=\int_{\widetilde l}\frac{d}{d\tau}\frac{d}{dt}w_\tau(s,1/2)ds=\frac{\pi}{2}\int_0^l\int_{-\infty}^{+\infty}\frac{1}{\left(\cosh\pi(\sigma-s)\right)^2}\frac{d}{d\tau}\widetilde f_\tau(\sigma)d\sigma ds,
\end{align*}
here $\widetilde l=\{t=1/2,s\in[0,l]\}$. For a suitable choice of "bumps" at $A_1,A_2$, this can clearly be made not-zero.
\end{proof}

\subsection{Generic asymptotic and relative asymptotic behavior}\label{Sec:generic_asymptotic_behavior_admissible_disks}

In this section, we discuss the asymptotic and the relative asymptotic behavior of $J$-holomorphic disks on $\lR\times\Lambda$.

For $u$ a $J$-holomorphic disk on $\lR\times\Lambda$, we define asymptotic representatives $\zeta_i=(\zeta^0_i,\zeta_i^1)$ of $u$ at each puncture $t_i$ as follows. The disk $u$ has a corner at $t_i$ at a self-intersection of $\pi_{xy}(\Lambda)$ that we denote by $k_i$. Let $\varphi_i:[0,+\infty)\times[0,1]\to\lD\backslash\{t_1,\dots,t_m\}$ be a holomorphic parameterization of a neighborhood of $t_i$ in $\lD^2\backslash\{t_1,\dots,t_m\}$, then we define 
\begin{align*}
\zeta_i^0=\lim_{s\to\infty}\frac{\pi_{xy}\circ u\circ\varphi_i(s,0)-k_i}{\|\pi_{xy}\circ u\circ\varphi_i(s,0)-k_i\|},\\
\zeta_i^1=\lim_{s\to\infty}\frac{\pi_{xy}\circ u\circ\varphi_i(s,1)-k_i}{\|\pi_{xy}\circ u\circ\varphi_i(s,1)-k_i\|}.
\end{align*}
Note that $\pi_{xy}\circ u\circ\varphi_i(s,\iota)\neq k_i$ for $s$ large enough, $\iota\in\{0,1\}$. 

Now we can define the notion of generic asymptotic behavior for a $J$-holomorphic disk $u$ on $L=\lR\times\Lambda$. We assume Legendrian knot $\Lambda$ is in general position. 
\begin{defi}\label{Def:generic_asym_beh}
We say a $J$-holomorphic disk $u$ on $L$ has \textit{generic asymptotic behavior} at a puncture $t_i$ if there exists a neighborhood $U_i\subset\lD\backslash\{t_1,\dots,t_m\}$ of $t_i$ such that $\pi_{xy}\circ u|_{U_i}$ is a bijection to a quadrant in $\lC$ at the corresponding self-intersection of $\pi_{xy}(\Lambda)$ and if $(-\zeta_i^0,\zeta_i^1)$ forms a positively oriented basis in $\lC$.
\end{defi}
Assume $u$ has generic asymptotic behavior at each puncture and let $t_j,t_l$ be two punctures on $u$ both positively or negatively asymptotic to some Reeb chord $\gamma_{k_{j,l}}$. Then, we define the relative asymptotic representative $\zeta_{j,l}=(\zeta_{j,l}^0,\zeta_{j,l}^1)$ for $t_j,t_l$ as follows. If small neighborhoods of $t_j,t_l$ occupy different quadrants at $k_{j,l}=\pi_{xy}(\gamma_{k_{j,l}})\in\lC$ in the Lagrangian projection, i.e. if $\zeta_j=-\zeta_l$, we define $(\zeta_{j,l}^0,\zeta_{j,l}^1)=(\zeta^0_j,\zeta^1_j)$. Otherwise, if $\zeta_j=\zeta_l$, we take a holomorphic parameterization $[0,+\infty)\times[0,1]$ of the corresponding quadrant and lift it to parameterizations $\varphi_j,\varphi_l:[0,+\infty)\times[0,1]\to\lD\backslash\{t_1,\dots,t_k\}$ of some neighborhoods of $t_j$ and $t_l$ in $\lD\backslash\{t_1,\dots,t_m\}$. Assume that each of the functions
$$F_0(s)\coloneq\pi_r\circ u\circ\varphi_j(s,0)-\pi_r\circ u\circ\varphi_l(s,0)$$ and $$F_1(s)\coloneq\pi_r\circ u\circ\varphi_j(s,1)-\pi_r\circ u\circ\varphi_l(s,1)$$ 
is either everywhere positive or everywhere negative for $s$ large enough. Then we define the relative asymptotic representative for $t_j,t_l$ by taking $\zeta_{j,l}^\iota=\operatorname{sgn}(F_\iota)\zeta_j^\iota,\iota\in\{0,1\}$.

Now we can define the notion of generic relative asymptotic behavior for a $J$-holomorphic disk $u$ on $L$. 
\begin{defi}\label{Def:generic_rel_asym_beh}
We say a $J$-holomorphic disk $u$ on $L$ with generic asymptotic behavior has \textit{generic relative asymptotic behavior} at punctures $t_j,t_l$ as above if $(-\zeta_{j,l}^0,\zeta_{j,l}^1)$ forms a positively oriented basis in $\lC$.
\end{defi}

\begin{lemma}\label{Lemma:index_zero_gen_asym_beh}
For $\Lambda$ a generic Legendrian knot, all index zero $J$-holomorphic disks on $\lR\times\Lambda$ have generic asymptotic behavior.
\end{lemma}
\begin{proof}
The proof follows easily from the fact that the projection of an index zero disk is an immersed polygon with convex corners at the punctures.

Moreover, for a generic $J$-holomorphic curve $u$ of any index, all corners in the projection are convex, therefore, $u$ has generic asymptotic behavior. Having non-generic asymptotic behavior at precisely $k$ punctures appears for a codimension $k$ subset of the moduli space.
\end{proof}

\begin{lemma}\label{Lemma:index_zero_gen_rel_asym_beh}
For $\Lambda$ a generic Legendrian knot, all index zero $J$-holomorphic disks on $\lR\times\Lambda$ have generic relative asymptotic behavior.
\end{lemma}
\begin{proof}
Let $u$ be an index zero $J$-holomorphic disk on $\lR\times\Lambda$ and $t_1,t_2$ two punctures on $u$ negatively asymptotic to $\gamma_i$. For $\Lambda$ generic, $u$ has generic asymptotic behavior at $t_1,t_2$. If the neighborhoods of the two punctures map to different quadrants at $\gamma_i$ in the Lagrangian projection, then $u$ has generic relative asymptotic behavior. Otherwise, let $\varphi:[0,\infty)\times [0,1]\to\lC$ be a holomorphic parameterization of the corresponding quadrant at $\gamma_i$ and $\varphi_1,\varphi_2:[0,\infty)\times [0,1]\to\lD$ parameterizations of some neighborhoods of $t_1,t_2$ obtained by lifting $\varphi$. By \cite[Lemma 7.1.]{lifting_Rizell}, there exist $c_1,c_2\in\lR,\lambda>0$ and $R_0\geq 0$ such that 
\begin{align*}
\|\pi_r\circ u\circ\varphi_1(s,t)+c_1+l(\gamma_i)s\|<e^{-\lambda s},\\
\|\pi_r\circ u\circ\varphi_2(s,t)+c_2+l(\gamma_i)s\|<e^{-\lambda s},
\end{align*}
for all $(s,t)\in[R_0,\infty)\times[0,1]$, where $l(\gamma_i)$ is the length of the Reeb chord $\gamma_i$. Moreover, for generic $\Lambda$ we have $c_1\neq c_2$. The proof of this follows similar to the proof of Lemma \ref{Lemma:obstruction_transversalityII}. This implies that $\pi_r\circ u\circ\varphi_1(s,t)-\pi_r\circ u\circ\varphi_2(s,t)$ is either everywhere positive or everywhere negative for $t\in\{0,1\}$ and $s$ large enough, from which the claim follows.
\end{proof}

\subsection{Coherent orientations}
\label{Section:Orientations}

Let $\Lambda\subset\lR^3$ be a Legendrian knot. We describe a combinatorial way to construct orientations on the moduli spaces of $J$-holomorphic disks and annuli on $\lR\times\Lambda$ of all dimensions. The orientations should additionally agree with gluing and string operations that will be defined in Section \ref{Sec:Invariant_definition}, i.e. they should satisfy certain coherency conditions (see Section \ref{Sec:orientation_and_algebraic_signs}.) 

The idea is to first define orientations on the moduli spaces of holomorphic curves in the Lagrangian projection. Then, these orientations are lifted to the moduli spaces $\cM_1,\cM_2$ of $J$-holomorphic disks and annuli on $\lR\times\Lambda$, using the obstruction section for annuli. Another more general way to construct orientations is by orienting Fredholm operators \cite{FloerHofer93,BMOrient04,EES05,FOOOII}, the two constructions are equivalent. The combinatorial construction gives us a more computable approach but is only applicable to $\lR^3$. Similar combinatorial construction for the moduli space of disks appears in \cite{ENS02,EES05}.

\subsubsection{Orienting the moduli space of disks}\label{Section:Orientations_disks}
First, we define orientations on the moduli space $\cM_1$ of $J$-holomorphic disks on $\lR\times\Lambda$. Let $\{\gamma_1,\dots,\gamma_n\}$ be the set of Reeb chords on $\Lambda$. As before, we denote by $\cM_1^\pi$ the moduli space of holomorphic disks in $\lC$ with boundary on $\pi_{xy}(\Lambda)$ and corners at the self-intersections of $\pi_{xy}(\Lambda)$. Fix an orientation on $\Lambda$ and define signs at each quadrant as in Figure \ref{Fig:final_signs}.

\begin{figure}
\def\svgwidth{66mm}
\begingroup%
  \makeatletter%
  \providecommand\rotatebox[2]{#2}%
  \newcommand*\fsize{\dimexpr\f@size pt\relax}%
  \newcommand*\lineheight[1]{\fontsize{\fsize}{#1\fsize}\selectfont}%
  \ifx\svgwidth\undefined%
    \setlength{\unitlength}{359.60308548bp}%
    \ifx\svgscale\undefined%
      \relax%
    \else%
      \setlength{\unitlength}{\unitlength * \real{\svgscale}}%
    \fi%
  \else%
    \setlength{\unitlength}{\svgwidth}%
  \fi%
  \global\let\svgwidth\undefined%
  \global\let\svgscale\undefined%
  \makeatother%
  \begin{picture}(1,0.40470379)%
    \lineheight{1}%
    \setlength\tabcolsep{0pt}%
    \put(0.09066555,0.18857455){\makebox(0,0)[lt]{\lineheight{1.25}\smash{\begin{tabular}[t]{l}$+1$\end{tabular}}}}%
    \put(0.23692899,0.18857286){\makebox(0,0)[lt]{\lineheight{1.25}\smash{\begin{tabular}[t]{l}$-1$\end{tabular}}}}%
    \put(0.1674483,0.27078066){\makebox(0,0)[lt]{\lineheight{1.25}\smash{\begin{tabular}[t]{l}$+1$\end{tabular}}}}%
    \put(0.16744673,0.1211127){\makebox(0,0)[lt]{\lineheight{1.25}\smash{\begin{tabular}[t]{l}$+1$\end{tabular}}}}%
    \put(0,0){\includegraphics[width=\unitlength,page=1]{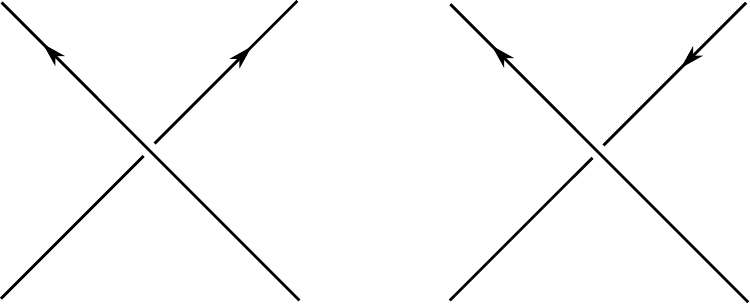}}%
    \put(0.68976338,0.18614019){\makebox(0,0)[lt]{\lineheight{1.25}\smash{\begin{tabular}[t]{l}$+1$\end{tabular}}}}%
    \put(0.83602598,0.1861385){\makebox(0,0)[lt]{\lineheight{1.25}\smash{\begin{tabular}[t]{l}$+1$\end{tabular}}}}%
    \put(0.76654586,0.2683463){\makebox(0,0)[lt]{\lineheight{1.25}\smash{\begin{tabular}[t]{l}$+1$\end{tabular}}}}%
    \put(0.76654465,0.11867834){\makebox(0,0)[lt]{\lineheight{1.25}\smash{\begin{tabular}[t]{l}$-1$\end{tabular}}}}%
  \end{picture}%
\endgroup%
\caption{Signs at the quadrants near an odd and an even self-intersection of $\pi_{xy}(\Lambda)$.}
\label{Fig:final_signs}
\end{figure}
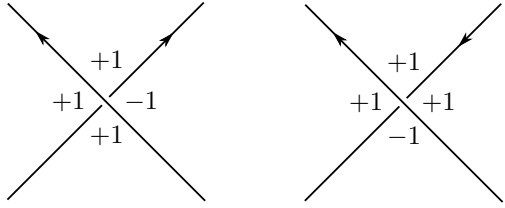

To define orientations on $\cM^\pi_1$, we use the fact that every $u_0\in\cM^\pi_1$ is uniquely determined by the image of its boundary and its interior branch points. Denote by $\widetilde\cA^\infty$ the tensor algebra generated by $q_i,p_i,t^\pm,i\in\{1,\dots,n\}$ with relation $t^+t^-=1=t^-t^+$. The space $\widetilde\cA^\infty$ is generated by words in $q_i,p_i,t^\pm$ as a vector space. We additionally define a grading on $\widetilde\cA^\infty$ by taking $|q_i|=\mu_{CZ}(\gamma_i),|p_i|=-\mu_{CZ}(\gamma_i)-1,|t^\pm|=\mp 2\operatorname{rot}(\Lambda)$. Fix a word $\omega=s_1\dots s_r\in\widetilde\cA^\infty$ with $s_j=t^{\pm},q_{i_j}$ or $p_{i_j}$ for $i_j\in\{1,\dots,n\},j\in\{1,\dots,r\}$. We consider unparameterized oriented loops on $\pi_{xy}(\Lambda)$ with a marked point $\alpha$ and corners at the self-intersections of $\pi_{xy}(\Lambda)$. We say a loop has a convex corner if it turns left at the corresponding crossing of $\pi_{xy}(\Lambda)$, see Figure \ref{Fig:turning_left}. Denote by $\cC_1(\omega)$ the space of loops as above, such that the order of the corners and the crossings over the base point on $\Lambda$ corresponds to the word $\omega$. More precisely, for $s_j=t^+$ ($s_j=t^-$) we have a positive (negative) crossing over the base point, and for $s_j=p_{i_j}$ ($s_j=q_{i_j}$) we have a positive (negative) corner at $\pi_{xy}(\gamma_{i_j})$. The topology on $\cC_1(\omega)$ is determined by the image of the branch (singular) points, where additionally a branch point can disappear at a corner, creating a non-convex corner, and appear on the other arc (see Figure \ref{Fig:orientation_extends}). The subspace of loops in $\cC_1(\omega)$ with $k$ non-convex corners is of codimension $k$.
 
Take an arbitrary loop $l\in\cC_{1}(\omega)$ such that $l$ has only convex corners and denote by $k\in\lN_0$ the number of the branch points. Label the branch points by $a_1,\dots,a_k$ in the order starting from the marked point and let $n(a_i)=n_i\in\{0,\dots,r\}$ such that $a_i$ lays on the arc between $s_{n_i}$ and $s_{n_{i+1}}$ for all $i\in\{1,\dots,k\}$ ($s_{n_0}$ and $s_{n_{r+1}}$ correspond to the marked point $\alpha$). We can parameterize the neighborhood of $l\in\cC_{1}$ by the images of its branch points. Denote by $v_i$ the tangent vector in $T_l\cC_{1}$ obtained by moving the branch point $a_i$ forward, i.e. in the direction of the loop $l$ right before the branch point. 

We define an orientation $\sigma(l)$ of $T_l\cC_{1}$ by taking
\begin{align*}
\sigma(l)=\epsilon(\alpha)\left(\prod_{i=1}^k(-1)^{ \sum_{j=1}^{n_i}|s_j|}\right)\left(\prod_{i=1}^r\epsilon_i\right)\langle v_1,\dots,v_k\rangle,
\end{align*}
where $\epsilon_i$ is the sign at the $i^{th}$ corner of $l$ (as shown in Figure \ref{Fig:final_signs}) and 1 if $s_i=t^\pm$, $\epsilon(\alpha)=+ 1$ if $l$ has positive orientation near the marked point $\alpha$ with respect to the orientation on $\Lambda$ and $-1$ otherwise.

The orientation above does not depend on the marked point, only on the ordering of the punctures. Therefore, two loops are seen as equivalent if one is obtained from the other by moving the marked point on the arc between the first and the last corner. Indeed, if $l_2\in\cC_1(\omega)$ is obtained from $l_1\in\cC_1(\omega)$ by crossing the first branch point with the marked point, then the corresponding signs $\epsilon(\alpha_1),\epsilon(\alpha_2)$ differ, but the ordering of the branch points changes and the orientation remains the same. More precisely, we have 
\begin{align*}
\sigma(l_2)&=\epsilon(\alpha_2)(-1)^{\sum_{j=1}^{r}|s_j|}\left(\prod_{i=2}^k(-1)^{ \sum_{j=1}^{n_i}|s_j|}\right)\left(\prod_{i=1}^r\epsilon_i\right)\langle v_2\dots,v_k,v_1\rangle=\\
&=-\epsilon(\alpha_1)(-1)^{k}\left(\prod_{i=2}^k(-1)^{ \sum_{j=1}^{n_i}|s_j|}\right)\left(\prod_{i=1}^r\epsilon_i\right)\langle v_2\dots,v_k,v_1\rangle=\\
&=\sigma(l_1).
\end{align*}
Here we use $n^{l_1}(a_1^{l_1})=0,n^{l_2}(a_1^{l_1}=a_{k}^{l_2})=r$ and $k+\sum_{j=1}^{r}|s_j|\equiv 0\pmod 2$. The latter follows from the fact that the orientation sign changes as we go along $l$ precisely $k+\sum_{j=1}^r|s_j|$ many times.

Next, we show that these orientations extend over the loops with non-convex negative and positive corners. We consider the first case shown in Figure \ref{Fig:orientation_extends}, other cases follow analogously. Denote the string on the left by $l_1$ and the string on the right by $l_2$. Since one branch point crosses over an odd negative end and the corner does not change the sign, we have $\sigma(l_1)=\pm\langle\dots,v_i,\dots\rangle$ and $\sigma(l_2)=\mp\langle\dots,\widetilde v_i,\dots\rangle$. These orientations can be glued over the string with the non-convex corner.

Define $\cC_1(-\omega)$ as the space $\cC_1(\omega)$ with the opposite orientation. We define the space of cyclic words $\widetilde\cA^{\infty,\cyc}$ in $q_i,p_i,t^\pm$ as the quotient of $\widetilde\cA^\infty$ by the vector subspace generated by $xy-(-1)^{|x||y|}yx,x,y\in\widetilde\cA^\infty$. Let $\omega_1,\omega_2\in\widetilde\cA^\infty$ be two words such that $\omega_1=\pm\omega_2$ when seen as cyclic words. Then there is a bijection between $\cC_1(\omega_1)$ and $\cC_1(\pm\omega_2)$ obtained by moving the marked point. Moreover, it is not difficult to verify that this identification preserves the orientation. This gives us well-defined orientations on the space $\cC_1^{\cyc}(\omega)$ of unparameterized oriented loops with no marked point for any cyclic word $\omega\in\widetilde\cA^{\infty,\cyc}$ that is not bad, i.e. such that $\omega\neq -\omega$ as a cyclic word. This is important later when we define orientations for annuli.

This gives us orientations on the moduli space of index zero and index one $J$-holomorphic disks by lifting from the Lagrangian projection. Note that the Lagrangian projection of an index zero $J$-holomorphic disk has no branch points, while for index one disks it has one boundary branch point.
\vspace{2.1mm}

Finally, we define orientation signs $\epsilon(u)\in\{+1,-1\}$ for index zero $J$-holomorphic disks $u$ on $\lR\times\Lambda$. For $u\in\cM_1$ an index zero disk with one positive puncture, we define
\begin{align*}
\epsilon(u)=\epsilon_1\prod\epsilon_\iota^u,
\end{align*}
where $\epsilon_1$ is the orientation sign at a marked point right after the positive puncture, and $\epsilon_\iota^u$ are the signs at the corners of $\pi_{xy}\circ u$ as shown in Figure \ref{Fig:final_signs}. For an index zero disk $v\in\cM_1$ with two positive punctures $t^1,t^2$, we define
\begin{align*}
\epsilon(v,t^1,t^2)=\epsilon_2\prod\epsilon_\iota^v,
\end{align*}
where $\epsilon_2$ is the orientation sign at a marked point right after the positive puncture $t^2$, and $\epsilon_\iota^v$ are the signs at the corners of $\pi_{xy}\circ v$. The sign can depend on the choice of the order of $t^1$ and $t^2$.

\begin{figure}
\def\svgwidth{126mm}
\begingroup%
  \makeatletter%
  \providecommand\rotatebox[2]{#2}%
  \newcommand*\fsize{\dimexpr\f@size pt\relax}%
  \newcommand*\lineheight[1]{\fontsize{\fsize}{#1\fsize}\selectfont}%
  \ifx\svgwidth\undefined%
    \setlength{\unitlength}{707.28431003bp}%
    \ifx\svgscale\undefined%
      \relax%
    \else%
      \setlength{\unitlength}{\unitlength * \real{\svgscale}}%
    \fi%
  \else%
    \setlength{\unitlength}{\svgwidth}%
  \fi%
  \global\let\svgwidth\undefined%
  \global\let\svgscale\undefined%
  \makeatother%
  \begin{picture}(1,0.61994416)%
    \lineheight{1}%
    \setlength\tabcolsep{0pt}%
    \put(0.22323124,0.57421941){\makebox(0,0)[lt]{\lineheight{1.25}\smash{\begin{tabular}[t]{l}$v$\end{tabular}}}}%
    \put(0.62585331,0.43376373){\makebox(0,0)[lt]{\lineheight{1.25}\smash{\begin{tabular}[t]{l}$\widetilde v$\end{tabular}}}}%
    \put(0.22430238,0.34820477){\makebox(0,0)[lt]{\lineheight{1.25}\smash{\begin{tabular}[t]{l}$\sigma_1=\pm \langle v\rangle$\end{tabular}}}}%
    \put(0.6294587,0.3478513){\makebox(0,0)[lt]{\lineheight{1.25}\smash{\begin{tabular}[t]{l}$\sigma_2=\mp \langle \widetilde v\rangle$\end{tabular}}}}%
    \put(0.22512622,0.23734689){\makebox(0,0)[lt]{\lineheight{1.25}\smash{\begin{tabular}[t]{l}$v$\end{tabular}}}}%
    \put(0.6256057,0.09582028){\makebox(0,0)[lt]{\lineheight{1.25}\smash{\begin{tabular}[t]{l}$\widetilde v$\end{tabular}}}}%
    \put(0.22182662,0.00490545){\makebox(0,0)[lt]{\lineheight{1.25}\smash{\begin{tabular}[t]{l}$\sigma_1=\pm \langle v\rangle$\end{tabular}}}}%
    \put(0.63895934,0.00459503){\makebox(0,0)[lt]{\lineheight{1.25}\smash{\begin{tabular}[t]{l}$\sigma_2=\mp \langle \widetilde v\rangle$\end{tabular}}}}%
    \put(0.29714124,0.50994986){\makebox(0,0)[lt]{\lineheight{1.25}\smash{\begin{tabular}[t]{l}$\gamma$\end{tabular}}}}%
    \put(0.29604496,0.17008177){\makebox(0,0)[lt]{\lineheight{1.25}\smash{\begin{tabular}[t]{l}$\gamma$\end{tabular}}}}%
    \put(0.26717029,0.54989548){\makebox(0,0)[lt]{\lineheight{1.25}\smash{\begin{tabular}[t]{l}$+$\end{tabular}}}}%
    \put(0.66537836,0.46099527){\makebox(0,0)[lt]{\lineheight{1.25}\smash{\begin{tabular}[t]{l}$+$\end{tabular}}}}%
    \put(0.66411217,0.12652374){\makebox(0,0)[lt]{\lineheight{1.25}\smash{\begin{tabular}[t]{l}$-$\end{tabular}}}}%
    \put(0.26797813,0.21353584){\makebox(0,0)[lt]{\lineheight{1.25}\smash{\begin{tabular}[t]{l}$+$\end{tabular}}}}%
    \put(0,0){\includegraphics[width=\unitlength,page=1]{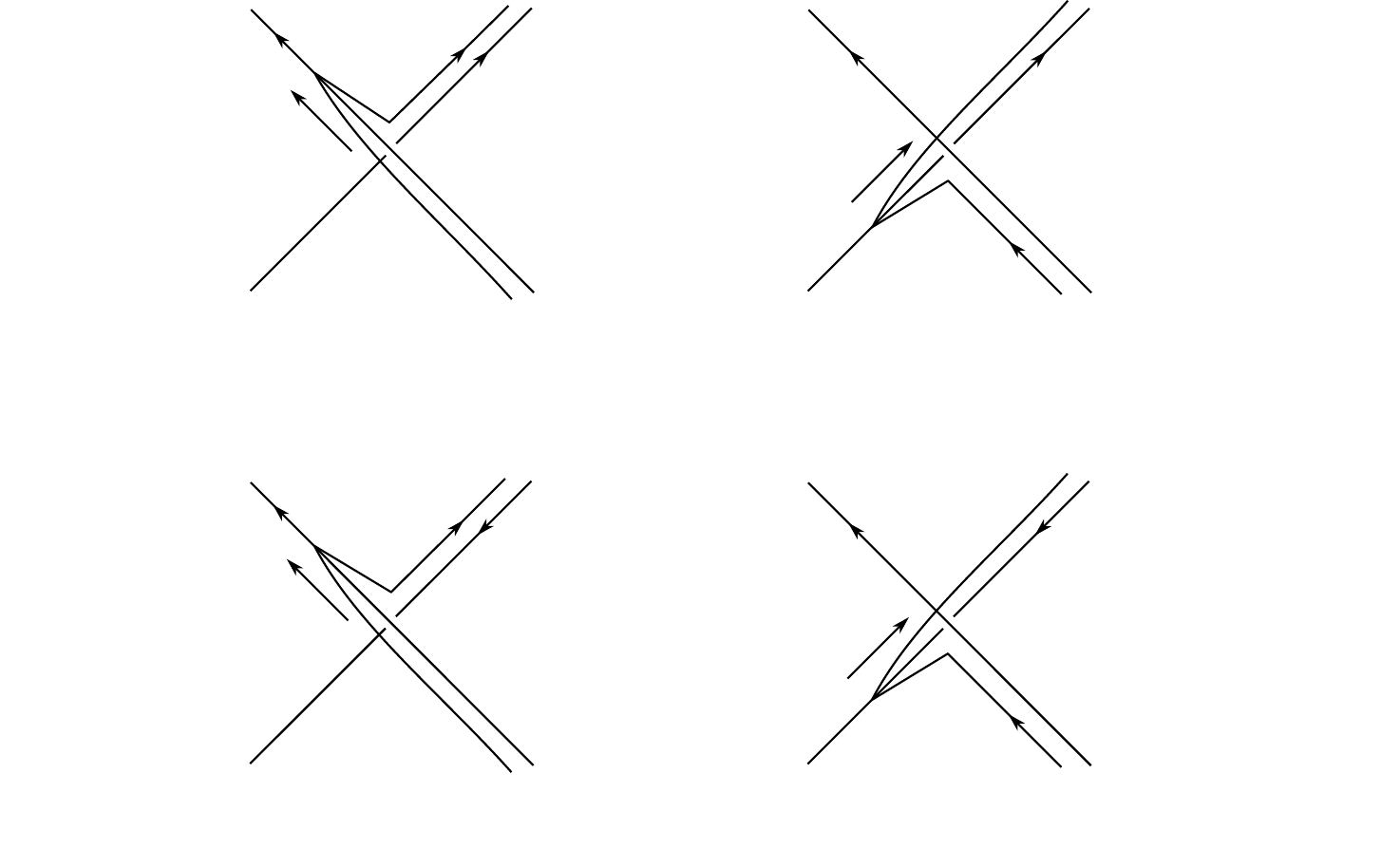}}%
    \put(-0.00101091,0.3165292){\makebox(0,0)[lt]{\lineheight{1.25}\smash{\begin{tabular}[t]{l}\textcolor{white}{.}\end{tabular}}}}%
  \end{picture}%
\endgroup%
\caption{The orientations extend over loops with non-convex corners, case of an odd and an even Reeb chord.}
\label{Fig:orientation_extends}
\end{figure}
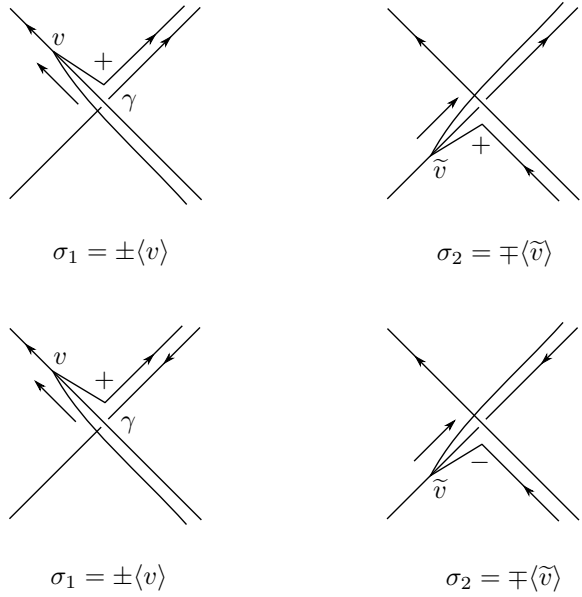
\vspace{2.1mm}

\subsubsection{Orienting the moduli space of annuli}
\label{Section:Orientations_annuli}
Next, we orient the moduli space $\cM_2$ of $J$-holomorphic annuli on $\lR\times\Lambda$. Denote by $\cM_2^\pi$ the moduli space of holomorphic annuli in $\lC$ with corners and boundary on $\pi_{xy}(\Lambda)$ as before. We first orient the space $\cM^\pi_2$, and then use the obstruction section $\Omega:\cM^\pi_2\to\lR$ to get orientations on $\cM_2\cong\Omega^{-1}(0)$.

Given two words $\omega_1=s_1^1\dots s_{r_1}^1,\omega_2=s_1^2\dots s_{r_2}^2$ in $q_i,p_i,t^\pm,i\in\{1,\dots,n\}$ as before, we denote $\cC_2(\omega_1,\omega_2)=\cC_1(\omega_1)\times\cC_1(\omega_2)$. Let $l=(l_1,l_2)$ be an element of $\cC_2$ with convex corners. We can parameterize the neighborhood of $l\in\cC_2$ by the images of its branch points. Denote the branch points on $l_\iota,\iota\in\{1,2\}$ by $a_1^\iota,\dots,a_{k_\iota}^\iota$ in the order starting from the marked point $\alpha_\iota$, and by $v_i^\iota$ the tangent vector in $T_l\cC_2$ corresponding to $a_i^\iota$ as before. Then we define an orientation $\sigma(l_1,l_2)$ at $l$ by taking
\begin{align*}
&\sigma(l_1,l_2)=\sigma(l_1)\#\sigma(l_2)=\\
&=\epsilon(\alpha_1)\epsilon(\alpha_2)\left(\prod_{i=1}^{k_1}(-1)^{ \sum_{j=1}^{n_i^1}|s_j^1|}\right)\left(\prod_{i=1}^{k_2}(-1)^{ \sum_{j=1}^{n_i^2}|s_j^2|}\right)\left(\prod_{i=1}^{r_1}\epsilon_i^1\right)\left(\prod_{i=1}^{r_2}\epsilon_i^2\right)\langle v_1^1,\dots,v_{k_1}^1,v_1^2,\dots,v_{k_2}^2\rangle,
\end{align*}
with the same notation as before. Independence from the choice of the marked points and gluing over strings with non-convex corners goes as before.

As before, we use these orientations on $\cC_2$ to orient the moduli space $\cM^\pi_2$. To orient the moduli space $\cM_2$ of $J$-holomorphic annuli on $\lR\times\Lambda$, we use the following simple observation. For $N,M$ oriented manifolds, $K\subset M$ an oriented submanifold and $f:N\to M$ smooth such that $f\pitchfork K$, there is a canonical way to orient $f^{-1}(K)\subset N$. We have a homeomorphism $\cM_2\cong\Omega^{-1}(0)$ for $\Omega:\cM_2^\pi\to\lR$ obstruction section defined in Section \ref{Sec:obstruction_section}. Additionally, $\Omega\pitchfork 0$ for generic $\Lambda$. Then the chosen orientations on $\cM^\pi_2$ give us orientations on $\cM_2$.

To orient the moduli space of index one $J$-holomorphic annuli on $\lR\times\Lambda$, we also need to take into account the potential branch point in the interior. The same holds for orienting higher dimensional moduli spaces of disks and annuli. Interior branch points increase the dimension of the moduli space by two, and the extra dimensions have a canonical orientation coming from the complex structure on $\lC$. Index 1 $J$-holomorphic annuli generically have either 2 boundary branch points or one interior branch point of degree 2 in the Lagrangian projection. In the (interior of the) moduli space of $J$-holomorphic annuli, two boundary branch points can disappear and an interior branch point can appear. From the definition, one can directly see that the orientations on the two sides can be glued. More precisely, we use the fact that the disappearing of the boundary branch points is modeled by the swallowtail singularity and that the orientations $\sigma(l)$ are stable under the operation of removing or adding two neighboring branch points. This gives us orientations on the moduli space $\cM_{2,1}$ of index one $J$-holomorphic annuli. Similarly, we get orientations on the moduli spaces of disks and annuli of higher dimensions.
\vspace{2.1mm}

The orientations defined above determine an orientation sign $\epsilon(u,e_2)\in\{+1,-1\}$ for any index zero annulus $u\in\cM_2$ with one positive puncture and $e_2$ a marked point on the inner boundary component of $u$. More precisely, for $\overline w(u,e_2)$ the pair of words obtained by looking at the boundary of $u$ with the order determined by the positive puncture and $e_2$, we have a corresponding sign $\epsilon(u,e_2)$ determined as follows. If the orientation of $\cC_2(\overline w(u,e_2))$ is given by $\sigma\langle v\rangle,\sigma\in\{+1,-1\}$ at $\pi_{xy}u|_\partial$ (where $v$ is the tangent vector corresponding to the boundary branch point of $\pi_{xy}(u)$), and $\nabla\Omega(\pi_{xy}(u))=\epsilon cv$ for some $c>0,\epsilon\in\{+1,-1\}$, then we define 
$$\epsilon(u,e_2)=\epsilon\sigma.$$ 
In other words, the sign is equal to 1 if the orientation vector of $\cM_{2,1}^\pi$ at $\pi_{xy}(u)$ points from the area with $\Omega<0$ to the area with $\Omega>0$, and $-1$ otherwise. Each $\epsilon(u,e_2)$,$\overline w(u,e_2)$ can depend on the choice of $e_2$, but $\epsilon(u,e_2)\overline w(u,e_2)$ seen as a tensor product of a word and a cyclic word does not.
\vspace{2.1mm}

There is a standard way to orient moduli spaces of $J$-holomorphic curves by orienting linear Cauchy--Riemann operators \cite{FloerHofer93,BMOrient04,EES05,FOOOII}. We give a quick overview for the moduli spaces of disks and annuli. Our combinatorial approach should correspond to orientations obtained this way for some suitable choices that are fixed along the way. 

For $D:V\to W$ a Fredholm operator, determinant of $D$ is defined as $\det D=\wedge^{\operatorname{top}} \operatorname{coker}(D)^*\otimes \wedge^{\operatorname{top}}\ker(D)$. Orientation of $D$ is a choice of an orientation of $\det D$. First, we consider the case of curves with no punctures. Let $L^n\subset M^{2n}$ be a Lagrangian submanifold with a fixed choice of relative spin structure. For $\beta\in\pi_2(M,L)$, denote by $\cM_1(\beta)$ the moduli space of $J$-holomorphic disks and by $\cM_2(\beta)$ the moduli space of $J$-holomorphic annuli with boundary on $L$ of class $\beta$. To orient moduli spaces $\cM_1(\beta),\cM_2(\beta)$, we orient linearized Cauchy--Riemann operators $D_u\overline\partial$ for $u\in\cM_1(\beta)$ and $u\in\cM_2(\beta)$.

\noindent There is a canonical way to orient the Dolbeault operator $\overline\partial_{(E,\lambda)}$ for $E$ any complex bundle over a disk or an annulus with totally real boundary subbundle $\lambda$ with a fixed trivialization (in case of an annulus, the boundary components should be ordered), see \cite{FOOOII}. In short words, the orientation is obtained by decomposing the operator into operators on the trivial bundles over disks and an operator over a closed surface. The operator over the closed surface has a canonical orientation induced by the complex structure, and the orientations of the operators over the disks are determined by the trivializations of $\lambda$ over the boundary components \cite{FOOOII}. Index $\operatorname{ind}=n-3+\mu$ of the operator over a disk is odd for $n=2$, therefore, changing the order of the boundary components of the annulus changes the orientation.

\noindent To orient the moduli spaces $\cM_1(\beta),\cM_2(\beta)$, we fix some $J$-holomorphic curves $u_i$ in each homotopy class $\pi_2(M,L)$ and orient $D_{u_i}\overline\partial$ as above (taking an isotopy of the zeroth order term to zero through a contractible space) using the trivializations over the boundary induced by the relative spin structure on $L$. The orientation of $\det D_u\overline\partial$ for any other $u$ in the same connected component as $u_i$ is then determined by choosing a path from $u$ to $u_i$ and showing that the orientation does not depend on the choice.

Orienting the moduli space of curves with boundary punctures in the symplectization goes similarly. Here we additionally first fix an orientation of an operator on a disk with precisely one positive puncture asymptotic to $\gamma$ for each Reeb chord $\gamma$, which also canonically determines an orientation of an operator on a disk with precisely one negative puncture at $\gamma$. Then we orient $\det D_u\overline\partial$ for any punctured curve $u$ by decomposing the operator into these and an operator on a curve without punctures. For more details see \cite{FloerHofer93,BMOrient04,EES05,FOOOII}. We note that we have a natural ordering of the boundary connected components of annuli since we have precisely one positive puncture.

\section{The chain complex on the space of strings}\label{Sec:Invariant_definition}
In this section we define a chain complex $(\cC(\Lambda),d)$ associated to a Legendrian knot $\Lambda:\lS^1\to\lR^3$ (or more generally a Legendrian link $\Lambda:\bigsqcup\lS^1\to\lR^3$) whose homology group is invariant under Legendrian knot isotopy. We assume $\Lambda$ is in general position. The vector space $\cC(\Lambda)$ is generated by strings and pairs of strings on $\Lambda$ with jumps at the Reeb chords of $\Lambda$, defined in Section \ref{Sec:strings_def}. The boundary operator $d:\cC(\Lambda)\to\cC(\Lambda)$ is defined in terms of the SFT bracket $\{\cdot,\cdot\}:\cC(\Lambda)\otimes\cC(\Lambda)\to\cC(\Lambda)$ and the string operator $d_{\str}:\cC(\Lambda)\to\cC(\Lambda)$, defined in Section \ref{Sec:def_SFT_bracket} and Section \ref{Sec:def_string_operator}, respectively. More precisely, we define
\begin{align*}
d=\{\cdot,H\}+d_{\str},
\end{align*} 
see Section \ref{Sec:The_differential}. Here, $H\in\cC(\Lambda)$ is the sum of strings obtained by looking at the boundaries of index zero pseudoholomorphic disks with up to two and annuli with one positive puncture, see Section \ref{Sec:hamiltonian_def} for the precise definition. The SFT bracket models breaking  of index one pseudoholomorphic curves into SFT buildings. The string operator $d_{\str}=\delta+\nabla$ is a string topological operation. It consists of two terms, $\delta$ and $\nabla$. The map $\delta:\cC(\Lambda)\to\cC(\Lambda)$ can be understood as the loop product with trivial strips over Reeb chords, and it corrects bubbling for disks with two positive punctures (see also \cite{Ng_rLSFT}). The map $\nabla:\cC(\Lambda)\to\cC(\Lambda)$ is a correction of the loop coproduct tailored to our setting, and it corrects nodal breaking of annuli. Our approach is in part motivated by \cite{CieLat}, but we avoid taking quotient by constant loops. This approach to Legendrian SFT can be applied to more general settings, we explore this in future work.

The chain complex is an extension of the Chekanov--Eliashberg differential graded algebra for Legendrian knots defined in \cite{Chekanov02}. We give an alternative definition of the chain complex with more algebraic structure in Section \ref{Section:Algebraic_definition}. This definition is more appropriate for computations.

The quasi-isomorphism class of the chain complex defined for the front resolution of a given Legendrian knot is an invariant of the knot up to Legendrian knot isotopy. This will follow from Section \ref{Sec:Invariance}, where we state a stronger invariance result and prove it combinatorially for Legendrian knots, similar to \cite{Chekanov02,Ng_rLSFT}. The same should follow for Legendrian links using a similar method.

\subsection{Broken closed strings}\label{Sec:strings_def}
In this section, we define the space $\cC=\cC(\Lambda)$ of broken closed strings on $\Lambda$. 

Denote by $\cR=\{\gamma_1,\dots,\gamma_n\}$ the set of Reeb chords on $\Lambda$. Fix an orientation on $\Lambda$ and a base point $T\in\Lambda$ different from the Reeb chord endpoints. For $t_1,\dots,t_k\in \lS^1$ distinct points on $\lS^1$, let $\gamma:\lS^1\backslash\{t_1,\dots,t_k\}\to\Lambda$ be a smooth map such that, at each puncture $t_i,i\in\{1,\dots,k\}$
$$\lim_{t\to t_i^\pm}\gamma(t)={n_i}^\pm$$
or
$$\lim_{t\to t_i^\pm}\gamma(t)={n_i}^\mp$$ 
for some Reeb chord $\gamma_{n_i}\in\cR$, where $n_i^-$ is the starting point and $n_i^+$ is the end point of $\gamma_{n_i}$. Additionally, we require $\gamma'(t_i^\pm)\neq 0$ for $i\in\{1,\dots,k\}$ and that the curve $\pi_{xy}\circ \gamma$ has convex corners, i.e. makes a left turn at every puncture (see Figure \ref{Fig:turning_left}), where $\pi_{xy}:\lR^3\to\lR^2$ is the Lagrangian projection. In that case, we say $\gamma$ has generic asymptotic behavior. This condition has to do with the asymptotic behavior of pseudoholomorphic curves on $\lR\times\Lambda$, see Section \ref{Sec:generic_asymptotic_behavior_admissible_disks}. The orientation of $\gamma$ is always determined by the positive orientation on $\lS^1$.

In case 
$$\lim_{t\to t_i^\pm}\gamma(t)=n_i^\pm,$$
we say $\gamma$ has a \textit{positive} puncture at $t_i$ asymptotic to Reeb chord $\gamma_{n_i}$. Otherwise, if 
$$\lim_{t\to t_i^\pm}\gamma(t)=n_i^\mp,$$
we say $\gamma$ has a \textit{negative} puncture at $t_i$ asymptotic to $\gamma_{n_i}$.

We say two curves $\gamma_0,\gamma_1:\lS^1\backslash\{t_1,\dots,t_k\}\to\Lambda$ as above are equivalent if they are equivalent up to homotopy preserving the ends, i.e. if there exists a smooth family $\gamma_s:\lS^1\backslash\{t_1,\dots,t_k\}\to\Lambda,s\in[0,1]$ such that 
\begin{align*}
&\lim_{t\to t_i^\pm}\gamma_s(t)=\lim_{t\to t_i^\pm}\gamma_0(t),\\
&{\gamma_s}'(t_i^\pm)\neq 0,
\end{align*}
for all $s$ and $i\in\{1,\dots,k\}$. Additionally, two maps are equivalent if they differ by the move shown in Figure \ref{Fig:simult_change}, which will be referred to as the change of the asymptotic behavior.

\begin{figure}
\def\svgwidth{100mm}
\begingroup%
  \makeatletter%
  \providecommand\rotatebox[2]{#2}%
  \newcommand*\fsize{\dimexpr\f@size pt\relax}%
  \newcommand*\lineheight[1]{\fontsize{\fsize}{#1\fsize}\selectfont}%
  \ifx\svgwidth\undefined%
    \setlength{\unitlength}{475.27917812bp}%
    \ifx\svgscale\undefined%
      \relax%
    \else%
      \setlength{\unitlength}{\unitlength * \real{\svgscale}}%
    \fi%
  \else%
    \setlength{\unitlength}{\svgwidth}%
  \fi%
  \global\let\svgwidth\undefined%
  \global\let\svgscale\undefined%
  \makeatother%
  \begin{picture}(1,0.53968292)%
    \lineheight{1}%
    \setlength\tabcolsep{0pt}%
    \put(0,0){\includegraphics[width=\unitlength,page=1]{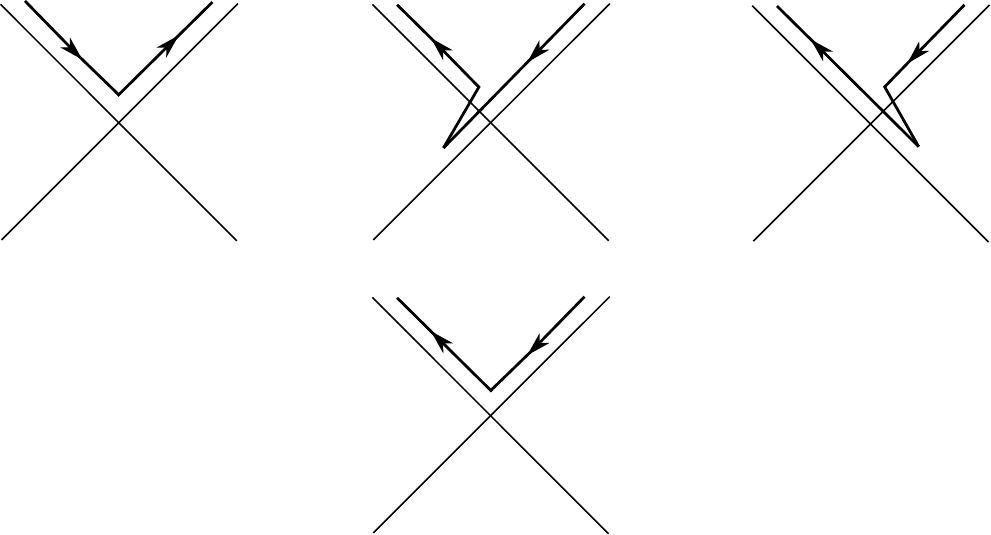}}%
  \end{picture}%
\endgroup%
\caption{Convex corners---top row, non-convex corner---bottom row.}
\label{Fig:turning_left}
\end{figure}

\begin{figure}
\def\svgwidth{71mm}
\begingroup%
  \makeatletter%
  \providecommand\rotatebox[2]{#2}%
  \newcommand*\fsize{\dimexpr\f@size pt\relax}%
  \newcommand*\lineheight[1]{\fontsize{\fsize}{#1\fsize}\selectfont}%
  \ifx\svgwidth\undefined%
    \setlength{\unitlength}{323.79168269bp}%
    \ifx\svgscale\undefined%
      \relax%
    \else%
      \setlength{\unitlength}{\unitlength * \real{\svgscale}}%
    \fi%
  \else%
    \setlength{\unitlength}{\svgwidth}%
  \fi%
  \global\let\svgwidth\undefined%
  \global\let\svgscale\undefined%
  \makeatother%
  \begin{picture}(1,0.35567267)%
    \lineheight{1}%
    \setlength\tabcolsep{0pt}%
    \put(0,0){\includegraphics[width=\unitlength,page=1]{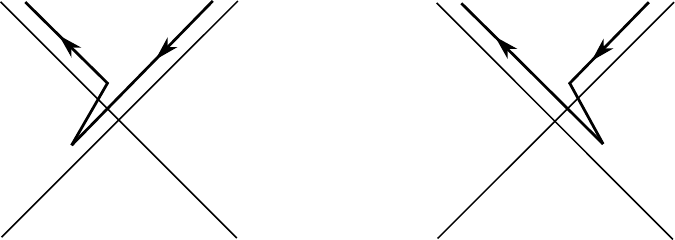}}%
    \put(0.48439721,0.15307739){\makebox(0,0)[lt]{\lineheight{1.25}\smash{\begin{tabular}[t]{l}$\cong$\end{tabular}}}}%
  \end{picture}%
\endgroup%
\caption{Change of the asymptotic behavior at a puncture.}
\label{Fig:simult_change}
\end{figure}

\begin{defi}
A \textit{string} on $\Lambda$ is an equivalence class of a map $\lS^1\backslash\{t_1,\dots,t_k\}\to\Lambda$ for some $\{t_1,\dots,t_k\}\subset\lS^1$ with generic asymptotic behavior at each puncture.
\end{defi}

Strings with one positive puncture have a natural choice of the ordering of the punctures, with the starting marked point right after the positive puncture. In general, we work with strings up to cyclic reordering of the punctures as described below. First we introduce grading. Define 
\begin{align*}
&|\gamma_i^-|=\mu_{CZ}(\gamma_i),\\
&|\gamma_i^+|=-\mu_{CZ}(\gamma_i)-1,
\end{align*}
for a Reeb chord $\gamma_i$. Let $\gamma_{i_1},\dots,\gamma_{i_k}$ be the Reeb chords at the negative punctures of a string $x$ and $\gamma_{j_1},\dots,\gamma_{j_l}$ the Reeb chords at the positive punctures. Then we define 
$$|x|=\sum_{a=1}^k\mu_{CZ}(\gamma_{i_a})-\sum_{b=1}^l(\mu_{CZ}(\gamma_{j_b})+1)-2 a \operatorname{rot}(\Lambda),$$ 
where $a\in\lZ$ is the algebraic count of the intersections between $\gamma$ and the base point $T\in\Lambda$ (we can perturb $\gamma$ to make it transverse to $T$). We say $\beta_1$ is equivalent to $(-1)^{AB}\beta_2$ for $\beta_\iota:\lS^1\backslash\{t_1^\iota,\dots,t_n^\iota\}\to\Lambda,\iota\in\{0,1\}$ strings such that $\beta_2\circ\varphi$ is equal to $\beta_1$, where $\varphi$ is a diffeomorphism of $\lS^1$ such that $t_1^1,\dots, t_n^1$ are sent to $t_{k+1}^2,\dots,t_n^2,t_1^2,\dots,t_k^2$, and 
\begin{align*}
&A=\sum_{\substack{j=1,\dots,k\\ t_j^2\text{ pos./neg. asympt. to }\gamma_{i_j}}}|\gamma_{i_j}^\pm|,\\
&B=\sum_{\substack{j=k+1,\dots,n\\ t_j^2\text{ pos./neg. asympt. to }\gamma_{i_j}}}|\gamma_{i_j}^\pm|.
\end{align*}
We say a string $\beta$ is \textit{bad} if  $\beta$ is equivalent to $-\beta$. This happens when punctures in $\beta$ form a word that is an even cover of an odd word.

Denote by $\widetilde\cC'$ the $\lQ$ vector space generated by strings on $\Lambda$ that have either one or two positive punctures (and arbitrarily many negative punctures), and by $\widetilde\cC$ the quotient of $\widetilde\cC'$ by the relation described above. We can write $\widetilde\cC=\widetilde\cC^1\oplus\widetilde\cC^2$, where $\widetilde\cC^\iota\subset\widetilde\cC,\iota\in\{1,2\}$ is the subspace generated by strings with exactly $\iota$ positive punctures. Let additionally $\overline\cC=\widetilde\cC^1\otimes\widetilde\cC^0$, where $\widetilde\cC^0$ is generated by strings with zero positive punctures.
\begin{defi}
The \textit{space of broken closed strings} on $\Lambda$ is defined as 
$$\cC\coloneq\cC(\Lambda)\coloneq\widetilde\cC\oplus\overline\cC.$$ 
\end{defi}
We extend the grading to $\cC$ by taking $|x\otimes y|=|x|+|y|-1$ for $x\in\widetilde\cC^1,y\in\widetilde\cC^0$.

\subsection{The SFT bracket}
\label{Sec:def_SFT_bracket}
In this section we define the SFT bracket, the first ingredient in the definition of the chain complex associated to $\Lambda$.

\subsubsection{Disk building contribution}\label{Sec:gluing_broken_strings}
We start by defining the restricted SFT bracket $\{\cdot,\cdot\}_1:\widetilde\cC\otimes\widetilde\cC\to\widetilde\cC$ which glues strings at one pair of punctures, following \cite{Ng_rLSFT}.

Let $\gamma:\lS^1\backslash\{t_1,\dots,t_r\}\to\Lambda$ and $\gamma':\lS^1\backslash\{t_1',\dots,t_{r'}'\}\to\Lambda$ be two strings and $t_i,t_j'$ a pair of punctures on $\gamma,\gamma'$ such that
$$\lim_{t\to t_i^\pm}\gamma(t)=\lim_{t\to t_j'^\mp}\gamma'(t),$$
i.e. one string has a positive, and the other a negative puncture at some Reeb chord. Then we define a string $\gamma\prescript{}{t_i\hspace{-0.65mm}}{*}_{t_j'}\gamma'$ on the domain with $r+r'-2$ punctures $t_1,\dots,t_{i-1},t'_{j+1},\dots,t'_{r'},t'_1,\dots,t'_{j-1},t_{i+1},\dots,t_r$ (in this order) equal to 
$$\gamma|_{[t_s,t_{s+1}]}\text{ on }[t_s,t_{s+1}]\text{, for }s\neq i-1,i,$$
$$\gamma'|_{[t'_s,t'_{s+1}]}\text{ on }[t'_s,t'_{s+1}],\text{ for }s\neq j-1,j,$$
and
\begin{align*}
\gamma|_{[t_{i-1},t_{i}]}*\gamma'|_{[t'_j,t'_{j+1}]}\text{ on }[t_{i-1},t'_{j+1}],\\
\gamma'|_{[t'_{j-1},t'_{j}]}*\gamma|_{[t_i,t_{i+1}]}\text{ on }[t'_{j-1},t_{i+1}],
\end{align*}
where $*$ is concatenation of curves (see Figure \ref{Fig:gluing}). 

\begin{figure}
\def\svgwidth{70mm}
\begingroup%
  \makeatletter%
  \providecommand\rotatebox[2]{#2}%
  \newcommand*\fsize{\dimexpr\f@size pt\relax}%
  \newcommand*\lineheight[1]{\fontsize{\fsize}{#1\fsize}\selectfont}%
  \ifx\svgwidth\undefined%
    \setlength{\unitlength}{338.55096651bp}%
    \ifx\svgscale\undefined%
      \relax%
    \else%
      \setlength{\unitlength}{\unitlength * \real{\svgscale}}%
    \fi%
  \else%
    \setlength{\unitlength}{\svgwidth}%
  \fi%
  \global\let\svgwidth\undefined%
  \global\let\svgscale\undefined%
  \makeatother%
  \begin{picture}(1,0.35265534)%
    \lineheight{1}%
    \setlength\tabcolsep{0pt}%
    \put(0,0){\includegraphics[width=\unitlength,page=1]{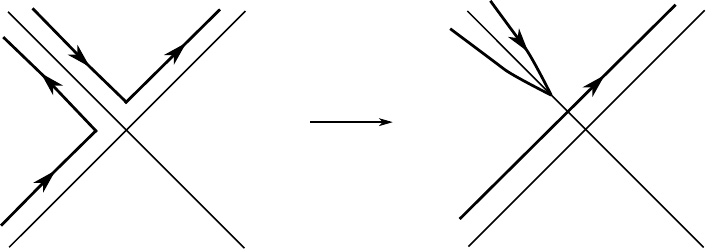}}%
  \end{picture}%
\endgroup%
\caption{Gluing at a puncture.}
\label{Fig:gluing}
\end{figure}

The restricted SFT bracket $\{\cdot,\cdot\}_1:\widetilde\cC\otimes\widetilde\cC\to\widetilde\cC$ is a degree 1 linear map defined as follows. For strings $\gamma,\gamma'\in\widetilde\cC^2$, we take $\{\gamma,\gamma'\}_1=0$. Otherwise, for two strings $\gamma,\gamma'$ we define
\begin{align*}
&\{\gamma,\gamma'\}_1=\sum_{t_i,t'_j}\epsilon(\gamma,\gamma',t_i,t_j')\gamma\prescript{}{t_i\hspace{-0.65mm}}{*}_{t'_j}\gamma',
\end{align*}
where the sum goes over all pairs of punctures $t_i,t'_j$ on $\gamma,\gamma'$ such that $\lim_{t\to t_i^\pm}\gamma(t)=\lim_{t\to t_j'^\mp}\gamma'(t)$. The signs $\epsilon(\gamma,\gamma',t_i,t_j')\in\{1,-1\}$ are defined as follows. For $\gamma$ a string and $t_I$ its puncture, we denote 
$$P(\gamma,t_I)=\sum_{\substack{j=1,\dots I-1,t_j\text{ pos./neg.}\\\text{asympt. to }\gamma_{k_j}}}|\gamma_{k_j}^\pm|.$$
Then we define (see also \cite[Section 3.1]{Ng_rLSFT}, note that here we write $\{y,x\}_1$ instead of $\{x,y\}$)
\begin{align*}
\epsilon(\gamma,\gamma',t_i,t_j')=\begin{cases}
(-1)^{P(\gamma,t_i)\left(|\gamma'|+1\right)+P(\gamma',t_{j+1}')\left(|\gamma'|+P(\gamma',t_{j+1}')\right)},& t_i\text{ negative puncture}\\
-(-1)^{P(\gamma,t_i)\left(|\gamma'|+1\right)+P(\gamma',t_{j+1}')\left(|\gamma'|+P(\gamma',t_{j+1}')\right)},& t_i\text{ positive puncture}
\end{cases}
\end{align*}
It is easy to see that the map is well defined, i.e. it does not depend on the representative of the broken closed string and descends to $\widetilde\cC$. We can extend $\{\cdot,\cdot\}_1$ to include strings with zero positive punctures in the obvious way.

\begin{lemma}[\cite{Ng_rLSFT}]\label{Lemma:PropertiesReduced}
The restricted SFT bracket $\{\cdot,\cdot\}_1:\widetilde\cC\otimes\widetilde\cC\to\widetilde\cC$ satisfies the following
\begin{itemize}
\item[a)]$\{x,y\}_1=-(-1)^{(|x|+1)(|y|+1)}\{y,x\}_1$,
\item[b)]$\{\{x,y\}_1,z\}_1+(-1)^{(|y|+|z|)(|x|+1)}\{\{y,z\}_1,x\}_1+(-1)^{(|x|+|y|)(|z|+1)}\{\{z,x\}_1,y\}_1=0$,
\end{itemize}
for all $x,y,z\in\widetilde\cC$.
\end{lemma}

\begin{proof}
The first property is trivial to check. The second property follows from (without signs)
\begin{align*}
&\{\gamma,\gamma'\}_1=\sum_{\substack{t_i,t_j'\\\lim_{t\to t_i^\pm}\gamma(t)=\lim_{ t\to t_j'^\mp}\gamma'( t)}}\gamma\prescript{}{ t_i\hspace{-0.65mm}}{*}_{ t_j'}\gamma';\\
&\{\{\gamma,\gamma'\}_1,\gamma''\}_1=\sum_{\substack{ t_i, t_j'}}\sum_{ t_k'',\tilde t\neq t_i, t_j'}\left(\gamma\prescript{}{ t_i\hspace{-0.65mm}}{*}_{ t_j'}\gamma'\right)\prescript{}{\tilde t}{*}_{ t_k''}\gamma''=\\
&=\sum_{\substack{ t_i, t_j'}}\sum_{ t_k'', t_s\neq t_i}\left(\gamma\prescript{}{ t_i\hspace{-0.65mm}}{*}_{ t_j'}\gamma'\right)\prescript{}{ t_s\hspace{-0.85mm}}{*}_{ t_k''}\gamma''+\sum_{\substack{ t_i, t_j'}}\sum_{ t_k'', t_s'\neq t_j'}\left(\gamma\prescript{}{ t_i\hspace{-0.65mm}}{*}_{ t_j'}\gamma'\right)\prescript{}{ t_s'\hspace{-0.45mm}}{*}_{ t_k''}\gamma''=\\
&=\sum_{ t_k'', t_s}\sum_{\substack{ t_i\neq t_s, t_j'}}\gamma'\prescript{}{ t_j'\hspace{-0.4mm}}{*}_{ t_i}\left(\gamma\prescript{}{ t_s\hspace{-0.85mm}}{*}_{ t_k''}\gamma''\right)+\sum_{ t_k'', t_s'}\sum_{\substack{ t_i, t_j'\neq t_s'}}\gamma\prescript{}{ t_i\hspace{-0.65mm}}{*}_{ t_j'}\left(\gamma'\prescript{}{ t_s'\hspace{-0.45mm}}{*}_{ t_k''}\gamma''\right)=\\
&=\sum_{ t_k'', t_s}\sum_{\substack{ t_i\neq t_s, t_j'}}\gamma'\prescript{}{ t_j'\hspace{-0.45mm}}{*}_{ t_i}\left(\gamma\prescript{}{ t_s\hspace{-0.85mm}}{*}_{ t_k''}\gamma''\right)+
\sum_{ t_k'', t_s}\sum_{\substack{ t_l''\neq t_k'', t_j'}}\gamma'\prescript{}{ t_j'\hspace{-0.45mm}}{*}_{ t_l''}\left(\gamma\prescript{}{ t_s\hspace{-0.85mm}}{*}_{ t_k''}\gamma''\right)+\\
&+\sum_{ t_l'', t_j'}\sum_{\substack{ t_k''\neq t_l'', t_s}}\gamma\prescript{}{ t_s\hspace{-0.85mm}}{*}_{ t_k''}\left(\gamma' \prescript{}{ t_j'\hspace{-0.4mm}}{*}_{ t_l''}\gamma''\right)
+\sum_{ t_k'', t_s'}\sum_{\substack{ t_i, t_j'\neq t_s'}}\gamma\prescript{}{ t_i\hspace{-0.65mm}}{*}_{ t_j'}\left(\gamma'\prescript{}{ t_s'\hspace{-0.4mm}}{*}_{ t_k''}\gamma''\right)=\\
&=\{\{\gamma'',\gamma\}_1,\gamma'\}_1+\{\{\gamma',\gamma''\}_1,\gamma\}_1.
\end{align*}
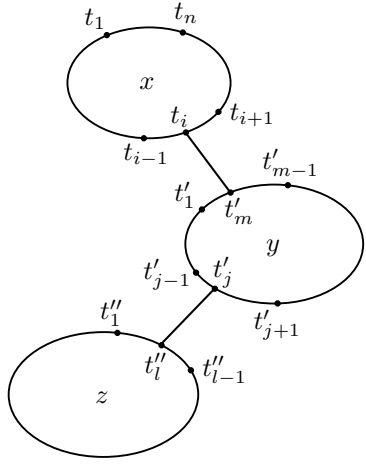
\begin{figure}
\def\svgwidth{49mm}
\begingroup%
  \makeatletter%
  \providecommand\rotatebox[2]{#2}%
  \newcommand*\fsize{\dimexpr\f@size pt\relax}%
  \newcommand*\lineheight[1]{\fontsize{\fsize}{#1\fsize}\selectfont}%
  \ifx\svgwidth\undefined%
    \setlength{\unitlength}{134.36606835bp}%
    \ifx\svgscale\undefined%
      \relax%
    \else%
      \setlength{\unitlength}{\unitlength * \real{\svgscale}}%
    \fi%
  \else%
    \setlength{\unitlength}{\svgwidth}%
  \fi%
  \global\let\svgwidth\undefined%
  \global\let\svgscale\undefined%
  \makeatother%
  \begin{picture}(1,1.24081089)%
    \lineheight{1}%
    \setlength\tabcolsep{0pt}%
    \put(0.35765306,1.00250746){\makebox(0,0)[lt]{\lineheight{1.25}\smash{\begin{tabular}[t]{l}$x$\end{tabular}}}}%
    \put(0.70158851,0.56262596){\makebox(0,0)[lt]{\lineheight{1.25}\smash{\begin{tabular}[t]{l}$y$\end{tabular}}}}%
    \put(0.23689738,0.14997418){\makebox(0,0)[lt]{\lineheight{1.25}\smash{\begin{tabular}[t]{l}$z$\end{tabular}}}}%
    \put(0.2085896,1.17542449){\makebox(0,0)[lt]{\lineheight{1.25}\smash{\begin{tabular}[t]{l}$t_1$\end{tabular}}}}%
    \put(0.45179515,1.190575){\makebox(0,0)[lt]{\lineheight{1.25}\smash{\begin{tabular}[t]{l}$t_n$\end{tabular}}}}%
    \put(0.31424434,0.80862276){\makebox(0,0)[lt]{\lineheight{1.25}\smash{\begin{tabular}[t]{l}$t_{i-1}$\end{tabular}}}}%
    \put(0.60250275,0.91627074){\makebox(0,0)[lt]{\lineheight{1.25}\smash{\begin{tabular}[t]{l}$t_{i+1}$\end{tabular}}}}%
    \put(0.45208084,0.696987){\makebox(0,0)[lt]{\lineheight{1.25}\smash{\begin{tabular}[t]{l}$t'_1$\end{tabular}}}}%
    \put(0.36874981,0.48567752){\makebox(0,0)[lt]{\lineheight{1.25}\smash{\begin{tabular}[t]{l}$t'_{j-1}$\end{tabular}}}}%
    \put(0.667362,0.34944105){\makebox(0,0)[lt]{\lineheight{1.25}\smash{\begin{tabular}[t]{l}$t'_{j+1}$\end{tabular}}}}%
    \put(0.69128365,0.7845371){\makebox(0,0)[lt]{\lineheight{1.25}\smash{\begin{tabular}[t]{l}$t'_{m-1}$\end{tabular}}}}%
    \put(0.25099204,0.37605786){\makebox(0,0)[lt]{\lineheight{1.25}\smash{\begin{tabular}[t]{l}$t''_1$\end{tabular}}}}%
    \put(0.52808658,0.22335686){\makebox(0,0)[lt]{\lineheight{1.25}\smash{\begin{tabular}[t]{l}$t''_{l-1}$\end{tabular}}}}%
    \put(0.44165929,0.90900629){\makebox(0,0)[lt]{\lineheight{1.25}\smash{\begin{tabular}[t]{l}$t_{i}$\end{tabular}}}}%
    \put(0.5883128,0.65364717){\makebox(0,0)[lt]{\lineheight{1.25}\smash{\begin{tabular}[t]{l}$t'_{m}$\end{tabular}}}}%
    \put(0.55732113,0.49114896){\makebox(0,0)[lt]{\lineheight{1.25}\smash{\begin{tabular}[t]{l}$t'_{j}$\end{tabular}}}}%
    \put(0,0){\includegraphics[width=\unitlength,page=1]{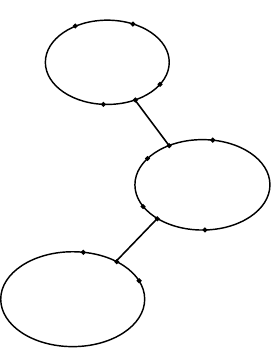}}%
    \put(0.37112966,0.23182525){\makebox(0,0)[lt]{\lineheight{1.25}\smash{\begin{tabular}[t]{l}$t''_{l}$\end{tabular}}}}%
  \end{picture}%
\endgroup%
\caption{The string can be glued in two ways. It appears as a summand in $\{\{x,y\}_1,z\}_1$ and $\{\{y,z\}_1,x\}_1$.}
\label{Fig:string_gluing_signs}
\end{figure}
To check that the signs cancel out, consider for example the configuration shown in Figure \ref{Fig:string_gluing_signs}. Other cases follow similarly. Without loss of generality, we assume that the last puncture on $x,y,z$ ($t_n,t'_m,t''_l$ respectively) is a positive puncture. The glued string appears in $\{\{x,y\}_1,z\}_1$ and $(-1)^{(|y|+|z|)(|x|+1)}\{\{y,z\},x\}=-\{x,\{y,z\}_1\}_1$. The signs for the two summands are given by
\begin{align*}
S_1=(-1)^{P(x,t_i)(|y|+1)+(P(x,t_i)+P(y,t_j'))(|z|+1)},
\end{align*}
and
\begin{align*}
S_2=-(-1)^{P(y,t_j')(|z|+1)+P(x,t_i)(|y|+|z|)}=-S_1.
\end{align*}
This finishes the proof.
\end{proof}

\subsubsection{Annulus building contribution}\label{Sec:gluing_annulus}
Next, we describe the part of the SFT bracket that glues strings at two pairs of punctures and define the SFT bracket $\{\cdot,\cdot\}:\cC\otimes\cC\to\cC$.
 
Let $\gamma:\lS^1\backslash\{t_1,\dots,t_n\}\to\Lambda,\gamma\in\widetilde\cC^1$ be a string with a positive puncture without loss of generality at $t_{n}$, $\gamma':\lS^1\backslash\{t_1',\dots,t_m'\}\to\Lambda,\gamma'\in\widetilde\cC^2$ a string with two positive punctures, and $t_i,t'_j;t_k,t'_l$ be two pairs of punctures on $\gamma,\gamma'$ such that $i<k, j\neq l$ and
\begin{align*}
&\lim_{t\to t_i^\pm}\gamma(t)=\lim_{t\to t_j'^\mp}\gamma'(t)=\gamma_1^\mp,\\
&\lim_{t\to t_k^\pm}\gamma(t)=\lim_{t\to t_l'^\mp}\gamma'(t)=\gamma_2^\mp,
\end{align*}
for some Reeb chords $\gamma_1,\gamma_2\in\cR$. Assume additionally without loss of generality that $l=m$ (take cyclic reordering of the punctures on $\gamma'$). We define a pair of strings $\gamma(t_i,t'_j;t_k,t'_l)\gamma'=\gamma_1\otimes\gamma_2\in\overline\cC$ with punctures $t_1,\dots,t_{i-1},t'_{j+1},\dots,t'_{l-1},t_{k+1},\dots,t_n$ on $\gamma_1\in\widetilde\cC^1$ and $t'_{l+1},\dots,t'_{j-1},t_{i+1},\dots,t_{k-1}$ on $\gamma_2\in\widetilde\cC^0$, consisting of the string $\gamma_1$ given by
\begin{align*}
\gamma|_{[t_s,t_{s+1}]}\text{ on }[t_s,t_{s+1}],\text{ for }s\in\{k+1,\dots,i-2\},\\
\gamma'|_{[t'_s,t'_{s+1}]}\text{ on }[t'_s,t'_{s+1}],\text{ for }s\in\{j+1,\dots,l-2\},
\end{align*}
and 
\begin{align*}
\gamma|_{[t_{i-1},t_{i}]}*\gamma'|_{[t'_j,t'_{j+1}]},\text{ on }[t_{i-1},t'_{j+1}],\\
\gamma'|_{[t'_{l-1},t'_{l}]}*\gamma|_{[t_k,t_{k+1}]},\text{ on }[t'_{l-1},t_{k+1}];
\end{align*}
and the string $\gamma_2$ given by 
\begin{align*}
\gamma'|_{[t'_s,t'_{s+1}]}\text{ on }[t'_s,t'_{s+1}],\text{ for }s\in\{l+1,\dots,j-2\},\\
\gamma|_{[t_s,t_{s+1}]}\text{ on }[t_s,t_{s+1}],\text{ for }s\in\{i+1,\dots,k-2\},
\end{align*}
and
\begin{align*}
\gamma'|_{[t'_{j-1},t'_{j}]}*\gamma|_{[t_i,t_{i+1}]},\text{ on }[t'_{j-1},t_{i+1}],\\
\gamma|_{[t_{k-1},t_{k}]}*\gamma'|_{[t'_l,t'_{l+1}]},\text{ on }[t_{k-1},t'_{l+1}].
\end{align*}
Now, we define $\{\cdot,\cdot\}_2:\cC\otimes\cC\to\cC$ as a degree 1 linear map given as follows. For strings $\gamma\in\widetilde\cC^1,\gamma'\in\widetilde\cC^2$ as above, we take
\begin{align*}
\{\gamma,\gamma'\}_2=\sum_{\substack{(t_i,t'_j);(t_k,t'_l)\\i<k}}\epsilon(\gamma,\gamma',t_i,t'_j;t_k,t'_l)\gamma(t_i,t'_j;t_k,t'_l)\gamma',
\end{align*}
where the sum goes over all pairs of punctures satisfying the conditions above. The sign $\epsilon(\gamma,\gamma',t_i,t'_j;t_k,t'_l)\in\{+1,-1\}$ is given by
\begin{align*}
\epsilon(\gamma,\gamma',t_i,t'_j;t_k,t'_l)=(-1)^
{P(\gamma,t_k)\left(P(\gamma',t_{l}',t'_j)+P(\gamma,t_i,t_k)+|\gamma'|\right)},
\end{align*}
where $P(x,t_I,t_J)=\sum_{\substack{j=I+1,\dots,J-1,\\t_j\text{ pos./neg. asym. to }\gamma_{i_j}}}|\gamma_{i_j}^\pm|$. To simplify, we assume $l=m$ as above. \newline
Additionally, we define $\{\gamma',\gamma\}_2=-(-1)^{(|\gamma|+1)(|\gamma'|+1)}\{\gamma,\gamma'\}_2$ for $\gamma\in\widetilde\cC^1,\gamma'\in\widetilde\cC^2$. In case $x,y\in\widetilde\cC^1$ or $x,y\in\widetilde\cC^2$ or $x\in\overline\cC$ or $y\in\overline\cC$, we define $\{x,y\}_2=0$.

We additionally extend the restricted SFT bracket to a linear map $\{\cdot,\cdot\}_1:\cC\otimes\cC\to\cC$ as follows. For generators $\gamma_1\otimes\gamma_2\in\overline\cC$ and $\gamma\in\widetilde\cC^1$, we define
\begin{align*}
&\{\gamma_1\otimes\gamma_2,\gamma\}_1=(-1)^{(|\gamma|+1)(|\gamma_2|+1)}\{\gamma_1,\gamma\}_1\otimes\gamma_2+\gamma_1\otimes\{\gamma_2,\gamma\}_1,\\
&\{\gamma,\gamma_1\otimes\gamma_2\}_1=\{\gamma,\gamma_1\}_1\otimes\gamma_2+(-1)^{(|\gamma|+1)(|\gamma_1|+1)}\gamma_1\otimes\{\gamma,\gamma_2\}_1,
\end{align*}
and zero otherwise.

Finally, the SFT bracket is a degree $1$ linear map $\{\cdot,\cdot\}:\cC\otimes
\cC\to\cC$  given by
\begin{align*}
\{\gamma,\gamma'\}=\{\gamma,\gamma'\}_1+\{\gamma,\gamma'\}_2.
\end{align*}

\begin{lemma}
For every $x,y,z\in\cC$, we have
$$\{\{x,y\},z\}+(-1)^{(|y|+|z|)(|x|+1)}\{\{y,z\},x\}+(-1)^{(|x|+|y|)(|z|+1)}\{\{z,x\},y\}=0.$$
\end{lemma}
The proof of the lemma is similar to Lemma \ref{Lemma:PropertiesReduced} and we omit it here.

\subsection{The string operator}
\label{Sec:def_string_operator}
In this section we define the string operator $d_{\str}:\cC\to\cC$. The map $d_{\str}$ consists of two parts
$$d_{\str}=\delta+\nabla,$$
defined below. 

Recall $\cC=\widetilde\cC\oplus\overline\cC$, where $\widetilde\cC=\widetilde\cC^1\oplus\widetilde\cC^2$ is the vector space generated by strings with one or two positive punctures, and $\overline\cC=\widetilde\cC^1\otimes\widetilde\cC^0$ is generated by pairs of strings with one and zero positive punctures.

\subsubsection{Disk bubble contribution---loop product}\label{Sec:loop_product_operator}
In this section, we define the map $\delta:\cC\to\cC$. Our definition is similar to the definition in \cite{Ng_rLSFT}. Note that here $\delta$ does not vanish only on strings with one positive puncture.

First, for every string $\gamma:\lS^1\backslash\{t_1,\dots,t_k\}\to\Lambda$ transverse to the Reeb chord endpoints and $\tau\in\lS^1\backslash\{t_1,\dots,t_k\}$ such that $\gamma(\tau)=i^\pm$ for some Reeb chord $\gamma_i\in\cR$, we introduce a string $\delta(\gamma,\tau)$ with $k+2$ punctures given by inserting a trivial strip at the point $\tau$. More precisely, close to $\tau$ we create two new punctures (with generic asymptotic behavior), first a negative puncture and then a positive puncture at $\gamma_i$ if $\gamma(\tau)=i^+$, and first a positive and then a negative puncture at $\gamma_i$ if $\gamma(\tau)=i^-$ (see Figure \ref{Fig:insert_triv_cyl}). See also \cite[Section 3.2]{Ng_rLSFT}. 

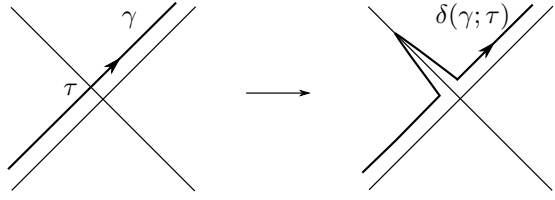
\begin{figure}
\def\svgwidth{96mm}
\begingroup%
  \makeatletter%
  \providecommand\rotatebox[2]{#2}%
  \newcommand*\fsize{\dimexpr\f@size pt\relax}%
  \newcommand*\lineheight[1]{\fontsize{\fsize}{#1\fsize}\selectfont}%
  \ifx\svgwidth\undefined%
    \setlength{\unitlength}{445.52728661bp}%
    \ifx\svgscale\undefined%
      \relax%
    \else%
      \setlength{\unitlength}{\unitlength * \real{\svgscale}}%
    \fi%
  \else%
    \setlength{\unitlength}{\svgwidth}%
  \fi%
  \global\let\svgwidth\undefined%
  \global\let\svgscale\undefined%
  \makeatother%
  \begin{picture}(1,0.2627446)%
    \lineheight{1}%
    \setlength\tabcolsep{0pt}%
    \put(0,0){\includegraphics[width=\unitlength,page=1]{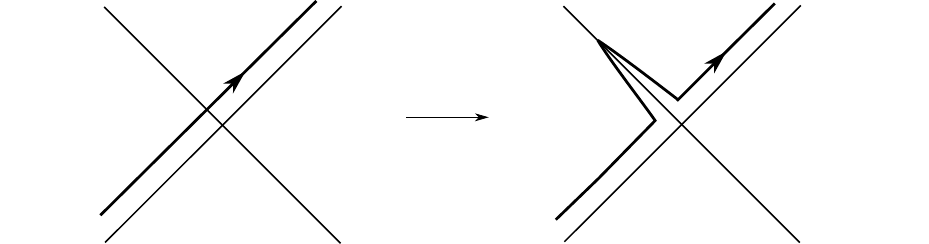}}%
    \put(0.26593548,0.23403062){\makebox(0,0)[lt]{\lineheight{1.25}\smash{\begin{tabular}[t]{l}$\gamma$\end{tabular}}}}%
    \put(0.7001291,0.23395004){\makebox(0,0)[lt]{\lineheight{1.25}\smash{\begin{tabular}[t]{l}$\delta(\gamma;\tau)$\end{tabular}}}}%
    \put(0.18477146,0.13302753){\makebox(0,0)[lt]{\lineheight{1.25}\smash{\begin{tabular}[t]{l}$\tau$\end{tabular}}}}%
    \put(-0.00160484,0.2181593){\makebox(0,0)[lt]{\lineheight{1.25}\smash{\begin{tabular}[t]{l}\textcolor{white}{.}\end{tabular}}}}%
  \end{picture}%
\endgroup%
\caption{Inserting a trivial strip.}
\label{Fig:insert_triv_cyl}
\end{figure}

Now, we define a degree $-1$ linear map $\delta:\cC\to\cC$ as follows. Choosing a representative of a string $\gamma\in\widetilde\cC^1$ that is transverse to all Reeb chord endpoints, we define
\begin{align*}
\delta(\gamma)=\sum_{\tau}\epsilon(\gamma,\tau)\delta(\gamma,\tau),
\end{align*}
where the sum goes over all $\tau\in\lS^1\backslash\{t_1,\dots,t_k\}$ such that $\gamma(\tau)=i^+$ or $\gamma(\tau)=i^-$ for some Reeb chord $\gamma_i\in\cR$, and $\delta(\gamma)=0$ for $\gamma\in\widetilde\cC^2$ or $\gamma\in\overline\cC$. The sign $\epsilon(\gamma,\tau)\in\{-1,+1\}$ is given by 
\begin{align*}
\epsilon(\gamma,\tau)=(-1)^{P(\gamma,\tau)}\epsilon(\tau)p(\tau),
\end{align*}
where $\epsilon(\tau)$ is $+1$ if the orientation of $\gamma$ at $\tau$ matches the orientation of $\Lambda$ and $-1$ otherwise, $p(\tau)$ is $1$ if $\gamma(\tau)=i^+$ and $-1$ if $\gamma(\tau)=i^-$, and $P(\gamma,\tau)$ is the sum of gradings at the punctures $t_1,\dots,t_j$ for $\tau\in(t_j,t_{j+1})$. The map $\delta$ is well defined, i.e. descends to $\cC(\Lambda)$ and does not depend on the representative of the broken closed string, see Figure \ref{Fig:cyl_insert_well_def}.

\begin{figure}
\def\svgwidth{145mm}
\begingroup%
  \makeatletter%
  \providecommand\rotatebox[2]{#2}%
  \newcommand*\fsize{\dimexpr\f@size pt\relax}%
  \newcommand*\lineheight[1]{\fontsize{\fsize}{#1\fsize}\selectfont}%
  \ifx\svgwidth\undefined%
    \setlength{\unitlength}{743.66935678bp}%
    \ifx\svgscale\undefined%
      \relax%
    \else%
      \setlength{\unitlength}{\unitlength * \real{\svgscale}}%
    \fi%
  \else%
    \setlength{\unitlength}{\svgwidth}%
  \fi%
  \global\let\svgwidth\undefined%
  \global\let\svgscale\undefined%
  \makeatother%
  \begin{picture}(1,0.32049218)%
    \lineheight{1}%
    \setlength\tabcolsep{0pt}%
    \put(0.9100545,0.06278912){\makebox(0,0)[lt]{\lineheight{1.25}\smash{\begin{tabular}[t]{l}$\cong$\end{tabular}}}}%
    \put(0.97477514,0.06083744){\makebox(0,0)[lt]{\lineheight{1.25}\smash{\begin{tabular}[t]{l}$0$\end{tabular}}}}%
    \put(0.26967242,0.05563939){\makebox(0,0)[lt]{\lineheight{1.25}\smash{\begin{tabular}[t]{l}$-$\end{tabular}}}}%
    \put(0.49054526,0.27276609){\makebox(0,0)[lt]{\lineheight{1.25}\smash{\begin{tabular}[t]{l}$\delta$\end{tabular}}}}%
    \put(0.71544223,0.2412365){\makebox(0,0)[lt]{\lineheight{1.25}\smash{\begin{tabular}[t]{l}$-$\end{tabular}}}}%
    \put(0.9093658,0.24831263){\makebox(0,0)[lt]{\lineheight{1.25}\smash{\begin{tabular}[t]{l}$\cong$\end{tabular}}}}%
    \put(0.48895706,0.08441443){\makebox(0,0)[lt]{\lineheight{1.25}\smash{\begin{tabular}[t]{l}$\delta$\end{tabular}}}}%
    \put(0.71457428,0.06080777){\makebox(0,0)[lt]{\lineheight{1.25}\smash{\begin{tabular}[t]{l}$-$\end{tabular}}}}%
    \put(0.27209395,0.23906243){\makebox(0,0)[lt]{\lineheight{1.25}\smash{\begin{tabular}[t]{l}$-$\end{tabular}}}}%
    \put(0.97408644,0.24636106){\makebox(0,0)[lt]{\lineheight{1.25}\smash{\begin{tabular}[t]{l}$0$\end{tabular}}}}%
    \put(-0.00096145,0.26156006){\makebox(0,0)[lt]{\lineheight{1.25}\smash{\begin{tabular}[t]{l}\textcolor{white}{.}\end{tabular}}}}%
    \put(0,0){\includegraphics[width=\unitlength,page=1]{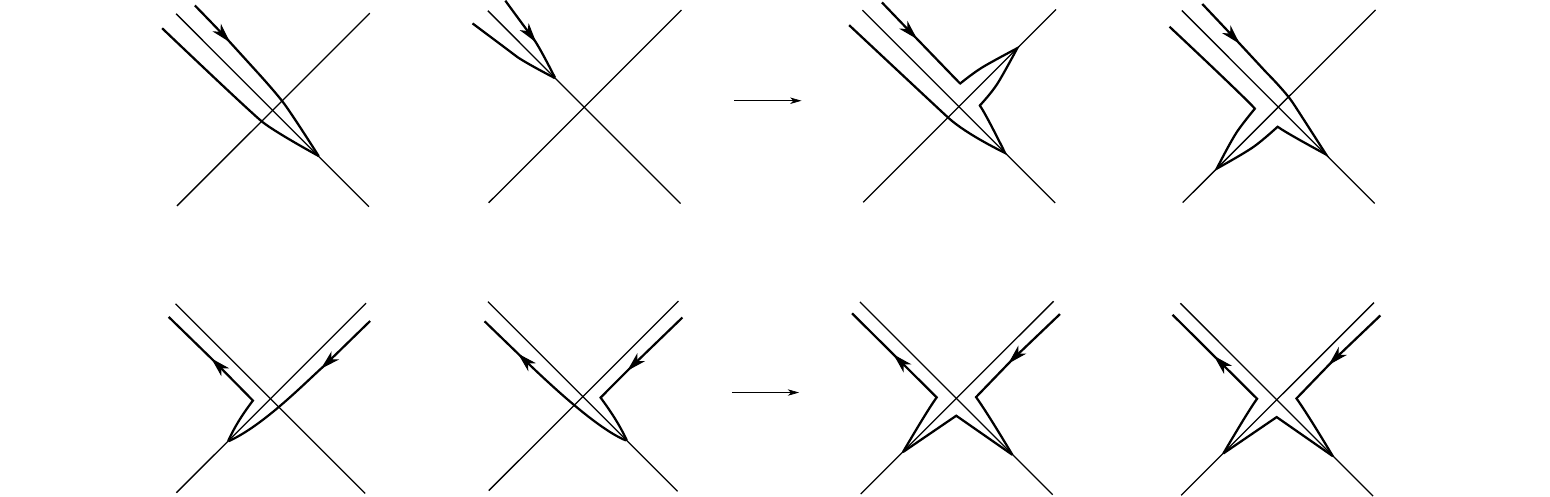}}%
  \end{picture}%
\endgroup%
\caption{Map $\delta:\cC\to\cC$ is well defined.}
\label{Fig:cyl_insert_well_def}
\end{figure}

\begin{lemma}[\cite{Ng_rLSFT}]\label{Lemma:reduced_prop}
The following holds
\begin{itemize}
\item[a)] $\delta\circ\delta(x)=0$,
\item[b)] $\delta\{x,y\}_1=\{x,\delta y\}_1-(-1)^{|y|}\{\delta x,y\}_1,$
\end{itemize}
for all $x,y\in\cC$.
\end{lemma}
\begin{proof}
The first property follows trivially by definition. The second property is proven in \cite[Proposition 3.8]{Ng_rLSFT}. We give another proof based on a similar idea. We omit the signs, which are easy to check. See also Lemma \ref{Lemma:signs_FH} for an extension to strings with a marked point.

Let $x:\lS^1\backslash\{t_1,\dots,t_k\}\to\Lambda$ and $y:\lS^1\backslash\{s_1,\dots,s_l\}\to\Lambda$ be strings in $\widetilde\cC^1$ and $x_i:(t_i,t_{i+1})\to\Lambda,y_j:(s_j,s_{j+1})\to\Lambda$ their restrictions to arcs $i\in\{1,\dots,k\},j\in\{1,\dots,l\}$. We can assume that $x$ and $y$ have distinct critical values in $\Lambda$. Then for $\Delta\subset\Lambda\times\Lambda$ the diagonal, the map $(x_i,y_j):(t_i,t_{i+1})\times (s_j,s_{j+1})\to\Lambda\times\Lambda$ is transverse to $\Delta$ and we have a 1-dimensional manifold
\begin{align*}
P_{i,j}=(x_i,y_j)^{-1}\Delta\subset (t_i,t_{i+1})\times(s_j,s_{j+1})\subset \lS^1\times\lS^1
\end{align*}
that can be compactified by adding points
\begin{align*}
A_{i,j}=\{(t_i,s)\,|\,y_j(s)=x_i(t_i^+),s\in(s_j,s_{j+1})\}\cup\{(t_{i+1},s)\,|\,y_j(s)=x_i(t_{i+1}^-),s\in(s_j,s_{j+1})\},
\end{align*}
\begin{align*}
B_{i,j}=\{(t,s_j)\,|\,x_i(t)=y_j(s_j^+),t\in(t_i,t_{i+1})\}\cup\{(t,s_{j+1})\,|\,x_i(t)=y_j(s_{j+1}^-),t\in(t_i,t_{i+1})\},
\end{align*}
and 
\begin{align*}
C_{i,j}=&\left\{(t_i,s_{j+1})\,\middle|\,x_i(t_i^+)=y_j(s_{j+1}^-),\frac{x_i'(t_i^+)}{\|x_i'(t_i^+)\|}=\frac{y_j'(s_{j+1}^-)}{\|y_j'(s_{j+1}^-)\|}\right\}\cup\\
\cup&\left\{(t_{i+1},s_j)\,\middle|\,x_i(t_{i+1}^-)=y_j(s_j^+),\frac{x_i'(t_{i+1}^-)}{\|x_i'(t_{i+1}^-)\|}=\frac{y_j'(s_j^+)}{\|y_j'(s_j^+)\|}\right\},
\\
D_{i,j}=&\left\{(t_i,s_j)\,\middle|\,x_i(t_i^+)=y_j(s_j^+),\frac{x_i'(t_i^+)}{\|x_i'(t_i^+)\|}=\frac{y_j'(s_j^+)}{\|y_j'(s_j^+)\|}\right\}\cup\\
\cup&\left\{(t_{i+1},s_{j+1})\,\middle|\,x_i(t_{i+1}^-)=y_j(s_{j+1}^-),\frac{x_i'(t_{i+1}^-)}{\|x_i'(t_{i+1}^-)\|}=\frac{y_j'(s_{j+1}^-)}{\|y_j'(s_{j+1}^-)\|}\right\}.
\end{align*}
We notice that the points in $\bigsqcup_{i,j} D_{i,j}$ come in pairs because of the generic asymptotic behavior at the punctures. We look at the 1-dimensional manifold with boundary $P=\bigsqcup_{1\leq i\leq k,1\leq j\leq l} \overline P_{i,j}/_\sim$ glued at the pairs of boundary points $\bigsqcup_{i,j}D_{i,j}$.

We can decorate the connected components of $P$ with broken closed strings as follows. For any $\tau=(s,t)\in P_{i,j}$ in the interior, we glue $x$ and $y$ at $t$ and $s$ into a broken closed string which we denote by $x{}_t{*}_s y$. At a boundary point in $A_{i,j}$, we first insert a trivial strip into $y$ at the point $s$, and then glue the positive (or negative) end of $x$ at $t_i$ or $t_{i+1}$ with the newly created negative (or positive) end on $y$. Similar works for boundary point in $B_{i,j}$. For a boundary point in $C_{i,j}$, we first glue the two strings at the punctures corresponding to the boundary point and then insert a trivial strip at the newly created crossing with the Reeb chord endpoint. 

It is not difficult to see that the map $\varphi:\{\text{Connected components of }P\}\to\widetilde\cC$ is well defined, see Figure \ref{Fig:decoration_AB} and Figure \ref{Fig:decoration_C}. 
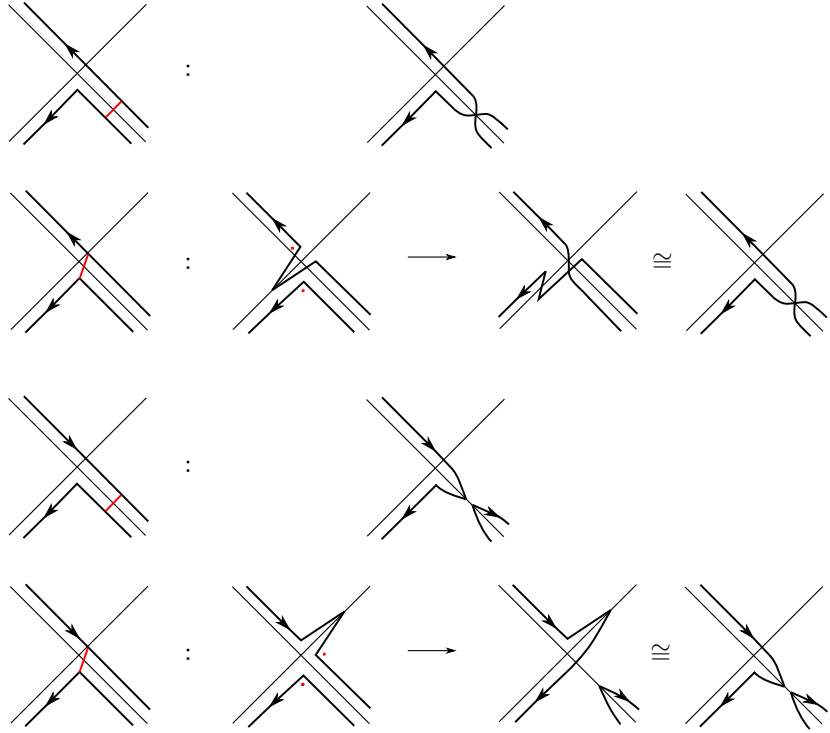
\begin{figure}
\def\svgwidth{109mm}
\begingroup%
  \makeatletter%
  \providecommand\rotatebox[2]{#2}%
  \newcommand*\fsize{\dimexpr\f@size pt\relax}%
  \newcommand*\lineheight[1]{\fontsize{\fsize}{#1\fsize}\selectfont}%
  \ifx\svgwidth\undefined%
    \setlength{\unitlength}{553.7258976bp}%
    \ifx\svgscale\undefined%
      \relax%
    \else%
      \setlength{\unitlength}{\unitlength * \real{\svgscale}}%
    \fi%
  \else%
    \setlength{\unitlength}{\svgwidth}%
  \fi%
  \global\let\svgwidth\undefined%
  \global\let\svgscale\undefined%
  \makeatother%
  \begin{picture}(1,0.88837339)%
    \lineheight{1}%
    \setlength\tabcolsep{0pt}%
    \put(0,0){\includegraphics[width=\unitlength,page=1]{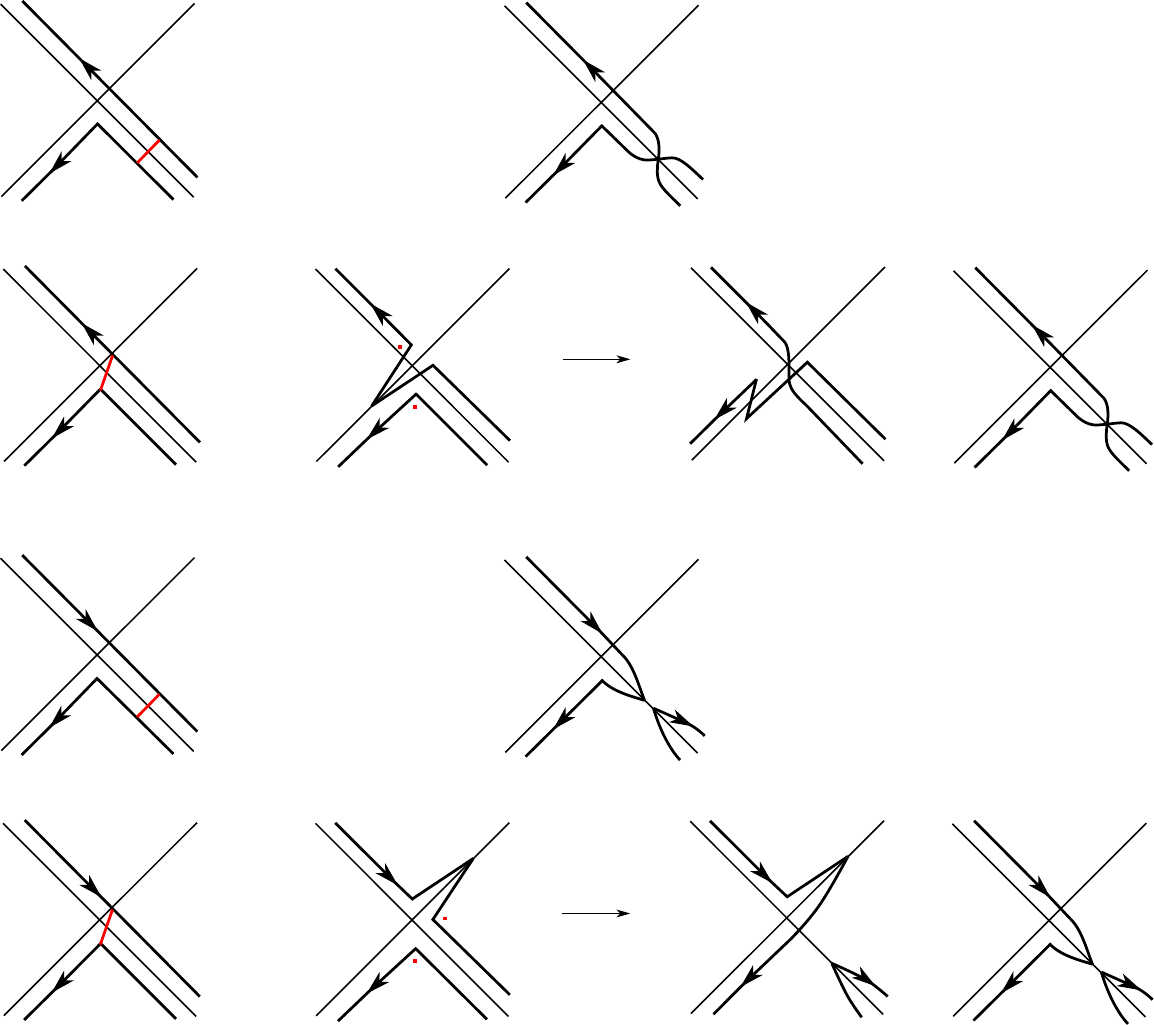}}%
    \put(0.78556592,0.56107189){\makebox(0,0)[lt]{\lineheight{1.25}\smash{\begin{tabular}[t]{l}$\cong$\end{tabular}}}}%
    \put(0.78352432,0.08355828){\makebox(0,0)[lt]{\lineheight{1.25}\smash{\begin{tabular}[t]{l}$\cong$\end{tabular}}}}%
    \put(0.21502392,0.79746311){\makebox(0,0)[lt]{\lineheight{1.25}\smash{\begin{tabular}[t]{l}:\end{tabular}}}}%
    \put(0.21511419,0.56239559){\makebox(0,0)[lt]{\lineheight{1.25}\smash{\begin{tabular}[t]{l}:\end{tabular}}}}%
    \put(0.21511419,0.31481943){\makebox(0,0)[lt]{\lineheight{1.25}\smash{\begin{tabular}[t]{l}:\end{tabular}}}}%
    \put(0.21511417,0.08465767){\makebox(0,0)[lt]{\lineheight{1.25}\smash{\begin{tabular}[t]{l}:\end{tabular}}}}%
  \end{picture}%
\endgroup%
\caption{Map $\varphi$ near boundary points $A_{i,j}$, $B_{i,j}$.}
\label{Fig:decoration_AB}
\end{figure}
\begin{figure}
\def\svgwidth{90mm}
\begingroup%
  \makeatletter%
  \providecommand\rotatebox[2]{#2}%
  \newcommand*\fsize{\dimexpr\f@size pt\relax}%
  \newcommand*\lineheight[1]{\fontsize{\fsize}{#1\fsize}\selectfont}%
  \ifx\svgwidth\undefined%
    \setlength{\unitlength}{426.7556192bp}%
    \ifx\svgscale\undefined%
      \relax%
    \else%
      \setlength{\unitlength}{\unitlength * \real{\svgscale}}%
    \fi%
  \else%
    \setlength{\unitlength}{\svgwidth}%
  \fi%
  \global\let\svgwidth\undefined%
  \global\let\svgscale\undefined%
  \makeatother%
  \begin{picture}(1,0.52450081)%
    \lineheight{1}%
    \setlength\tabcolsep{0pt}%
    \put(0,0){\includegraphics[width=\unitlength,page=1]{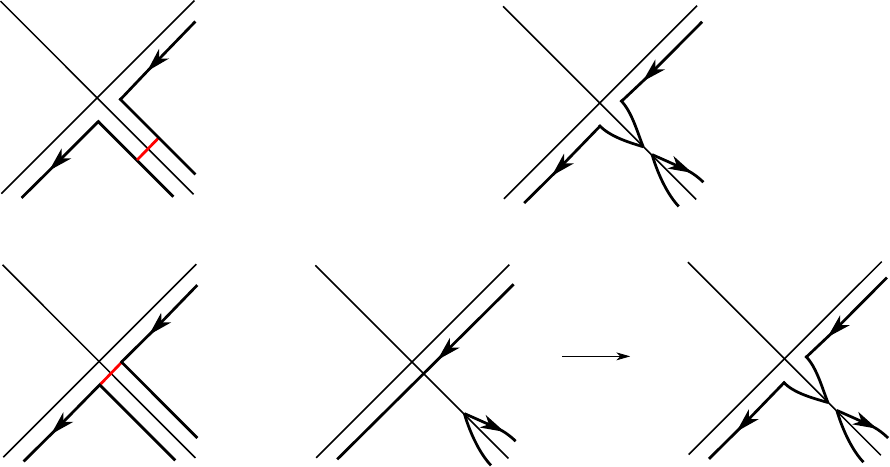}}%
    \put(0.27885893,0.40982409){\makebox(0,0)[lt]{\lineheight{1.25}\smash{\begin{tabular}[t]{l}:\end{tabular}}}}%
    \put(0.27897601,0.10481821){\makebox(0,0)[lt]{\lineheight{1.25}\smash{\begin{tabular}[t]{l}:\end{tabular}}}}%
  \end{picture}%
\endgroup%
\caption{Map $\varphi$ near boundary points $C_{i,j}$.}
\label{Fig:decoration_C}
\end{figure}
Moreover, strings in 
$$\sum_{\alpha\in\partial P}\varphi(\alpha)$$ 
correspond to summands in
$$\delta\{x,y\}_1+\{\delta x,y\}_1+\{x,\delta y\}_1$$
that arise in one of the following ways
\begin{itemize}
\item[i)] after gluing $x$ and $y$ at a gluing pair, we insert a trivial strip at the newly created crossing of the glued curve with the Reeb chord endpoint;
\item[ii)] after inserting a trivial strip into $x$, we glue the newly created puncture to a puncture in $y$;
\item[iii)] after inserting a trivial strip into $y$, we glue the newly created puncture to a puncture in $x$.
\end{itemize} 
Using the fact that $\varphi$ is well defined, we conclude that these summands cancel out. It is easy to see that all the other summands in $\delta\{x,y\}_1+\{\delta x,y\}_1+\{x,\delta y\}_1$ come in pairs and cancel out since the order of the operations for these summands can be reversed. This finishes the proof when $x,y\in\widetilde\cC^1$. For $x$ or $y$ in $\widetilde\cC^2$ or $\overline\cC$, the statement follows trivially.
\end{proof}
\subsubsection{Nodal annulus contribution---corrected loop coproduct}\label{Sec:loop_coproduct}
In this section, we define the corrected loop coproduct $\nabla:\cC\to\cC$. Here $\nabla$ does not vanish only on strings with one positive puncture.

Let $u:\lD\backslash\{t_1,\dots,t_m\}\to\lR^4$ be a smooth map. We say disk $u$ is positively asymptotic to a Reeb chord $\gamma\in\cR$ at a puncture $t_i,i\in\{1,\dots,m\}$ if for $\phi:[0,\infty)\times[0,1]\to\lD\backslash\{t_1,\dots,t_m\}$ a holomorphic parameterization of a neighborhood of $t_i$ and some $\alpha_i\in\lR$, 
$$u\circ\phi(s+R,t)|_{(0,\infty)\times[0,1]}-(l_\gamma R+\alpha_i)\partial_r$$ 
$C^1$-converges to the trivial strip $(s,t)\to (l_\gamma s,\gamma(t))$ over $\gamma$ as $R\to\infty$. Similarly, it is negatively asymptotic to $\gamma$ if $u\circ\phi(s+R,t)|_{(0,\infty)\times[0,1]}+(l_\gamma R+\alpha_i)\partial_r$ converges to $(s,t)\to (-l_\gamma s,\gamma(1-t))$ as $R\to\infty$. 

A \textit{punctured disk} on $L=\lR\times\Lambda$ is defined as a smooth map $u:\lD\backslash\{t_1,\dots,t_m\}\to\lR^4$ that is positively or negatively asymptotic to some Reeb chord on $\Lambda$ at each puncture $t_i$, with boundary mapped to $L=\lR\times\Lambda$, such that there exist neighborhoods $U_i\subset\lD\backslash\{t_1,\dots,t_k\}$ of $t_i$ such that $u$ is an embedding when restricted to $\bigsqcup_{i=1}^m \partial U_i$ and $\pi_{xy}\circ u|_{U_i}$ is a local embedding for all $i$.

For a punctured disk $u$ on $L$, we define generic asymptotic and generic relative asymptotic behavior same as before, see Definition \ref{Def:generic_asym_beh} and Definition \ref{Def:generic_rel_asym_beh}.

\begin{defi}
A punctured disk $u$ on $L$ is called \textit{admissible} if its restriction to the boundary is an immersion, it has generic asymptotic behavior at every positive or negative puncture and generic relative asymptotic behavior at all pairs of positive and pairs of negative punctures asymptotic to the same Reeb chord. 
\end{defi}

Let $u$ be an admissible punctured disk with one positive puncture. The boundary of $u$ is immersed and, after a small perturbation away from the punctures, we can assume that the self-intersections of $\overline\gamma\coloneq u|_\partial\subset\lR\times\Lambda=L$ are transverse. We say a self-intersection $A$ of $\overline\gamma$ is positive if the tangent vectors to $\overline\gamma$ at $A$ in the order of appearance starting from the positive puncture form a positive basis in $T_AL$. If the intersection $A$ is positive, we define $\epsilon(A)=1$, and $-1$ otherwise.

Consider the map $\operatorname{sh}(u):\lD\backslash\{t_1,\dots,t_k\}\to\lR^4$ obtained from $u$ by taking a small shift near the boundary in direction $J\overline\gamma'(t),t\in\lS^1\backslash\{t_1,\dots,t_k\}$. We can assume that $\operatorname{sh}(u)$ intersects $L$ transversally by perturbing away from the boundary. Note that $J\overline\gamma'(t)\not\in TL$ since $L$ is a Lagrangian submanifold and $\overline\gamma'\neq 0$, therefore, the boundary of $\operatorname{sh}(u)$ does not intersect $L$. We say an intersection $B$ of the image of $\operatorname{sh}(u)$ and the Lagrangian cylinder $L$ is positive if a positive basis of $L$ and a positive basis of $\operatorname{sh}(u)$ at $B\in\operatorname{sh}(u)\cap L\subset\lR^4$ form a positive basis in $\lR^4$. If the intersection $B$ is positive, we define $\epsilon(B)=1$, and $-1$ otherwise.

\vspace{2.1mm}
Next, we define the corrected loop coproduct $\nabla:\cC\to\cC$. Let $\gamma:\lS^1\backslash\{t_1,\dots,t_k\}\to\Lambda$ be a string in $\widetilde\cC^1$ with a positive puncture without loss of generality at $t_k$, and $u:\lD\backslash\{t_1,\dots,t_k\}\to\lR^4$ an admissible punctured disk on $L$ such that $\pi_{xyz}\circ u|_\partial=\gamma$. Denote $\overline\gamma=u|_\partial$. We can assume that the self-intersections of $\overline\gamma\subset L$ are transverse and that $\operatorname{sh}(u)$ is transverse to $L$ as above. 

For every self-intersection $A$ of $\overline\gamma$, we define a pair of strings
$$\nabla_1(u,A)=\pi_{xyz}\circ\overline\gamma_1\otimes \pi_{xyz}\circ\overline\gamma_2,$$
where $\overline\gamma_1,\overline\gamma_2$ are two punctured loops on $L$ obtained from $\overline\gamma$ by resolving the self-intersection $A$ as shown in Figure \ref{Fig:resolv_intersect} (here $\overline \gamma_1$ contains the positive puncture). If $A$ is in the intersection of the arcs $\overline\gamma(t_i,t_{i+1})$ and $\overline\gamma(t_j,t_{j+1})$ for $i\leq j$, the ordering of the punctures on $\pi_{xyz}\circ\overline\gamma_2$ is given by $i+1,\dots, j$ and on $\pi_{xyz}\circ\overline\gamma_1$ by $1,\dots,i,j+1,\dots,k$.
\begin{figure}
\def\svgwidth{76mm}
\begingroup%
  \makeatletter%
  \providecommand\rotatebox[2]{#2}%
  \newcommand*\fsize{\dimexpr\f@size pt\relax}%
  \newcommand*\lineheight[1]{\fontsize{\fsize}{#1\fsize}\selectfont}%
  \ifx\svgwidth\undefined%
    \setlength{\unitlength}{334.89893964bp}%
    \ifx\svgscale\undefined%
      \relax%
    \else%
      \setlength{\unitlength}{\unitlength * \real{\svgscale}}%
    \fi%
  \else%
    \setlength{\unitlength}{\svgwidth}%
  \fi%
  \global\let\svgwidth\undefined%
  \global\let\svgscale\undefined%
  \makeatother%
  \begin{picture}(1,0.34350899)%
    \lineheight{1}%
    \setlength\tabcolsep{0pt}%
    \put(0,0){\includegraphics[width=\unitlength,page=1]{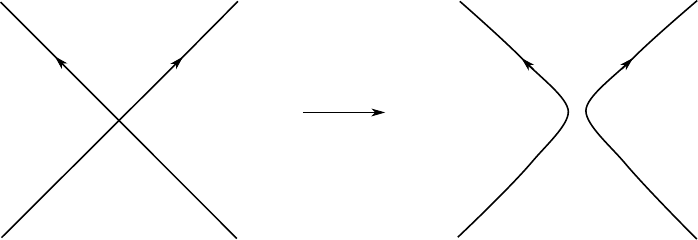}}%
  \end{picture}%
\endgroup%
\caption{Resolving a string at a self-intersection.}
\label{Fig:resolv_intersect}
\end{figure}
Furthermore, for every intersection $B$ between the Lagrangian cylinder $L$ and the shifted map $\operatorname{sh}(u)$ as before, we define a pair of strings
$$\nabla_2(u,B)=\gamma\otimes \pi_{xyz}(B),$$
where $\pi_{xyz}(B)\in\widetilde\cC^0$ is the constant string.

Now we define a degree $-1$ linear map $\nabla:\cC\to\cC$ as follows. For $\gamma\in\widetilde\cC^1$ a string and $u$ an admissible punctured disk such that $\pi_{xyz}\circ u|_\partial=\gamma$, we define
\begin{align*}
&\nabla(\gamma)=\nabla(u)=\sum_A
\epsilon_1(u,A)\nabla_1(u,A)+\sum_B\epsilon_2(u,B)\nabla_2(u,B),
\end{align*}
where the first sum goes over all self-intersections $A$ of $u|_\partial$ and the second sum goes over all intersections $B$ between $L$ and the shifted map $\operatorname{sh}(u)$. The signs are given by 
\begin{align*}
&\epsilon_1(u,A)=(-1)^{\left(P(\gamma,t_{i+1})+1\right)|\gamma_2|+1}\epsilon(A),\\
&\epsilon_2(u,B)=\epsilon(B),
\end{align*}
where $\epsilon(A),\epsilon(B)$ are the signs of the intersection points $A,B$, $t_i$ is the puncture of $\gamma$ such that $A$ is in the intersection of arcs $u|_\partial(t_i,t_{i+1})$ and $ u|_\partial(t_j,t_{j+1})$ for $i\leq j$, and $\gamma_2=\pi_{xyz}\circ\overline\gamma_2$ is the second string in $\nabla_1(u,A)$. Additionally, for $\gamma\in\widetilde\cC^2$ and $\gamma\in\overline\cC$, we define $\nabla(\gamma)=0$.
\begin{lemma}\label{Lemma:corrected_intersection_class}
The map $\nabla:\cC\to\cC$ is well defined, i.e. $\nabla(\gamma)\in\cC$ does not depend on the representative of the broken closed string $\gamma$ and the choice of an admissible punctured disk $u$ such that $\pi_{xyz}\circ u|_\partial=\gamma$.
\end{lemma}
\begin{proof}
Let $\gamma,\widetilde\gamma$ be two representatives of a broken closed string in $\widetilde\cC^1$ and $u,\widetilde u$ admissible punctured disks such that $\pi_{xyz}\circ u|_\partial=\gamma$ and $\pi_{xyz}\circ \widetilde u|_\partial=\widetilde\gamma$ that satisfy the transversality conditions above. We can find a generic smooth isotopy $u_s,s\in[0,1]$ from $u$ to $\widetilde u$ such that $\pi_{xyz}\circ u_s|_\partial$ is a broken closed string and $u_s$ is an admissible punctured disk satisfying the transversality conditions as before except for finitely many singular points $s_0\in(0,1)$ of the following four types
\begin{enumerate}
\item[i)] $\gamma_{s_0}$ is not immersed,
\item[ii)] $\gamma_{s_0}$ has a non-transverse self-intersection,
\item[iii)] $\operatorname{sh}(u_{s_0})$ has a non-transverse intersection with $L$,
\item[iv)] at $s_0$ we have a change of the asymptotic behavior,
\end{enumerate}
where $\gamma_s=u_s|_\partial$. If there are no such singularities in the interval $[s_1,s_2]$, then the isotopy easily gives us $\nabla(u_{s_1})=\nabla(u_{s_2})$.

First, assume $[s_0-\theta,s_0+\theta]$ has one singularity at $s_0$ of the first type. We can see that one boundary self-intersection appears/disappears while one interior intersection disappears/appears (see Figure \ref{Fig:wind_degeneration}). More precisely, since the chosen isotopy is generic, we can model a neighborhood of the singular boundary point by $\widetilde\gamma_s(t)=t^2\partial_\tau+(t^3-(s-s_0)t)\partial_r$ in local coordinates $(r,\tau)$ on $L$ (or reversed). The boundaries of the shifted maps $\operatorname{sh}(u_{s}),s\in[s_0-\theta,s_0+\theta]$ in this neighborhood are given by 
$$B(s,t)=\widetilde\gamma_s(t)+\varepsilon J\widetilde\gamma'_s(t)=t^2\partial_\tau+(t^3-(s-s_0)t)\partial_r+2\varepsilon tJ\partial_\tau+(3\varepsilon t^2-\varepsilon(s-s_0))J\partial_r,$$ 
for $\varepsilon>0$ a small constant. The image of the map $B(s,t)$ gives us a cobordism between the boundaries of $\operatorname{sh}(u_{s_0-\theta})$ and $\operatorname{sh}(u_{s_0+\theta})$, therefore, the change of the interior intersection number is equal to the count of the intersections between $L$ and $B$. It is easy to see that they intersect only at $(s,t)=(s_0,0)$. Moreover, this intersection is negative since the vectors $\partial_r, \partial_\tau,\partial_tB(s_0,0)=2\varepsilon J\partial_\tau,\partial_sB(s_0,0)=-\varepsilon J\partial_r$ form a negative basis in $\lR^4$. This shows that one negative interior intersection disappears in the isotopy, which we denote by $B$.  Additionally, it is clear that one positive boundary self-intersection appears, which we denote by $A$. Then, we have (see Figure \ref{Fig:wind_degeneration})
\begin{align*}
\nabla(u_{s_0-\theta})-\nabla(u_{s_0+\theta})=-\nabla_2(u_{s_0-\theta},B)+\nabla_1(u_{s_0+\theta},A)=0.
\end{align*}

If $s_0$ is a singularity of the second type, clearly two boundary self-intersections of opposite signs appear or disappear. Then we have (see Figure \ref{Fig:trans_knot_degeneration})
\begin{align*}
\nabla(u_{s_0-\theta})-\nabla(u_{s_0+\theta})=\pm\left(\nabla_1(u_{s_0-\theta},A_1)-\nabla_1(u_{s_0-\theta},A_2)\right)=0.
\end{align*}

Similarly if $s_0$ is a singularity of the third type, we have (see Figure \ref{Fig:trans_cyl_degeneration})
\begin{align*}
\nabla(u_{s_0-\theta})-\nabla(u_{s_0+\theta})=\pm\left(\nabla_2(u_{s_0-\theta},B_1)-\nabla_2(u_{s_0-\theta},B_2)\right)=0.
\end{align*}

Finally, assume $s_0$ is a singularity of the fourth type, i.e. we have a change of the relative asymptotic behavior at two punctures $t_j,t_k,j< k$ asymptotic to a Reeb chord $\gamma_\iota$. This can be seen as two boundary intersections $A,A'$ simultaneously appearing/disappearing at the two ends of $t_j,t_k$, see Figure \ref{Fig:asympt_degeneration}. It is not difficult to see that $\epsilon(A)=-\epsilon(A')$ if $|\gamma_\iota|$ is even and $\epsilon(A)=\epsilon(A')$ if $|\gamma_\iota|$ is odd. Additionally, the order of the punctures of the second word of $\nabla_1(u_{s_0+\theta},A_1)$ is shifted by one compared to the second word of $\nabla_1(u_{s_0+\theta},A_2)$. In particular, we have 
$$\nabla_1(u_{s_0+\theta},A_1)=(-1)^{|\gamma_\iota^-|(|\gamma_2|+|\gamma_\iota^-|)}\nabla_1(u_{s_0+\theta},A_2)$$ 
in $\overline\cC$. Additionally, 
$$\epsilon_1(u_{s_0+\theta},A_1)=-(-1)^{(P(\gamma_{s_0},t_j)+1)|\gamma_2|}\epsilon(A_1)$$ 
and 
$$\epsilon_1(u_{s_0+\theta},A_2)=-(-1)^{(P(\gamma_{s_0},t_{j+1})+1)|\gamma_2|}\epsilon(A_2)=-(-1)^{(P(\gamma_{s_0},t_{j})+1)|\gamma_2|+|\gamma_\iota^-||\gamma_2|+|\gamma_\iota^-|+1}\epsilon(A_1).$$ 
Then we have (see Figure \ref{Fig:asympt_degeneration})
\begin{align*}
\nabla(u_{s_0-\theta})-\nabla(u_{s_0+\theta})=\pm\left(\nabla_1(u_{s_0+\theta},A_1)+(-1)^{|\gamma_\iota^-||\gamma_2|+|\gamma_\iota^-|+1}\nabla_1(u_{s_0+\theta},A_2)\right)=0,
\end{align*}
which finishes the proof.
\end{proof}

\begin{figure}
\def\svgwidth{70mm}
\begingroup%
  \makeatletter%
  \providecommand\rotatebox[2]{#2}%
  \newcommand*\fsize{\dimexpr\f@size pt\relax}%
  \newcommand*\lineheight[1]{\fontsize{\fsize}{#1\fsize}\selectfont}%
  \ifx\svgwidth\undefined%
    \setlength{\unitlength}{414.58593086bp}%
    \ifx\svgscale\undefined%
      \relax%
    \else%
      \setlength{\unitlength}{\unitlength * \real{\svgscale}}%
    \fi%
  \else%
    \setlength{\unitlength}{\svgwidth}%
  \fi%
  \global\let\svgwidth\undefined%
  \global\let\svgscale\undefined%
  \makeatother%
  \begin{picture}(1,0.54034273)%
    \lineheight{1}%
    \setlength\tabcolsep{0pt}%
    \put(0,0){\includegraphics[width=\unitlength,page=1]{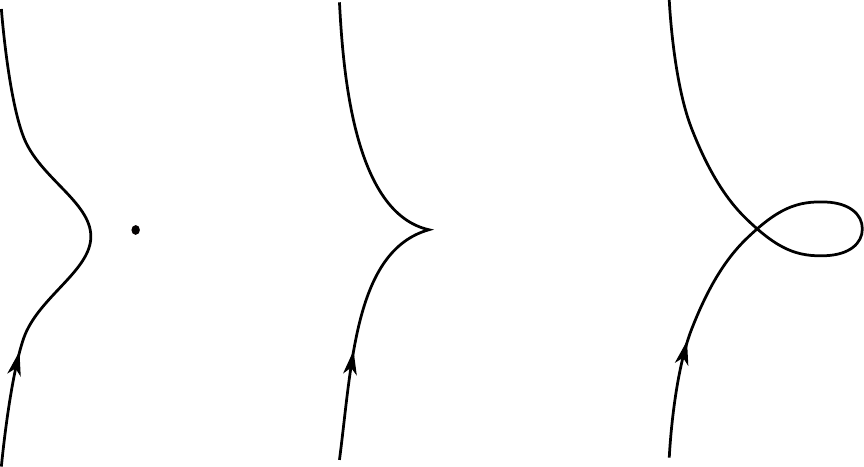}}%
    \put(0.85652708,0.31672838){\makebox(0,0)[lt]{\lineheight{1.25}\smash{\begin{tabular}[t]{l}$A$\end{tabular}}}}%
    \put(0.13934903,0.31230818){\makebox(0,0)[lt]{\lineheight{1.25}\smash{\begin{tabular}[t]{l}$B$\end{tabular}}}}%
  \end{picture}%
\endgroup%
\caption{Singularity of the first type.}
\label{Fig:wind_degeneration}
\end{figure}

\begin{figure}
\def\svgwidth{101mm}
\begingroup%
  \makeatletter%
  \providecommand\rotatebox[2]{#2}%
  \newcommand*\fsize{\dimexpr\f@size pt\relax}%
  \newcommand*\lineheight[1]{\fontsize{\fsize}{#1\fsize}\selectfont}%
  \ifx\svgwidth\undefined%
    \setlength{\unitlength}{650.85231583bp}%
    \ifx\svgscale\undefined%
      \relax%
    \else%
      \setlength{\unitlength}{\unitlength * \real{\svgscale}}%
    \fi%
  \else%
    \setlength{\unitlength}{\svgwidth}%
  \fi%
  \global\let\svgwidth\undefined%
  \global\let\svgscale\undefined%
  \makeatother%
  \begin{picture}(1,0.80829641)%
    \lineheight{1}%
    \setlength\tabcolsep{0pt}%
    \put(0,0){\includegraphics[width=\unitlength,page=1]{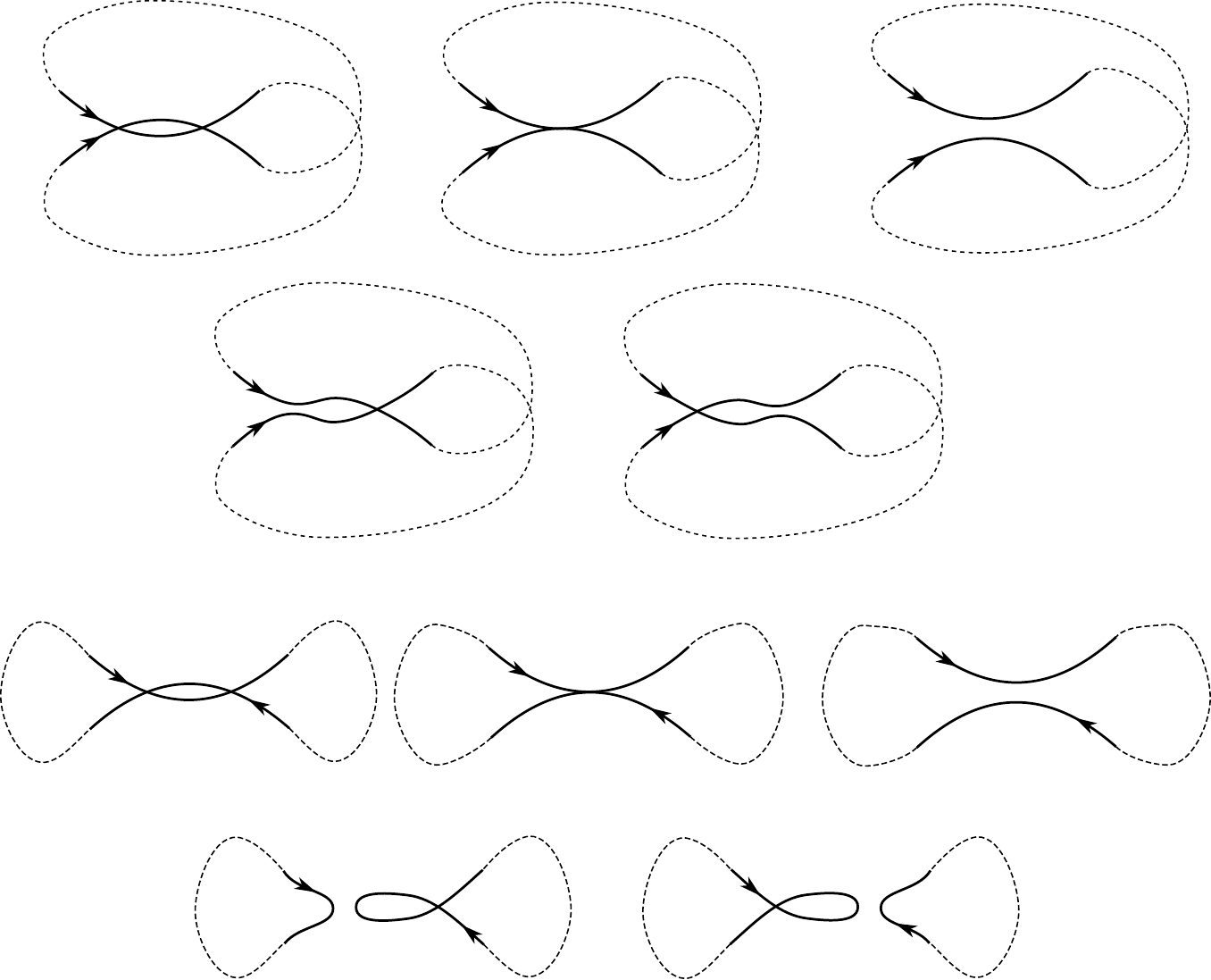}}%
    \put(0.46916431,0.46223683){\makebox(0,0)[lt]{\lineheight{1.25}\smash{\begin{tabular}[t]{l}$\cong$\end{tabular}}}}%
    \put(0.48757857,0.05213678){\makebox(0,0)[lt]{\lineheight{1.25}\smash{\begin{tabular}[t]{l}$\cong$\end{tabular}}}}%
  \end{picture}%
\endgroup%
\caption{Singularity of the second type.}
\label{Fig:trans_knot_degeneration}
\end{figure}

\begin{figure}
\def\svgwidth{90mm}
\begingroup%
  \makeatletter%
  \providecommand\rotatebox[2]{#2}%
  \newcommand*\fsize{\dimexpr\f@size pt\relax}%
  \newcommand*\lineheight[1]{\fontsize{\fsize}{#1\fsize}\selectfont}%
  \ifx\svgwidth\undefined%
    \setlength{\unitlength}{347.73590195bp}%
    \ifx\svgscale\undefined%
      \relax%
    \else%
      \setlength{\unitlength}{\unitlength * \real{\svgscale}}%
    \fi%
  \else%
    \setlength{\unitlength}{\svgwidth}%
  \fi%
  \global\let\svgwidth\undefined%
  \global\let\svgscale\undefined%
  \makeatother%
  \begin{picture}(1,0.24091618)%
    \lineheight{1}%
    \setlength\tabcolsep{0pt}%
    \put(0,0){\includegraphics[width=\unitlength,page=1]{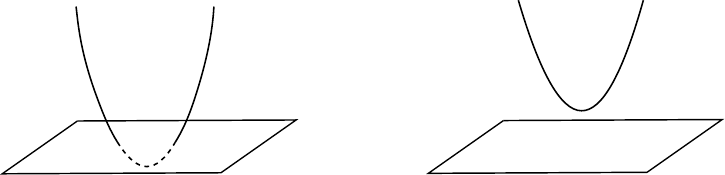}}%
    \put(0.11163284,0.02565332){\makebox(0,0)[lt]{\lineheight{1.25}\smash{\begin{tabular}[t]{l}$B$\end{tabular}}}}%
    \put(0.25071945,0.02552956){\makebox(0,0)[lt]{\lineheight{1.25}\smash{\begin{tabular}[t]{l}$B'$\end{tabular}}}}%
  \end{picture}%
\endgroup%
\caption{Singularity of the third type, $\gamma\otimes \pi_{xyz}(B)\cong\gamma\otimes\pi_{xyz}(B')$.}
\label{Fig:trans_cyl_degeneration}
\end{figure}

\begin{figure}
\def\svgwidth{60mm}
\begingroup%
  \makeatletter%
  \providecommand\rotatebox[2]{#2}%
  \newcommand*\fsize{\dimexpr\f@size pt\relax}%
  \newcommand*\lineheight[1]{\fontsize{\fsize}{#1\fsize}\selectfont}%
  \ifx\svgwidth\undefined%
    \setlength{\unitlength}{343.09611846bp}%
    \ifx\svgscale\undefined%
      \relax%
    \else%
      \setlength{\unitlength}{\unitlength * \real{\svgscale}}%
    \fi%
  \else%
    \setlength{\unitlength}{\svgwidth}%
  \fi%
  \global\let\svgwidth\undefined%
  \global\let\svgscale\undefined%
  \makeatother%
  \begin{picture}(1,0.80103024)%
    \lineheight{1}%
    \setlength\tabcolsep{0pt}%
    \put(0.50086528,0.13789051){\makebox(0,0)[lt]{\lineheight{1.25}\smash{\begin{tabular}[t]{l}$\cong$\end{tabular}}}}%
    \put(0.60429727,0.58751657){\makebox(0,0)[lt]{\lineheight{1.25}\smash{\begin{tabular}[t]{l}$A'$\end{tabular}}}}%
    \put(0.93585195,0.58660195){\makebox(0,0)[lt]{\lineheight{1.25}\smash{\begin{tabular}[t]{l}$A$\end{tabular}}}}%
    \put(0,0){\includegraphics[width=\unitlength,page=1]{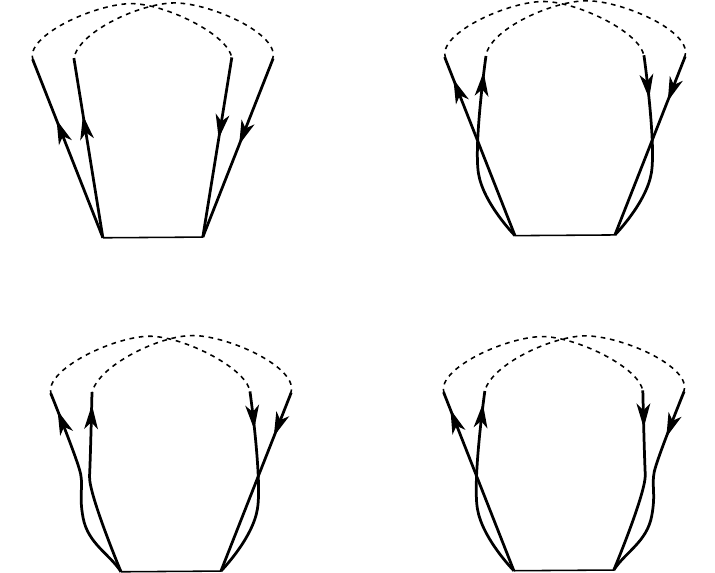}}%
    \put(-0.00208396,0.75309725){\makebox(0,0)[lt]{\lineheight{1.25}\smash{\begin{tabular}[t]{l}\textcolor{white}{.}\end{tabular}}}}%
  \end{picture}%
\endgroup%
\caption{Singularity of the fourth type at an even Reeb chord.}
\label{Fig:asympt_degeneration}
\end{figure}
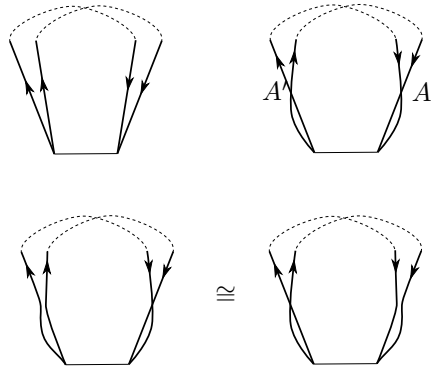

Finally, we define the string operator.
\begin{defi}
The string operator $d_{\str}:\cC\to\cC$ is defined as
$$d_{\str}=\delta+\nabla.$$
\end{defi}
\subsection{Properties of the SFT bracket and the string operator}\label{Sec:properties}
In this section we discuss some additional properties of the SFT bracket and the string operator.

We trivially get that the string operator squares to zero by definition, since $d_{\str}$ vanishes on $\widetilde\cC^2,\overline\cC$ and $d_{str}(\widetilde\cC^1)\subset\widetilde\cC^2\oplus\overline\cC$.
\begin{lemma}
For all $s\in\cC$ we have $d_{\str}^2(s)=0$.
\end{lemma}

The string operator is additionally a derivation with respect to the SFT bracket.
\begin{lemma}
The string operator $d_{\str}:\cC\to\cC$ is a derivation with respect to the SFT bracket $\{\cdot,\cdot\}$, i.e. 
$$d_{\str}\{x,y\}=\{x,d_{\str}y\}-(-1)^{|y|}\{d_{\str}x,y\}$$
for all $x,y\in\cC$.
\end{lemma}
\begin{proof}
Let $x:\lS^1\backslash\{\tau_1,\dots,\tau_s\}\to\Lambda$ and $y:\lS^1\backslash\{\tau_1',\dots,\tau_{s'}'\}\to\Lambda$ be representatives of two broken closed strings in $\widetilde\cC^1$. From Lemma \ref{Lemma:reduced_prop} we have $\delta\{x,y\}_1=\{x,\delta y\}_1-(-1)^{|y|}\{\delta x,y\}_1$, therefore, it is enough to show
\begin{align}\label{Eq:FullDeriv}
\nabla\{x,y\}_1=\{x,\nabla y\}_1-(-1)^{|y|}\{\nabla x,y\}_1+\{x,\delta y\}_2-(-1)^{|y|}\{\delta x,y\}_2.
\end{align}

First, we work without signs. Let $u_x:\lD^2\backslash\{\tau_1,\dots,\tau_s\}\to\lR^4$ and $u_y:\lD^2\backslash\{\tau_1',\dots,\tau_{s'}'\}\to\lR^4$ be two admissible punctured disks on $L$ such that $\pi_{xyz}\circ u_x|_\partial=x,\pi_{xyz}\circ u_y|_\partial=y$. Denote by $\bm{A}_x$ and $\bm{A}_y$ the sets of self-intersections of $u_x|_\partial$ and $u_y|_\partial$, and by $\bm{B}_x$ and $\bm{B}_y$ the sets of intersections of the shifted maps $\operatorname{sh}(u_x)$ and $\operatorname{sh}(u_y)$ with $L$.

Let $\tau_j'$ be the positive puncture on $y$ and $\tau_i$ a negative puncture on $x$ asymptotic to the same Reeb chord. We construct an admissible disk $u_{\tau_i,\tau_j'}$ such that $\pi_{xyz}\circ u_{\tau_i,\tau_j'}|_\partial=x\prescript{}{\tau_i\hspace{-0.65mm}}{*}_{\tau_j'}y$ by gluing punctured disks $u_x,u_y$ at $\tau_i,\tau_j'$. Let $\gamma$ be the Reeb chord at the negative puncture $\tau_k,k\neq i$ of $x$. Denote the intersection points of $y$ with the trivial strip over $\gamma$ by $\tau'_{k,r}\in\lS^1\backslash\{\tau_1',\dots,\tau_{s'}'\},r=1,\dots, r_k$ and $A_{k,r}=y(\tau'_{k,r})$. These points appear as boundary self-intersections of $u_{\tau_i,\tau_j'}$. We additionally have one boundary self-intersection point $C_{i,l}$ for each negative puncture $\tau_l\neq\tau_i$ on $x$ negatively asymptotic to the same Reeb orbit as $\tau_i$ because of the difference between the asymptotic representatives at a negative and a positive puncture at a Reeb chord (see Figure \ref{Fig:add_intersections_cancelation}). The self-intersections of $u_{\tau_i,\tau_j'}|_\partial$ consist of self-intersections $\bm{A}_x$ of $u_x|_\partial$, self-intersections $\bm{A}_y$ of $u_y|_\partial$, intersections $A_{k,r},k\neq i,r=1,\dots,r_k$ and intersections $C_{i,l}$ as above. Additionally, it is easy to see that
\begin{align*}
&\nabla_1(u_{\tau_i,\tau_j'},A_x)=\nabla_1(u_x,A_x)\prescript{}{\tau_i\hspace{-0.65mm}}{*}_{\tau_j'}y,\text{ for }A_x\in\bm{A}_x\\
&\nabla_1(u_{\tau_i,\tau_j'},A_y)=x\prescript{}{\tau_i\hspace{-0.65mm}}{*}_{\tau_j'}\nabla_1(u_y,A_y),\text{ for }A_y\in\bm{A}_y\\
&\nabla_1(u_{\tau_i,\tau_j'},A_{k,r})=x(\tau_i,\tau_j';\tau_k,\tau_+')\delta(y,\tau'_{k,r}),
\end{align*}
where $\tau_+'$ denotes the positive puncture on $\delta(y,\tau'_{k,r})$ coming from the inserted trivial strip. Similar holds when we glue the positive puncture of $x$ to a negative puncture on $y$. Moreover, it is not difficult to see that the strings $\nabla_1(u_{\tau_i,\tau_j'},C_{i,l})$ and $\nabla_1(u_{\tau_l,\tau_j'},C_{l,i})$ are equivalent (see Figure \ref{Fig:add_intersections_cancelation}).
\begin{figure}
\def\svgwidth{79mm}
\begingroup%
  \makeatletter%
  \providecommand\rotatebox[2]{#2}%
  \newcommand*\fsize{\dimexpr\f@size pt\relax}%
  \newcommand*\lineheight[1]{\fontsize{\fsize}{#1\fsize}\selectfont}%
  \ifx\svgwidth\undefined%
    \setlength{\unitlength}{957.27540989bp}%
    \ifx\svgscale\undefined%
      \relax%
    \else%
      \setlength{\unitlength}{\unitlength * \real{\svgscale}}%
    \fi%
  \else%
    \setlength{\unitlength}{\svgwidth}%
  \fi%
  \global\let\svgwidth\undefined%
  \global\let\svgscale\undefined%
  \makeatother%
  \begin{picture}(1,0.85896806)%
    \lineheight{1}%
    \setlength\tabcolsep{0pt}%
    \put(0,0){\includegraphics[width=\unitlength,page=1]{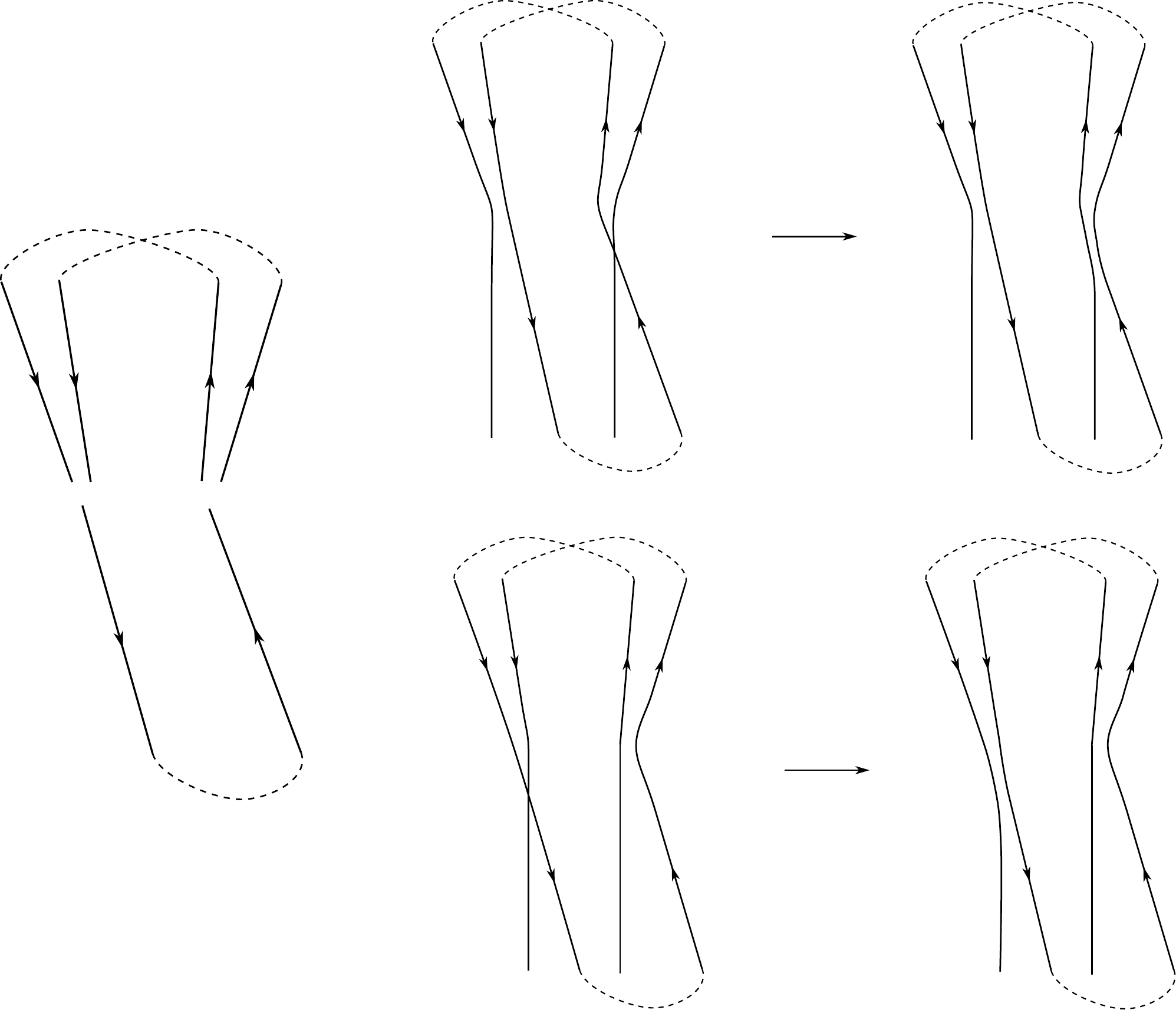}}%
    \put(0.10844368,0.53909753){\makebox(0,0)[lt]{\lineheight{1.25}\smash{\begin{tabular}[t]{l}$x$\end{tabular}}}}%
    \put(0.14591056,0.30799818){\makebox(0,0)[lt]{\lineheight{1.25}\smash{\begin{tabular}[t]{l}$y$\end{tabular}}}}%
    \put(0.02876428,0.45081424){\makebox(0,0)[lt]{\lineheight{1.25}\smash{\begin{tabular}[t]{l}$i$\end{tabular}}}}%
    \put(0.08435191,0.45126192){\makebox(0,0)[lt]{\lineheight{1.25}\smash{\begin{tabular}[t]{l}$l$\end{tabular}}}}%
    \put(0.53479557,0.64027765){\makebox(0,0)[lt]{\lineheight{1.25}\smash{\begin{tabular}[t]{l}$C_{l,i}$\end{tabular}}}}%
    \put(0.36591439,0.17325275){\makebox(0,0)[lt]{\lineheight{1.25}\smash{\begin{tabular}[t]{l}$C_{i,l}$\end{tabular}}}}%
  \end{picture}%
\endgroup%
\caption{Strings $\nabla_1(u_{\tau_l,\tau_j'},C_{l,i})$ and $\nabla_1(u_{\tau_i,\tau_j'},C_{i,l})$ are equivalent.}
\label{Fig:add_intersections_cancelation}
\end{figure}
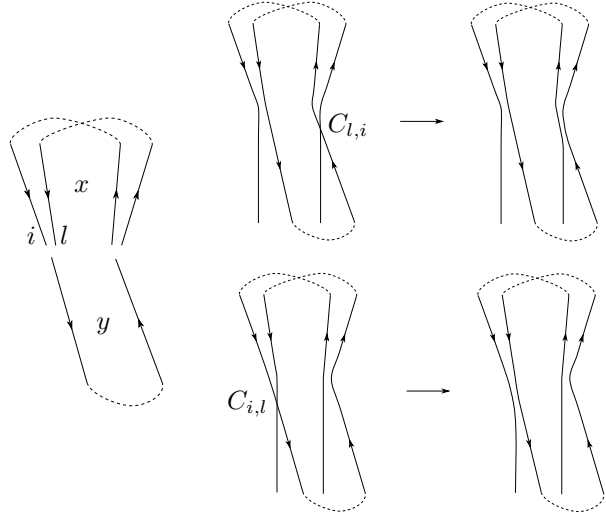
Additionally, the intersections of the shift of $u_{\tau_i,\tau_j'}$ with the cylinder $L$ consist of the intersections $\bm{B}_x$ of $\operatorname{sh}(u_x)$ and the intersections $\bm{B}_y$ of $\operatorname{sh}(u_y)$ with $L$, and we have 
\begin{align*}
&\nabla_2(u_{\tau_i,\tau_j'},B_x)=\nabla_2(u_x,B_x)\prescript{}{\tau_i\hspace{-0.65mm}}{*}_{\tau_j'}y,\text{ for }B_x\in\bm{B}_x\\
&\nabla_2(u_{\tau_i,\tau_j'},B_y)=x\prescript{}{\tau_i\hspace{-0.65mm}}{*}_{\tau_j'}\nabla_2(u_y,B_y),\text{ for }B_y\in\bm{B}_y.
\end{align*}
Taking the sum over all gluing pairs of punctures on $x$ and $y$, we get (\ref{Eq:FullDeriv}) up to signs.

\begin{figure}
\def\svgwidth{120mm}
\begingroup%
  \makeatletter%
  \providecommand\rotatebox[2]{#2}%
  \newcommand*\fsize{\dimexpr\f@size pt\relax}%
  \newcommand*\lineheight[1]{\fontsize{\fsize}{#1\fsize}\selectfont}%
  \ifx\svgwidth\undefined%
    \setlength{\unitlength}{587.05727895bp}%
    \ifx\svgscale\undefined%
      \relax%
    \else%
      \setlength{\unitlength}{\unitlength * \real{\svgscale}}%
    \fi%
  \else%
    \setlength{\unitlength}{\svgwidth}%
  \fi%
  \global\let\svgwidth\undefined%
  \global\let\svgscale\undefined%
  \makeatother%
  \begin{picture}(1,0.45626169)%
    \lineheight{1}%
    \setlength\tabcolsep{0pt}%
    \put(0.16455939,0.15621343){\makebox(0,0)[lt]{\lineheight{1.25}\smash{\begin{tabular}[t]{l}$t'_k$\end{tabular}}}}%
    \put(0.09214768,0.13602519){\makebox(0,0)[lt]{\lineheight{1.25}\smash{\begin{tabular}[t]{l}$t'_{k+1}$\end{tabular}}}}%
    \put(0.33670302,0.219487){\makebox(0,0)[lt]{\lineheight{1.25}\smash{\begin{tabular}[t]{l}$t'_1$\end{tabular}}}}%
    \put(0.20073371,0.44169887){\makebox(0,0)[lt]{\lineheight{1.25}\smash{\begin{tabular}[t]{l}$t_1$\end{tabular}}}}%
    \put(0.09552766,0.30108298){\makebox(0,0)[lt]{\lineheight{1.25}\smash{\begin{tabular}[t]{l}$t_{i-1}$\end{tabular}}}}%
    \put(0.24349308,0.29312906){\makebox(0,0)[lt]{\lineheight{1.25}\smash{\begin{tabular}[t]{l}$t_{i+1}$\end{tabular}}}}%
    \put(0.27516604,0.26827847){\makebox(0,0)[lt]{\lineheight{1.25}\smash{\begin{tabular}[t]{l}$t_{j-1}$\end{tabular}}}}%
    \put(0.42464327,0.29162802){\makebox(0,0)[lt]{\lineheight{1.25}\smash{\begin{tabular}[t]{l}$t_{j+1}$\end{tabular}}}}%
    \put(0.72057829,0.44476366){\makebox(0,0)[lt]{\lineheight{1.25}\smash{\begin{tabular}[t]{l}$t_1$\end{tabular}}}}%
    \put(0.6907814,0.29038934){\makebox(0,0)[lt]{\lineheight{1.25}\smash{\begin{tabular}[t]{l}$t_{i-1}$\end{tabular}}}}%
    \put(0.83955928,0.27815532){\makebox(0,0)[lt]{\lineheight{1.25}\smash{\begin{tabular}[t]{l}$t_{i+1}$\end{tabular}}}}%
    \put(0.72705253,0.25266218){\makebox(0,0)[lt]{\lineheight{1.25}\smash{\begin{tabular}[t]{l}$t'_1$\end{tabular}}}}%
    \put(0.94009981,0.10634692){\makebox(0,0)[lt]{\lineheight{1.25}\smash{\begin{tabular}[t]{l}$t'_k$\end{tabular}}}}%
    \put(0.9271945,0.17942107){\makebox(0,0)[lt]{\lineheight{1.25}\smash{\begin{tabular}[t]{l}$t'_{k+1}$\end{tabular}}}}%
    \put(0.5996336,0.10376583){\makebox(0,0)[lt]{\lineheight{1.25}\smash{\begin{tabular}[t]{l}$t'_{j+1}$\end{tabular}}}}%
    \put(0.63365111,0.17861651){\makebox(0,0)[lt]{\lineheight{1.25}\smash{\begin{tabular}[t]{l}$t'_{j}$\end{tabular}}}}%
    \put(0.27290869,0.37482816){\makebox(0,0)[lt]{\lineheight{1.25}\smash{\begin{tabular}[t]{l}$x$\end{tabular}}}}%
    \put(0.31036001,0.13349929){\makebox(0,0)[lt]{\lineheight{1.25}\smash{\begin{tabular}[t]{l}$y$\end{tabular}}}}%
    \put(0.78041109,0.36198198){\makebox(0,0)[lt]{\lineheight{1.25}\smash{\begin{tabular}[t]{l}$x$\end{tabular}}}}%
    \put(0.77912187,0.20324695){\makebox(0,0)[lt]{\lineheight{1.25}\smash{\begin{tabular}[t]{l}$y$\end{tabular}}}}%
    \put(0.1122898,0.19802575){\makebox(0,0)[lt]{\lineheight{1.25}\smash{\begin{tabular}[t]{l}$\tau$\end{tabular}}}}%
    \put(-0.00121794,0.09117889){\makebox(0,0)[lt]{\lineheight{1.25}\smash{\begin{tabular}[t]{l}\textcolor{white}{.}\end{tabular}}}}%
    \put(0.17258558,0.2690697){\makebox(0,0)[lt]{\lineheight{1.25}\smash{\begin{tabular}[t]{l}$t_{i}$\end{tabular}}}}%
    \put(0.36215511,0.27097553){\makebox(0,0)[lt]{\lineheight{1.25}\smash{\begin{tabular}[t]{l}$t_{j}$\end{tabular}}}}%
    \put(0.61453946,0.14001103){\makebox(0,0)[lt]{\lineheight{1.25}\smash{\begin{tabular}[t]{l}$A$\end{tabular}}}}%
    \put(0.94369324,0.14100347){\makebox(0,0)[lt]{\lineheight{1.25}\smash{\begin{tabular}[t]{l}$A$\end{tabular}}}}%
    \put(0,0){\includegraphics[width=\unitlength,page=1]{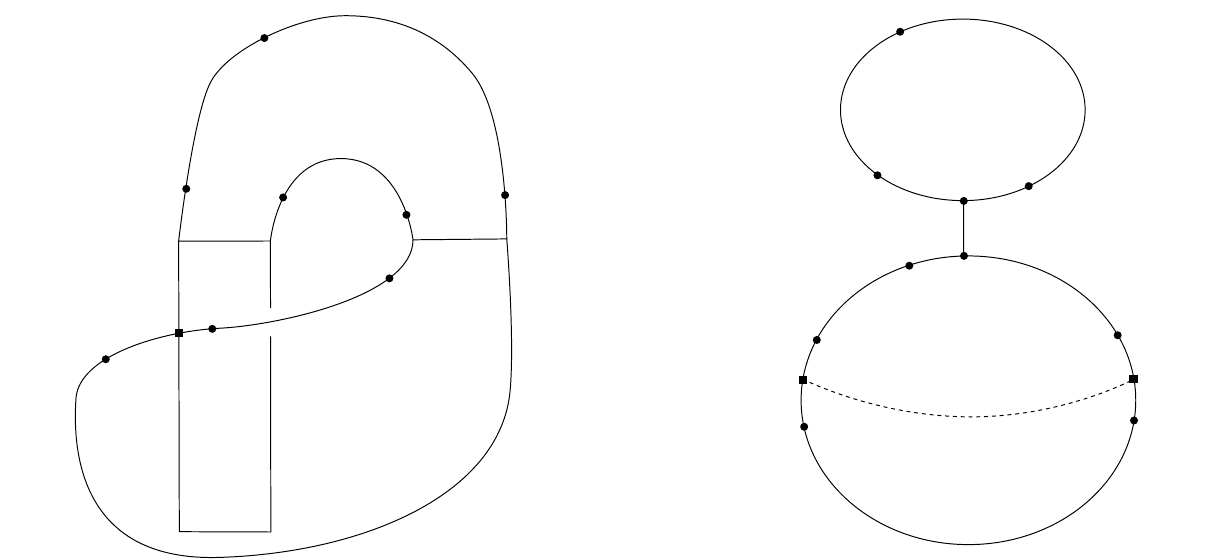}}%
    \put(0.77679866,0.30270604){\makebox(0,0)[lt]{\lineheight{1.25}\smash{\begin{tabular}[t]{l}$t_{i}$\end{tabular}}}}%
  \end{picture}%
\endgroup%
\caption{The first string can be glued in two ways and appears in $\nabla\{x,y\}_1$ and $\{x,\delta y\}_2$. Similarly for the second string, which appears in $\nabla\{x,y\}_1$ and $\{x,\nabla y\}_1$.}
\label{Fig:string_op_derivation_signs}
\end{figure}
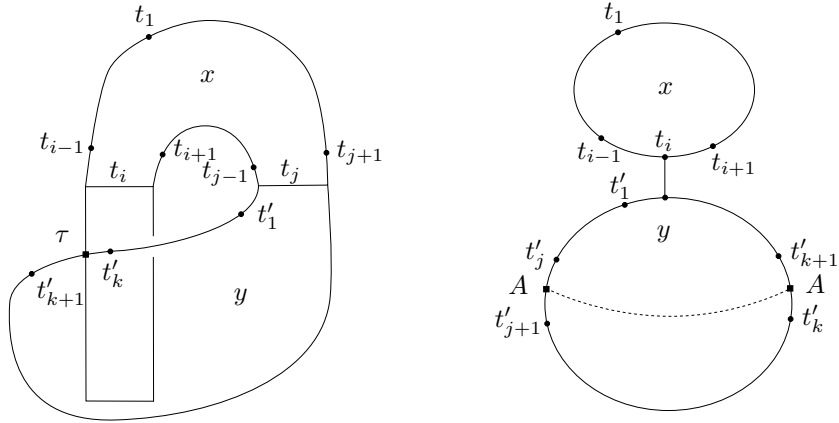

To check the signs, we first consider for example the gluing configuration shown in Figure \ref{Fig:string_op_derivation_signs}, left. Without loss of generality, assume $\tau_s$ and $\tau_{s'}'$ are the positive punctures on $x$ and $y$, i.e. the marked point is right after the positive puncture. The glued string appears in $\nabla\{x,y\}_1$ and $\{x,\delta y\}_2$. It is not difficult to see that the string in $\nabla\{x,y\}_1$ appears with the sign
\begin{align*}
e_1=\epsilon(\tau)(-1)^{P(x,t_j)(|y|+1)+(P(x,t_i)+1)(P(y,t_{k+1}')+P(x,t_{i-1},t_j))}
\end{align*} 
and the second marked point (on the string with no positive puncture) after the puncture $t_k'$. The corresponding string in  $\{x,\delta y\}_2$ appears with the sign
\begin{align*}
e_2&=\epsilon(\tau)(-1)^{P(y,t_{k+1}')+P(x,t_j)(P(y,t_{k+1}')+P(x,t_{i-1},t_j)+|y|+1)}=\\
&=(-1)^{P(y,t_{k+1}')P(x,t_{i-1},t_j)} e_1
\end{align*}
and the second marked point before the puncture $t_1'$. Punctures are labeled as in the figure. Moving the second marked point cancels the remaining sign $(-1)^{P(y,t_{k+1}')P(x,t_{i-1},t_j)}$.

Similarly for the gluing configuration shown in Figure \ref{Fig:string_op_derivation_signs}, right. The glued curve appears in $\nabla\{x,y\}_1$ and $\{x,\nabla y\}_1$. It is not difficult to see that the string in $\nabla\{x,y\}_1$ appears with the sign
\begin{align*}
e_1'=-\epsilon(A)(-1)^{P(x,t_i)(|y|+1)+(P(x,t_i)+P(y,t_{j+1}')+1)P(y,t_j',t_k')}.
\end{align*} 
The corresponding string in  $\{x,\nabla y\}_1$ appears with the sign
\begin{align*}
e_2'=-\epsilon(A)(-1)^{(P(y,t_{j+1}')+1)P(y,t_j',t_k')+P(x,t_i)(P(y,t_j',t_k')+|y|+1)}=e_1'.
\end{align*}
Other cases go similarly.

Finally, it is not difficult to see that the signs corresponding to summands $\nabla_1(u_{t_i,t_{s'}'},C_{i,j})$ and $\nabla_1(u_{t_j,t_{s'}'},C_{j,i})$ in $\nabla\{x,y\}_1$ for $t_i,t_j,i<j$ punctures on $x$ negatively asymptotic to the same Reeb chord as the positive puncture $t_{s'}'$ on $y$ cancel out. Assume for example the Reeb chord at $t_i,t_j$ is even. For strings $x,y$ as shown in Figure \ref{Fig:add_intersections_cancelation}, $\nabla_1(u_{t_i,t_{s'}'},C_{i,l})$ comes with the sign
\begin{align*}
e_1''=(-1)^{P(x,t_i)(|y|+1)+(P(x,t_i)+1)(P(y,t_{s'}')+P(x,t_i,t_j))}
\end{align*}  
and the second marked point after the positive puncture at $y$. The string $\nabla_1(u_{t_j,t_{s'}'},C_{l,i})$ comes with the sign
\begin{align*}
e_2''=-(-1)^{P(x,t_j)(|y|+1)+(P(x,t_i)+1)(P(y,t_{s'}')+P(x,t_i,t_j))}=-(-1)^{P(y,t_{s'}')P(x,t_i,t_j)}e_1''
\end{align*}
and the second marked point right before $t_{i+1}$, therefore, the summands cancel out.
\end{proof}

\subsection{The Hamiltonian}\label{Sec:hamiltonian_def}
Next, we define the Hamiltonian $H\in\cC(\Lambda)$ associated to a Legendrian knot $\Lambda$ and prove the SFT master equation (\ref{Eq:master_eq_prop}). Let $J$ be the almost complex structure given by (\ref{Eq:intro_J_formula}). 

Let $u:\lD^2\backslash\{t_1,\dots,t_k\}\to\lR\times\lR^3$ be a punctured $J$-holomorphic disk with boundary on $\lR\times\Lambda$ and one positive puncture at $t_k$. Then $\pi_{xyz}\circ u|_\partial$ gives us a string that we denote by $\omega(u)\in\widetilde\cC$, with a marked point right after the positive puncture and the corresponding sign $\epsilon(u)$ defined in Section \ref{Section:Orientations_disks}. Similarly, for $u$ a disk with two positive punctures at $t^1,t^2$, we have a string $\omega(u,t^1,t^2)\in\widetilde\cC$ obtained by looking at the boundary of $u$ with a marked point right after $t^2$ and the corresponding sign $\epsilon(u,t^1,t^2)$. Additionally, for $v$ a punctured $J$-holomorphic annulus with one positive puncture and $e_2$ a marked point on the inner boundary component, we have a pair of strings $\omega(v,e_2)=\omega_1(v)\otimes\omega_2(v,e_2)$ obtained by looking at the two boundary components, with a marked point right after the positive puncture for the first and at $e_2$ for the second component, and the corresponding sign $\epsilon(u,e_2)$ defined in Section \ref{Section:Orientations_annuli}.

The Hamiltonian $H=H(\Lambda)$ associated to the Legendrian knot $\Lambda$ is an element in $\cC$ given by
\begin{align*}
H(\Lambda)\coloneq &\sum_{u\in\cM_{1}^1(J),\operatorname{ind}(u)=0}\epsilon(u)\omega(u)+\sum_{u\in\cM^2_{1}(J),\operatorname{ind}(u)=0}\epsilon(u,t^1,t^2)\omega(u,t^1,t^2)+\\
&+\sum_{v\in\cM_2(J),\operatorname{ind}(v)=0}(-1)^{|\omega_2(v,e_2)|}\epsilon(v,e_2)\omega(v,e_2),
\end{align*}
where $\cM_{1}^\iota(J),\iota\in\{1,2\}$ is the moduli space of $J$-holomorphic disks with $\iota$ positive punctures and $\cM_2(J)$ is the moduli space of $J$-holomorphic annuli with one positive puncture. Note that
$$|H(\Lambda)|=-2.$$

The following lemma is a corollary of Lemma \ref{Lemma:index_zero_gen_asym_beh} and Lemma \ref{Lemma:index_zero_gen_rel_asym_beh}. 
\begin{lemma}\label{Lemma:u_are_admissible}
For a generic Legendrian knot $\Lambda$ and $u$ an index zero $J$-holomorphic disk on $\lR\times\Lambda$, $u$ is an admissible punctured disk.
\end{lemma}

Next, we prove the crucial property of the Hamiltonian $H(\Lambda)$ for the definition of the chain complex, often referred to as the \textit{master equation}. 
\begin{prop}\label{Prop:Master_eq1}
For $H\in\cC(\Lambda)$ the Hamiltonian associated to a generic Legendrian knot $\Lambda$, we have
\begin{equation}
\label{Eq:master_eq_prop}
\begin{aligned}
\frac{1}{2}\{H,H\}+d_{\str}H=0.
\end{aligned}
\end{equation}
\end{prop}
\begin{rmk}
We can write $\{x,y\}$ as a sum of two terms $x\leftarrow y$ and $(-1)^{|x||y|+|x|+|y|} y\leftarrow x$, where $s_1\leftarrow s_2$ consists of terms obtained by gluing positive punctures of $s_2$ to negative punctures of $s_1$. Then we can write $\frac{1}{2}\{H,H\}$ as  $H\leftarrow H$, see also \cite[Remark 3.14]{Ng_rLSFT}. 
\end{rmk}
\begin{proof}
The boundary points of the compactified moduli space $\overline\cM_1(\bm{\gamma},a)$ of index 1 $J$-holomorphic disks with up to two positive punctures and the compactified moduli space $\overline\cM_2(\bm{\gamma},\bm{\beta},a,b)$ of index 1 $J$-holomorphic annuli with one positive puncture were described in Proposition \ref{Prop:compactness_disk} and Proposition \ref{Prop:compactness_annulus}. There are three kinds of boundary points, pseudoholomorphic disk and annulus buildings, disks with a trivial strip bubble, and hyperbolic and elliptic nodal annuli. We notice that the boundaries of the glued buildings are in correspondence with the summands in $H\leftarrow H$, the boundaries of the disks with an inserted trivial strip are in correspondence with the summands in $\delta H$, while the boundaries of the nodal annuli are in correspondence with the summands in $\nabla H$ (see Lemma \ref{Lemma:sing_removal}) by definition. Here we use the fact that every index zero $J$-holomorphic disk $u$ is admissible (see Lemma \ref{Lemma:u_are_admissible}). For $u$ a $J$-holomorphic disk, the intersections of the shifted disk $\operatorname{sh}(u)$ with the Lagrangian cylinder $L$ correspond to interior intersections of the disk with the cylinder. Moreover, the self-intersections of the boundary of $u$ in $L$ and the interior intersections with $L$ are generically transverse. 

To complete the proof, we need to consider index zero $J$-holomorphic curves with a bad string on a boundary component. Strings in (\ref{Eq:master_eq_prop}) corresponding to a building obtained by gluing a disk to an annulus with a bad inner boundary are seen as zero, while the other boundary point of the corresponding connected component of the moduli space $\overline\cM_2$ is potentially not. In general, one can show that there is an even number of such bad buildings in $\partial\overline\cM_2$ and a pairing of summands in (\ref{Eq:master_eq_prop}) that gives us cancellations (see Remark \ref{Rem:Cancelations_bad_curves}). In summary, since a bad word is an even cover of an odd word, we have an even number of bad buildings obtained by attaching the disk at different iterations of the odd word. Moreover, precisely half of them come with a sign $+1$. For Legendrians in $\lR^3$, we show in Lemma \ref{Lemma:no_bad_words} that there can actually be no such bad index zero curve.

The observations above imply $\frac{1}{2}\{H, H\}+d_{\str}H=0$ up to signs. The sign cancellation follows from Section \ref{Sec:orientation_and_algebraic_signs}.  
\end{proof}
In the proof of the master equation, we use Propositions \ref{Prop:moduliI}-\ref{Prop:compactness_annulus}. This holds since all the curves are regular.

\begin{lemma}\label{Lemma:regularity_main}
For a generic Legendrian knot $\Lambda$, all $J$-holomorphic annuli with one positive puncture and all $J$-holomorphic disks with arbitrarily many positive punctures, with boundary on $\lR\times\Lambda$ and of index 0 or 1 are regular.
\end{lemma}
\begin{proof}
Let $u$ be a $J$-holomorphic disk with $k>1$ positive punctures. For $\Lambda$ generic, $u$ is either somewhere injective or a multiple cover of a somewhere injective curve. In the second case, there exists a branched cover $\psi:\lD\to\lD$ of degree $d\in\lN,d\geq 2$ and a somewhere injective pseudoholomorphic disk $v$ such that $u=v\circ\psi$. Denote by $l=\frac{k}{d}$ the number of positive punctures of $v$ and by $\widetilde\gamma_i,i=1,\dots,r$ the Reeb chords at the punctures of $v$. Then 
\begin{align*}
&\operatorname{ind} u=k+\mu_L([u])+d\sum_i\epsilon_i\mu_{CZ}(\widetilde\gamma_i)-2=\\
&=dl+d\mu_L([v])+d\sum_i\epsilon_i\mu_{CZ}(\widetilde\gamma_i)-2=\\
&=d\operatorname{ind} v+2d-2\geq 2,
\end{align*}
from which we conclude that disks of index $<2$ are not multiple covers.

For disks with one positive puncture, we can use \cite{EES07} to perturb $\Lambda$ near the positive puncture to get regularity. Similar as in \cite{EES07}, we can achieve transversality for index zero and one disks with arbitrarily many positive punctures. To get regularity for the moduli space of index zero and one $J$-holomorphic annuli, we perturb the Legendrian knot as in \cite{EES07} to get regularity of the holomorphic annuli in the Lagrangian projection. Then, we use Lemma \ref{Lemma:obstruction_transversalityI} and Lemma \ref{Lemma:obstruction_transversalityII} to get regularity for their lifts to $\lR^4$.
\end{proof}

In the proof of the master equation, we additionally need the following fact. Using the removal of boundary and interior singularities for pseudoholomorphic maps, we can see any nodal annulus in the boundary of the 1-dimensional moduli space as an index zero pseudoholomorphic disk together with a boundary self-intersection or an interior intersection with $\lR\times\Lambda$.
\begin{lemma}[\cite{oh_rem_sing}]\label{Lemma:sing_removal}
Let $u:\mathring\Sigma\backslash\{\tau\}\to\lR^4$ be a pseudoholomorphic map with boundary on a Lagrangian $L$, where $\mathring\Sigma$ is a punctured Riemann surface and $\tau\in\mathring\Sigma$ a boundary or an interior point. Assume $u$ can be continuously extended at $\tau$. Then the extension $u:\mathring\Sigma\to \lR^4$ is smooth and pseudoholomorphic at $\tau$.
\end{lemma}

\subsection{The boundary operator}\label{Sec:The_differential}
Let $\Lambda$ be a Legendrian knot and $\cC$ the vector space generated by strings and string pairs as before. We define a degree $-1$ linear map $d_\Lambda=d:\cC\to\cC$ by
$$d \alpha=\{\alpha,H\}+d_{\str}\alpha.$$
Now, it is easy to show $d\circ d=0$ using Proposition \ref{Prop:Master_eq1} and the properties of $d_{\str}$ and $\{\cdot,\cdot\}$.

\begin{prop}\label{Prop:d_square_vanishes}
The map $d:\cC\to\cC$ satisfies $d\circ d=0$.
\end{prop}
\begin{proof}
Follows from
\begin{align*}
d\circ d(\alpha)=&\{\{\alpha,H\},H\}+\{d_{\str}\alpha,H\}+\\
&+d_{\str}\{\alpha,H\}+d_{\str}^2\alpha=\\
=&-\frac{1}{2}\{\{H,H\},\alpha\}+\{\alpha,d_{\str}H\}=\\
=&-\{1/2\{H,H\}+d_{\str}H,\alpha\}=\\
=&0.
\end{align*}
\end{proof}

\begin{rmk}
We defined $\cC(\Lambda)$ as the space generated by strings and string pairs with at least one positive puncture in order to simplify the definition of the SFT bracket. Alternatively, we can work with the space $\cA(\Lambda)$ generated by strings with zero or one positive puncture, and define the SFT bracket as an action of $\cC(\Lambda)$ on $\cA(\Lambda)$. We take this approach in Section \ref{Section:IIinvar_def}. The definition in Section \ref{Section:Algebraic_definition} additionally contains more algebraic structure and is easier to work with when it comes to computations.
\end{rmk}

There is a standard procedure for Legendrian knots called front resolution that gives us Lagrangian projection of a Legendrian knot from the front projection ($xz$-projection) of $\Lambda$ by smoothing out right cusps and replacing left cusps by loops.

\begin{thm}\label{Theorem:Invariance0}
Let $\Lambda_0,\Lambda_1$ be Legendrian isotopic knots and $\left(\cC(\widetilde\Lambda_0),d_{\widetilde\Lambda_0}\right),\left(\cC(\widetilde\Lambda_1),d_{\widetilde\Lambda_1}\right)$ the chain complexes associated to front resolutions $\widetilde\Lambda_0$ and $\widetilde\Lambda_1$ of $\Lambda_0$ and $\Lambda_1$. Then
$$H_*\left(\cC(\widetilde\Lambda_0),d_{\widetilde\Lambda_0}\right)\cong H_*\left(\cC(\widetilde\Lambda_1),d_{\widetilde\Lambda_1}\right).$$
\end{thm}

After introducing the structure of a second-order differential graded algebra (quantum BV-algebra) in Section \ref{Section:Algebraic_definition}, we prove a stronger version of Theorem \ref{Theorem:Invariance0} in Section \ref{Sec:Invariance}, which can be seen as an analogue of stable tame equivalence introduced in \cite{Chekanov02}. Invariance up to Reidemeister II move is shown for a class of moves that we call admissible. This is the reason why we define the invariant of a Legendrian knot by looking at its front resolution. Similar situation appears in \cite{Ng_rLSFT} due to a different reason. We show that for front resolutions of two Legendrian isotopic knots there exists an isotopy that does not contain non-admissible Reidemeister II moves.

\section{Second-order dga for Legendrian knots}\label{Section:Algebraic_definition}
In this section we reformulate the definition of the Legendrian knot invariant defined in Section \ref{Sec:Invariant_definition}, giving it more algebraic structure and, hence, making it more suitable for computations. More precisely, we give a definition of a second-order differential graded algebra associated to a Legendrian knot $\Lambda$, which, seen as a chain complex, is analogous to the first definition. For simplicity, we describe only the part of the chain complex restricted to strings with one positive puncture. The full chain complex is constructed by combining our definition with the differential in \cite{Ng_rLSFT}, adding cyclic words with one letter $p$. 

\subsection{Second-order differential graded algebras}\label{Section:IIord_dga_definition}
In this section, we define the notion of a second-order differential graded algebra (second-order dga). This structure comes with an operator $d$ such that $d^2=0$ and an antibracket $\{\cdot,\cdot\}$ that measures the failure of $d$ to be a derivation. Additionally, we define the notion of a morphism. This section is purely algebraic.

Let $\widetilde\cA$ be the tensor algebra generated by $q_1,\dots,q_n$ over $\lQ$. Define a grading on $\widetilde\cA$ by taking $|q_i|=a_i$ for some $a_i\in\lZ$. Let $\widetilde\cA^{\cyc}$ be the corresponding vector space of cyclic words, i.e. the quotient space $\widetilde\cA/\cI$ for $\cI$ the vector subspace generated by $\{vw-(-1)^{|v||w|}wv\,|\,v,w\in\widetilde\cA\text{ words}\}$. We consider the graded vector space 
$$\cA=\widetilde\cA\oplus\hbar\,(\widetilde\cA\otimes \widetilde\cA^{\cyc}),$$ 
where $\hbar$ is a formal variable such that $|\hbar|=-1$. Elements in $\cA$ are denoted by $u+\hbar\, w$ for $u\in\widetilde\cA,w\in\widetilde\cA\otimes\widetilde\cA^{\cyc}$. The algebra structure on $\cA$ is given by 
\begin{align*} 
&w\cdot\hbar(v_1\otimes v_2)=(-1)^{|w|(|v_2|+1)}\hbar(wv_1\otimes v_2),\\
&\hbar(v_1\otimes v_2)\cdot w=\hbar(v_1w\otimes v_2),\\
&\hbar(v_1\otimes v_2)\cdot \hbar(w_1\otimes w_2)=0,
\end{align*} 
and by concatenation of words on $\widetilde\cA$. For $\omega\in\hbar\,(\widetilde\cA\otimes\widetilde\cA^{\cyc})$ and $s\in\cA$, we sometimes write $\omega\otimes s$ or $s\otimes\omega$, which should be seen as zero.

Recall the definition of a differential graded algebra (dga) structure on $\cA$ as a choice of a degree $-1$ linear map $d:\cA\to \cA$ such that
\begin{align*}
&d(w_1w_2)=d(w_1)w_2+(-1)^{|w_1|}w_1d(w_2),\\
&d(\hbar(v\otimes w))=(-1)^{|w|+1}\hbar(d_0v\otimes w)-\hbar(v\otimes d_0^{\cyc}w),\\
&d^2(w)=0,
\end{align*}
where $d_0\coloneq\pi_{\widetilde\cA}\circ d\circ \iota_{\widetilde\cA}:\widetilde\cA\to\widetilde\cA$ and $d_0^{\cyc}:\widetilde\cA^{\cyc}\to\widetilde\cA^{\cyc}$ is the linear map induced by $d_0$ on the cyclic quotient. A map $d$ that satisfies the first two conditions is called a derivation on $\cA$. For any $Q_i\in \cA$ such that $|Q_i|=|q_i|-1$, there exists a unique derivation $d:\cA\to \cA$ such that $d(q_i)=Q_i$. This map is a differential if and only if $d^2(s)=0$ for all $s=q_i,i\in\{1,\dots,n\}$.

A morphism of dg algebras $(\cA,d),(\cA',d')$ is a degree 0 linear map $f:\cA\to \cA'$ such that
\begin{align*}
&f(vw)=f(v)f(w),\\
&f(\hbar(v\otimes w))=\hbar(f_0v\otimes f_0^{\cyc}w),\\
&d'\circ f(w)=f\circ d(w),
\end{align*}
for all $v,w\in\cA$, where $f_0\coloneq\pi_{\widetilde\cA'}\circ f\circ \iota_{\widetilde\cA}:\widetilde\cA\to\widetilde\cA'$ and $f_0^{\cyc}:\widetilde\cA^{\cyc}\to\widetilde\cA'^{\cyc}$ is the map induced by $f_0$ on the cyclic quotient. For any $F_i\in\cA'$ such that $|F_i|=|q_i|$, there exists a unique linear map $f:\cA\to\cA'$ such that $f(q_i)=F_i$ that satisfies the first two conditions. This map is a dga morphism if and only if $d'\circ f(s)=f\circ d(s)$ for all $s=q_i,i=1,\dots,n$.

Next, we introduce the notion of a second-order dg algebra and a second-order dga morphism. Consider the algebra structure on $\widetilde\cA\otimes\widetilde\cA$ given by 
\begin{align*}
(v_1\otimes v_2)\cdot(w_1\otimes w_2)=(-1)^{|v_1||w_2|}(v_1 w_1\otimes v_2 w_2).
\end{align*}

\begin{defi}
A degree $0$ bilinear map $\{\cdot,\cdot\}:\widetilde\cA\times \widetilde\cA\to\widetilde\cA\otimes \widetilde\cA$ is called an \textit{antibracket} if
\begin{align*}
&\{v,w_1w_2\}=\{v,w_1\}\cdot(w_2\otimes 1)+(-1)^{|v||w_1|}(1\otimes w_1)\cdot\{v,w_2\},\\
&\{v_1v_2,w\}=(v_1\otimes 1)\cdot\{v_2,w\}+(-1)^{|v_2||w|}\{v_1,w\}\cdot(1\otimes v_2),
\end{align*}
for all words $v,v_1,v_2,w,w_1,w_2\in\widetilde\cA$. 
\end{defi}
Antibracket induces a degree $-1$ linear map $\{\cdot,\cdot\}_\hbar:\cA\otimes\cA\to\cA$ given by 
$$\{v,w\}_\hbar=\hbar\,\pi_{\cyc}\{\pi_{\widetilde\cA}v,\pi_{\widetilde\cA}w\},$$ 
where $\pi_{\cyc}:\widetilde\cA\otimes\widetilde\cA\to\widetilde\cA\otimes\widetilde\cA^{\cyc}$ is induced by the cyclic quotient.

For $f,g:\widetilde \cA\to\widetilde\cA'$ graded linear maps, we define a linear map $f\otimes g:\widetilde\cA\otimes\widetilde\cA\to\widetilde\cA'\otimes\widetilde\cA'$
\begin{align*}
(f\otimes g)(v_1\otimes v_2)=(-1)^{|f||v_2|}f(v_1)\otimes g(v_2).
\end{align*}

\begin{defi}\label{Definition:IIord_derivation}
A degree $-1$ linear map $d:\cA\to\cA$ is a \textit{second-order derivation} with respect to an antibracket $\{\cdot,\cdot\}$ on $\widetilde\cA$ if
\begin{align*}
&d(vw)=d(v)w+(-1)^{|v|}v d(w)+\{v,w\}_\hbar,\\
& d(\hbar(v\otimes w))=(-1)^{|w|+1}\hbar(d_0v\otimes w)-\hbar(v\otimes d_0^{\cyc}w),
\end{align*}
for all generators $v,w\in\cA$, where $d_0\coloneq\pi_{\widetilde\cA}\circ d\circ \iota_{\widetilde\cA}$ and $d_0^{\cyc}:\widetilde\cA^{\cyc}\to\widetilde\cA^{\cyc}$ is the linear map induced by $d_0$ on the cyclic quotient. Furthermore, we say $d:\cA\to\cA$ is a \textit{strong} second-order derivation with respect to $\{\cdot,\cdot\}$ if $d$ is additionally a derivation with respect to $\{\cdot,\cdot\}$, i.e. if
\begin{align}\label{Equation:third_condition}
(d_0\otimes 1+1\otimes d_0)\{v,w\}=\{d_0v,w\}+(-1)^{|v|}\{v,d_0w\}\in\widetilde\cA\otimes\widetilde\cA.
\end{align}
\end{defi}
A strong second-order derivation $d:\cA\to\cA$ such that $d^2(s)=0$ for all $s=q_i,i\in\{1,\dots,n\}$ satisfies $d^2(s)=0$ for all $s\in\cA$. It is not difficult to show that (\ref{Equation:third_condition}) holds if and only if 
$$(d_0\otimes 1+1\otimes d_0)\{q_i,q_j\}=\{d_0q_i,q_j\}+(-1)^{|q_i|}\{q_i,d_0q_j\}$$ 
for all $i,j\in\{1,\dots,n\}$.

\begin{defi}
A \textit{second-order differential graded algebra} structure $(\cA,d,\{\cdot,\cdot\})$ on $\cA$ consists of an antibracket $\{\cdot,\cdot\}$ on $\widetilde\cA$ and a strong second-order derivation $d:\cA\to\cA$ with respect to $\{\cdot,\cdot\}$ such that $d^2=0$. 
\end{defi}
\begin{rmk}
The structure of a second-order dg algebra is similar to quantum Batalin--Vilkovisky algebra (also known as Beilinson--Drinfeld algebra). One difference is that second-order dg algebra is not commutative or graded-commutative and is more suitable for recording the cyclic ordering of the boundary punctures of pseudoholomorphic curves.
\end{rmk}

Before we define second-order graded algebra morphisms, we need to introduce the notion of an $f$-antibracket for $f:\widetilde\cA\to\widetilde\cA'$ a degree zero algebra map. 
\begin{defi}
A degree 1 bilinear map $\{\cdot,\cdot\}_f:\widetilde\cA\times \widetilde\cA\to\widetilde\cA'\otimes \widetilde\cA'$ is called an $f$-\textit{antibracket} if
\begin{align*}
&\{v,w_1w_2\}_f=\{v,w_1\}_f\cdot(fw_2\otimes 1)+(-1)^{|w_1|(|v|+1)}(1\otimes fw_1)\cdot\{v,w_2\}_f,\\
&\{v_1v_2,w\}_f=(-1)^{|v_1|}(fv_1\otimes 1)\cdot\{v_2,w\}_f+(-1)^{|v_2||w|}\{v_1,w\}_f\cdot(1\otimes fv_2),
\end{align*}
for all words $v,v_1,v_2,w,w_1,w_2\in\widetilde\cA$. 
\end{defi}
The map $\{\cdot,\cdot\}_f$ induces a degree 0 linear map $\{\cdot,\cdot\}_{f,\hbar}:\cA\otimes\cA\to\cA'$ given by 
$$\{v,w\}_{f,\hbar}=\hbar\,\pi_{\cyc}\{\pi_{\widetilde\cA}v,\pi_{\widetilde\cA}w\}_f.$$

\begin{defi}A \textit{second-order graded algebra morphism} is a degree 0 linear map $f:\cA\to\cA'$ such that $f_0\coloneq\pi_{\widetilde\cA'}\circ f\circ i_{\widetilde\cA}$ is an algebra map, together with an $f_0$-antibracket $\{\cdot,\cdot\}_f$, such that
\begin{align*}
&f(1)=1,\\
&f(vw)=fv fw+\{v,w\}_{f,\hbar},\\
&f(\hbar(v\otimes w))=\hbar(f_0v\otimes f_0^{\cyc}w),
\end{align*}
for all $v,w\in\cA$, were $f_0^{\cyc}$ is the map induced by $f_0$ on the cyclic quotient as before.
\end{defi}
\begin{defi}
Let $(\cA,d,\{\cdot,\cdot\}_d)$ and $(\cA',d',\{\cdot,\cdot\}_{d'})$ be second-order dg algebras. We say a second-order graded algebra morphism $f:\cA\to\cA'$ with respect to an $f_0$-antibracket $\{\cdot,\cdot\}_f$ preserves the second-order dga structure on $\cA,\cA'$ if
\begin{align*}
&(f_0\otimes f_0)\{v,w\}_{d}+\{d_0v,w\}_f+(-1)^{|v|}\{v,d_0w\}_f=\\
&=\{f_0v,f_0w\}_{d'}-(d'_0\otimes 1+1\otimes d'_0)\{v,w\}_f,
\end{align*}
and
\begin{align*}
&d'\circ f(w)=f\circ d(w),
\end{align*}
for all $v,w\in\cA$. A map $f$ that satisfies these properties will be referred to as a \textit{second-order dga morphism}. It is not difficult to see that the second condition follows from the first one and $d'\circ f(s)=f\circ d(s)$ for all $s=q_i,i\in\{1,\dots,n\}$.
\end{defi}

The following lemmas are straightforward.

\begin{lemma}\label{Lemma:Second_order_derivation}
For any $Q_i\in\cA,i\in\{1,\dots,n\}$ and $R_{ij}=\sum R_{ij}^2\otimes R_{ij}^1\in\widetilde\cA\otimes\widetilde\cA,i,j\in\{1,\dots,n\}$, where each $R_{ij}^\bullet$ is a scalar multiple of a word, such that $|Q_i|=|q_i|-1,|R_{ij}|=|q_i|+|q_j|$, there exists a unique antibracket $\{\cdot,\cdot\}:\widetilde\cA\times \widetilde\cA\to \widetilde\cA\otimes\widetilde\cA$ such that $\{q_i,q_j\}=\sum(-1)^{(1+|q_i|)|R_{ij}^1|}R_{ij}^2\otimes R_{ij}^1$, and a unique second-order derivation $d:\cA\to\cA$ with respect to $\{\cdot,\cdot\}$ such that $d(q_i)=Q_i$. The map $d$ is a strong second-order derivation if additionally 
$$(d_0\otimes 1+1\otimes d_0)\{q_i,q_j\}=\{\pi_{\widetilde\cA}Q_i,q_j\}+(-1)^{|q_i|}\{q_i,\pi_{\widetilde\cA}Q_j\}$$ 
for all $i,j$.
\end{lemma} 
\begin{proof}
The antibracket is given by
\begin{align*}
&\{q_{s_1}\dots q_{s_l},q_{s_{l+1}}\dots q_{s_k}\}=\\
&=\sum_{i=1}^l\sum_{j={l+1}}^k  \sum(-1)^{\bullet}q_{s_1}\dots q_{s_{i-1}}R_{s_is_j}^2 q_{s_{j+1}}\dots q_{s_k} \otimes  q_{s_{l+1}}\dots q_{s_{j-1}} R_{s_is_j}^1 q_{s_{i+1}}\dots q_{s_l},
\end{align*}
where  
$$\bullet=\left(1+\sum_{a=1}^{j-1}|q_{s_a}|\right)\left(|R_{s_is_j}^1|+\sum_{b=i+1}^{j-1}|q_{s_b}|\right)+\sum_{a=l+1}^{j-1}|q_{s_a}|\left(|R_{s_is_j}^1|+\sum_{b=i+1}^l|q_{s_b}|\right).$$ 
The map $d:\cA\to\cA$ is given by
\begin{equation}
\label{Eq:Second_order_differential_formula}
\begin{aligned}
d(q_{s_1}\dots q_{s_k})=&\sum_{i=1}^k(-1)^{\sum_{a=1}^{i-1}|q_{s_a}|}q_{s_1}\dots q_{s_{i-1}} Q_{s_i} q_{s_{i+1}}\dots q_{s_k}+\\
&+\hbar\sum_{1\leq i<j\leq k}\sum(-1)^{\left(1+\sum_{a=1}^{j-1}|q_{s_a}|\right)\left(|R_{s_is_j}^1|+\sum_{b=i+1}^{j-1}|q_{s_b}|\right)}\cdot\\
&\cdot q_{s_1}\dots q_{s_{i-1}}R_{s_is_j}^2q_{s_{j+1}}\dots q_{s_k}\otimes R_{s_is_j}^1 q_{s_{i+1}}\dots q_{s_{j-1}},
\end{aligned}
\end{equation}
and
\begin{align*}
d(\hbar(q_{s_1}\dots q_{s_l}\otimes q_{s_{l+1}}\dots q_{s_k}))=&(-1)^{1+\sum_{a=l+1}^k|q_{s_a}|}\hbar(d(q_{s_1}\dots q_{s_l})\otimes q_{s_{l+1}}\dots q_{s_k})-\\
&-\hbar(q_{s_1}\dots q_{s_l}\otimes d(q_{s_{l+1}}\dots q_{s_k})).
\end{align*}
\end{proof}

\begin{lemma}\label{Lemma:Second_order_morphism}
For any $S_i\in\cA',i\in\{1,\dots,n\}$ and $U_{ij}=\sum U_{ij}^2\otimes U_{ij}^1\in\widetilde\cA'\otimes\widetilde\cA',i,j\in\{1,\dots,n\}$, where each $U_{ij}^\bullet$ is a scalar multiple of a word, such that $|S_i|=|q_i|,|U_{ij}|=|q_i|+|q_j|+1$, there exists a unique second-order graded algebra morphism $(f,\{\cdot,\cdot\}_f):\cA\to\cA'$ such that $\{q_i,q_j\}_f=\sum (-1)^{(1+|q_i|)(1+|U_{ij}^1|)}U_{ij}^2\otimes U_{ij}^1$ and $f(q_i)=S_i$. Moreover, if $(\cA,d,\{\cdot,\cdot\}_d)$, $(\cA',d',\{\cdot,\cdot\}_{d'})$ are second-order dg algebras, then $f$ is a second-order dga morphism if additionally $d'\circ f(q_i)=f\circ d(q_i)$ for all $i\in\{1,\dots,n\}$ and 
$$(f_0\otimes f_0)\{q_i,q_j\}_{d}+\{d_0q_i,q_j\}_f+(-1)^{|q_i|}\{q_i,d_0q_j\}_f=\{f_0q_i,f_0q_j\}_{d'}-(d'_0\otimes 1+1\otimes d'_0)\{q_i,q_j\}_f$$ 
for all $i,j\in\{1,\dots,n\}$.
\end{lemma} 
\begin{proof}
The $f$-antibracket is given by
\begin{align*}
&\{q_{s_1}\dots q_{s_l},q_{s_{l+1}}\dots q_{s_k}\}_f=\\
&=\sum_{i=1}^l\sum_{j={l+1}}^k  \sum(-1)^{\bullet}(S_{s_1}\dots S_{s_{i-1}}U_{s_is_j}^2 S_{s_{j+1}}\dots S_{s_k} \otimes  S_{s_{l+1}}\dots S_{s_{j-1}} U_{s_is_j}^1 S_{s_{i+1}}\dots S_{s_l}),
\end{align*}
where  
$$\bullet=\left(1+\sum_{a=1}^{j-1}|q_{s_a}|\right)\left(|U_{s_is_j}^1|+\sum_{b=i+1}^{j-1}|q_{s_b}|+1\right)+\sum_{a=l+1}^{j-1}|q_{s_a}|\left(|U_{s_is_j}^1|+\sum_{b=i+1}^l|q_{s_b}|\right).$$
The morphism $f:\cA\to\cA'$ is given by
\begin{equation}
\label{Eq:Second_order_morphism_formula}
\begin{aligned}
f(q_{s_1}\dots q_{s_k})=&S_{s_1}\dots S_{s_k}+\\
&+\hbar\sum_{1\leq i<j\leq k}\sum(-1)^{\left(1+\sum_{a=1}^{j-1}|q_{s_a}|\right)\left(|U_{s_is_j}^1|+\sum_{b=i+1}^{j-1}|q_{s_b}|+1\right)}\\
&S_{s_1}\dots S_{s_{i-1}}U_{s_is_j}^2 S_{s_{j+1}}\dots S_{s_k}\otimes U_{s_is_j}^1 S_{s_{i+1}}\dots S_{s_{j-1}},
\end{aligned}
\end{equation}
and
\begin{align*}
f(\hbar(q_{s_1}\dots q_{s_l}\otimes q_{s_{l+1}}\dots q_{s_k}))=&\hbar(S_{s_1}\dots S_{s_l}\otimes S_{s_{l+1}}\dots S_{s_k}).
\end{align*}
\end{proof}

We briefly discuss some elementary properties of second-order derivations and second-order graded algebra morphisms.

\begin{lemma}\label{Lemma:Second_order_properties}\begin{enumerate}
\item Composition of second-order graded algebra morphisms $(f,\{\cdot,\cdot\}_f):\cA\to\cA'$ and $(g,\{\cdot,\cdot\}_g):\cA'\to\cA''$ is a second-order graded algebra morphism with respect to the $(g_0\circ f_0)$-antibracket 
$$\{v,w\}_{g\circ f}=\{f_0v,f_0w\}_g+(g_0\otimes g_0)\{v,w\}_f.$$ 
Moreover, if $(\cA,d,\{\cdot,\cdot\}_{d})$,$(\cA',d',\{\cdot,\cdot\}_{d'})$ and $(\cA'',d'',\{\cdot,\cdot\}_{d''})$ are second-order dg algebras and $f,g$ are second-order dga morphisms, then so is $g\circ f$.
\item If a second-order graded algebra morphism $f:\cA\to\cA'$ is a bijection, then $f^{-1}$ is a second-order graded algebra morphism with respect to the $(f^{-1}_0)$-antibracket 
$$\{u,v\}_{f^{-1}}=-(f_0^{-1}\otimes f_0^{-1})\{f_0^{-1}u,f_0^{-1}v\}_f.$$ 
Moreover, if $(\cA,d,\{\cdot,\cdot\}_d),(\cA',d',\{\cdot,\cdot\}_{d'})$ are second-order dg algebras and $f$ is a second-order dga morphism, then so is $f^{-1}$.
\item For $(f,\{\cdot,\cdot\}_f):\cA'\to\cA$ an invertible second-order graded algebra morphism and $(d,\{\cdot,\cdot\}_d):\cA\to\cA$ a (strong) second-order derivation, the map $\widehat d=f^{-1}\circ d\circ f$ is a (strong) second-order derivation with respect to the antibracket 
\begin{align*}
\{u,v\}_{\widehat d}=&-(f_0^{-1}\otimes f_0^{-1})\circ (d_0\otimes 1+1\otimes d_0)\{u,v\}_f+(f_0^{-1}\otimes f_0^{-1})\{f_0u,f_0v\}_d+\\
&+\{d_0f_0u,f_0v\}_{f^{-1}}+(-1)^{|u|}\{f_0u,d_0f_0v\}_{f^{-1}}.
\end{align*}
\end{enumerate}
\end{lemma}

We define the notion of action on $\cA$. Let $l(q_i)>0,i=1,\dots,n$ be $\lZ$-linearly independent positive real numbers. For any word $w=q_{i_1}\dots q_{i_k}\in\widetilde\cA$, we define the action of $w$ as
$$l(w)=\sum_{j=1}^k l(q_{i_j}),$$
and similarly for $w=\hbar(w_1\otimes w_2)\in\hbar\,(\widetilde\cA\otimes\widetilde\cA^{\cyc})$
$$l(\hbar(w_1\otimes w_2))=l(w_1)+l(w_2).$$
Additionally, we define $l(\sum_{i=1}^k a_iw_i)=\max_{i=1,\dots, k}l(w_i)$ for $a_i\in\lQ,a_i\neq 0,w_i\in\cA,i\in\{1,\dots,k\}$. 

\begin{defi}
We say a linear map $f:\cA\to\cA'$ is \textit{filtered} if 
$$l(f(s))\leq l(s)$$ 
for all $s\in\cA$ and $f(\hbar\,(\widetilde\cA\otimes\widetilde\cA^{\cyc}))\subset\hbar\,(\widetilde\cA'\otimes\widetilde\cA'^{\cyc})$.
\end{defi}

\begin{lemma}\label{Lemma:inverse_existence_I_order}
Let $\phi:\cA\to\cA$ be a filtered algebra morphism such that
\begin{align*}
&\phi(q_i)=q_i+\omega_i
\end{align*} 
for some $\omega_i\in\cA,i\in\{1,\dots, n\}$ with $l(\omega_i)< l(q_i)$ for all $i$. Then $\phi$ is invertible and the inverse is a filtered graded algebra morphism.
\end{lemma}
\begin{proof}
Without loss of generality, $q_i$ are ordered by action
\begin{align*}
l(q_1)<\dots<l(q_n).
\end{align*}
Then $\phi$ can be seen as a composition $\phi=\phi_n\circ\dots\phi_1$ of graded algebra morphisms given by
\begin{align*}
\phi_i(q_j)=\begin{cases}
q_j,&j\neq i\\
q_i+\omega_i,&j=i
\end{cases}
\end{align*}
Each $\phi_i$ is invertible with the inverse given by
\begin{align*}
\phi_i^{-1}(q_j)=\begin{cases}
q_j,&j\neq i\\
q_i-\omega_i,&j=i
\end{cases}
\end{align*}
\end{proof}

\begin{lemma}\label{Lemma:inverse_existence}
Let $\phi:\cA\to\cA$ be a second-order graded algebra morphism such that
\begin{align*}
&\phi(q_i)=q_i,\\
&\{q_i,q_j\}_{\phi}=\omega_{ij},
\end{align*} 
for some $\omega_{ij}\in\widetilde\cA\otimes\widetilde\cA, i,j\in\{1,\dots,n\}$. Then $\phi$ is invertible and the inverse is a second-order graded algebra morphism. If additionally  $l(\omega_{ij})\leq l(q_i q_j)$ for all $i,j\in\{1,\dots,n\}$, then the inverse is filtered.
\end{lemma}
\begin{proof}
It is easy to check that the inverse is given by
\begin{align*}
&\phi^{-1}(q_i)=q_i,\\
&\{q_i,q_j\}_{\phi^{-1}}=-\omega_{ij}.
\end{align*}
\end{proof}

\subsection{Definition of Legendrian knot second-order dg algebra}\label{Section:IIinvar_def}
In this section we introduce a second-order differential graded algebra associated to a Legendrian knot following the ideas from Section \ref{Sec:Invariant_definition}. The definition can be generalized to include Legendrian links in $\lR^3$.

Let $\Lambda\subset\lR^3$ be a Legendrian knot with Reeb chords $\cR=\{\gamma_1,\dots,\gamma_n\}$. Fix an orientation on $\Lambda$ and a base point $T\in\Lambda$ different from the Reeb chord endpoints. We denote by $\widetilde\cA=\widetilde\cA(\Lambda)$ the tensor algebra generated by $q_1,\dots,q_n,t^+,t^-$ with relations $t^-t^+=t^+t^-=1$, grading given by 
$$|q_i|=\mu_{CZ}(\gamma_i),|t^\pm|=\mp 2\operatorname{rot}(\Lambda),$$ 
and by $\cA$ be the graded algebra $\cA=\cA(\Lambda)=\widetilde\cA\oplus\hbar\,(\widetilde\cA\otimes\widetilde\cA^{\cyc})$ with $|\hbar|=-1$ as before. The action of $q_i$ is set to be the length of the Reeb chord $\gamma_i$. Fix $J$ to be the almost complex structure given by (\ref{Eq:intro_J_formula}). We define a second-order differential $(d,\{\cdot,\cdot\}_d):\cA(\Lambda)\to\cA(\Lambda)$ associated to the Legendrian knot $\Lambda$. We can think of $\cA(\Lambda)$ as the space of strings and string pairs on $\Lambda$ with negative punctures. The second-order differential $d$ is defined by counting index zero $J$-holomorphic disks with up to two and annuli with one positive puncture with boundary on $\lR\times\Lambda$. In addition to that, we add the corrected loop coproduct discussed in Section \ref{Sec:def_string_operator} in order to cancel out the nodal breaking. 
\begin{rmk}
The map $d$ corresponds to the boundary map defined in Section \ref{Sec:The_differential} modulo strings with two positive punctures. 
The part of the differential concerning strings with "unglued" positive punctures was introduced in \cite{Ng_rLSFT} and we omit it here for simplicity. The precise formula for the extension can be seen from (\ref{Eg:delta_formula}) and (\ref{Eq:d_unglued_pos_formula}) below. 
\end{rmk}

In order to define the second-order derivation $d:\cA\to\cA$, we need to fix $d(s)\in\cA$ and $\{s_1,s_2\}_{d}\in\widetilde\cA\otimes\widetilde\cA$ for $s,s_1,s_2\in\{q_i,t^\pm\,|\,i=1,\dots,n\}$.

First, we define $d(q_i),i\in\{1,\dots,n\}$. It consists of three parts: contribution of $J$-holomorphic disks $d_\lD(q_i)$, contribution of $J$-holomorphic annuli $d_A(q_i)$, and the first order part $d_f(q_i)$ of the corrected loop coproduct. We define $d_\lD(q_i),d_A(q_i),d_f(q_i)\in\cA$ below, and take $$d(q_i)=d_\lD(q_i)+d_A(q_i)+d_f(q_i).$$

To define $d_\lD(q_i)$, we consider the moduli space $\cM_{1}(J,\gamma_i^+)$ of index zero $J$-holomorphic disks in $\lR^4$ with boundary on $\lR\times\Lambda$, one positive puncture asymptotic to $\gamma_i$ and arbitrarily many negative punctures. As we have seen in Section \ref{Sec:moduli_spaces_all}, $\cM_1(J,\gamma_i^+)$ is in bijection with immersed polygons in $\lC$ with boundary on $\pi_{xy}(\Lambda)$, convex corners at the self-intersections of $\pi_{xy}(\Lambda)$ and one positive corner at $\gamma_i$. For every $u\in\cM_{1}(J,\gamma_i^+)$, we denote by $\widetilde w(u)$ the word in $\widetilde\cA$ obtained as follows. Let $\gamma_{i_1},\dots,\gamma_{i_k}$ be the Reeb chords at the negative punctures of $u$ in the order starting from the positive puncture, and $a_i\in\lZ,i\in\{0,\dots,k\}$ be the algebraic intersection number of the boundary of $u$ between the $i^{th}$ and the $(i+1)^{st}$ puncture and $\lR\times \{T\}$ (where the positive puncture is the $0^{th}$ and the $(k+1)^{st}$ puncture). Then we define 
$$\widetilde w(u)=t^{a_0}q_{i_1}t^{a_1}\dots t^{a_{k-1}}q_{i_k}t^{a_k}.$$ 
Additionally, we have an orientation sign $\epsilon(u)\in\{1,-1\}$ defined in Section \ref{Section:Orientations_disks} for every $u\in\cM_{1}(J,\gamma_i^+)$. Then we define
\begin{align*}
d_\lD(q_i)=\sum_{u\in\cM_{1}(J,\gamma_i^+)} \epsilon(u)\widetilde w(u).
\end{align*}

Next, we define $d_A(q_i),i\in\{1,\dots,n\}$. Denote by $\cM_{2}(J,\gamma_i^+)$ the moduli space of index zero $J$-holomorphic annuli in $\lR^4$ with boundary on $\lR\times\Lambda$ and one positive puncture asymptotic to $\gamma_i$. For every $u\in\cM_{2}(J,\gamma_i^+)$ and $e_2$ a (generic) marked point on its inner boundary component, we define $\overline w(u,e_2)\in\widetilde\cA\otimes \widetilde\cA$ as follows. As before, we define $\overline w_1(u),\overline w_2(u,e_2)\in\widetilde\cA$ by looking at the outer, inner boundary component of $u$. More precisely, we look at the negative punctures and the crossings of $\partial u$ over the base point in the order starting from the positive puncture for $\overline w_1(u)$ and in the order starting from the marked point $e_2$ for $\overline w_2(u,e_2)$. Then 
$$\overline w(u,e_2)=\overline w_1(u)\otimes \overline w_2(u,e_2).$$ 
Additionally, we have an orientation sign $\epsilon(u,e_2)\in\{1,-1\}$ defined in Section \ref{Section:Orientations_annuli}. Then we define
\begin{align*}
d_A(q_i)=\hbar\sum_{u\in\cM_{2}(J,\gamma_i^+)} \epsilon(u,e_2)\overline w(u,e_2).
\end{align*}
Recall from Section \ref{Section:Orientations} that $\epsilon(u,e_2)\overline w(u,e_2)\in\widetilde\cA\otimes\widetilde\cA^{\cyc}$ is independent of the choice of $e_2$.
\begin{rmk}
Using a similar formula as above, a virtual count of annuli for any combinatorial obstruction section can be used to define $d_A(q_i)$ instead of the count of actual $J$-holomorphic annuli. We prove in Section \ref{Section:Invar_IV_degeneration} that the second-order dg algebras obtained this way are isomorphic.
\end{rmk}
For two distinct points $A,B$ on $\Lambda$ different from $T$, we say $A<B$ if $B$ is on the arc of $\Lambda$ that starts at $T$ and ends at $A$ with respect to the orientation on $\Lambda$. Define
\begin{align*}
\delta(A,B)=\begin{cases}
1,&A<B\\
0,&\text{otherwise}
\end{cases}
\end{align*}

\begin{defi}
We say a Reeb chord $\gamma_i$ on $\Lambda$ is \textit{even} if the orientation of $\Lambda$ near $\gamma_i$ has the form as show in Figure \ref{Figure:twisted_chord}, left, in the Lagrangian projection (i.e. if $\mu_{CZ}(\gamma_i)$ is even). Otherwise, we say $\gamma_i$ is \textit{odd}.
\end{defi}

To motivate the definition of $d_f(q_i)$, we introduce the capping paths for $\gamma_i,i\in\{1,\dots,n\}$. Denote by $i^-$ and $i^+$ the starting point and the endpoint of the Reeb chord $\gamma_i$. There exists a unique embedded path $\widetilde c_i$ on $\Lambda$ starting at $i^-$ and ending at $i^+$ that does not pass through $T$. Denote by $c_i$ the knot obtained by shifting $\widetilde c_i$ in direction $(-1)^{|q_i|}J|_\xi\frac{d}{dt} \widetilde c_i$ and closing it off as shown in Figure \ref{Figure:shifted_cap}. Then we define $l_i\in\lZ$ as the linking number between knots $\Lambda$ and $c_i$. In other words,
\begin{align*}
l_i=\frac{1}{2}\sum_{j\neq i}(-1)^{|q_j|}\left(\delta(i^{+},j^{+})\delta(j^{+},i^{-})+\delta(i^{+},j^{-})\delta(j^{-},i^{-})\right)
\end{align*}
if $\delta(i^+,i^-)=1$, and
\begin{align*}
l_i=-\frac{1}{2}\sum_{j\neq i}(-1)^{|q_j|}\left(\delta(i^{-},j^{+})\delta(j^{+},i^{+})+\delta(i^{-},j^{-})\delta(j^{-},i^{+})\right)
\end{align*}
if $\delta(i^-,i^+)=1$. Then we define 
\begin{align*}
d_f(q_i)=\left(l_i-\delta(i^-,i^+)+1\right)\hbar(q_i\otimes 1)-\delta(i^-,i^+)\hbar(1\otimes q_i)
\end{align*}
for $|q_i|$ even, and
\begin{align*}
d_f(q_i)=\left(l_i+\delta(i^-,i^+)\right)\hbar(q_i\otimes 1)-\delta(i^-,i^+)\hbar(1\otimes q_i)
\end{align*}
for $|q_i|$ odd.

\begin{figure}
\def\svgwidth{74mm}
\begingroup%
  \makeatletter%
  \providecommand\rotatebox[2]{#2}%
  \newcommand*\fsize{\dimexpr\f@size pt\relax}%
  \newcommand*\lineheight[1]{\fontsize{\fsize}{#1\fsize}\selectfont}%
  \ifx\svgwidth\undefined%
    \setlength{\unitlength}{408.93776713bp}%
    \ifx\svgscale\undefined%
      \relax%
    \else%
      \setlength{\unitlength}{\unitlength * \real{\svgscale}}%
    \fi%
  \else%
    \setlength{\unitlength}{\svgwidth}%
  \fi%
  \global\let\svgwidth\undefined%
  \global\let\svgscale\undefined%
  \makeatother%
  \begin{picture}(1,0.35252701)%
    \lineheight{1}%
    \setlength\tabcolsep{0pt}%
    \put(0.15964599,0.10859364){\makebox(0,0)[lt]{\lineheight{1.25}\smash{\begin{tabular}[t]{l}$\gamma_i$\end{tabular}}}}%
    \put(0,0){\includegraphics[width=\unitlength,page=1]{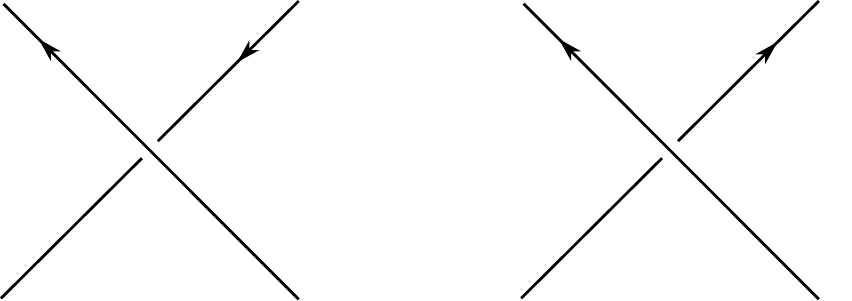}}%
    \put(0.77023397,0.1086902){\makebox(0,0)[lt]{\lineheight{1.25}\smash{\begin{tabular}[t]{l}$\gamma_i$\end{tabular}}}}%
  \end{picture}%
\endgroup%
\caption{Even and odd Reeb chord, respectively.}
\label{Figure:twisted_chord}
\end{figure}

\begin{figure}
\def\svgwidth{140mm}
\begingroup%
  \makeatletter%
  \providecommand\rotatebox[2]{#2}%
  \newcommand*\fsize{\dimexpr\f@size pt\relax}%
  \newcommand*\lineheight[1]{\fontsize{\fsize}{#1\fsize}\selectfont}%
  \ifx\svgwidth\undefined%
    \setlength{\unitlength}{551.28757944bp}%
    \ifx\svgscale\undefined%
      \relax%
    \else%
      \setlength{\unitlength}{\unitlength * \real{\svgscale}}%
    \fi%
  \else%
    \setlength{\unitlength}{\svgwidth}%
  \fi%
  \global\let\svgwidth\undefined%
  \global\let\svgscale\undefined%
  \makeatother%
  \begin{picture}(1,0.2560605)%
    \lineheight{1}%
    \setlength\tabcolsep{0pt}%
    \put(0.47481802,0.08331496){\makebox(0,0)[lt]{\lineheight{1.25}\smash{\begin{tabular}[t]{l}$T$\end{tabular}}}}%
    \put(0.96073746,0.08175991){\makebox(0,0)[lt]{\lineheight{1.25}\smash{\begin{tabular}[t]{l}$T$\end{tabular}}}}%
    \put(0.1137944,0.03068492){\makebox(0,0)[lt]{\lineheight{1.25}\smash{\begin{tabular}[t]{l}$c_i$\end{tabular}}}}%
    \put(0.28466676,0.21583589){\makebox(0,0)[lt]{\lineheight{1.25}\smash{\begin{tabular}[t]{l}$c_i$\end{tabular}}}}%
    \put(0.67045189,0.14198285){\makebox(0,0)[lt]{\lineheight{1.25}\smash{\begin{tabular}[t]{l}$c_i$\end{tabular}}}}%
    \put(0.76704401,0.10252974){\makebox(0,0)[lt]{\lineheight{1.25}\smash{\begin{tabular}[t]{l}$c_i$\end{tabular}}}}%
    \put(0.10815824,0.22600681){\makebox(0,0)[lt]{\lineheight{1.25}\smash{\begin{tabular}[t]{l}$T$\end{tabular}}}}%
    \put(0.33449724,0.01395081){\makebox(0,0)[lt]{\lineheight{1.25}\smash{\begin{tabular}[t]{l}$T$\end{tabular}}}}%
    \put(0,0){\includegraphics[width=\unitlength,page=1]{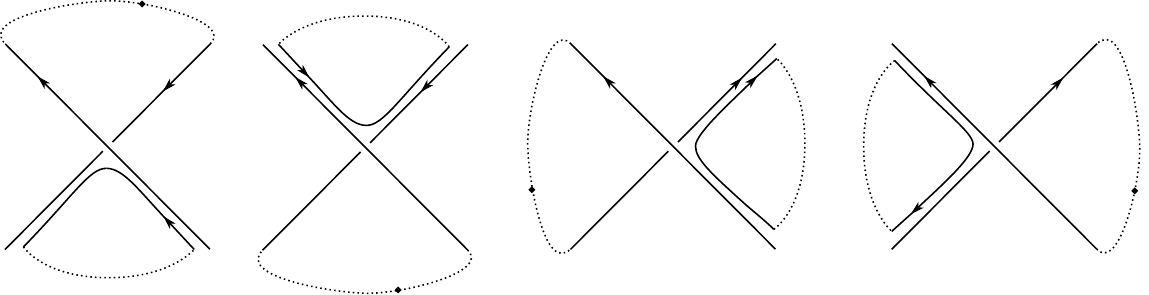}}%
  \end{picture}%
\endgroup%
\caption{Shift of $\widetilde c_i$ for $\gamma_i$ even when $\delta(i^+,i^-)=1$ and $\delta(i^-,i^+)=1$, and for $\gamma_i$ odd when $\delta(i^+,i^-)=1$ and $\delta(i^-,i^+)=1$, respectively.}
\label{Figure:shifted_cap}
\end{figure}
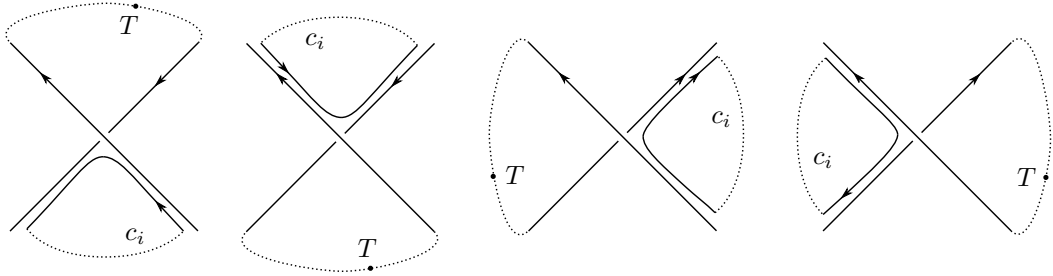

Additionally, we define $d_\lD(t^\pm)=d_A(t^\pm)=0$ and
\begin{align*}
&d_f(t^+)=\left(\operatorname{tb}(\Lambda)+1\right)\hbar(t^+\otimes 1),\\ 
&d_f(t^-)=-\operatorname{tb}(\Lambda)\hbar(t^-\otimes 1)-\hbar(1\otimes t^-),
\end{align*}
where $\operatorname{tb}(\Lambda)$ is the Thurston--Bennequin number of $\Lambda$. See the proof of Proposition \ref{Prop:dcircdvanishes} and Figure \ref{Figure:corrected_loop_coproduct_and_d_f} for the motivation behind the definition. Note that we are defining the normalized version (such that $d(1)=0$) of the operator from Section \ref{Sec:Invariant_definition}, see Remark \ref{Remark:normalization}.

Next, we define $\{q_i,q_j\}_d\in\widetilde\cA\otimes\widetilde\cA$ for $i,j\in\{1,\dots,n\}$. It consists of two parts: the contribution of $J$-holomorphic disks with two positive punctures $d_\lD(q_i,q_j)$, and the second-order part of the corrected loop coproduct $d_f(q_i,q_j)$. We define $d_\lD(q_i,q_j),d_f(q_i,q_j)$ below, and take $$\{q_i,q_j\}_d=d_\lD(q_i,q_j)+d_f(q_i,q_j).$$

To define $d_\lD(q_i,q_j)$, consider the moduli space $\cM_1(J,\gamma_i^+,\gamma_j^+)$ of index zero $J$-holomorphic disks on $\lR\times\Lambda$ with two positive punctures asymptotic to Reeb chords $\gamma_i,\gamma_j$. We have seen that $\cM_1(J,\gamma_i^+,\gamma_j^+)$ is in bijection with immersed polygons in $\lC$ with boundary on $\pi_{xy}(\Lambda)$, convex corners at the self intersections of $\pi_{xy}(\Lambda)$ and two positive corners at $\gamma_i$ and $\gamma_j$. Note that the positive punctures of the curves in $\cM_1(J,\gamma_i^+,\gamma_j^+)$ are ordered. Denote by $t^1,t^2$ the positive punctures of $u\in\cM_1(J,\gamma_i^+,\gamma_j^+)$ asymptotic to $\gamma_i,\gamma_j$, by $t_1,\dots, t_{k_1}$ the negative punctures on the arc $(t^1,t^2)$, and by $t_1',\dots, t_{k_2}'$ the negative punctures on the arc $(t^2,t^1)$. Let $\gamma_{i_j},\gamma_{i_j'}$ be the Reeb chord at $t_j,t_j'$, and $a_j\in\lZ$ ($a_j'\in\lZ$) be the algebraic intersection number of the boundary of $u$ restricted to the arc $(t_j,t_{j+1})$ $\left((t_j',t_{j+1}')\right)$ and $\lR\times\{T\}$, where we take $t_0=t_{k_2+1}'=t^1$ and $t_0'=t_{k_1+1}=t^2$. Then we denote 
\begin{align*}
&\widehat w_1(u,t^1,t^2)=t^{a_0}q_{i_1}t^{a_1}\dots t^{a_{k_1-1}}q_{i_{k_1}}t^{a_{k_1}},\\
&\widehat w_2(u,t^1,t^2)=t^{a_0'}q_{i_1'}t^{a_1'}\dots t^{a_{k_2-1}'}q_{i_{k_2}'}t^{a_{k_2}'},\\
&\widehat w(u,t^1,t^2)=\widehat w_1(u,t^1,t^2)\otimes\widehat w_2(u,t^1,t^2)\in\widetilde\cA\otimes \widetilde\cA.
\end{align*}
Additionally, we have an orientation sign $\epsilon(u,t^1,t^2)\in\{1,-1\}$ for every $u\in\cM_1(J,\gamma_i^+,\gamma_j^+)$ defined in Section \ref{Section:Orientations_disks}. Then we define
\begin{align*}
d_\lD(q_i,q_j)=\sum_{u\in\cM_1(J,\gamma_i^+,\gamma_j^+)}(-1)^{(1+|q_i|)|\widehat w_2|}\epsilon(u,t^1,t^2)\widehat w(u,t^1,t^2).
\end{align*}
If one of the entries $s_1,s_2$ is equal to $t^\pm$, we define $d_{\lD}(s_1,s_2)=0$.

Next, we define $d_f(q_i,q_j)$. Given $i\in\{1,\dots,n\}$, we denote
\begin{align*}
\delta(i)=\begin{cases}
1,&\gamma_i\text{ even}\\
0,&\gamma_i\text{ odd}
\end{cases}
\end{align*}
Then for $i,j\in\{1,\dots,n\}$ such that $i\neq j$, we define
\begin{align*}
d_f(q_i,q_j)=&\delta(j^+,i^+)q_j\otimes q_i+(-1)^{|q_i||q_j|}\delta(j^-,i^-)q_i\otimes q_j-\\
-&\delta(j^+,i^-) q_iq_j\otimes 1-(-1)^{|q_i||q_j|}\delta(j^-,i^+)1\otimes q_jq_i,
\end{align*}
and
\begin{align*}
d_f(q_i,q_i)=-\delta(i^+,i^-)q_iq_i\otimes 1-(-1)^{|q_i|}\delta(i^-,i^+)1\otimes q_iq_i+\delta(i)q_i\otimes q_i
\end{align*}
for $i=j$.

Additionally, we define
\begin{equation}
\label{Eq:bracket_for_t}
\begin{aligned}
&d_f(s,t^+)=\{s,t^+\}_d=t^+\otimes s-st^+\otimes 1,\\
&d_f(s,t^-)=\{s,t^-\}_d=s\otimes t^--1\otimes t^-s,\\
&d_f(t^+,s)=\{t^+,s\}_d=-t^+s\otimes 1+t^+\otimes s,\\
&d_f(t^-,s)=\{t^-,s\}_d=-1\otimes st^-+s\otimes t^-,
\end{aligned}
\end{equation}
for $s\in\{q_i,t^\pm\,|\,i=1,\dots,n\}$.

We have relations $t^+t^-=t^-t^+=1$ on $\cA(\Lambda)$. It is not difficult to check that all the operators above are well-defined. This uniquely determines a second-order derivation $(d,\{\cdot,\cdot\}_d)$ on $\cA(\Lambda)$. In Proposition \ref{Prop:strong_derivation} and Proposition \ref{Prop:dcircdvanishes} below, we show that $(\cA(\Lambda),d,\{\cdot,\cdot\}_d)$ is a second-order dg algebra.

\begin{rmk}\label{Remark:normalization}
The operator defined in Section \ref{Sec:The_differential} corresponds modulo strings with two positive punctures to the non-normalized operator
\begin{align}
\label{Eq:normalization_f}
\widehat d(s)=d(s)-\hbar(s\otimes 1),
\end{align}
that satisfies
\begin{align*}
&\widehat d(uv)=\widehat d(u)v+(-1)^{|u|}u\widehat d(v)-(-1)^{|u|}u\widehat d(1) v+\{u,v\}_\hbar,\\
&\widehat d(1)=-\hbar(1\otimes 1).
\end{align*}
It is not difficult to check that $d^2=0$ if and only if $\widehat d^2=0$.
\end{rmk}
\begin{rmk}
\label{Rmk:sign_difference}
The signs here are slightly different compared to the definition in Section \ref{Sec:Invariant_definition}. Summands of the form $\zeta=\hbar(\gamma_1\otimes\gamma_2)$ appear with an extra sign $(-1)^{|\gamma_2|}$. 
\end{rmk}

\vspace{3mm}
Denote by $\widetilde\cA^r$ the tensor algebra generated by $q_i,p_i,t^+,t^-,i\in\{1,\dots,n\}$ over $\lQ$, quotient by relations $t^+t^-=t^-t^+=1$ and words that contain more than $r$ letters $p$, with the grading given by $|q_i|=-1-|p_i|=\mu_{CZ}(\gamma_i),|t^\pm|=\mp2\operatorname{rot}(\Lambda)$.

Define a differential $\delta:\widetilde\cA^1\to\widetilde\cA^1$
\begin{equation}
\label{Eg:delta_formula}
\begin{aligned}
&\begin{aligned}
\delta(q_j)=\sum_{i\neq j}\Big(&\delta(j^+,i^+)q_ip_iq_j+(-1)^{|q_j|}\delta(j^-,i^-)q_jp_iq_i-\\
&-\delta(j^+,i^-)p_iq_iq_j+(-1)^{|q_j|+1}\delta(j^-,i^+)q_jq_ip_i\Big)+\\
&+(-1)^{|q_j|+1}\delta(j^-,j^+)q_jq_jp_j-\delta(j^+,j^-)p_jq_jq_j+\delta(j)q_jp_j q_j,
\end{aligned}\\
&\begin{aligned}
\delta(t)&=\sum_{i}(q_ip_it-p_iq_it),\\
\delta(t^-)&=\sum_{i}(t^-p_iq_i-t^-q_ip_i),\\
\delta(p_i)&=0.
\end{aligned}
\end{aligned}
\end{equation}
Map $\delta$ corresponds to the operator defined in Section \ref{Sec:loop_product_operator}.

\begin{prop}\label{Prop:strong_derivation}
The map $d:\cA(\Lambda)\to\cA(\Lambda)$ is a derivation with respect to the antibracket $\{\cdot,\cdot\}_d$, i.e.
\begin{align*}
(d_0\otimes 1+1\otimes d_0)\{v,w\}_d=\{d_0v,w\}_d+(-1)^{|v|}\{v,d_0w\}_d.
\end{align*}
\end{prop}
\begin{proof}
If one of the entries is equal to $t^+$ or $t^-$, the statement is easy to check directly. Now, it is enough to show
\begin{align*}
(d_0\otimes 1+1\otimes d_0)\{q_i,q_j\}_d=\{d_0q_i,q_j\}_d+(-1)^{|q_i|}\{q_i,d_0q_j\}_d
\end{align*}
for all $i,j\in\{1,\dots,n\}$.

Let $\operatorname{cl}(\cdot,\cdot):\widetilde\cA\times\widetilde\cA^1\to\widetilde\cA\otimes\widetilde\cA$ be the bilinear map given by
\begin{align*}
&\operatorname{cl}(q_{j_1}\dots q_{j_l},q_{j_{l+1}}\dots q_{j_s}p_{I} q_{j_{s+1}}\dots q_{j_k})=\\
&=\sum_{\substack{i=1,\dots,l\\j_i=I}}(-1)^{\left(1+\sum_{a=1}^l|q_{j_a}|\right)\left(\sum_{b=i+1}^s|q_{j_b}|\right)}
q_{j_1}\dots q_{j_{i-1}}q_{j_{s+1}}\dots q_{j_k}\otimes q_{j_{l+1}}\dots q_{j_s} q_{j_{i+1}}\dots q_{j_l},
\end{align*}
and zero on $\widetilde\cA\times\widetilde\cA$. It is not difficult to see that $\operatorname{cl}(\cdot,\cdot)$ is the unique bilinear map that satisfies
\begin{align*}
&\operatorname{cl}(q_i,p_j)=\delta_{ij}1\otimes 1,\\
&\operatorname{cl}(q_i,q_j)=0,
\end{align*}
where $\delta_{ij}=1$ if $i=j$ and zero otherwise, and 
\begin{align*}
&\operatorname{cl}(v,w_1w_2)=\operatorname{cl}(v,w_1)\cdot(w_2\otimes 1)+(-1)^{(1+|v|)|w_1|}(1\otimes w_1)\cdot\operatorname{cl}(v,w_2),\\
&\operatorname{cl}(v_1v_2,w)=(v_1\otimes 1)
\cdot\operatorname{cl}(v_2,w)+(-1)^{|v_2|(|w|+1)}\operatorname{cl}(v_1,w)\cdot(1\otimes v_2). 
\end{align*}
Let $\{\cdot,\cdot\}_{d,1}$ be the antibracket given by 
\begin{align*}
&\{q_i,s\}_{d,1}=\{q_i,s\}_d,i\in\{1,\dots,n\}\\
&\{t^\pm,s\}_{d,1}=0.
\end{align*}
We additionally define a derivation $\widetilde d:\widetilde\cA^1\to\widetilde\cA^1$ given by
\begin{equation}
\label{Eq:d_unglued_pos_formula}
\begin{aligned}
&\widetilde d(q_j)=\delta(q_j)+\sum_{u=wp_j\in\cM_1(J,\gamma_j^+)}\epsilon(u)w+\sum_{u=w_1p_iw_2p_j\in\cM_1(J,\gamma_i^+,\gamma_j^+)}\epsilon(u,\gamma_j^+)w_1p_iw_2,\\
&\widetilde d(p_j)=\sum_{u=w_2q_jw_1p_i\in\cM_1(J,\gamma_i^+)}(-1)^{|q_i|+|w_1|}\epsilon(u)w_1p_iw_2=-\sum_{u=w_2q_jw_1p_i\in\cM_1(J,\gamma_i^+)}\epsilon(u,\gamma_j^-)w_1p_iw_2,\\
&\widetilde d(t^\pm)=\delta(t^\pm),
\end{aligned}
\end{equation}
where $\epsilon(u,\gamma_j^+)=\epsilon(u,t^1,t^2)$ denotes the product of the signs at the corners of $u\in\cM_1(J,\gamma_i^+,\gamma_j^+)$ and the sign at the marked point right after the positive puncture at $\gamma_j$, and $\epsilon(u,\gamma_j^-)$ is the product of the signs at the corners of $u\in\cM_1(J,\gamma_i^+)$ with a negative puncture at $\gamma_j$ and the sign at the marked point right after this puncture. Note that $\widetilde d$ is part of the operator defined in \cite{Ng_rLSFT} (with slightly different sign rules). Denote additionally 
\begin{align*}
\widetilde d_{0,0}=\pi_{0}\circ \widetilde d\circ \iota_{0},\\
\widetilde d_{1,0}=\pi_{1}\circ \widetilde d\circ \iota_{0},\\
\widetilde d_{1,1}=\pi_{1}\circ \widetilde d\circ \iota_{1},
\end{align*}
where $\pi_{0}$ is the projection to the subspace generated by words with no letter $p$, $\pi_1$ is the projection to the subspace generated by words with exactly one letter $p$, $\iota_0,\iota_1$ are the corresponding inclusions, and
\begin{align*}
&\widetilde d_{1,p}(q_{j_1}\dots q_{j_s} p_I q_{j_{s+1}}\dots q_{j_k})=(-1)^{\sum_{a=1}^s|q_{j_a}|}q_{j_1}\dots q_{j_s} \widetilde d(p_I)q_{j_{s+1}}\dots q_{j_k},\\
&\widetilde d_{1,q}=\widetilde d_{1,1}-\widetilde d_{1,p}.
\end{align*} 

Then by definition we have
\begin{equation}
\begin{aligned}\label{Eq:l2_bracket_cl_operator}
\{v,w\}_{d,1}=\operatorname{cl}(v,\widetilde d_{1,0} w).
\end{aligned}
\end{equation}
Moreover, it is not difficult to check that
\begin{align*}
&(\widetilde d_{0,0}\otimes 1+1\otimes \widetilde d_{0,0})\operatorname{cl}(q_i, \widetilde d_{1,0}q_j)=-(-1)^{|q_i|}\operatorname{cl}(q_i,\widetilde d_{1,q}\widetilde d_{1,0}q_j),
\end{align*}
and
\begin{align*}
&\operatorname{cl}(\widetilde d_{0,0}q_i,\widetilde d_{1,0}q_j)=(-1)^{|q_i|}\operatorname{cl}(q_i,\widetilde d_{1,p}\widetilde d_{1,0}q_j).
\end{align*} 
The first equality follows from
\begin{align*}
&(-1)^{|q_i|}\operatorname{cl}(q_i,\widetilde d_{1,q}\widetilde d_{1,0}q_j)=(-1)^{|q_i|}\sum\operatorname{cl}\left(q_i,\widetilde d_{0,0}(R_{ij}^1)p_iR_{ij}^2+(-1)^{|R_{ij}^1|+|q_i|+1}R_{ij}^1p_i\widetilde d_{0,0}(R_{ij}^2)\right)=\\
&=\sum(-1)^{(1+|q_i|)|R_{ij}^1|+1}R_{ij}^2\otimes \widetilde d_{0,0}(R_{ij}^1)+\sum(-1)^{|R_{ij}^1|+1+(1+|q_i|)|R_{ij}^1|}\widetilde d_{0,0}(R_{ij}^2)\otimes R_{ij}^1=\\
&=-\sum(\widetilde d_{0,0}\otimes 1+1\otimes \widetilde d_{0,0})\left((-1)^{(1+|q_i|)|R_{ij}^1|}R_{ij}^2\otimes R_{ij}^1\right)=-(\widetilde d_{0,0}\otimes 1+1\otimes \widetilde d_{0,0})\operatorname{cl}(q_i, \widetilde d_{1,0}q_j),
\end{align*}
where we write $\widetilde d_{1,0}(q_j)|_{p_j=0,j\neq i}=\sum R_{ij}^1 p_iR_{ij}^2$. \newline
The second equality follows from
\begin{align*}
&\operatorname{cl}(\widetilde d_{0,0}q_i,\widetilde d_{1,0}q_j)=\sum_{v=\overline h_1q_k\overline h_2p_i\in\cM_1}\sum_{u=h_1p_kh_2p_j\in\cM_1}\epsilon(u,\gamma_j^+)\epsilon(v)\operatorname{cl}(\overline h_1q_k\overline h_2,h_1p_kh_2)=\\
&=\sum_{v=\overline h_1q_k\overline h_2p_i\in\cM_1}\sum_{u=h_1p_kh_2p_j\in\cM_1}\epsilon(u,\gamma_j^+)\epsilon(v)(-1)^{|q_i|(|h_1|+|\overline h_2|)}\overline h_1 h_2\otimes h_1\overline h_2=\\
&=\sum_{v=\overline h_1q_k\overline h_2p_i\in\cM_1}\sum_{u=h_1p_kh_2p_j\in\cM_1}\epsilon(u,\gamma_j^+)\epsilon(v)(-1)^{|h_1|+|\overline h_2|}\operatorname{cl}(q_i,h_1\overline h_2p_i\overline h_1 h_2)=\\
&=(-1)^{|q_i|}\operatorname{cl}(q_i,\widetilde d_{1,p}\widetilde d_{1,0}q_j).
\end{align*}
This gives us
\begin{align*}
&-(d_{0}\otimes 1+1\otimes d_{0})\{q_i,q_j\}_{d,1}+\{d_{0}q_i,q_j\}_{d,1}+(-1)^{|q_i|}\{q_i,d_{0}q_j\}_{d,1}=\\
=&-(\widetilde d_{0,0}\otimes 1+1\otimes \widetilde d_{0,0})\operatorname{cl}(q_i, \widetilde d_{1,0}q_j)+\operatorname{cl}(\widetilde d_{0,0}q_i,\widetilde d_{1,0}q_j)+(-1)^{|q_i|}\operatorname{cl}(q_i,\widetilde d_{1,0}\widetilde d_{0,0}q_j)=\\
=&(-1)^{|q_i|}\operatorname{cl}(q_i,\pi_1\circ\widetilde d\circ \widetilde d(q_j)).
\end{align*}
From \cite{Ng_rLSFT} we have $\widetilde d\circ \widetilde d(q_j)=[F(H_0),q_j]$. Here $H_0=\sum_{u=wp_i\in\cM_1(J,\gamma_i^+)}\epsilon (u,\gamma_i)wp_i\in\widetilde\cA^1$ is the part of the Hamiltonian containing disks with one positive puncture, $[x,y]=xy-(-1)^{|x||y|}yx$ is the commutator and $F:\widetilde\cA^1\to\widetilde\cA^1$ is the linear map given by
\begin{align*}
F(h)=&-\sum_{h=\widetilde h_1 t\widetilde h_2p_i}(-1)^{|\widetilde h_1|(|h|-|\widetilde h_1|)}\widetilde h_2p_i\widetilde h_1 t  +\\
&+\sum_{h=\overline h_1 t^-\overline h_2p_i}(-1)^{|\overline h_1|(|h|-|\overline h_1|)}t^-\overline h_2p_i\overline h_1 
\end{align*}
for $h\in\widetilde\cA^1$ word. This is obtained by considering the compactification of the 1-dimensional moduli space of disks with two positive punctures. See also Lemma \ref{Lemma:signs_FH} and Lemma \ref{Lemma:signs_disksII} for more details.

Next, let $\{\cdot,\cdot\}_{d,2}:\widetilde\cA\times\widetilde\cA\to\widetilde\cA\otimes\widetilde\cA$ be the antibracket given by
\begin{align*}
\{v,w\}_{d,2}=&-\sum_{v=v_1t^+v_2}(-1)^{|v_1||v_2|}v_1t^+w\otimes v_2+\sum_{v=\widetilde v_1t^-\widetilde v_2}(-1)^{|\widetilde v_1||\widetilde v_2|}\widetilde v_1w\otimes t^-\widetilde v_2+\\
&+\sum_{v=v_1t^+v_2}(-1)^{|v_1||v_2|+|v||w|}v_1t^+\otimes wv_2-\sum_{v=\widetilde v_1t^-\widetilde v_2}(-1)^{|\widetilde v_1||\widetilde v_2|+|v||w|}\widetilde v_1\otimes wt^-\widetilde v_2
\end{align*}
for words $v,w\in\widetilde\cA$, where the sums go over all different ways to write $v$ as $v_1t^+v_2$ or $\widetilde v_1t^-\widetilde v_2$ for $v_1,v_2,\widetilde v_1,\widetilde v_2\in\widetilde\cA$. In other words, $\{\cdot,\cdot\}_{d,2}$ is the antibracket given by
\begin{align*}
&\{t^+,s\}_{d,2}=t^+\otimes s-t^+s\otimes 1,\\
&\{t^-,s\}_{d,2}=s\otimes t^--1\otimes st^-,\\
&\{q_i,s\}_{d,2}=0,
\end{align*}
for $s\in\{q_i,t^\pm\,|\,i=1,\dots,n\}$, i.e.
$$\{\cdot,\cdot\}_{d,2}=\{\cdot,\cdot\}_d-\{\cdot,\cdot\}_{d,1}.$$

Now, it is enough to show 
$$\{\widetilde d_{0,0}q_i,q_j\}_{d,2}=(-1)^{|q_i|+1}\operatorname{cl}\left(q_i,[F(H_0),q_j]\right).$$
This follows from
\begin{align*}
&\{\widetilde d_{0,0}q_i,q_j\}_{d,2}=\\
=&\sum_{u=\widetilde h p_i\in H_0}\sum_{\widetilde h=\widetilde h_1 t\widetilde h_2}\epsilon(u)\left(-(-1)^{|\widetilde h_1||\widetilde h_2|}\widetilde h_1 t q_j\otimes \widetilde h_2+ (-1)^{|\widetilde h_1||\widetilde h_2|+(|q_i|+1)|q_j|}\widetilde h_1 t\otimes q_j\widetilde h_2\right)+\\
+&\sum_{u=\overline h p_i\in H_0}\sum_{\overline h=\overline h_1 t^-\overline h_2}\epsilon(u)\left((-1)^{|\overline h_1||\overline h_2|}\overline h_1 q_j\otimes t^-\overline h_2-(-1)^{|\overline h_1||\overline h_2|+(|q_i|+1)|q_j|} \overline h_1 \otimes q_j t^-\overline h_2\right)=\\
=&\sum_{u=\widetilde h p_i\in H_0}\sum_{\widetilde h=\widetilde h_1 t\widetilde h_2}\epsilon(u)\left(-(-1)^{|\widetilde h_2|}\operatorname{cl}(q_i,\widetilde h_2p_i\widetilde h_1 t q_j)+(-1)^{|\widetilde h_2|}\operatorname{cl}(q_i,q_j\widetilde h_2p_i\widetilde h_1 t )\right)+\\
+&\sum_{u=\overline h p_i\in H_0}\sum_{\overline h=\overline h_1 t^-\overline h_2}\epsilon(u)\left((-1)^{|\overline h_2|}\operatorname{cl}(q_i,t^-\overline h_2p_i\overline h_1 q_j)-(-1)^{|\overline h_2|}\operatorname{cl}(q_i,q_j t^-\overline h_2p_i\overline h_1)\right)=\\
=&\operatorname{cl}\Big(q_i,\sum_{u=\widetilde h p_i\in H_0}\sum_{\widetilde h=\widetilde h_1 t\widetilde h_2}\epsilon(u)\big(-(-1)^{|\widetilde h_2|}\widetilde h_2p_i\widetilde h_1 t q_j+(-1)^{|\widetilde h_2|}q_j\widetilde h_2p_i\widetilde h_1 t\big) +\\
+&\sum_{u=\overline h p_i\in H_0}\sum_{\overline h=\overline h_1 t^-\overline h_2}\epsilon(u)\big((-1)^{|\overline h_2|}t^-\overline h_2p_i\overline h_1 q_j-(-1)^{|\overline h_2|}q_j t^-\overline h_2p_i\overline h_1\big)\Big)=\\
=&(-1)^{|q_i|+1}\operatorname{cl}\left(q_i,[F(H_0),q_j]\right).
\end{align*}
Here we use $|H_0|\equiv 0\pmod 2$.
\end{proof}

\begin{prop}\label{Prop:dcircdvanishes}
The map $d:\cA(\Lambda)\to\cA(\Lambda)$ satisfies $d\circ d=0$.
\end{prop}
\begin{proof}
We prove the proposition up to signs. Checking the signs is postponed until Section \ref{Sec:orientation_and_algebraic_signs}. Using Proposition \ref{Prop:strong_derivation}, it is enough to show $d^2(s)=0$ for $s\in\{q_i,t^\pm\,|\,i=1,\dots,n\}$.

The definition of $d_\lD(\cdot),d_{A}(\cdot)$ up to signs clearly corresponds to $\{\cdot,H_1^1\}_1+\{\cdot,H_1^2\}_2,\{\cdot,H_2\}_1$, where $H_1^\iota,\iota\in\{1,2\}$ is the part of the Hamiltonian $H=H(\Lambda)$ that contains disks with $\iota$ positive punctures, and $H_2$ the part that contains annuli.

Let $S\in\Lambda$ be a point on $\Lambda$ right after the base point $T$ and $\beta_i$ be an immersed string on $\lR\times\Lambda$ with the following properties. Let $\beta_i^+$ be an embedded string on $\lR\times\Lambda$ that goes from $(M,S),M>>1$ down to $(-\infty,i^+)$, $\beta_i^+$ an embedded string on $\lR\times\Lambda$ that goes from $(-\infty,i^-)$ up to $(M,S)$. Let $\beta_i$ be a string obtained by concatenating $\beta_i^+$ and $\beta_i^-$ such that $\beta_i$ does not pass through $T$, passes through $(M,S)$ at a marked point $t_0$ with tangency $\dot\beta_i(t_0)=(1,0)$, and in the Lagrangian projection has precisely one non-regular point at $t_0$. Let $u_i$ be a generic punctured disk with one negative puncture at $\gamma_i$ and embedded boundary that does not pass through the base point. Then by definition 
$$|\operatorname{sh}(u_i)\cap L|=\operatorname{lk}(\Lambda,c_i)$$
for $|q_i|$ even, and
$$|\operatorname{sh}(u_i)\cap L|=\operatorname{lk}(\Lambda,c_i)-(-1)^{\delta(i^-,i^+)}$$
for $|q_i|$ odd. Let $\widetilde u_i$ be a punctured disk with boundary $\beta_i$ and marked point $t_0$. Then
$$|\operatorname{sh}(\widetilde u_i)\cap L|=\operatorname{lk}(\Lambda,c_i)-\delta(i^-,i^+)$$
for $|q_i|$ even, and
$$|\operatorname{sh}(\widetilde u_i)\cap L|=\operatorname{lk}(\Lambda,c_i)-(-1)^{\delta(i^-,i^+)}-\delta(i^-,i^+)=\operatorname{lk}(\Lambda,c_i)+\delta(i^-,i^+)-1$$
for $|q_i|$ odd. Additionally
\begin{align*}
|\beta_i^-\cap\beta_i^+|= \delta(i^-,i^+),
\end{align*}
these points correspond to the words of the form $\hbar(1\otimes q_i)$ in $d_f(q_i)$. In other words, $d_f(q_i)$ corresponds to $\nabla(q_i)+\hbar(q_i\otimes 1)$. Similar holds for $t^\pm$.

\begin{figure}
\def\svgwidth{91mm}
\begingroup%
  \makeatletter%
  \providecommand\rotatebox[2]{#2}%
  \newcommand*\fsize{\dimexpr\f@size pt\relax}%
  \newcommand*\lineheight[1]{\fontsize{\fsize}{#1\fsize}\selectfont}%
  \ifx\svgwidth\undefined%
    \setlength{\unitlength}{678.71145526bp}%
    \ifx\svgscale\undefined%
      \relax%
    \else%
      \setlength{\unitlength}{\unitlength * \real{\svgscale}}%
    \fi%
  \else%
    \setlength{\unitlength}{\svgwidth}%
  \fi%
  \global\let\svgwidth\undefined%
  \global\let\svgscale\undefined%
  \makeatother%
  \begin{picture}(1,0.64491341)%
    \lineheight{1}%
    \setlength\tabcolsep{0pt}%
    \put(0.85435899,0.0018152){\makebox(0,0)[lt]{\lineheight{1.25}\smash{\begin{tabular}[t]{l}$S$\end{tabular}}}}%
    \put(0.90944603,0.0018152){\makebox(0,0)[lt]{\lineheight{1.25}\smash{\begin{tabular}[t]{l}$T$\end{tabular}}}}%
    \put(0.22978674,0.22675385){\makebox(0,0)[lt]{\lineheight{1.25}\smash{\begin{tabular}[t]{l}$\operatorname{cl}(w_1,\delta w_2)$\end{tabular}}}}%
    \put(0.85687798,0.38673175){\makebox(0,0)[lt]{\lineheight{1.25}\smash{\begin{tabular}[t]{l}${\{w_1,w_2\}_{d,2}}$\end{tabular}}}}%
    \put(0,0){\includegraphics[width=\unitlength,page=1]{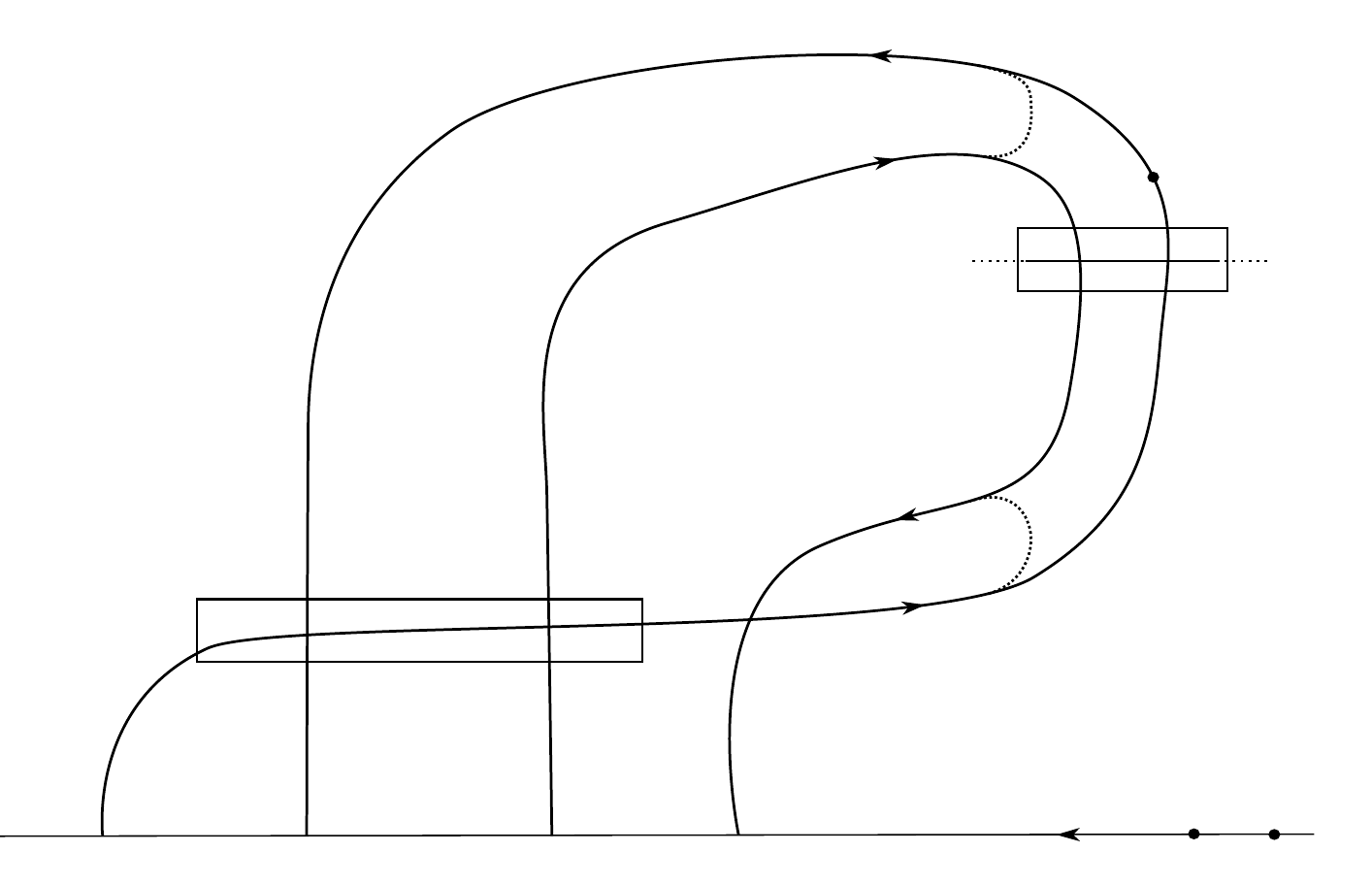}}%
    \put(0.6666481,0.61397242){\makebox(0,0)[lt]{\lineheight{1.25}\smash{\begin{tabular}[t]{l}$w_1$\end{tabular}}}}%
    \put(0.6688006,0.2897705){\makebox(0,0)[lt]{\lineheight{1.25}\smash{\begin{tabular}[t]{l}$w_2$\end{tabular}}}}%
  \end{picture}%
\endgroup%
\caption{Correspondence between $d_f(w_1w_2)$ and $\nabla(w_1w_2)+\hbar(w_1w_2\otimes 1)$.}
\label{Figure:corrected_loop_coproduct_and_d_f}
\end{figure}
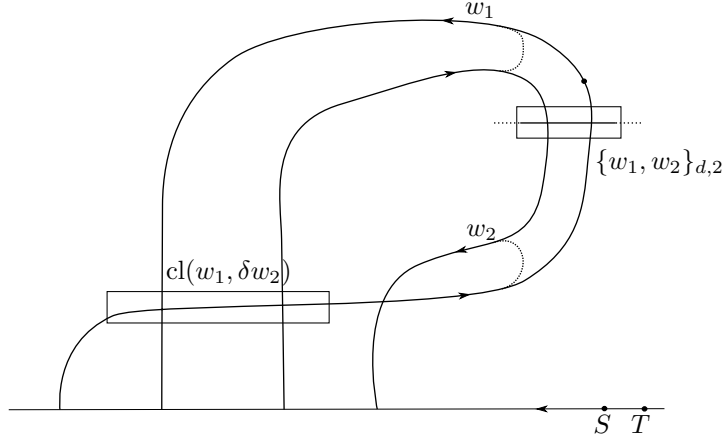

Moreover, it is not difficult to see that $d_f(w)$ corresponds to $\nabla(w)+\hbar(w\otimes 1)$ for any word $w$. Let $w=w_1w_2$, where this is known for words $w_1,w_2$. Let $\beta_1,\beta_2$ be strings on $\lR\times\Lambda$ corresponding to $w_1,w_2$ as above (recursive construction). Then we construct a string $\beta$ for $w_1w_2$ and the corresponding spanning disk $\widetilde u$ by shifting $\beta_1,\widetilde u_i$ far above $\beta_2,\widetilde u_2$, and connecting them as shown in Figure \ref{Figure:corrected_loop_coproduct_and_d_f}. Then, $(d_f(w_1)-\hbar(w_1\otimes 1))w_2$ and $(-1)^{|w_1|}w_1 (d_f(w_2)-\hbar(w_2\otimes 1))$ correspond to the interior and the boundary intersection points of $\widetilde u$ coming form $\widetilde u_1$ and $\widetilde u_2$ separately, while $d_f(w_1,w_2)$ corresponds to boundary intersections of trivial strips at the negative punctures of $\beta_1$ with $\beta_2$ ($\operatorname{cl}(w_1,\delta w_2)$ by definition) and intersections of the connecting strip between $(M_1,S)$ and $(M_2,S)$ ($M_1>M_2$) with the parts of $\beta_1$ passing through $T$ ($\{w_1,w_2\}_{d,2}$ by definition) as depicted in Figure \ref{Figure:corrected_loop_coproduct_and_d_f}. An additional positive interior intersection corresponding to $\hbar(w_1w_2\otimes 1)$ of the shifted disk $\operatorname{sh}(\widetilde u)$ appears in the shift of the connecting strip. This shows (\ref{Eq:normalization_f}).

Denote by $\cM$ the 1-dimensional moduli space of disks and annuli with boundary on $\lR\times\Lambda$ and a positive puncture at $\gamma_i$. As we have seen in Section \ref{Sec:The_differential}, the summands in $d^2(q_i)$ correspond to points in $\partial\cM$. The boundary of $\cM$ consists of disk and annulus buildings and nodal annuli. The nodal annuli can be seen as index zero disks together with a boundary self-intersection or an interior intersection with $\lR\times\Lambda$. This gives us
\begin{align*}
d\circ d(q_i)=0
\end{align*}
over $\lZ_2$ coefficients as in Proposition \ref{Prop:d_square_vanishes}. Here we also use Lemma \ref{Lemma:no_bad_words}, which states that there are no bad index zero curves, see below. 

In the following section we show that the signs of the corresponding summands in $d\circ d(q_i)$ cancel out, which finishes the proof.
\end{proof}

\begin{defi}
We say a word $w\in\widetilde\cA$ is \textit{bad} if it is an even cover of an odd word, i.e. if there exists $w'\in\widetilde\cA,|w'|\equiv 1\pmod 2$ such that $w=\underbrace{w'\dots w'}_{2p}$ for some $p\in\lN$.
\end{defi}
A word $w$ is bad if it is equal to zero in $\widetilde\cA^{\cyc}$. To get cancellations in $d\circ d$, we additionally need to deal with index zero $J$-holomorphic curves that contain a bad word in a boundary component, which we call \textit{bad curves}. In our case, we show that there exists no such bad annulus or disk, see also Remark \ref{Rem:Cancelations_bad_curves} for a more general setting.

\begin{lemma}\label{Lemma:no_bad_words}
For $\Lambda$ a Legendrian knot, there exists no bad annulus or disk on $\lR\times\Lambda$ of index zero.
\end{lemma}
\begin{proof}
Assume there is such a curve $u_0$ and denote by $w$ the bad word in its boundary component. Then $w$ can be written as an even cover of an odd word $w'=t^{a_0}q_{i_1}t^{a_1}\dots q_{i_k}$, $w=\underbrace{w'\dots w'}_{2p},\,|w'|\equiv 1\pmod 2,\,p\in\lN$. There is at most one boundary branch point of $\pi_{xy}\circ u_0$ since $u_0$ is of index zero (one if $u_0$ is an annulus and zero if it is a disk). Since $|w|$ is even and $|w|+\#(\text{branch pts on }w$) is always even, this branch point cannot be on the boundary component corresponding to $w$. From this, we can conclude that two corners at $q_{i_j}$ in consecutive iterations of $w'$ have to be in opposite quadrants when seen in the Lagrangian projection. Then, the two corresponding boundary arcs between $q_{i_j}$ and $q_{i_{j+1}}$ in two consecutive iterations of $w'$ together pass through every point on $\pi_{xy}(\Lambda)$ when seen in the Lagrangian projection. This is clearly impossible for a projection of an index zero curve for $\Lambda$ a knot.

Similar also holds for Legendrian links when working with loop coefficients. Here we additionally use the fact that precisely one of the two arcs above passes through the base point of the corresponding link component, see Figure \ref{Fig:no_bad_words}. 
\end{proof}
\begin{rmk}\label{Rem:Cancelations_bad_curves}
In case we are working with a setting where Lemma \ref{Lemma:no_bad_words} does not hold, we can still get cancellations in $d\circ d$ as follows. Let for example $u$ be a bad annulus of index zero and $v$ an index zero disk that can be glued to the bad boundary component of $u$. These buildings are seen as zero in $d\circ d$, but the second boundary point of the corresponding connected component of the moduli space potentially has a non-zero contribution. We consider boundary points of $\overline\cM$ that are seen as zero in $d\circ d$ but do not contain a bad word when glued, i.e. that consist of a building that contains a bad curve and at least one of the gluing punctures comes from the bad component. For $u,v$ as above, the bad boundary component of $u$ is an even cover of an odd word, therefore, there is an even number of such buildings obtained by gluing $v$ to a different iteration of the odd word. Note that these buildings are not equivalent since annuli with one positive puncture have no non-trivial automorphisms. Moreover, it is easy to check that exactly half of such $(u,v)$ buildings come with orientation sign $+1$ and half with sign $-1$ (see below). Similar holds for bad index zero disks. This shows that half of the non-zero boundary points cobordant to a bad building come with orientation sign $+1$ and half with sign $-1$, and they cancel each other out. 

To see that exactly half of the buildings obtained by gluing $v$ to the bad boundary of $u$ come with a positive orientation sign, choose a marked point $e_2$ on the bad boundary component and let $\overline w(u,e_2)=q_{i_1}\dots q_{i_n}\otimes  (q_{j_1}\dots q_{j_m})\dots (q_{j_1}\dots q_{j_m})$, where $\overline w_2(u,e_2)=(q_{j_1}\dots q_{j_m})\dots (q_{j_1}\dots q_{j_m})$ is a $2p$-cover of a word $w_2=q_{j_1}\dots q_{j_m},|w_2|\cong 1\pmod 2$. Without loss of generality, $v$ has a positive puncture at $\gamma_{j_1}$. Denote by $B_i,i\in\{1,\dots,2p\}$ the $(u,v)$ building obtained by gluing $v$ to the $i^{th}$ iteration of $w_2$ with respect to the marked point $e_2$. The orientation near $\pi_{xy}B_i$ in the Lagrangian projection is given by
\begin{align*}
\epsilon_1\epsilon_2\epsilon(v)\left(\prod\epsilon_\bullet^u\right)(-1)^{|q_{i_1}\dots q_{i_l}|}(-1)^{i-1}\langle v_1,v_2\rangle,
\end{align*}
where $v_1$ is the branch point on $u$ (on the arc between the $l^{th}$ and the $(l+1)^{st}$ puncture on the outer boundary component), $v_2$ is the branch point that appears after gluing and $\epsilon_\bullet^u$ are the signs at the corners of $u$. The orientation of the boundary point $B_i$ is therefore given by
\begin{align*}
\pm\epsilon_1\epsilon_2\epsilon(v)\left(\prod\epsilon_\bullet^u\right)(-1)^{|q_{i_1}\dots q_{i_l}|}(-1)^{i-1},
\end{align*}
where the orientation normal in the $v_1$ direction (direction pointing from $\Omega<0$ to $\Omega>0$), $n=\pm v_1$, is the same for all $i$ by Lemma \ref{Lemma:OSection_elliptic}. Clearly, the sign is $+1$ for exactly half of the values of $i$.
\end{rmk}

\begin{figure}
\def\svgwidth{40mm}
\begingroup%
  \makeatletter%
  \providecommand\rotatebox[2]{#2}%
  \newcommand*\fsize{\dimexpr\f@size pt\relax}%
  \newcommand*\lineheight[1]{\fontsize{\fsize}{#1\fsize}\selectfont}%
  \ifx\svgwidth\undefined%
    \setlength{\unitlength}{277.97664226bp}%
    \ifx\svgscale\undefined%
      \relax%
    \else%
      \setlength{\unitlength}{\unitlength * \real{\svgscale}}%
    \fi%
  \else%
    \setlength{\unitlength}{\svgwidth}%
  \fi%
  \global\let\svgwidth\undefined%
  \global\let\svgscale\undefined%
  \makeatother%
  \begin{picture}(1,0.67219506)%
    \lineheight{1}%
    \setlength\tabcolsep{0pt}%
    \put(0,0){\includegraphics[width=\unitlength,page=1]{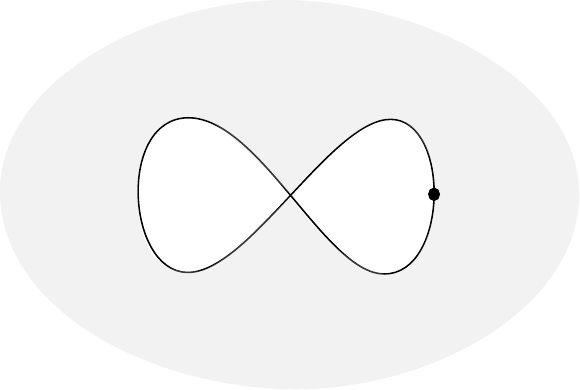}}%
    \put(0.77544506,0.31564197){\makebox(0,0)[lt]{\lineheight{1.25}\smash{\begin{tabular}[t]{l}$T$\end{tabular}}}}%
    \put(0.46711917,0.24783258){\makebox(0,0)[lt]{\lineheight{1.25}\smash{\begin{tabular}[t]{l}$q_i$\end{tabular}}}}%
    \put(0.46822479,0.40420719){\makebox(0,0)[lt]{\lineheight{1.25}\smash{\begin{tabular}[t]{l}$q_i$\end{tabular}}}}%
  \end{picture}%
\endgroup%
\caption{Bad word when $t=1$.}
\label{Fig:no_bad_words}
\end{figure}
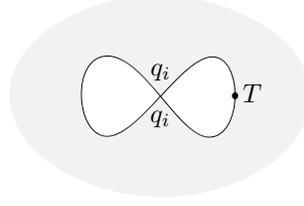

As a final point in this section, we briefly discuss an important action filtration property of the differential $d$. The proof of the lemma below follows easily from the definition.

\begin{lemma}
For all $s\in\cA(\Lambda)$ and $s_1,s_2\in\widetilde\cA(\Lambda)$, we have
\begin{equation}\label{Eq:diff_action_property}
\begin{aligned}
&l(d(s))\leq l(s),\\
&l(\pi_{\widetilde\cA}\circ d(s))<l(s),\\
&l(\{s_1,s_2\}_d)\leq l(s_1s_2),
\end{aligned}
\end{equation}
where $\pi_{\widetilde\cA}:\cA(\Lambda)=\widetilde\cA\oplus\hbar\,(\widetilde\cA\otimes\widetilde\cA^{\cyc})\to\widetilde\cA$ is the projection onto the first coordinate.
\end{lemma}
\begin{proof}
The proof follows easily from the definition, similar to \cite[Lemma 6.1]{Chekanov02}.
\end{proof}

\subsection{Algebraic and orientation signs}
\label{Sec:orientation_and_algebraic_signs}
In this section we complete the proof of Proposition \ref{Prop:dcircdvanishes}, i.e., we show that the signs cancel out. This follows from the fact that the orientations constructed in Section \ref{Section:Orientations} are coherent, see Definition \ref{Def:coherency} below.

We have seen that the summands in $d\circ d(q_i)$ correspond to boundary points of the 1-dimensional moduli space $\cM(J,\gamma_i^+)$ of disks and annuli with one positive puncture at $\gamma_i$, which consist of disk and annulus 2-buildings and nodal annuli. For every such point $\zeta\in\partial\cM(J,\gamma_i^+)$, we have an \textit{orientation sign} $\epsilon_O(\zeta)$ with respect to the orientations constructed in Section \ref{Section:Orientations}. Additionally, we define the \textit{algebraic sign} $\epsilon_A(\zeta)$ as the sign of the summand corresponding to $\zeta$ in $d\circ d(q_i)$. In Proposition \ref{Prop:strong_derivation}, we have seen that the summands in $\{d_0q_i,q_j\}_d+(-1)^{|q_i|}\{q_i,d_0q_j\}_d-(d_0\otimes 1+1\otimes d_0)\{q_i,q_j\}_d$, which correspond to summands in $\widetilde d\circ \widetilde d(q_j)-[F(H_0),q_j]$ that contain $p_i$, correspond to boundary points of the 1-dimensional moduli space $\cM(J,\gamma_i^+,\gamma_j^+)$ of disks with two positive punctures at $\gamma_i,\gamma_j$, which consist of disk 2-buildings and nodal disks with a trivial strip bubble. For every such point $\zeta\in\partial\cM(J,\gamma_i^+,\gamma_j^+)$, we have an orientation sign $\epsilon_O(\zeta)$ as before, and an algebraic sign $\epsilon_A(\zeta)$ defined as the sign of the summand corresponding to $\zeta$ in $\widetilde d\circ \widetilde d(q_j)$.

\begin{defi}\label{Def:coherency}
We say a set of orientations on the moduli space $\cM$ of $J$-holomorphic disks and annuli of index zero and one is \textit{coherent} if for all $\zeta\in\partial\cM$
\begin{align*}
\epsilon_O(\zeta)=(-1)^{|\partial\zeta|+1}\epsilon_A(\zeta),
\end{align*}
where $|\partial\zeta|$ is the number of boundary components of $\zeta$ after gluing. 
\end{defi}

We show that the orientations from Section \ref{Section:Orientations} are coherent in several steps.
\begin{lemma}\label{Lemma:signs_disks} 
Let $\zeta=(u,v)$ be a building in the boundary of the 1-dimensional moduli space of disks with one positive puncture. Then 
\begin{align*}
\epsilon_O(\zeta)=\epsilon_A(\zeta).
\end{align*}
\end{lemma}
\begin{proof}
Let $u\in\cM_1(J,\gamma_{I_1}^+)$ and $v\in\cM_1(J,\gamma_{I_2}^+)$ be index zero disks with one positive puncture that form a building $\zeta=(u,v)$ on $\lR\times\Lambda$ (with $u$ as the top level), and $u\# v$ a glued disk close to breaking. When gluing $\zeta$, we get a branch point $w_{R}$ on $u\# v$ near the glued corners as shown in Figure \ref{Fig:sign_rule_lemma}. It is not difficult to see that the orientation of the corresponding 1-dimensional moduli space near $u\# v$ is given by $A\langle w_R\rangle$ for
\begin{align*}
A=(-1)^{\sum_{j=1}^{k}|q_{i_j}|}\epsilon(u)\epsilon(v),
\end{align*}
where $\gamma_{i_1},\dots,\gamma_{i_{k}}$ are the Reeb chords at the negative punctures of $u$ that come before the negative corner at $\gamma_{I_2}$ glued to $v$ (starting from the positive puncture). Here we use $(-1)^{|q_{I_2}|+|q_{j_{1}}|+\dots+|q_{j_{l}}|}=-1$, for $\gamma_{j_{1}},\dots,\gamma_{j_{l}}$ the Reeb chords at the negative punctures of $v$, which follows from the fact that $v$ has index zero. Additionally, the outward-pointing vector at the boundary of the moduli space near $u\# v$ is given by $w_R$. Therefore, $\epsilon_O(\zeta)=A$. It is not difficult to see that this is equal to the sign of the summand in $d\circ d(q_{I_1})$ corresponding to $\zeta$. This finishes the proof.
\end{proof}
\begin{figure}
\def\svgwidth{126mm}
\begingroup%
  \makeatletter%
  \providecommand\rotatebox[2]{#2}%
  \newcommand*\fsize{\dimexpr\f@size pt\relax}%
  \newcommand*\lineheight[1]{\fontsize{\fsize}{#1\fsize}\selectfont}%
  \ifx\svgwidth\undefined%
    \setlength{\unitlength}{778.09755792bp}%
    \ifx\svgscale\undefined%
      \relax%
    \else%
      \setlength{\unitlength}{\unitlength * \real{\svgscale}}%
    \fi%
  \else%
    \setlength{\unitlength}{\svgwidth}%
  \fi%
  \global\let\svgwidth\undefined%
  \global\let\svgscale\undefined%
  \makeatother%
  \begin{picture}(1,0.20103277)%
    \lineheight{1}%
    \setlength\tabcolsep{0pt}%
    \put(0.06534756,0.14737284){\makebox(0,0)[lt]{\lineheight{1.25}\smash{\begin{tabular}[t]{l}$w_R$\end{tabular}}}}%
    \put(0.41106383,0.11988256){\makebox(0,0)[lt]{\lineheight{1.25}\smash{\begin{tabular}[t]{l}$w_R$\end{tabular}}}}%
    \put(0.67785085,0.06730411){\makebox(0,0)[lt]{\lineheight{1.25}\smash{\begin{tabular}[t]{l}$w_R$\end{tabular}}}}%
    \put(0.8745334,0.05075169){\makebox(0,0)[lt]{\lineheight{1.25}\smash{\begin{tabular}[t]{l}$w_R$\end{tabular}}}}%
    \put(0,0){\includegraphics[width=\unitlength,page=1]{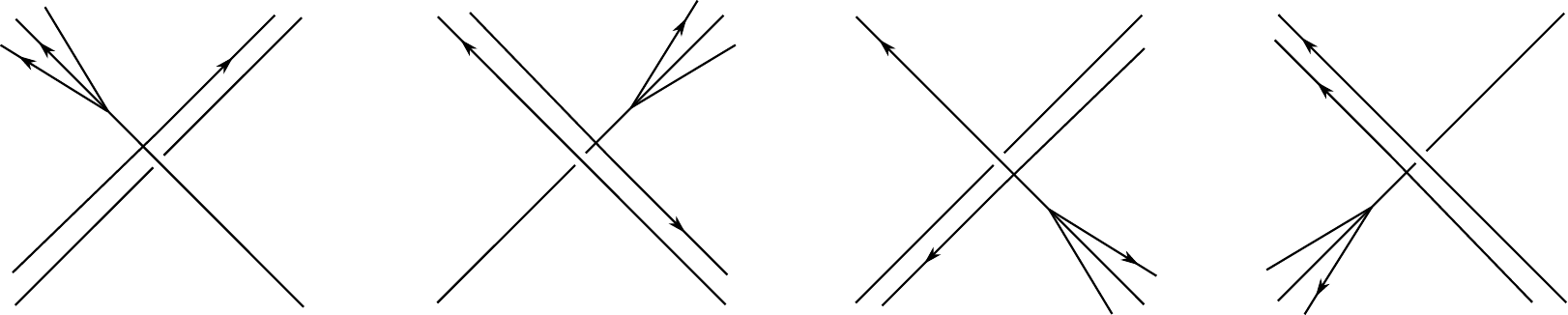}}%
  \end{picture}%
\endgroup%
\caption{Boundary branch point $w_{R}$ after gluing.}
\label{Fig:sign_rule_lemma}
\end{figure}
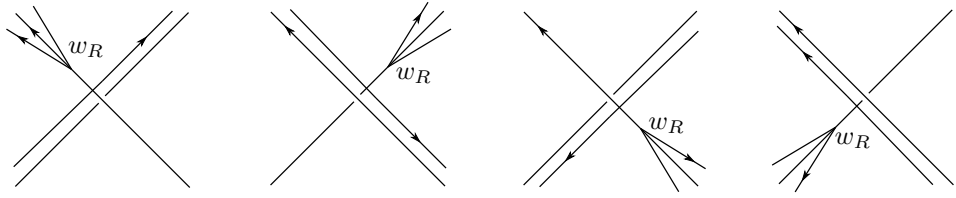

Let $F:\widetilde\cA^{r,\cyc}\to\widetilde\cA^r$ be the linear map given by
\begin{align*}
F(x)=\sum_{x=\widetilde w_1 t^{-} \widetilde w_2}(-1)^{|\widetilde w_1||\widetilde w_2|} t^{-}\widetilde w_2\widetilde w_1-\sum_{x=w_1 t^+ w_2}(-1)^{|w_1||w_2|}w_2w_1 t^+.
\end{align*}

We can see the SFT bracket $\{\cdot,\cdot\}_1$ defined in Section \ref{Sec:def_SFT_bracket} as a map $\{\cdot,\cdot\}_1:\widetilde\cA^{r_1}\otimes\widetilde\cA^{r_2,\cyc}\to\widetilde \cA^{r_1+r_2-1}$. Denote additionally by $\widetilde\delta:\widetilde\cA^{r,\cyc}\to\widetilde\cA^{r+1,\cyc}$ the map given by the first three lines in (\ref{Eg:delta_formula}) and 
\begin{align*}
\widetilde\delta(p_j)=\sum_{i\neq j}\Big(&(-1)^{|p_j|}\delta(j^+,i^-)p_jp_iq_i+\delta(j^-,i^+)q_ip_ip_j-\\
&-(-1)^{|p_j|}\delta(j^+,i^+)p_jq_ip_i-\delta(j^-,i^-)p_iq_ip_j\Big)+\\
&+(-1)^{|p_j|}\delta(j^+,j^-)p_jp_jq_j+\delta(j^-,j^+)q_jp_jp_j-(1-\delta(j))p_jq_j p_j.
\end{align*}
The following lemma is an extension of Lemma \ref{Lemma:reduced_prop} for words with a marked point.
\begin{lemma}{\cite[Proposition 2.30]{ng_linf}}\label{Lemma:signs_FH}
For all $x_1\in\widetilde\cA,x_2\in\widetilde\cA^{1,\cyc}$,
\begin{align*}
\delta\{x_1,x_2\}_1=\{x_1,\widetilde\delta x_2\}_1-(-1)^{|x_2|}\{\delta x_1,x_2\}_1+(-1)^{|x_2|}F(x_2)x_1-(-1)^{(|x_1|+1)|x_2|}x_1 F(x_2).
\end{align*}
\end{lemma}
\begin{proof}
The proof goes as the proof of Lemma \ref{Lemma:reduced_prop}. The appearance of the marked point gives us additional summands which cancel out when seen as cyclic words. These appear when gluing a puncture from an inserted trivial strip to a puncture on $x_1$ or $x_2$ right before or after passing through the base point $T$ on $x_2$. This is precisely what the summands in $(-1)^{|x_2|}F(x_2)x_1-(-1)^{(|x_1|+1)|x_2|}x_1F(x_2)$ correspond to. See for example Figure \ref{Fig:delta_deriv_deform}. Note that here $x_1$ starts at $T$ in the positive and ends at $T$ in the negative direction with respect to the orientation on $\Lambda$. 

To check the signs, consider for example the first case shown in Figure \ref{Fig:delta_deriv_deform}. Let $x_2=x_{21} q_It^-x_{22}$. We have a summand $-p_Iq_I x_1$ in $\delta x_1$ and $-(-1)^{|x_2|+|x_{21}q_I||x_{22}|}t^- x_{22}x_{21}q_Ix_1$ in $-(-1)^{|x_2|}\{\delta x_1,x_2\}$. This summands appears in $(-1)^{|x_2|}F(x_2)x_1$ with the opposite sign.

Similarly for the second case in Figure \ref{Fig:delta_deriv_deform}, where $x_1=x_1'q_I,x_2=x_{21}t^+ x_{22}$.  We have a summand $-(-1)^{|x_{21}|}x_{21}t^+ p_Iq_I x_{22}$ in $\widetilde\delta x_2$ and $-(-1)^{(|x_1|+1)|x_2|+|x_{21}||x_{22}|}x_1x_{22}x_{21}t^+$ in $\{x_1,\widetilde\delta x_2\}$. This summands appears in $-(-1)^{(|x_1|+1)|x_2|}x_1 F(x_2)$ with the opposite sign. Other cases go similarly.
\end{proof}

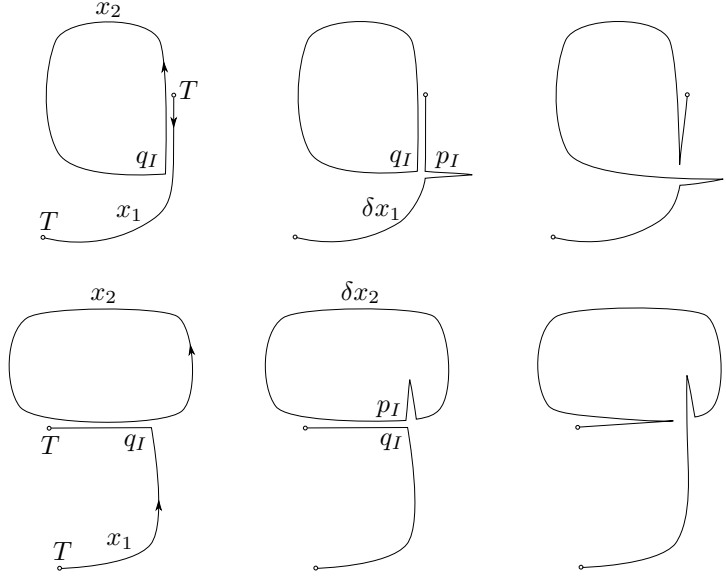
\begin{figure}
\def\svgwidth{95mm}
\begingroup%
  \makeatletter%
  \providecommand\rotatebox[2]{#2}%
  \newcommand*\fsize{\dimexpr\f@size pt\relax}%
  \newcommand*\lineheight[1]{\fontsize{\fsize}{#1\fsize}\selectfont}%
  \ifx\svgwidth\undefined%
    \setlength{\unitlength}{599.14648426bp}%
    \ifx\svgscale\undefined%
      \relax%
    \else%
      \setlength{\unitlength}{\unitlength * \real{\svgscale}}%
    \fi%
  \else%
    \setlength{\unitlength}{\svgwidth}%
  \fi%
  \global\let\svgwidth\undefined%
  \global\let\svgscale\undefined%
  \makeatother%
  \begin{picture}(1,0.79643274)%
    \lineheight{1}%
    \setlength\tabcolsep{0pt}%
    \put(0.46464809,0.37960372){\makebox(0,0)[lt]{\lineheight{1.25}\smash{\begin{tabular}[t]{l}$\delta x_2$\end{tabular}}}}%
    \put(0.12136946,0.78015963){\makebox(0,0)[lt]{\lineheight{1.25}\smash{\begin{tabular}[t]{l}$x_2$\end{tabular}}}}%
    \put(0.15042739,0.49562406){\makebox(0,0)[lt]{\lineheight{1.25}\smash{\begin{tabular}[t]{l}$x_1$\end{tabular}}}}%
    \put(0.49439603,0.49691398){\makebox(0,0)[lt]{\lineheight{1.25}\smash{\begin{tabular}[t]{l}$\delta x_1$\end{tabular}}}}%
    \put(0.13704641,0.03519139){\makebox(0,0)[lt]{\lineheight{1.25}\smash{\begin{tabular}[t]{l}$x_1$\end{tabular}}}}%
    \put(0.11426018,0.37960372){\makebox(0,0)[lt]{\lineheight{1.25}\smash{\begin{tabular}[t]{l}$x_2$\end{tabular}}}}%
    \put(0,0){\includegraphics[width=\unitlength,page=1]{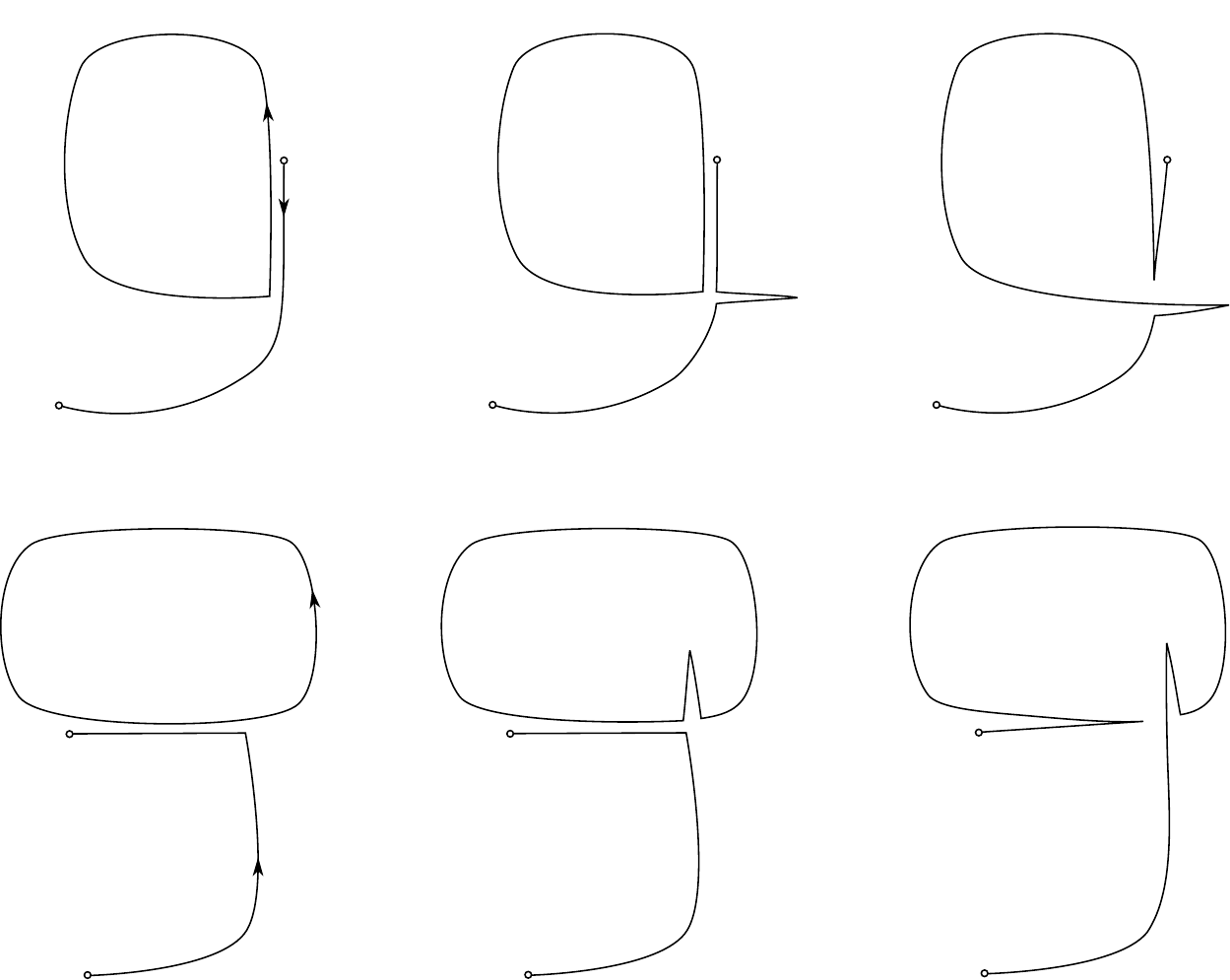}}%
    \put(0.17793623,0.57110377){\makebox(0,0)[lt]{\lineheight{1.25}\smash{\begin{tabular}[t]{l}$q_I$\end{tabular}}}}%
    \put(0.53295205,0.57204601){\makebox(0,0)[lt]{\lineheight{1.25}\smash{\begin{tabular}[t]{l}$q_I$\end{tabular}}}}%
    \put(0.59786574,0.57070484){\makebox(0,0)[lt]{\lineheight{1.25}\smash{\begin{tabular}[t]{l}$p_I$\end{tabular}}}}%
    \put(0.16094785,0.17473554){\makebox(0,0)[lt]{\lineheight{1.25}\smash{\begin{tabular}[t]{l}$q_I$\end{tabular}}}}%
    \put(0.51954003,0.17630368){\makebox(0,0)[lt]{\lineheight{1.25}\smash{\begin{tabular}[t]{l}$q_I$\end{tabular}}}}%
    \put(0.51506935,0.22217248){\makebox(0,0)[lt]{\lineheight{1.25}\smash{\begin{tabular}[t]{l}$p_I$\end{tabular}}}}%
    \put(0.23864765,0.66051669){\makebox(0,0)[lt]{\lineheight{1.25}\smash{\begin{tabular}[t]{l}$T$\end{tabular}}}}%
    \put(0.04514884,0.1634684){\makebox(0,0)[lt]{\lineheight{1.25}\smash{\begin{tabular}[t]{l}$T$\end{tabular}}}}%
    \put(0.06200858,0.01540141){\makebox(0,0)[lt]{\lineheight{1.25}\smash{\begin{tabular}[t]{l}$T$\end{tabular}}}}%
    \put(0.04301746,0.47703007){\makebox(0,0)[lt]{\lineheight{1.25}\smash{\begin{tabular}[t]{l}$T$\end{tabular}}}}%
  \end{picture}%
\endgroup%
\caption{Summands in $\delta\{x_1,x_2\}_1-\{x_1,\widetilde\delta x_2\}_1+(-1)^{|x_2|}\{\delta x_1,x_2\}_1$ that are not canceled out when seen as words with a marked starting point.}
\label{Fig:delta_deriv_deform}
\end{figure}

\begin{lemma}
\label{Lemma:signs_disksII}
Let $\zeta$ be a boundary point of the 1-dimensional moduli space of disks with two positive punctures. Then 
\begin{align*}
\epsilon_O(\zeta)=\epsilon_A(\zeta).
\end{align*}
This implies $\widetilde d\circ \widetilde d(q_J)=[F(H_0),q_J]$.
\end{lemma}
\begin{proof}
For $\zeta$ a building, the proof follows similar as in Lemma \ref{Lemma:signs_disks}. More precisely, for $\zeta$ a building that consists of two disks $u=w_1q_iw_2p_J$ ($u=w_1p_iw_2p_J,v$) and $v$, such that $v$ is glued to the negative (positive) puncture $q_i$ ($p_i$) of $u$, the orientation sign of $\zeta$ is equal to $(-1)^{|w_1|}\epsilon(u)\epsilon(v,\gamma_i^+)$ ($-(-1)^{|w_1|}\epsilon(u,\gamma_J^+)\epsilon(v,\gamma_i^-)$), which is equal to the sign of the corresponding summand in $\widetilde d\circ\widetilde d(q_J)$.

Now, let $\zeta$ be a boundary point corresponding to an index zero disk $u\in\cM_1(J,\gamma_J^+)$ together with a trivial strip bubble at a point $\tau\in(t_{I-1},t_{I})\subset\lS^1$ such that $u(\tau)=i^\pm$. Let additionally $\epsilon_\tau^1$ be $1$ if $u(\tau)=i^+$ and $-1$ if $u(\tau)=i^-$, $\epsilon_\tau^2$ be $1$ if the orientation of $u$ at $\tau$ matches the orientation of $\Lambda$ and $-1$ otherwise, and $\epsilon_\tau=\epsilon_\tau^1\epsilon_\tau^2$. Denote by $w_R$ the branch point introduced by inserting the trivial strip at $\tau$. Then the orientation of the space of strings near $\zeta$ is given by $B\langle w_R\rangle$ for
\begin{align*}
B=-(-1)^{|q_{i_1}|+\dots+|q_{i_{I-1}}|}\epsilon_\tau\epsilon(u),
\end{align*}
where $\gamma_{i_k}$ is the Reeb chord at the $k^{th}$ negative puncture $t_k$ of $u$. The outward-pointing vector is given by $-w_R$. The orientation sign of $\zeta$ is, therefore, equal to $-B$, and it is not difficult to see that this is equal to the sign of the corresponding summand in $\widetilde d\circ\widetilde d(q_J)$. 
\end{proof}

\begin{figure}
\def\svgwidth{176mm}
\begingroup%
  \makeatletter%
  \providecommand\rotatebox[2]{#2}%
  \newcommand*\fsize{\dimexpr\f@size pt\relax}%
  \newcommand*\lineheight[1]{\fontsize{\fsize}{#1\fsize}\selectfont}%
  \ifx\svgwidth\undefined%
    \setlength{\unitlength}{1563.24391341bp}%
    \ifx\svgscale\undefined%
      \relax%
    \else%
      \setlength{\unitlength}{\unitlength * \real{\svgscale}}%
    \fi%
  \else%
    \setlength{\unitlength}{\svgwidth}%
  \fi%
  \global\let\svgwidth\undefined%
  \global\let\svgscale\undefined%
  \makeatother%
  \begin{picture}(1,0.67719789)%
    \lineheight{1}%
    \setlength\tabcolsep{0pt}%
    \put(0.08628123,0.6408714){\makebox(0,0)[lt]{\lineheight{1.25}\smash{\begin{tabular}[t]{l}$\epsilon_1$\end{tabular}}}}%
    \put(0.04315607,0.57684752){\makebox(0,0)[lt]{\lineheight{1.25}\smash{\begin{tabular}[t]{l}$s_{1}$\end{tabular}}}}%
    \put(0.09830691,0.51113992){\makebox(0,0)[lt]{\lineheight{1.25}\smash{\begin{tabular}[t]{l}$s_{i-1}$\end{tabular}}}}%
    \put(0.1566202,0.45031513){\makebox(0,0)[lt]{\lineheight{1.25}\smash{\begin{tabular}[t]{l}$\epsilon_3$\end{tabular}}}}%
    \put(0.26759357,0.47537518){\makebox(0,0)[lt]{\lineheight{1.25}\smash{\begin{tabular}[t]{l}$\epsilon_2$\end{tabular}}}}%
    \put(0.24242914,0.44349601){\makebox(0,0)[lt]{\lineheight{1.25}\smash{\begin{tabular}[t]{l}$s_{k+1}'$\end{tabular}}}}%
    \put(0.21300087,0.45223038){\makebox(0,0)[lt]{\lineheight{1.25}\smash{\begin{tabular}[t]{l}$s_{l}'$\end{tabular}}}}%
    \put(0.20725913,0.38092436){\makebox(0,0)[lt]{\lineheight{1.25}\smash{\begin{tabular}[t]{l}$s_{1}'$\end{tabular}}}}%
    \put(0.30996497,0.51555973){\makebox(0,0)[lt]{\lineheight{1.25}\smash{\begin{tabular}[t]{l}$s_{k}'$\end{tabular}}}}%
    \put(0.19526233,0.57741578){\makebox(0,0)[lt]{\lineheight{1.25}\smash{\begin{tabular}[t]{l}$s_{i+1}$\end{tabular}}}}%
    \put(0.15403128,0.66570645){\makebox(0,0)[lt]{\lineheight{1.25}\smash{\begin{tabular}[t]{l}$s_{n}$\end{tabular}}}}%
    \put(0.15346312,0.49886971){\makebox(0,0)[lt]{\lineheight{1.25}\smash{\begin{tabular}[t]{l}$w_2$\end{tabular}}}}%
    \put(0.18624055,0.47872728){\makebox(0,0)[lt]{\lineheight{1.25}\smash{\begin{tabular}[t]{l}$w_1$\end{tabular}}}}%
    \put(0.58055859,0.41067154){\makebox(0,0)[lt]{\lineheight{1.25}\smash{\begin{tabular}[t]{l}$s_{1}$\end{tabular}}}}%
    \put(0.58106371,0.46213125){\makebox(0,0)[lt]{\lineheight{1.25}\smash{\begin{tabular}[t]{l}$s_{i-1}$\end{tabular}}}}%
    \put(0.51951034,0.36678864){\makebox(0,0)[lt]{\lineheight{1.25}\smash{\begin{tabular}[t]{l}$\epsilon_1$\end{tabular}}}}%
    \put(0.6013386,0.49666912){\makebox(0,0)[lt]{\lineheight{1.25}\smash{\begin{tabular}[t]{l}$\epsilon_3$\end{tabular}}}}%
    \put(0.62318465,0.53682652){\makebox(0,0)[lt]{\lineheight{1.25}\smash{\begin{tabular}[t]{l}$\epsilon_2$\end{tabular}}}}%
    \put(0.67359177,0.48669305){\makebox(0,0)[lt]{\lineheight{1.25}\smash{\begin{tabular}[t]{l}$s_{1}'$\end{tabular}}}}%
    \put(0.6959726,0.57377641){\makebox(0,0)[lt]{\lineheight{1.25}\smash{\begin{tabular}[t]{l}$s_{l}'$\end{tabular}}}}%
    \put(0.65701565,0.62176275){\makebox(0,0)[lt]{\lineheight{1.25}\smash{\begin{tabular}[t]{l}$s_{l+1}'$\end{tabular}}}}%
    \put(0.56912623,0.60244196){\makebox(0,0)[lt]{\lineheight{1.25}\smash{\begin{tabular}[t]{l}$s_{k}'$\end{tabular}}}}%
    \put(0.45781204,0.49453547){\makebox(0,0)[lt]{\lineheight{1.25}\smash{\begin{tabular}[t]{l}$s_{i+1}$\end{tabular}}}}%
    \put(0.45327079,0.40909292){\makebox(0,0)[lt]{\lineheight{1.25}\smash{\begin{tabular}[t]{l}$s_{n}$\end{tabular}}}}%
    \put(0.11964186,0.22499052){\makebox(0,0)[lt]{\lineheight{1.25}\smash{\begin{tabular}[t]{l}$s_{1}$\end{tabular}}}}%
    \put(0.07818741,0.17599469){\makebox(0,0)[lt]{\lineheight{1.25}\smash{\begin{tabular}[t]{l}$s_{i-1}$\end{tabular}}}}%
    \put(0.12510309,0.17927783){\makebox(0,0)[lt]{\lineheight{1.25}\smash{\begin{tabular}[t]{l}$s_{i+1}$\end{tabular}}}}%
    \put(0.24380849,0.1667764){\makebox(0,0)[lt]{\lineheight{1.25}\smash{\begin{tabular}[t]{l}$s_{j-1}$\end{tabular}}}}%
    \put(0.32481914,0.1542114){\makebox(0,0)[lt]{\lineheight{1.25}\smash{\begin{tabular}[t]{l}$s_{j+1}$\end{tabular}}}}%
    \put(0.24854413,0.22618986){\makebox(0,0)[lt]{\lineheight{1.25}\smash{\begin{tabular}[t]{l}$s_{n}$\end{tabular}}}}%
    \put(0.16104014,0.25376027){\makebox(0,0)[lt]{\lineheight{1.25}\smash{\begin{tabular}[t]{l}$\epsilon_1$\end{tabular}}}}%
    \put(0.23466292,0.12363077){\makebox(0,0)[lt]{\lineheight{1.25}\smash{\begin{tabular}[t]{l}$\epsilon_2$\end{tabular}}}}%
    \put(0.17662599,0.09201843){\makebox(0,0)[lt]{\lineheight{1.25}\smash{\begin{tabular}[t]{l}$s_{1}'$\end{tabular}}}}%
    \put(0.06833954,0.05716517){\makebox(0,0)[lt]{\lineheight{1.25}\smash{\begin{tabular}[t]{l}$s_{k}'$\end{tabular}}}}%
    \put(0.00855505,0.12302034){\makebox(0,0)[lt]{\lineheight{1.25}\smash{\begin{tabular}[t]{l}$s_{k+1}'$\end{tabular}}}}%
    \put(0.22790257,0.05097763){\makebox(0,0)[lt]{\lineheight{1.25}\smash{\begin{tabular}[t]{l}$s_{l}'$\end{tabular}}}}%
    \put(0.08969104,0.15305367){\makebox(0,0)[lt]{\lineheight{1.25}\smash{\begin{tabular}[t]{l}$w_1$\end{tabular}}}}%
    \put(0.26819116,0.13531126){\makebox(0,0)[lt]{\lineheight{1.25}\smash{\begin{tabular}[t]{l}$w_2$\end{tabular}}}}%
    \put(0.64796936,0.57715198){\makebox(0,0)[lt]{\lineheight{1.25}\smash{\begin{tabular}[t]{l}$w_1$\end{tabular}}}}%
    \put(0.53905466,0.53118608){\makebox(0,0)[lt]{\lineheight{1.25}\smash{\begin{tabular}[t]{l}$w_2$\end{tabular}}}}%
    \put(0.50155,0.25925106){\makebox(0,0)[lt]{\lineheight{1.25}\smash{\begin{tabular}[t]{l}$\epsilon_1$\end{tabular}}}}%
    \put(0.60213049,0.09282105){\makebox(0,0)[lt]{\lineheight{1.25}\smash{\begin{tabular}[t]{l}$\epsilon_2$\end{tabular}}}}%
    \put(0.4308302,0.19887215){\makebox(0,0)[lt]{\lineheight{1.25}\smash{\begin{tabular}[t]{l}$s_{1}$\end{tabular}}}}%
    \put(0.42426359,0.12828385){\makebox(0,0)[lt]{\lineheight{1.25}\smash{\begin{tabular}[t]{l}$s_{i-1}$\end{tabular}}}}%
    \put(0.4878446,0.15360229){\makebox(0,0)[lt]{\lineheight{1.25}\smash{\begin{tabular}[t]{l}$s_{i+1}$\end{tabular}}}}%
    \put(0.62207811,0.14362649){\makebox(0,0)[lt]{\lineheight{1.25}\smash{\begin{tabular}[t]{l}$s_{j-1}$\end{tabular}}}}%
    \put(0.7200699,0.14406836){\makebox(0,0)[lt]{\lineheight{1.25}\smash{\begin{tabular}[t]{l}$s_{j+1}$\end{tabular}}}}%
    \put(0.60667229,0.23688005){\makebox(0,0)[lt]{\lineheight{1.25}\smash{\begin{tabular}[t]{l}$s_{n}$\end{tabular}}}}%
    \put(0.5615278,0.10485952){\makebox(0,0)[lt]{\lineheight{1.25}\smash{\begin{tabular}[t]{l}$s_{1}'$\end{tabular}}}}%
    \put(0.49769429,0.08699149){\makebox(0,0)[lt]{\lineheight{1.25}\smash{\begin{tabular}[t]{l}$s_{k}'$\end{tabular}}}}%
    \put(0.506155,0.04159576){\makebox(0,0)[lt]{\lineheight{1.25}\smash{\begin{tabular}[t]{l}$s_{k+1}'$\end{tabular}}}}%
    \put(0.60490435,0.0546649){\makebox(0,0)[lt]{\lineheight{1.25}\smash{\begin{tabular}[t]{l}$s_{l}'$\end{tabular}}}}%
    \put(0.47237978,0.06838636){\makebox(0,0)[lt]{\lineheight{1.25}\smash{\begin{tabular}[t]{l}$w_1$\end{tabular}}}}%
    \put(0.62789124,0.07078575){\makebox(0,0)[lt]{\lineheight{1.25}\smash{\begin{tabular}[t]{l}$w_2$\end{tabular}}}}%
    \put(0.79379261,0.52262855){\makebox(0,0)[lt]{\lineheight{1.25}\smash{\begin{tabular}[t]{l}$\Omega<0$\end{tabular}}}}%
    \put(0.88438838,0.52224004){\makebox(0,0)[lt]{\lineheight{1.25}\smash{\begin{tabular}[t]{l}$\Omega>0$\end{tabular}}}}%
    \put(0.89957518,0.55106304){\makebox(0,0)[lt]{\lineheight{1.25}\smash{\begin{tabular}[t]{l}$\Omega^{-1}(0)$\end{tabular}}}}%
    \put(0.83981619,0.46525316){\makebox(0,0)[lt]{\lineheight{1.25}\smash{\begin{tabular}[t]{l}$w_2$\end{tabular}}}}%
    \put(0.81753519,0.07441569){\makebox(0,0)[lt]{\lineheight{1.25}\smash{\begin{tabular}[t]{l}$w_2$\end{tabular}}}}%
    \put(0.78331694,0.08280346){\makebox(0,0)[lt]{\lineheight{1.25}\smash{\begin{tabular}[t]{l}$w_1$\end{tabular}}}}%
    \put(0.82420119,0.1684704){\makebox(0,0)[lt]{\lineheight{1.25}\smash{\begin{tabular}[t]{l}$\Omega<0$\end{tabular}}}}%
    \put(0.89769524,0.11940877){\makebox(0,0)[lt]{\lineheight{1.25}\smash{\begin{tabular}[t]{l}$\Omega>0$\end{tabular}}}}%
    \put(0.90651554,0.16588519){\makebox(0,0)[lt]{\lineheight{1.25}\smash{\begin{tabular}[t]{l}$\Omega^{-1}(0)$\end{tabular}}}}%
    \put(0,0){\includegraphics[width=\unitlength,page=1]{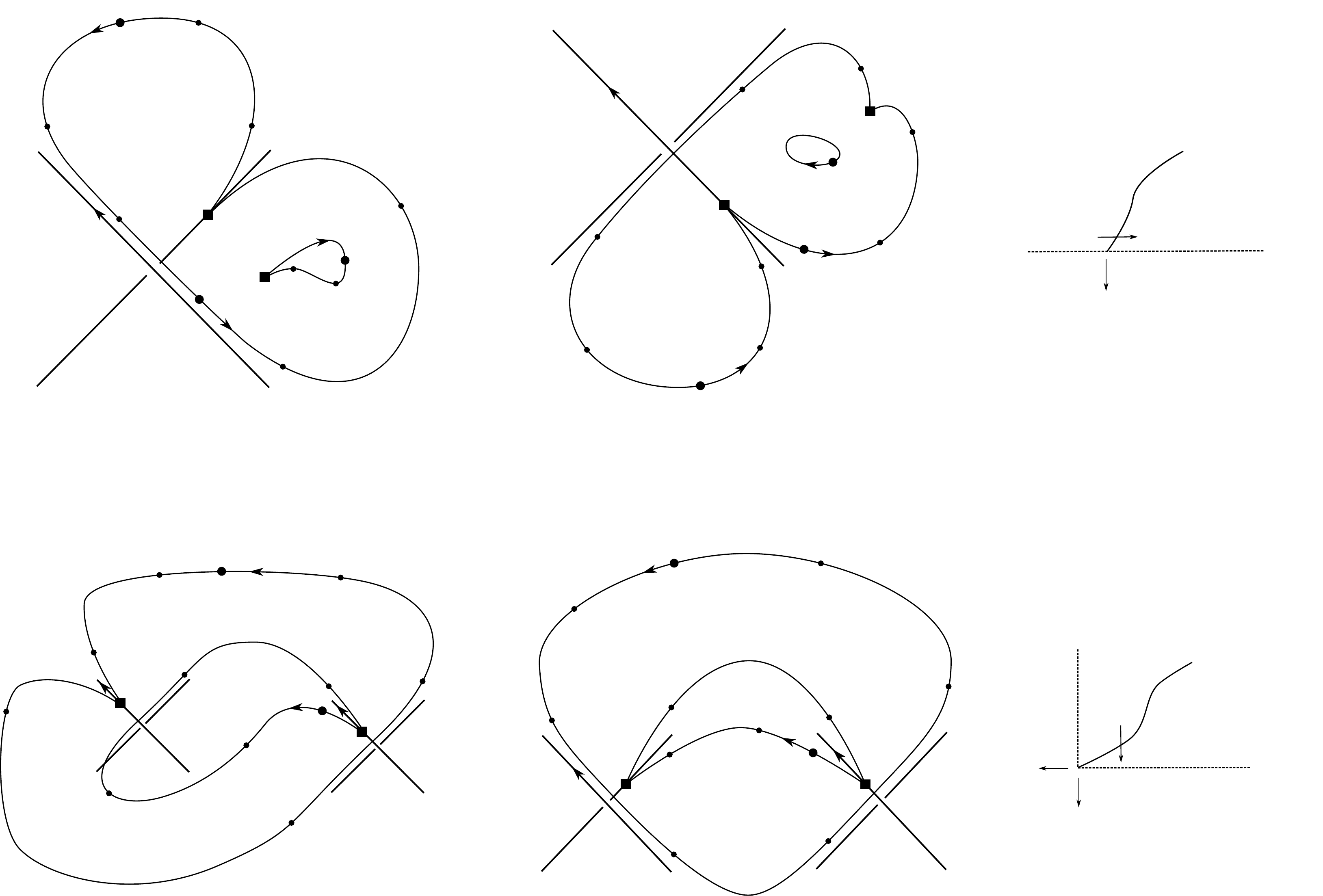}}%
  \end{picture}%
\endgroup%
\caption{Coherent orientations for annuli I.}
\label{Figure:annulus_orientation}
\end{figure}
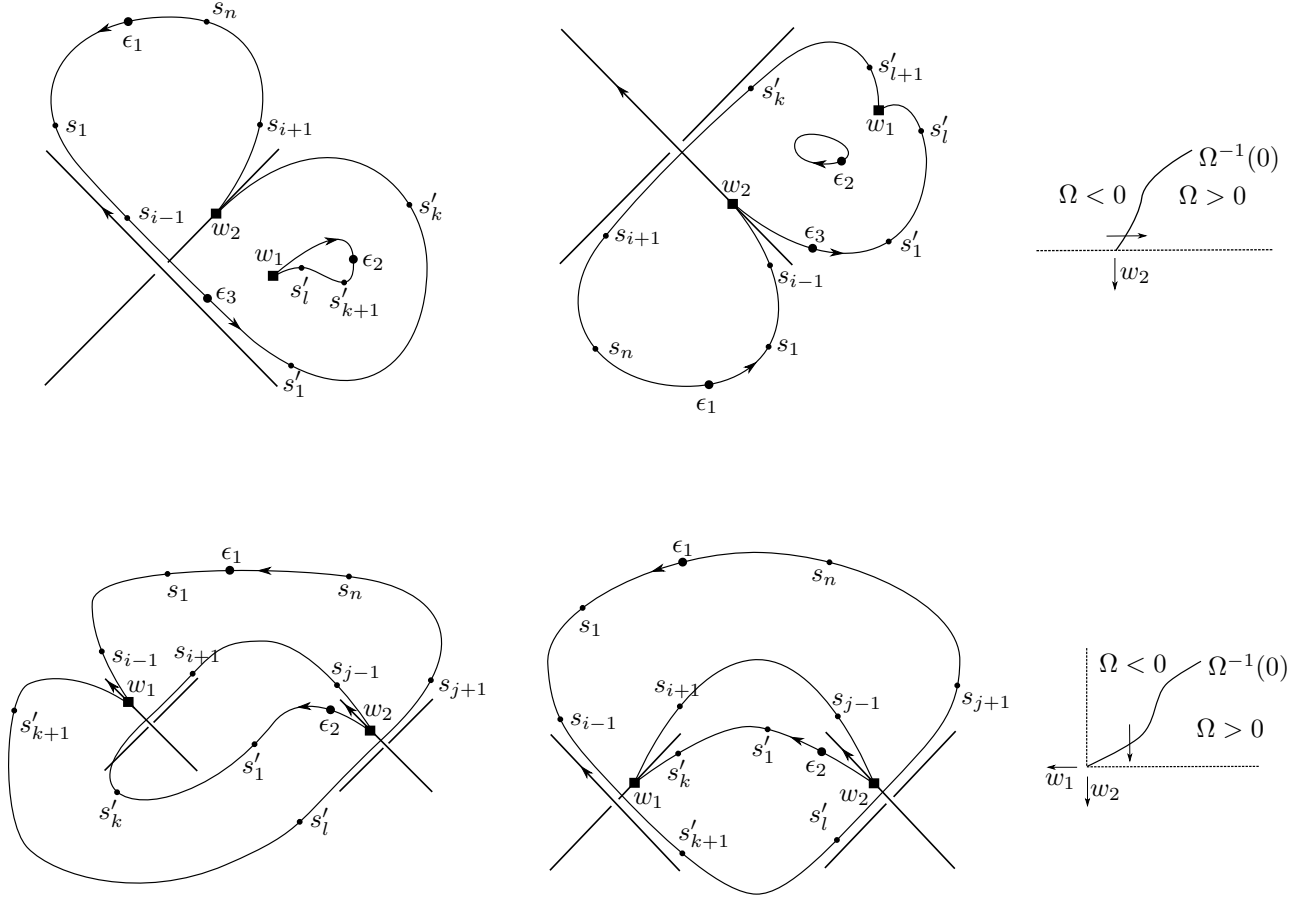

\begin{lemma}\label{Lemma:signs_annulus_buildings}
Let $\zeta=(u,v)$ be a building in the boundary of the 1-dimensional moduli space of annuli with one positive puncture. Then 
\begin{align*}
\epsilon_O(\zeta)=-\epsilon_A(\zeta).
\end{align*}
\end{lemma}
\begin{proof}
First, we consider buildings $\zeta=(u,v)$ where the top curve $u$ is a disk and the bottom curve $v$ an annulus. In particular, we consider breaking as shown in Figure \ref{Figure:annulus_orientation}, top. We look at examples where the branch points are on different boundary components and on the same boundary component.

Let first $\zeta$ be a building breaking as shown in Figure \ref{Figure:annulus_orientation}, top left. The orientation of the glued string is given by
\begin{align*}
&\epsilon_1\epsilon_2(-1)^{\sum_{a=1}^{i-1}|q_{s_a}|+\sum_{b=1}^{l}|q_{s_b'}|}\left(\prod\epsilon_\bullet^u\right)\left(\prod\epsilon_\bullet^v\right)(-1)^{|q_{s_i}|}\langle w_2,w_1\rangle=\\
&=-\epsilon_2\epsilon(u)(-1)^{\sum_{a=1}^{i-1}|q_{s_a}|+\sum_{b=k+1}^{l}|q_{s_b'}|}\left(\prod\epsilon_\bullet^v\right)\langle w_2,w_1\rangle,
\end{align*}
where $\prod\epsilon_\bullet^u,\prod\epsilon_\bullet^v$ are the products of the signs at the corners of $u,v$, $w_1$ is the branch point on $v$ and $w_2$ is the branch point coming from gluing. The orientation of $\cM^\pi_2$ near $\pi_{xy}(v)$ is given by
\begin{align*}
-\epsilon_2(-1)^{\sum_{b=k+1}^l|q_{s_b'}|}\left(\prod\epsilon_\bullet^v\right)\langle w_1\rangle.
\end{align*}
The top right figure depicts the neighborhood of the moduli space near $\zeta$. The outward-pointing vector at the boundary of $\Omega^{-1}(0)$ is given by $w_2$, and the vector pointing from the region with $\Omega<0$ to the region with $\Omega>0$ by
\begin{align*}
-\epsilon(v)\epsilon_2(-1)^{\sum_{b=k+1}^l|q_{s_b'}|}\left(\prod\epsilon_\bullet^v\right) w_1.
\end{align*}
The orientation vector on $\Omega^{-1}(0)$ near $\zeta$ is then given by
\begin{align*}
-(-1)^{\sum_{a=1}^{i-1}|q_{s_a}|}\epsilon(u)\epsilon(v)\, w_2,
\end{align*}
i.e. $A=-(-1)^{\sum_{a=1}^{i-1}|q_{s_a}|}\epsilon(u)\epsilon(v)$ times the outward-pointing vector.

Similarly for the second figure, the orientation of the glued string is given by 
\begin{align*}
-\epsilon_1\epsilon_2(-1)^{\sum_{a=1}^l|q_{s_a'}|}\left(\prod\epsilon_\bullet^u\right)\left(\prod\epsilon_\bullet^v\right)\langle w_2,w_1\rangle,
\end{align*}
and the orientation on the moduli space containing $\pi_{xy}(v)$ by
\begin{align*}
-\epsilon_2(-1)^{\sum_{a=1}^l|q_{s_a'}|}\left(\prod\epsilon_\bullet^v\right)\langle w_1\rangle.
\end{align*}
The vector pointing from $\Omega<0$ to $\Omega>0$ is given by
\begin{align*}
-\epsilon(v)\epsilon_2(-1)^{\sum_{a=1}^l|q_{s_a'}|}\left(\prod\epsilon_\bullet^v\right) w_1,
\end{align*}
and the orientation vector on $\Omega^{-1}(0)$ near $\zeta$ is then
\begin{align*}
-\epsilon(u)\epsilon(v)\, w_2,
\end{align*}
i.e. $A=-\epsilon(u)\epsilon(v)$ times the outward-pointing vector.

Next, we consider the case where $\zeta$ is a building $(u,v)$ consisting of two disks glued at two punctures at $\gamma_{s_i},\gamma_{s_j}$. In particular, we first consider breaking as shown in Figure \ref{Figure:annulus_orientation}, bottom left. The orientation of the glued string is given by
\begin{align*}
&\epsilon_1\epsilon_2(-1)^{\sum_{a=1}^{i-1}|q_{s_a}|+\sum_{b=1}^k|q_{s_b'}|+\sum_{c=i+1}^{j-1}|q_{s_c}|}\left(\prod\epsilon_\bullet^u\right)\left(\prod\epsilon_\bullet^v\right)\langle w_1,w_2\rangle=\\
&=-(-1)^{\sum_{a=1}^{j-1}|q_{s_a}|}\epsilon(u)\epsilon(v)\langle w_1,w_2\rangle.
\end{align*}
The equality is a consequence of $\sum_{b=1}^k|q_{s_b'}|\equiv |q_i|+1\pmod{2}$, which follows from the fact that the part of the string $v$ starting at $\epsilon_2$ and ending right after the positive puncture at $\gamma_{s_i}$ changes the sign an even number of times. As in Lemma \ref{Lemma:OSection_hyperbolic}, we get that $w_2$ points from the region with $\Omega<0$ to the region with $\Omega>0$. Then the orientation of $\Omega^{-1}(0)$ near $\zeta$ is given by
\begin{align*}
(-1)^{\sum_{a=1}^{j-1}|q_{s_a}|}\epsilon(u)\epsilon(v)\, w_1,
\end{align*}  
i.e. $A=(-1)^{\sum_{a=1}^{j-1}|q_{s_a}|}\epsilon(u)\epsilon(v)$ times the outward-pointing vector $w_1$. 

Similarly for the last building in Figure \ref{Figure:annulus_orientation}. The orientation of the string is given by
\begin{align*}
&\epsilon_1\epsilon_2(-1)^{\sum_{a=i+1}^{j-1}|q_{s_a}|}\left(\prod\epsilon_\bullet^u\right)\left(\prod\epsilon_\bullet^v\right)(-1)^{|q_{s_i}|}\langle w_1,w_2\rangle=\\
&=(-1)^{\sum_{a=i}^{j-1}|q_{s_a}|}\epsilon(u)\epsilon(v)\langle w_1,w_2\rangle.
\end{align*}
Similar as above, we notice that $\sum_{a=i}^{j-1}|q_{s_a}|\equiv 0\pmod{2}$. Therefore, the orientation on $\Omega^{-1}(0)$ near $\zeta$ is given by
\begin{align*}
-\epsilon(u)\epsilon(v) w_1,
\end{align*}  
i.e. $A=-\epsilon(u)\epsilon(v)$ times the outward-pointing vector. Other cases go analogously.

We notice that in each case, the orientation sign is equal to minus the algebraic sign. 
\vspace{2.1mm}

Next, we consider buildings $\zeta=(u,v)$ where a disk $v$ (bottom level) is glued to an annulus $u$ (top level).

Let $\zeta$ be such a building with breaking as shown in Figure \ref{Figure:annulus_orientationII}, left.  The orientation of the glued string is given by
\begin{align*}
\epsilon_1\epsilon_2(-1)^{\sum_{a=1}^{i-1}|q_{s_a}|+\sum_{b=l}^{m}|q_{s_b}|}\left(\prod\epsilon_\bullet^u\right)\left(\prod\epsilon_\bullet^v\right)\langle w_2,w_1\rangle,
\end{align*}
and the orientation at $\pi_{xy}(u)$ is given by
\begin{align*}
\epsilon_1\epsilon_2(-1)^{\sum_{b=l}^{m}|q_{s_b}|}\left(\prod\epsilon_\bullet^u\right)\langle w_1\rangle.
\end{align*}
Then the orientation vector on $\Omega^{-1}(0)$ near $\zeta$ is given by
\begin{align*}
-(-1)^{\sum_{a=1}^{i-1}|q_{s_a}|}\epsilon(u)\epsilon(v)\, w_2,
\end{align*}
i.e. $A=-(-1)^{\sum_{a=1}^{i-1}|q_{s_a}|}\epsilon(u)\epsilon(v)$ times the outward-pointing vector.

Similarly for the building in Figure \ref{Figure:annulus_orientationII}, right. The orientation of the glued string is given by
\begin{align*}
\epsilon_1\epsilon_2(-1)^{\sum_{a=p}^{i-1}|q_{s_a}|}\left(\prod\epsilon_\bullet^u\right)\left(\prod\epsilon_\bullet^v\right)\langle w_1,w_2\rangle,
\end{align*}
and the orientation at $\pi_{xy}(u)$ by
\begin{align*}
\epsilon_1\epsilon_2(-1)^{\sum_{b=1}^{p-1}|q_{s_b}|}\left(\prod\epsilon_\bullet^u\right)\langle w_1\rangle.
\end{align*}
Then the orientation vector on $\Omega^{-1}(0)$ near $\zeta$ is given by
\begin{align*}
(-1)^{\sum_{a=1}^{i-1}|q_{s_a}|}\epsilon(u)\epsilon(v)\, w_2,
\end{align*}
i.e. $A=(-1)^{\sum_{a=1}^{i-1}|q_{s_a}|}\epsilon(u)\epsilon(v)$ times the outward-pointing vector. The calculation goes similarly when the disk is glued to the inner boundary component.

As before, we notice that the orientation sign is equal to minus the algebraic sign, which finishes the proof.
\end{proof}

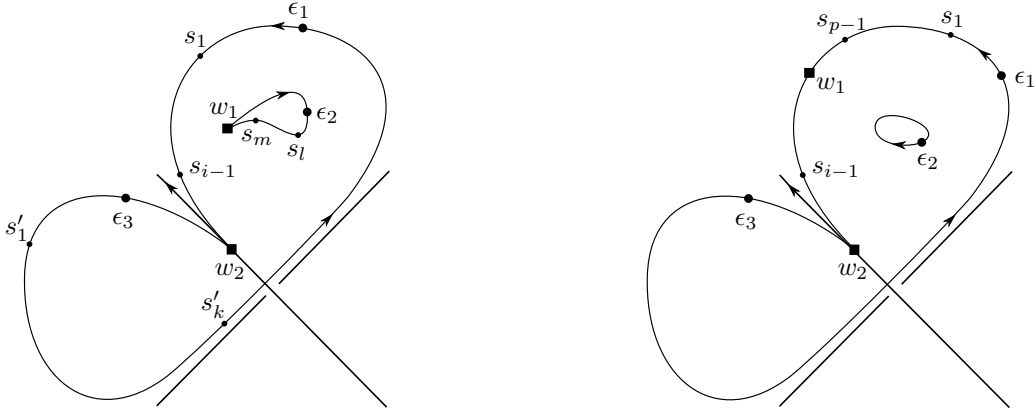
\begin{figure}
\def\svgwidth{155mm}
\begingroup%
  \makeatletter%
  \providecommand\rotatebox[2]{#2}%
  \newcommand*\fsize{\dimexpr\f@size pt\relax}%
  \newcommand*\lineheight[1]{\fontsize{\fsize}{#1\fsize}\selectfont}%
  \ifx\svgwidth\undefined%
    \setlength{\unitlength}{1004.45105512bp}%
    \ifx\svgscale\undefined%
      \relax%
    \else%
      \setlength{\unitlength}{\unitlength * \real{\svgscale}}%
    \fi%
  \else%
    \setlength{\unitlength}{\svgwidth}%
  \fi%
  \global\let\svgwidth\undefined%
  \global\let\svgscale\undefined%
  \makeatother%
  \begin{picture}(1,0.35124424)%
    \lineheight{1}%
    \setlength\tabcolsep{0pt}%
    \put(0.84535187,0.33084799){\makebox(0,0)[lt]{\lineheight{1.25}\smash{\begin{tabular}[t]{l}$s_{1}$\end{tabular}}}}%
    \put(0.73872526,0.27427057){\makebox(0,0)[lt]{\lineheight{1.25}\smash{\begin{tabular}[t]{l}$w_1$\end{tabular}}}}%
    \put(0.28555162,0.33827411){\makebox(0,0)[lt]{\lineheight{1.25}\smash{\begin{tabular}[t]{l}$\epsilon_1$\end{tabular}}}}%
    \put(0.19615694,0.31359878){\makebox(0,0)[lt]{\lineheight{1.25}\smash{\begin{tabular}[t]{l}$s_{1}$\end{tabular}}}}%
    \put(0.21079914,0.08378207){\makebox(0,0)[lt]{\lineheight{1.25}\smash{\begin{tabular}[t]{l}$s_k'$\end{tabular}}}}%
    \put(0.13478461,0.15778367){\makebox(0,0)[lt]{\lineheight{1.25}\smash{\begin{tabular}[t]{l}$\epsilon_3$\end{tabular}}}}%
    \put(0.20057613,0.20183549){\makebox(0,0)[lt]{\lineheight{1.25}\smash{\begin{tabular}[t]{l}$s_{i-1}$\end{tabular}}}}%
    \put(0.04271392,0.14935564){\makebox(0,0)[lt]{\lineheight{1.25}\smash{\begin{tabular}[t]{l}$s_1'$\end{tabular}}}}%
    \put(0.22291425,0.11464396){\makebox(0,0)[lt]{\lineheight{1.25}\smash{\begin{tabular}[t]{l}$w_2$\end{tabular}}}}%
    \put(0.30967037,0.24849214){\makebox(0,0)[lt]{\lineheight{1.25}\smash{\begin{tabular}[t]{l}$\epsilon_2$\end{tabular}}}}%
    \put(0.28429209,0.21603822){\makebox(0,0)[lt]{\lineheight{1.25}\smash{\begin{tabular}[t]{l}$s_{l}$\end{tabular}}}}%
    \put(0.24503111,0.22894034){\makebox(0,0)[lt]{\lineheight{1.25}\smash{\begin{tabular}[t]{l}$s_{m}$\end{tabular}}}}%
    \put(0.21785141,0.2522756){\makebox(0,0)[lt]{\lineheight{1.25}\smash{\begin{tabular}[t]{l}$w_1$\end{tabular}}}}%
    \put(0.90865572,0.27945512){\makebox(0,0)[lt]{\lineheight{1.25}\smash{\begin{tabular}[t]{l}$\epsilon_1$\end{tabular}}}}%
    \put(0.73870262,0.32756479){\makebox(0,0)[lt]{\lineheight{1.25}\smash{\begin{tabular}[t]{l}$s_{p-1}$\end{tabular}}}}%
    \put(0.6689274,0.15741876){\makebox(0,0)[lt]{\lineheight{1.25}\smash{\begin{tabular}[t]{l}$\epsilon_3$\end{tabular}}}}%
    \put(0.73471893,0.20147069){\makebox(0,0)[lt]{\lineheight{1.25}\smash{\begin{tabular}[t]{l}$s_{i-1}$\end{tabular}}}}%
    \put(0.75705692,0.11427905){\makebox(0,0)[lt]{\lineheight{1.25}\smash{\begin{tabular}[t]{l}$w_2$\end{tabular}}}}%
    \put(0.82491149,0.20871748){\makebox(0,0)[lt]{\lineheight{1.25}\smash{\begin{tabular}[t]{l}$\epsilon_2$\end{tabular}}}}%
    \put(-0.00071183,0.24018735){\makebox(0,0)[lt]{\lineheight{1.25}\smash{\begin{tabular}[t]{l}\textcolor{white}{.}\end{tabular}}}}%
    \put(0,0){\includegraphics[width=\unitlength,page=1]{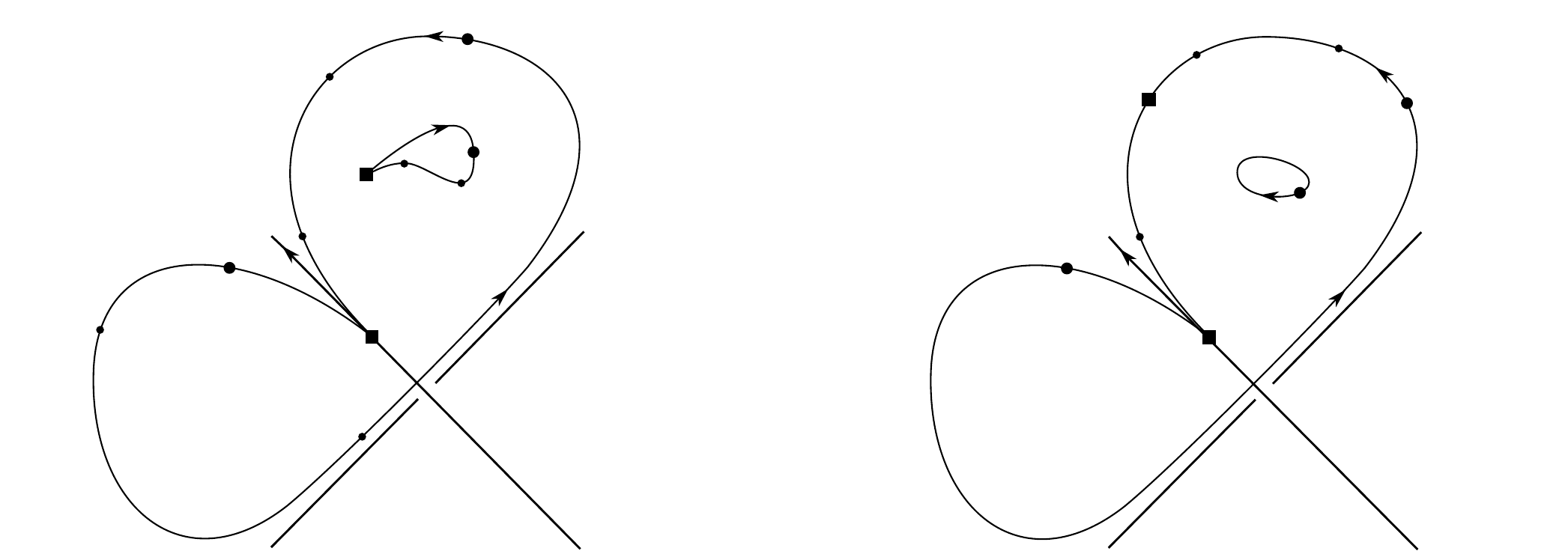}}%
  \end{picture}%
\endgroup%
\caption{Coherent orientations for annuli II.}
\label{Figure:annulus_orientationII}
\end{figure}

Let $u$ be a $J$-holomorphic disk on $\lR\times\Lambda$ and $P$ an interior intersection of $u$ with $\lR\times\Lambda$. We say the intersection at $P$ is positive and define $\epsilon(P)=1$ if 
$$\left(D(\pi_{xyz}\circ u)(X),D(\pi_{xyz}\circ u)(iX),e_\Lambda(P)\right)$$ 
is a positively oriented basis in $\lR^3$, where $e_\Lambda(P)$ is the positive unit vector tangent to $\Lambda$ at $P$ and $X\in T_{u^{-1}(P)}\Sigma$. Otherwise, we say the intersection is negative and define $\epsilon(P)=-1$. 

Similarly, let $Q$ be a boundary self-intersection of a $J$-holomorphic disk $v$. Let $\tau_1,\tau_2\in\lS^1\cong\partial\lD$ be two distinct points in the order of appearance starting from the positive puncture on $v$ such that $v|_\partial(\tau_1)=v|_\partial(\tau_2)=Q$. We say the boundary intersection at $Q$ is positive and define $\epsilon(Q)=1$ if 
$$\left(\frac{d}{dt}(v|_\partial)(\tau_1),\frac{d}{dt}(v|_\partial)(\tau_2)\right)$$ 
is a positively oriented basis in $\lR\times\Lambda$. Otherwise, we say the intersection is negative and define $\epsilon(Q)=-1$.

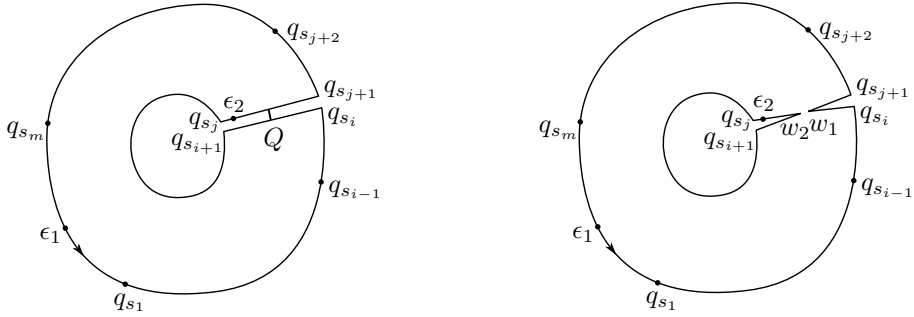
\begin{figure}
\def\svgwidth{145mm}
\begingroup%
  \makeatletter%
  \providecommand\rotatebox[2]{#2}%
  \newcommand*\fsize{\dimexpr\f@size pt\relax}%
  \newcommand*\lineheight[1]{\fontsize{\fsize}{#1\fsize}\selectfont}%
  \ifx\svgwidth\undefined%
    \setlength{\unitlength}{581.1604873bp}%
    \ifx\svgscale\undefined%
      \relax%
    \else%
      \setlength{\unitlength}{\unitlength * \real{\svgscale}}%
    \fi%
  \else%
    \setlength{\unitlength}{\svgwidth}%
  \fi%
  \global\let\svgwidth\undefined%
  \global\let\svgscale\undefined%
  \makeatother%
  \begin{picture}(1,0.28298252)%
    \lineheight{1}%
    \setlength\tabcolsep{0pt}%
    \put(0.11454789,0.06569366){\makebox(0,0)[lt]{\lineheight{1.25}\smash{\begin{tabular}[t]{l}$\epsilon_1$\end{tabular}}}}%
    \put(0.28289884,0.18620369){\makebox(0,0)[lt]{\lineheight{1.25}\smash{\begin{tabular}[t]{l}$\epsilon_2$\end{tabular}}}}%
    \put(0.18220808,0.00559225){\makebox(0,0)[lt]{\lineheight{1.25}\smash{\begin{tabular}[t]{l}$q_{s_1}$\end{tabular}}}}%
    \put(0.37851985,0.17543749){\makebox(0,0)[lt]{\lineheight{1.25}\smash{\begin{tabular}[t]{l}$q_{s_i}$\end{tabular}}}}%
    \put(0.23508758,0.15183942){\makebox(0,0)[lt]{\lineheight{1.25}\smash{\begin{tabular}[t]{l}$q_{s_{i+1}}$\end{tabular}}}}%
    \put(0.25129351,0.17145483){\makebox(0,0)[lt]{\lineheight{1.25}\smash{\begin{tabular}[t]{l}$q_{s_j}$\end{tabular}}}}%
    \put(0.37506679,0.20095247){\makebox(0,0)[lt]{\lineheight{1.25}\smash{\begin{tabular}[t]{l}$q_{s_{j+1}}$\end{tabular}}}}%
    \put(0.0842094,0.16260561){\makebox(0,0)[lt]{\lineheight{1.25}\smash{\begin{tabular}[t]{l}$q_{s_m}$\end{tabular}}}}%
    \put(0.37956035,0.10920571){\makebox(0,0)[lt]{\lineheight{1.25}\smash{\begin{tabular}[t]{l}$q_{s_{i-1}}$\end{tabular}}}}%
    \put(0.34126775,0.25266708){\makebox(0,0)[lt]{\lineheight{1.25}\smash{\begin{tabular}[t]{l}$q_{s_{j+2}}$\end{tabular}}}}%
    \put(0.60269738,0.06697259){\makebox(0,0)[lt]{\lineheight{1.25}\smash{\begin{tabular}[t]{l}$\epsilon_1$\end{tabular}}}}%
    \put(0.76846609,0.18699476){\makebox(0,0)[lt]{\lineheight{1.25}\smash{\begin{tabular}[t]{l}$\epsilon_2$\end{tabular}}}}%
    \put(0.67035753,0.00687118){\makebox(0,0)[lt]{\lineheight{1.25}\smash{\begin{tabular}[t]{l}$q_{s_1}$\end{tabular}}}}%
    \put(0.86666674,0.1767165){\makebox(0,0)[lt]{\lineheight{1.25}\smash{\begin{tabular}[t]{l}$q_{s_i}$\end{tabular}}}}%
    \put(0.72323634,0.15311842){\makebox(0,0)[lt]{\lineheight{1.25}\smash{\begin{tabular}[t]{l}$q_{s_{i+1}}$\end{tabular}}}}%
    \put(0.73944469,0.1732205){\makebox(0,0)[lt]{\lineheight{1.25}\smash{\begin{tabular}[t]{l}$q_{s_{j}}$\end{tabular}}}}%
    \put(0.86321369,0.20223148){\makebox(0,0)[lt]{\lineheight{1.25}\smash{\begin{tabular}[t]{l}$q_{s_{j+1}}$\end{tabular}}}}%
    \put(0.57236051,0.16388454){\makebox(0,0)[lt]{\lineheight{1.25}\smash{\begin{tabular}[t]{l}$q_{s_m}$\end{tabular}}}}%
    \put(0.86770714,0.11048464){\makebox(0,0)[lt]{\lineheight{1.25}\smash{\begin{tabular}[t]{l}$q_{s_{i-1}}$\end{tabular}}}}%
    \put(0.82941521,0.25394601){\makebox(0,0)[lt]{\lineheight{1.25}\smash{\begin{tabular}[t]{l}$q_{s_{j+2}}$\end{tabular}}}}%
    \put(0.79333684,0.16100777){\makebox(0,0)[lt]{\lineheight{1.25}\smash{\begin{tabular}[t]{l}$w_2$\end{tabular}}}}%
    \put(0.81974396,0.16520076){\makebox(0,0)[lt]{\lineheight{1.25}\smash{\begin{tabular}[t]{l}$w_1$\end{tabular}}}}%
    \put(-0.00123029,0.24599779){\makebox(0,0)[lt]{\lineheight{1.25}\smash{\begin{tabular}[t]{l}\textcolor{white}{.}\end{tabular}}}}%
    \put(0.31987057,0.15088737){\makebox(0,0)[lt]{\lineheight{1.25}\smash{\begin{tabular}[t]{l}$Q$\end{tabular}}}}%
    \put(0,0){\includegraphics[width=\unitlength,page=1]{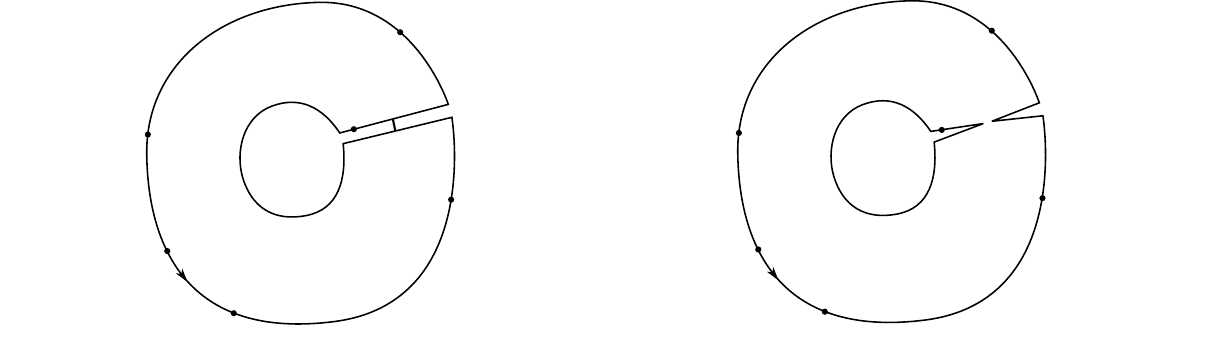}}%
  \end{picture}%
\endgroup%
\caption{Disk with a boundary self-intersection $Q$.}
\label{Fig:boundary_self-intersection}
\end{figure}

\begin{lemma}
\label{Lemma:signs_nabla}
Let $\zeta$ be a nodal annulus in the boundary of the 1-dimensional moduli space of annuli with one positive puncture. Then 
\begin{align*}
\epsilon_O(\zeta)=-\epsilon_A(\zeta).
\end{align*}
\end{lemma}
\begin{proof}
First, let $v:\lD\backslash\{t_0,\dots,t_{m}\}\to\lR^4$ be an index zero $J$-holomorphic disk and $Q\in v((t_i,t_{i+1}))\cap v((t_j,t_{j+1})),i,j\in\{0,\dots,m\},i<j$ a boundary self-intersection as shown in Figure \ref{Fig:boundary_self-intersection}, left. Note that $i\neq j$ since $v$ is of index zero. Let $\gamma_{s_k}$ be the Reeb chord at the negative puncture $t_k,k\in\{1,\dots,m\}$ and $s_k^\pm\in\Lambda$ the endpoints of $\gamma_{s_k}$ as before. We first consider four possible orderings of points $s_i^-,s_{i+1}^+,s_j^-,s_{j+1}^+$ on $\Lambda$ shown in Figure \ref{Fig:signs_nabla_orderings}.

We have a 1-parameter family $v_t,t\in(0,\varepsilon)$ of $J$-holomorphic annuli on $\lR\times\Lambda$ converging to the nodal annulus $\zeta\coloneq(v,Q)$ when $t\to 0$.
\begin{figure}
\def\svgwidth{165mm}
\begingroup%
  \makeatletter%
  \providecommand\rotatebox[2]{#2}%
  \newcommand*\fsize{\dimexpr\f@size pt\relax}%
  \newcommand*\lineheight[1]{\fontsize{\fsize}{#1\fsize}\selectfont}%
  \ifx\svgwidth\undefined%
    \setlength{\unitlength}{721.26822832bp}%
    \ifx\svgscale\undefined%
      \relax%
    \else%
      \setlength{\unitlength}{\unitlength * \real{\svgscale}}%
    \fi%
  \else%
    \setlength{\unitlength}{\svgwidth}%
  \fi%
  \global\let\svgwidth\undefined%
  \global\let\svgscale\undefined%
  \makeatother%
  \begin{picture}(1,0.14733582)%
    \lineheight{1}%
    \setlength\tabcolsep{0pt}%
    \put(0.12283958,0.01745689){\makebox(0,0)[lt]{\lineheight{1.25}\smash{\begin{tabular}[t]{l}$s_j^-$\end{tabular}}}}%
    \put(0.15552841,0.03487289){\makebox(0,0)[lt]{\lineheight{1.25}\smash{\begin{tabular}[t]{l}$s_{i+1}^+$\end{tabular}}}}%
    \put(0.19974733,0.06020052){\makebox(0,0)[lt]{\lineheight{1.25}\smash{\begin{tabular}[t]{l}$s_{j+1}^+$\end{tabular}}}}%
    \put(0.24798091,0.08538867){\makebox(0,0)[lt]{\lineheight{1.25}\smash{\begin{tabular}[t]{l}$s_i^-$\end{tabular}}}}%
    \put(0.33167693,0.01378647){\makebox(0,0)[lt]{\lineheight{1.25}\smash{\begin{tabular}[t]{l}$s_{i+1}^+$\end{tabular}}}}%
    \put(0.38100314,0.03952104){\makebox(0,0)[lt]{\lineheight{1.25}\smash{\begin{tabular}[t]{l}$s_j^-$\end{tabular}}}}%
    \put(0.42522206,0.06276903){\makebox(0,0)[lt]{\lineheight{1.25}\smash{\begin{tabular}[t]{l}$s_i^-$\end{tabular}}}}%
    \put(0.46097737,0.08171813){\makebox(0,0)[lt]{\lineheight{1.25}\smash{\begin{tabular}[t]{l}$s_{j+1}^+$\end{tabular}}}}%
    \put(0.56162953,0.01764843){\makebox(0,0)[lt]{\lineheight{1.25}\smash{\begin{tabular}[t]{l}$s_i^-$\end{tabular}}}}%
    \put(0.59431845,0.03506443){\makebox(0,0)[lt]{\lineheight{1.25}\smash{\begin{tabular}[t]{l}$s_{j+1}^+$\end{tabular}}}}%
    \put(0.63853737,0.05831242){\makebox(0,0)[lt]{\lineheight{1.25}\smash{\begin{tabular}[t]{l}$s_{i+1}^+$\end{tabular}}}}%
    \put(0.68677087,0.08558042){\makebox(0,0)[lt]{\lineheight{1.25}\smash{\begin{tabular}[t]{l}$s_j^-$\end{tabular}}}}%
    \put(0.77464656,0.01230081){\makebox(0,0)[lt]{\lineheight{1.25}\smash{\begin{tabular}[t]{l}$s_{j+1}^+$\end{tabular}}}}%
    \put(0.82397262,0.03803539){\makebox(0,0)[lt]{\lineheight{1.25}\smash{\begin{tabular}[t]{l}$s_i^-$\end{tabular}}}}%
    \put(0.86403225,0.06128337){\makebox(0,0)[lt]{\lineheight{1.25}\smash{\begin{tabular}[t]{l}$s_j^-$\end{tabular}}}}%
    \put(0.89978789,0.08023271){\makebox(0,0)[lt]{\lineheight{1.25}\smash{\begin{tabular}[t]{l}$s_{i+1}^+$\end{tabular}}}}%
    \put(0,0){\includegraphics[width=\unitlength,page=1]{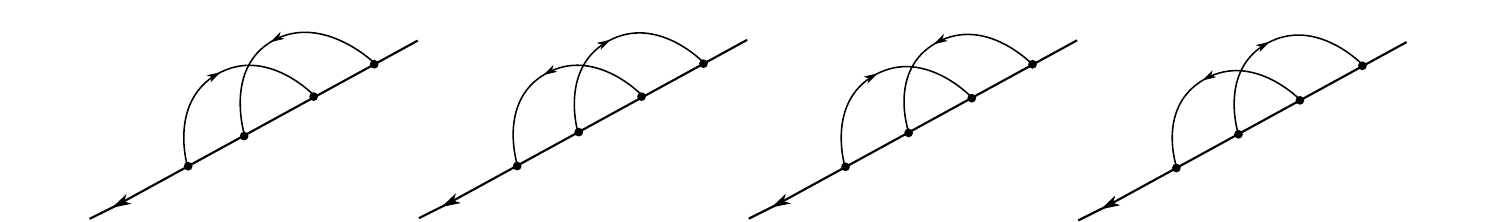}}%
    \put(0.11676946,0.13381799){\makebox(0,0)[lt]{\lineheight{1.25}\smash{\begin{tabular}[t]{l}$\epsilon(Q)=1$\end{tabular}}}}%
    \put(0.34011732,0.13132823){\makebox(0,0)[lt]{\lineheight{1.25}\smash{\begin{tabular}[t]{l}$\epsilon(Q)=-1$\end{tabular}}}}%
    \put(0.5518367,0.13134957){\makebox(0,0)[lt]{\lineheight{1.25}\smash{\begin{tabular}[t]{l}$\epsilon(Q)=-1$\end{tabular}}}}%
    \put(0.77128401,0.12979496){\makebox(0,0)[lt]{\lineheight{1.25}\smash{\begin{tabular}[t]{l}$\epsilon(Q)=1$\end{tabular}}}}%
    \put(-0.00099131,0.01898272){\makebox(0,0)[lt]{\lineheight{1.25}\smash{\begin{tabular}[t]{l}\textcolor{white}{.}\end{tabular}}}}%
  \end{picture}%
\endgroup%
\caption{Four linked orderings of points $s_i^-,s_{i+1}^+,s_j^-,s_{j+1}^+$ on $\Lambda$.}
\label{Fig:signs_nabla_orderings}
\end{figure}
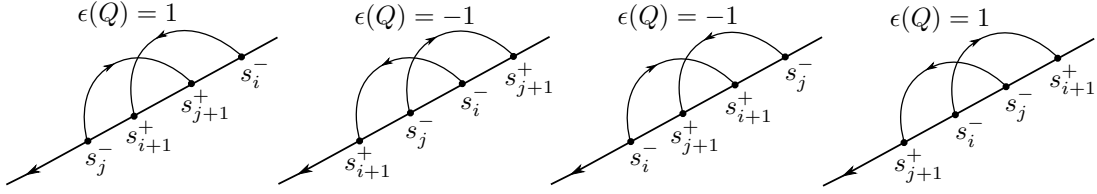
In the Lagrangian projection, annuli $v_t$ have the form as shown in Figure \ref{Fig:boundary_self-intersection}, right. Denote by $w_1$ and $w_2$ the branch points on the outer and the inner boundary component of $v_t$. Let $e_2$ be a marked point on the inner boundary component right after the negative puncture at $t_{s_j}$ and $\epsilon_2$ be the orientation sign at $e_2$. For a neighborhood of the moduli space near the boundary point $\zeta\in\Omega^{-1}(0)$, see Figure \ref{Fig:orientation_normal_nabla}, we have an outward-pointing vector $w_1$ and the corresponding orientation normal $-\epsilon_2\epsilon(Q) w_2$. To see this, let $Q_t,t\in(-\varepsilon,\varepsilon)$ be a parameterization of a neighborhood of $\pi_{xy}(Q)\in\pi_{xy}(\Lambda)$ such that $Q_0=\pi_{xy}(Q)$ and $Q_{\varepsilon}$ is closer to $s_i^-$ than $Q_{-\varepsilon}$.  Additionally, let $w_t\in\cM^\pi_2,t\in(-\varepsilon,\varepsilon)$ be a family of nodal holomorphic annuli in $\lC$ obtained by gluing $\pi_{xy}(v)$ at $Q_t$. We can see that for $0<t<\varepsilon$ we have $-\epsilon_2\epsilon(Q)\Omega(w_t)>0$. For example, with the ordering as shown in the first image in Figure \ref{Fig:signs_nabla_orderings} we have $\operatorname{sgn}(\Omega(w_t))=\operatorname{sgn}(t)$. This follows from $\pi_z\circ v\left((\pi_{xy}v)^{-1}Q_t\cap(t_i,t_{i+1})\right)>\pi_z\circ v\left((\pi_{xy}v)^{-1}Q_t\cap(t_j,t_{j+1})\right)$ for all $0<t<\varepsilon$. 

The orientation of $\cM^\pi_2$ near $\zeta$ is given by
\begin{align*}
\epsilon_2\epsilon(v)(-1)^{|q_{s_1}|+\dots+|q_{s_i}|}\langle w_1,w_2\rangle,
\end{align*}
therefore, the orientation sign of $\zeta$ is equal to
\begin{align*}
\epsilon(Q)(-1)^{|q_{s_1}|+\dots+|q_{s_i}|}\epsilon(v).
\end{align*}

Now, we notice that there can be potentially four summands in $d_f(\epsilon(v)q_{s_1}\dots q_{s_m})$ corresponding to the word $\hbar(q_{s_1}\dots q_{s_i}q_{s_{j+1}}\dots q_{s_m}\otimes q_{s_{i+1}}\dots q_{s_j})$. These come from
\begin{align*}
&d_f(q_{s_i},q_{s_j})\ni (-1)^{|q_{s_i}||q_{s_j}|}\delta(s_j^-,s_i^-)q_{s_i}\otimes q_{s_j},\\
&d_f(q_{s_{i+1}},q_{s_{j+1}})\ni \delta(s_{j+1}^+,s_{i+1}^+)q_{s_{j+1}}\otimes q_{s_{i+1}},\\
&d_f(q_{s_{i}},q_{s_{j+1}})\ni -\delta(s_{j+1}^+,s_i^-)q_{s_i}q_{s_{j+1}}\otimes 1,\\
&d_f(q_{s_{i+1}},q_{s_j})\ni -(-1)^{|q_{s_{i+1}}||q_{s_j}|}\delta(s_j^-,s_{i+1}^+)1\otimes q_{s_j}q_{s_{i+1}}.
\end{align*}
More precisely, the summand $\hbar(q_{s_1}\dots q_{s_i}q_{s_{j+1}}\dots q_{s_m}\otimes q_{s_{i+1}}\dots q_{s_j})$ in $d_f(\epsilon(v)q_{s_1}\dots q_{s_m})$ comes with the coefficient
\begin{align*}
(-1)^{|q_{s_1}|+\dots+|q_{s_i}|}\epsilon(v)&\left(\delta(s_j^-,s_i^-)+\delta(s_{j+1}^+,s_{i+1}^+)-\delta(s_{j+1}^+,s_i^-)-\delta(s_j^-,s_{i+1}^+)\right).
\end{align*}
Moreover, we notice that 
$$\delta(s_j^-,s_i^-)+\delta(s_{j+1}^+,s_{i+1}^+)-\delta(s_{j+1}^+,s_i^-)-\delta(s_j^-,s_{i+1}^+)=-\epsilon(Q).$$
The summand in $d\circ d$ corresponding to $\zeta$ therefore comes with the sign
\begin{align*}
-\epsilon(Q)(-1)^{|q_{s_1}|+\dots+|q_{s_i}|}\epsilon(v),
\end{align*}
i.e. minus the corresponding orientation sign, as desired. For all the other orderings of $s_i^-,s_{i+1}^+,s_j^-,s_{j+1}^+$ such that there is an intersection  between $v(t_i,t_{i+1})$ and $v(t_j,t_{j+1})$ (see for example Figure \ref{Fig:boundary_self-intersectionII}), we check similarly that $\epsilon_O(\zeta)=-\epsilon_A(\zeta)$. Additionally, for the rest of the orderings of $s_i^-,s_{i+1}^+,s_j^-,s_{j+1}^+$ where the intersection number vanishes, we have $\delta(s_j^-,s_i^-)+\delta(s_{j+1}^+,s_{i+1}^+)-\delta(s_{j+1}^+,s_i^-)-\delta(s_j^-,s_{i+1}^+)=0$.

\begin{figure}
\def\svgwidth{175mm}
\begingroup%
  \makeatletter%
  \providecommand\rotatebox[2]{#2}%
  \newcommand*\fsize{\dimexpr\f@size pt\relax}%
  \newcommand*\lineheight[1]{\fontsize{\fsize}{#1\fsize}\selectfont}%
  \ifx\svgwidth\undefined%
    \setlength{\unitlength}{839.02986714bp}%
    \ifx\svgscale\undefined%
      \relax%
    \else%
      \setlength{\unitlength}{\unitlength * \real{\svgscale}}%
    \fi%
  \else%
    \setlength{\unitlength}{\svgwidth}%
  \fi%
  \global\let\svgwidth\undefined%
  \global\let\svgscale\undefined%
  \makeatother%
  \begin{picture}(1,0.13105108)%
    \lineheight{1}%
    \setlength\tabcolsep{0pt}%
    \put(0.22539915,0.00248987){\makebox(0,0)[lt]{\lineheight{1.25}\smash{\begin{tabular}[t]{l}$w_1$\end{tabular}}}}%
    \put(0.11245725,0.10815957){\makebox(0,0)[lt]{\lineheight{1.25}\smash{\begin{tabular}[t]{l}$w_2$\end{tabular}}}}%
    \put(0.15925876,0.0593787){\makebox(0,0)[lt]{\lineheight{1.25}\smash{\begin{tabular}[t]{l}$w_1$\end{tabular}}}}%
    \put(0.18715346,0.02860422){\makebox(0,0)[lt]{\lineheight{1.25}\smash{\begin{tabular}[t]{l}${\scriptstyle \Omega<0}$\end{tabular}}}}%
    \put(0.21482494,0.08289738){\makebox(0,0)[lt]{\lineheight{1.25}\smash{\begin{tabular}[t]{l}${\scriptstyle \Omega>0}$\end{tabular}}}}%
    \put(0.21353065,0.05522852){\makebox(0,0)[lt]{\lineheight{1.25}\smash{\begin{tabular}[t]{l}$w_2$\end{tabular}}}}%
    \put(0.42480409,0.00216917){\makebox(0,0)[lt]{\lineheight{1.25}\smash{\begin{tabular}[t]{l}$w_1$\end{tabular}}}}%
    \put(0.3118621,0.10783882){\makebox(0,0)[lt]{\lineheight{1.25}\smash{\begin{tabular}[t]{l}$w_2$\end{tabular}}}}%
    \put(0.35866363,0.05905802){\makebox(0,0)[lt]{\lineheight{1.25}\smash{\begin{tabular}[t]{l}$w_1$\end{tabular}}}}%
    \put(0.38655836,0.02828352){\makebox(0,0)[lt]{\lineheight{1.25}\smash{\begin{tabular}[t]{l}${\scriptstyle \Omega>0}$\end{tabular}}}}%
    \put(0.41422987,0.08257668){\makebox(0,0)[lt]{\lineheight{1.25}\smash{\begin{tabular}[t]{l}${\scriptstyle \Omega<0}$\end{tabular}}}}%
    \put(0.4129356,0.05490784){\makebox(0,0)[lt]{\lineheight{1.25}\smash{\begin{tabular}[t]{l}$-w_2$\end{tabular}}}}%
    \put(0.62378691,0.00322334){\makebox(0,0)[lt]{\lineheight{1.25}\smash{\begin{tabular}[t]{l}$w_1$\end{tabular}}}}%
    \put(0.51084641,0.10889302){\makebox(0,0)[lt]{\lineheight{1.25}\smash{\begin{tabular}[t]{l}$w_2$\end{tabular}}}}%
    \put(0.55764735,0.06011217){\makebox(0,0)[lt]{\lineheight{1.25}\smash{\begin{tabular}[t]{l}$w_1$\end{tabular}}}}%
    \put(0.58554161,0.0293377){\makebox(0,0)[lt]{\lineheight{1.25}\smash{\begin{tabular}[t]{l}${\scriptstyle \Omega<0}$\end{tabular}}}}%
    \put(0.61321282,0.08363085){\makebox(0,0)[lt]{\lineheight{1.25}\smash{\begin{tabular}[t]{l}${\scriptstyle \Omega>0}$\end{tabular}}}}%
    \put(0.61191862,0.05596199){\makebox(0,0)[lt]{\lineheight{1.25}\smash{\begin{tabular}[t]{l}$w_2$\end{tabular}}}}%
    \put(0.82318909,0.00290264){\makebox(0,0)[lt]{\lineheight{1.25}\smash{\begin{tabular}[t]{l}$w_1$\end{tabular}}}}%
    \put(0.7102487,0.10857227){\makebox(0,0)[lt]{\lineheight{1.25}\smash{\begin{tabular}[t]{l}$w_2$\end{tabular}}}}%
    \put(0.75704948,0.05979145){\makebox(0,0)[lt]{\lineheight{1.25}\smash{\begin{tabular}[t]{l}$w_1$\end{tabular}}}}%
    \put(0.7849437,0.029017){\makebox(0,0)[lt]{\lineheight{1.25}\smash{\begin{tabular}[t]{l}${\scriptstyle \Omega>0}$\end{tabular}}}}%
    \put(0.812615,0.08331015){\makebox(0,0)[lt]{\lineheight{1.25}\smash{\begin{tabular}[t]{l}${\scriptstyle \Omega<0}$\end{tabular}}}}%
    \put(0,0){\includegraphics[width=\unitlength,page=1]{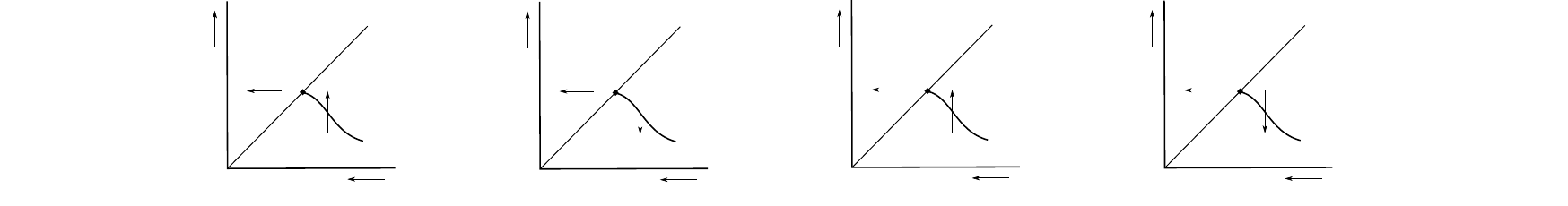}}%
    \put(0.8113207,0.05564129){\makebox(0,0)[lt]{\lineheight{1.25}\smash{\begin{tabular}[t]{l}$-w_2$\end{tabular}}}}%
    \put(-0.00085217,0.04007438){\makebox(0,0)[lt]{\lineheight{1.25}\smash{\begin{tabular}[t]{l}\textcolor{white}{.}\end{tabular}}}}%
  \end{picture}%
\endgroup%
\caption{Orientation of $\cM^\pi_2$ near $v_0$ for four ordering in Figure \ref{Fig:signs_nabla_orderings}.}
\label{Fig:orientation_normal_nabla}
\end{figure}
\vspace{1.5mm}

Next, let $u$ be an index zero $J$-holomorphic disk together with an interior intersection $P$ with $\lR\times\Lambda$ (see Figure \ref{Fig:interior_intersection}). We have a 1-parameter family $u_t,t\in(0,\varepsilon)$ of $J$-holomorphic annuli on $\lR\times\Lambda$ converging to the nodal annulus $\zeta\coloneq (u,P)$ when $t\to 0$. Annuli $u_t$ have the form as shown in Figure \ref{Fig:interior_intersection} in the Lagrangian projection. Denote the branch points on $u_t$ by $w_1,w_2$, and let $e_2$ be a marked point right before $w_1$, with the corresponding orientation sign $\epsilon_2$. 
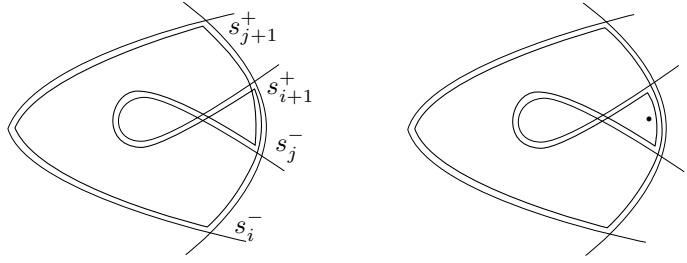
\begin{figure}
\def\svgwidth{90mm}
\begingroup%
  \makeatletter%
  \providecommand\rotatebox[2]{#2}%
  \newcommand*\fsize{\dimexpr\f@size pt\relax}%
  \newcommand*\lineheight[1]{\fontsize{\fsize}{#1\fsize}\selectfont}%
  \ifx\svgwidth\undefined%
    \setlength{\unitlength}{654.79971817bp}%
    \ifx\svgscale\undefined%
      \relax%
    \else%
      \setlength{\unitlength}{\unitlength * \real{\svgscale}}%
    \fi%
  \else%
    \setlength{\unitlength}{\svgwidth}%
  \fi%
  \global\let\svgwidth\undefined%
  \global\let\svgscale\undefined%
  \makeatother%
  \begin{picture}(1,0.37541339)%
    \lineheight{1}%
    \setlength\tabcolsep{0pt}%
    \put(0,0){\includegraphics[width=\unitlength,page=1]{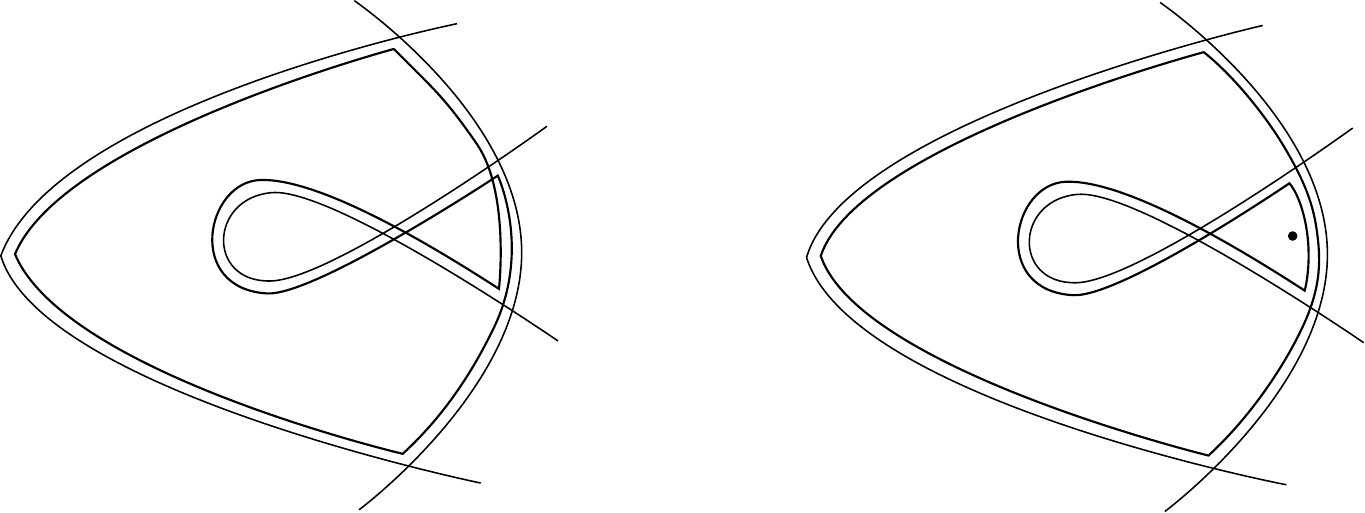}}%
    \put(0.33498158,0.03313838){\makebox(0,0)[lt]{\lineheight{1.25}\smash{\begin{tabular}[t]{l}$s_i^-$\end{tabular}}}}%
    \put(0.3954426,0.15439908){\makebox(0,0)[lt]{\lineheight{1.25}\smash{\begin{tabular}[t]{l}$s_j^-$\end{tabular}}}}%
    \put(0.38276182,0.23895435){\makebox(0,0)[lt]{\lineheight{1.25}\smash{\begin{tabular}[t]{l}$s_{i+1}^+$\end{tabular}}}}%
    \put(0.32606809,0.32811443){\makebox(0,0)[lt]{\lineheight{1.25}\smash{\begin{tabular}[t]{l}$s_{j+1}^+$\end{tabular}}}}%
  \end{picture}%
\endgroup%
\caption{Disk with a boundary self-intersection II.}
\label{Fig:boundary_self-intersectionII}
\end{figure}

\begin{figure}
\def\svgwidth{115mm}
\begingroup%
  \makeatletter%
  \providecommand\rotatebox[2]{#2}%
  \newcommand*\fsize{\dimexpr\f@size pt\relax}%
  \newcommand*\lineheight[1]{\fontsize{\fsize}{#1\fsize}\selectfont}%
  \ifx\svgwidth\undefined%
    \setlength{\unitlength}{540.06623269bp}%
    \ifx\svgscale\undefined%
      \relax%
    \else%
      \setlength{\unitlength}{\unitlength * \real{\svgscale}}%
    \fi%
  \else%
    \setlength{\unitlength}{\svgwidth}%
  \fi%
  \global\let\svgwidth\undefined%
  \global\let\svgscale\undefined%
  \makeatother%
  \begin{picture}(1,0.24751582)%
    \lineheight{1}%
    \setlength\tabcolsep{0pt}%
    \put(0.18166152,0.10685848){\makebox(0,0)[lt]{\lineheight{1.25}\smash{\begin{tabular}[t]{l}$P$\end{tabular}}}}%
    \put(0.10905143,0.22946252){\makebox(0,0)[lt]{\lineheight{1.25}\smash{\begin{tabular}[t]{l}$\epsilon(P)=1$\end{tabular}}}}%
    \put(0.74630975,0.11240887){\makebox(0,0)[lt]{\lineheight{1.25}\smash{\begin{tabular}[t]{l}$P$\end{tabular}}}}%
    \put(0.670922,0.22945791){\makebox(0,0)[lt]{\lineheight{1.25}\smash{\begin{tabular}[t]{l}$\epsilon(P)=-1$\end{tabular}}}}%
    \put(0.19073692,0.04494198){\makebox(0,0)[lt]{\lineheight{1.25}\smash{\begin{tabular}[t]{l}$\epsilon_2$\end{tabular}}}}%
    \put(0.12893912,0.08521488){\makebox(0,0)[lt]{\lineheight{1.25}\smash{\begin{tabular}[t]{l}$w_1$\end{tabular}}}}%
    \put(0.23517572,0.09468794){\makebox(0,0)[lt]{\lineheight{1.25}\smash{\begin{tabular}[t]{l}$w_2$\end{tabular}}}}%
    \put(0.76247728,0.0485298){\makebox(0,0)[lt]{\lineheight{1.25}\smash{\begin{tabular}[t]{l}$\epsilon_2$\end{tabular}}}}%
    \put(0.6995137,0.08719059){\makebox(0,0)[lt]{\lineheight{1.25}\smash{\begin{tabular}[t]{l}$w_1$\end{tabular}}}}%
    \put(0.80297323,0.09388641){\makebox(0,0)[lt]{\lineheight{1.25}\smash{\begin{tabular}[t]{l}$w_2$\end{tabular}}}}%
    \put(0,0){\includegraphics[width=\unitlength,page=1]{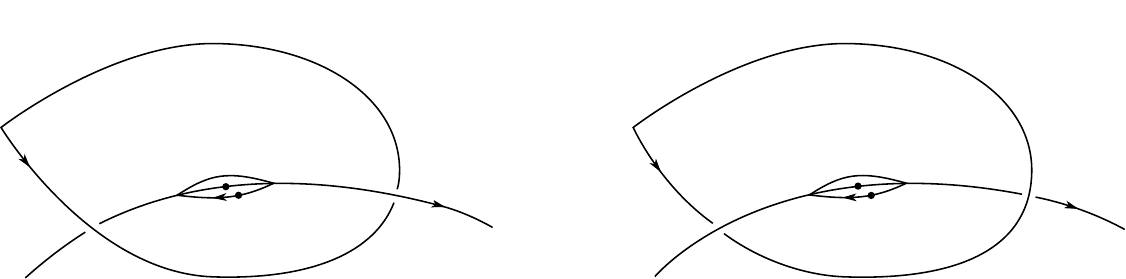}}%
  \end{picture}%
\endgroup%
\caption{Disk with an interior intersection $P$ with $\lR\times\Lambda$.}
\label{Fig:interior_intersection}
\end{figure}
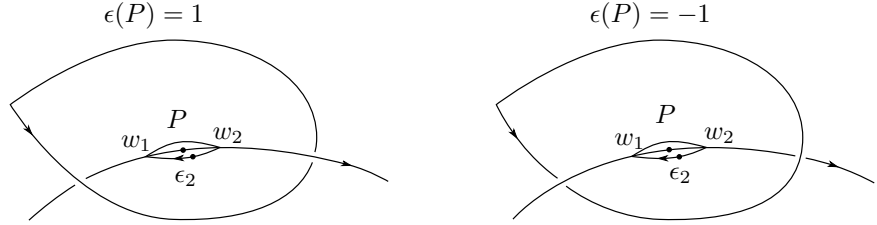

\begin{figure}
\def\svgwidth{143mm}
\begingroup%
  \makeatletter%
  \providecommand\rotatebox[2]{#2}%
  \newcommand*\fsize{\dimexpr\f@size pt\relax}%
  \newcommand*\lineheight[1]{\fontsize{\fsize}{#1\fsize}\selectfont}%
  \ifx\svgwidth\undefined%
    \setlength{\unitlength}{576.13495184bp}%
    \ifx\svgscale\undefined%
      \relax%
    \else%
      \setlength{\unitlength}{\unitlength * \real{\svgscale}}%
    \fi%
  \else%
    \setlength{\unitlength}{\svgwidth}%
  \fi%
  \global\let\svgwidth\undefined%
  \global\let\svgscale\undefined%
  \makeatother%
  \begin{picture}(1,0.19038368)%
    \lineheight{1}%
    \setlength\tabcolsep{0pt}%
    \put(0.33937365,0.00315897){\makebox(0,0)[lt]{\lineheight{1.25}\smash{\begin{tabular}[t]{l}$w_1$\end{tabular}}}}%
    \put(0.17489548,0.1570466){\makebox(0,0)[lt]{\lineheight{1.25}\smash{\begin{tabular}[t]{l}$w_2$\end{tabular}}}}%
    \put(0.22482796,0.08600663){\makebox(0,0)[lt]{\lineheight{1.25}\smash{\begin{tabular}[t]{l}$-w_1$\end{tabular}}}}%
    \put(0.28367615,0.04118951){\makebox(0,0)[lt]{\lineheight{1.25}\smash{\begin{tabular}[t]{l}${\scriptstyle \Omega>0}$\end{tabular}}}}%
    \put(0.32397435,0.12025706){\makebox(0,0)[lt]{\lineheight{1.25}\smash{\begin{tabular}[t]{l}${\scriptstyle \Omega<0}$\end{tabular}}}}%
    \put(0.32208947,0.07996268){\makebox(0,0)[lt]{\lineheight{1.25}\smash{\begin{tabular}[t]{l}$w_2$\end{tabular}}}}%
    \put(0.74250911,0.00422714){\makebox(0,0)[lt]{\lineheight{1.25}\smash{\begin{tabular}[t]{l}$w_1$\end{tabular}}}}%
    \put(0.57803094,0.15811473){\makebox(0,0)[lt]{\lineheight{1.25}\smash{\begin{tabular}[t]{l}$w_2$\end{tabular}}}}%
    \put(0.62796336,0.08707479){\makebox(0,0)[lt]{\lineheight{1.25}\smash{\begin{tabular}[t]{l}$-w_1$\end{tabular}}}}%
    \put(0.68681132,0.04225768){\makebox(0,0)[lt]{\lineheight{1.25}\smash{\begin{tabular}[t]{l}${\scriptstyle \Omega<0}$\end{tabular}}}}%
    \put(0.72710975,0.12132522){\makebox(0,0)[lt]{\lineheight{1.25}\smash{\begin{tabular}[t]{l}${\scriptstyle \Omega>0}$\end{tabular}}}}%
    \put(0,0){\includegraphics[width=\unitlength,page=1]{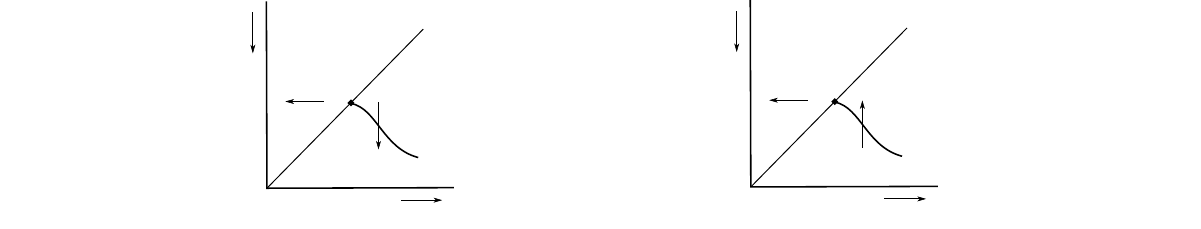}}%
    \put(0.72522492,0.08103085){\makebox(0,0)[lt]{\lineheight{1.25}\smash{\begin{tabular}[t]{l}$-w_2$\end{tabular}}}}%
    \put(-0.00169231,0.11725054){\makebox(0,0)[lt]{\lineheight{1.25}\smash{\begin{tabular}[t]{l}{\textcolor{white} t}\end{tabular}}}}%
  \end{picture}%
\endgroup%
\caption{Orientation of $\cM^\pi_2$ near $u_0$.}
\label{Fig:orientation_normal_nablaII}
\end{figure}
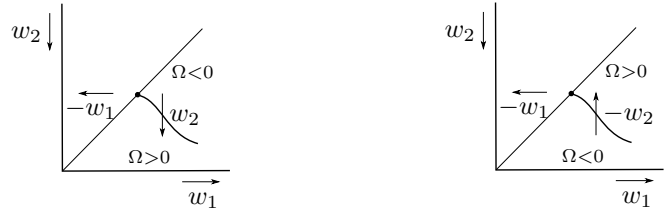

A neighborhood of the boundary point $\zeta\in\Omega^{-1}(0)$, as shown in Figure \ref{Fig:orientation_normal_nablaII}, with the outward-pointing vector $-w_1$, has the corresponding orientation normal $-\epsilon_2\epsilon(P)w_2$ using Lemma \ref{Lemma:OSection_hyperbolic}. 
The orientation of $\cM^\pi_2$ near $\zeta$ is given by
\begin{align*}
\epsilon_2\epsilon(u)\langle w_1,w_2\rangle,
\end{align*}
therefore, the orientation sign of $\zeta$ is equal to
\begin{align*}
-\epsilon(P)\epsilon(u).
\end{align*}
The sign of the corresponding summand in $d_f(\epsilon(u)q_{s_1}\dots q_{s_n})$ is equal to
$$\epsilon(P)\epsilon(u)$$
by construction, i.e. equal to minus the orientation sign, as desired.
\end{proof}

Since the constructed orientations descend to cyclic words as seen in Section \ref{Section:Orientations}, summands in $d\circ d$ cancel out when seen as elements in $\widetilde\cA\oplus\hbar(\widetilde\cA\otimes\widetilde\cA^{\cyc})$. We can now complete the proof of Proposition \ref{Prop:dcircdvanishes}.
\begin{cor}
For $\Lambda$ a Legendrian knot and $d:\cA(\Lambda)\to\cA(\Lambda)$ the second-order derivation defined in Section \ref{Section:IIinvar_def}, we have $d\circ d=0$ over $\lZ$ $(\lQ)$ coefficients.
\end{cor}

\section{Invariance}\label{Sec:Invariance}
The main goal of this section is to show invariance of the second-order dg algebra defined in Section \ref{Section:IIinvar_def} under Legendrian knot isotopy. More precisely, we define the notion of stable tame equivalence for second-order dg algebras (Section \ref{Sec:stabilizationI}), which is an analogue of stable tame equivalence of dg algebras defined in \cite{Chekanov02}, and prove the following theorem.

\begin{thm}\label{Theorem:Invariance}
Let $\Lambda_0$ and $\Lambda_1$ be front resolutions of two Legendrian isotopic knots, then the second-order dg algebras $(\cA(\Lambda_0),d_{\Lambda_0},\{\cdot,\cdot\}_{d_{\Lambda_0}}),(\cA(\Lambda_1),d_{\Lambda_1},\{\cdot,\cdot\}_{d_{\Lambda_1}})$ associated to $\Lambda_0,\Lambda_1$ are stable tame equivalent. In particular, their homology groups are isomorphic
\begin{align*}
H_*(\cA(\Lambda_0),d_{\Lambda_0})\cong H_*(\cA(\Lambda_1),d_{\Lambda_1}).
\end{align*}
\end{thm}

The proof follows methods similar to \cite{Chekanov02} and is done by splitting a (generic) Legendrian knot isotopy into steps consisting of Reidemeister II and III moves in the Lagrangian projection (Figure \ref{Figure:ReidII} and Figure \ref{Figure:ReidIII}), crossings of the base point over a Reeb chord endpoint (Figure \ref{Figure:change_basept} and Figure \ref{Figure:change_basept_IIcase}), and passing through a Legendrian knot with a degenerate annulus of index $-1$. More precisely, we say a Legendrian knot $\Lambda$ is degenerate of type IV if the obstruction section $\Omega:\overline{\cM}^\pi_{2,1}\sqcup\cM^\pi_{2,0} \to\lR\cup\{\pm\infty\}$ defined in Section \ref{Sec:obstruction_section} maps some point in $\cM^\pi_{2,0}$ to zero, or equivalently, some boundary point in $\partial\overline{\cM}^\pi_{2,1}$ to zero. When passing through a degenerate knot of type IV, the count of annuli can change, while the count of disks remains the same. We say an isotopy that passes through one degenerate knot of type IV and no Reidemeister moves is an \textit{isotopy of type IV}. The reason we take front resolutions is that we are able to prove equivalence only under a special class of Reidemeister II moves which we call admissible, see Definition \ref{Def:admissible_ReidII}. There always exists a Legendrian knot isotopy between front resolutions of two Legendrian isotopic knots that is composed out of Reidemeister III moves, admissible Reidemeister II moves, and isotopies of type IV, see Remark \ref{Remark:typeII_extra_assumption}.

In Section \ref{Sec:invar_base_pt} we show invariance up to tame second-order dga isomorphism under any change of the base point, in Section \ref{Section:Invar_IV_degeneration} under type IV isotopy and in Section \ref{Section:Invar_III_degeneration} under Reidemeister III move. In Section \ref{Section:Invar_II_degeneration} we show invariance under Reidemeister II move up to stable tame equivalence. These four steps imply that for any two Legendrian isotopic knots, the second-order dg algebras associated to their front resolutions are stable tame equivalent.

\begin{figure}
\def\svgwidth{100mm}
\begingroup%
  \makeatletter%
  \providecommand\rotatebox[2]{#2}%
  \newcommand*\fsize{\dimexpr\f@size pt\relax}%
  \newcommand*\lineheight[1]{\fontsize{\fsize}{#1\fsize}\selectfont}%
  \ifx\svgwidth\undefined%
    \setlength{\unitlength}{578.76348251bp}%
    \ifx\svgscale\undefined%
      \relax%
    \else%
      \setlength{\unitlength}{\unitlength * \real{\svgscale}}%
    \fi%
  \else%
    \setlength{\unitlength}{\svgwidth}%
  \fi%
  \global\let\svgwidth\undefined%
  \global\let\svgscale\undefined%
  \makeatother%
  \begin{picture}(1,0.2076535)%
    \lineheight{1}%
    \setlength\tabcolsep{0pt}%
    \put(0,0){\includegraphics[width=\unitlength,page=1]{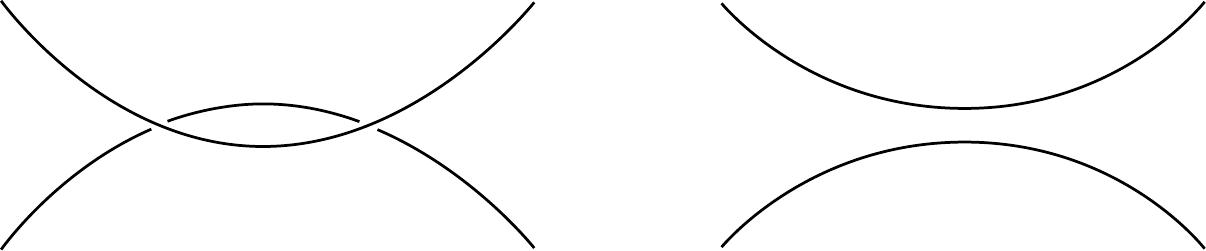}}%
    \put(0.49831363,0.09743584){\makebox(0,0)[lt]{\lineheight{1.25}\smash{\begin{tabular}[t]{l}$\longrightarrow$\end{tabular}}}}%
    \put(0.07492381,0.09612673){\makebox(0,0)[lt]{\lineheight{1.25}\smash{\begin{tabular}[t]{l}$a$\end{tabular}}}}%
    \put(0.34450171,0.09359504){\makebox(0,0)[lt]{\lineheight{1.25}\smash{\begin{tabular}[t]{l}$b$\end{tabular}}}}%
  \end{picture}%
\endgroup%
\caption{Reidemeister II move.}
\label{Figure:ReidII}
\end{figure}
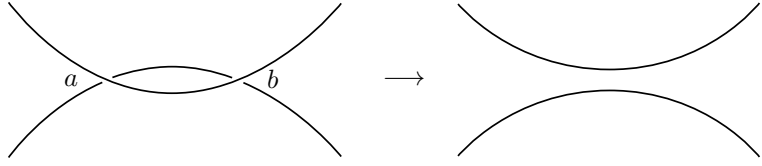

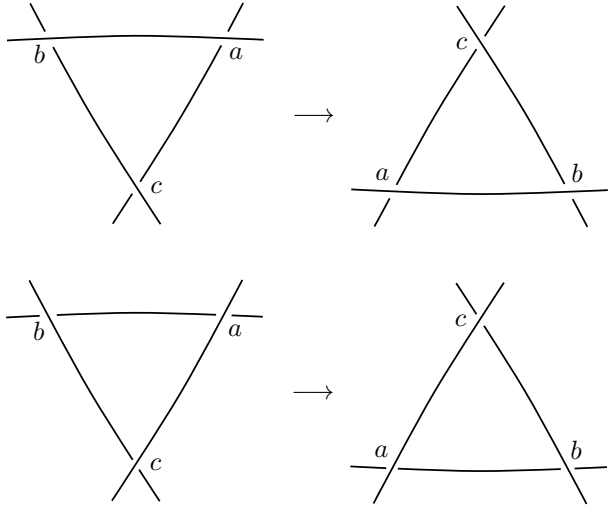
\begin{figure}
\def\svgwidth{80mm}
\begingroup%
  \makeatletter%
  \providecommand\rotatebox[2]{#2}%
  \newcommand*\fsize{\dimexpr\f@size pt\relax}%
  \newcommand*\lineheight[1]{\fontsize{\fsize}{#1\fsize}\selectfont}%
  \ifx\svgwidth\undefined%
    \setlength{\unitlength}{447.2562528bp}%
    \ifx\svgscale\undefined%
      \relax%
    \else%
      \setlength{\unitlength}{\unitlength * \real{\svgscale}}%
    \fi%
  \else%
    \setlength{\unitlength}{\svgwidth}%
  \fi%
  \global\let\svgwidth\undefined%
  \global\let\svgscale\undefined%
  \makeatother%
  \begin{picture}(1,0.83461124)%
    \lineheight{1}%
    \setlength\tabcolsep{0pt}%
    \put(0.47581413,0.63589709){\makebox(0,0)[lt]{\lineheight{1.25}\smash{\begin{tabular}[t]{l}$\longrightarrow$\end{tabular}}}}%
    \put(0.3705494,0.74062745){\makebox(0,0)[lt]{\lineheight{1.25}\smash{\begin{tabular}[t]{l}$a$\end{tabular}}}}%
    \put(0.04619055,0.73415016){\makebox(0,0)[lt]{\lineheight{1.25}\smash{\begin{tabular}[t]{l}$b$\end{tabular}}}}%
    \put(0.24010364,0.51859586){\makebox(0,0)[lt]{\lineheight{1.25}\smash{\begin{tabular}[t]{l}$c$\end{tabular}}}}%
    \put(0.61369905,0.53973643){\makebox(0,0)[lt]{\lineheight{1.25}\smash{\begin{tabular}[t]{l}$a$\end{tabular}}}}%
    \put(0.93957713,0.53553502){\makebox(0,0)[lt]{\lineheight{1.25}\smash{\begin{tabular}[t]{l}$b$\end{tabular}}}}%
    \put(0.74727597,0.75806922){\makebox(0,0)[lt]{\lineheight{1.25}\smash{\begin{tabular}[t]{l}$c$\end{tabular}}}}%
    \put(0,0){\includegraphics[width=\unitlength,page=1]{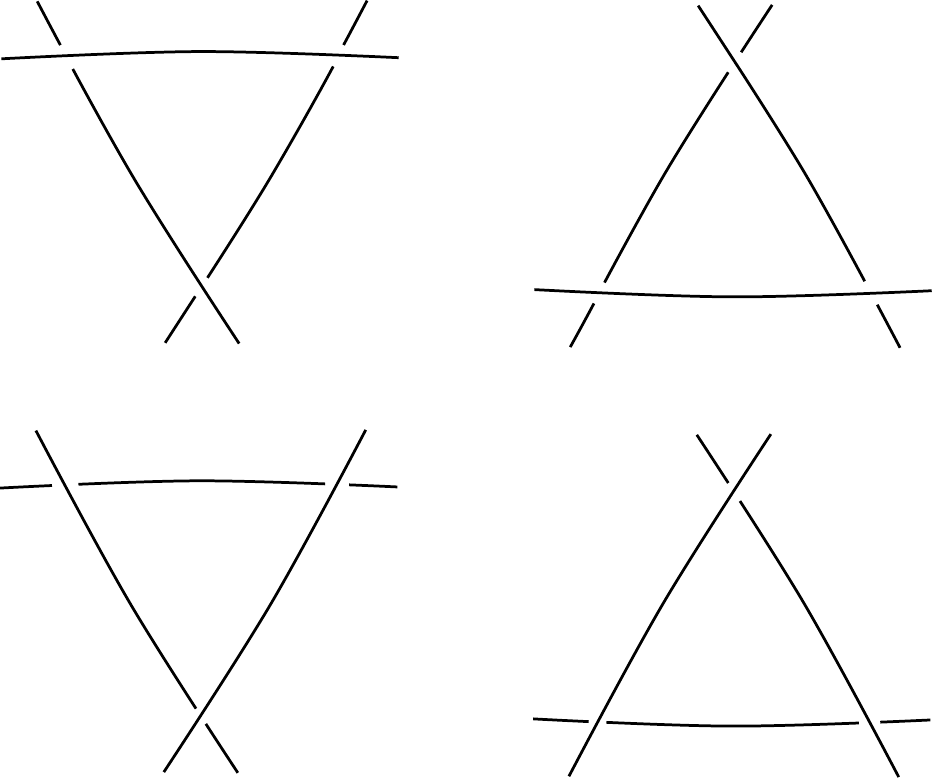}}%
    \put(0.47439774,0.1752241){\makebox(0,0)[lt]{\lineheight{1.25}\smash{\begin{tabular}[t]{l}$\longrightarrow$\end{tabular}}}}%
    \put(0.36913301,0.27995442){\makebox(0,0)[lt]{\lineheight{1.25}\smash{\begin{tabular}[t]{l}$a$\end{tabular}}}}%
    \put(0.04477427,0.27347731){\makebox(0,0)[lt]{\lineheight{1.25}\smash{\begin{tabular}[t]{l}$b$\end{tabular}}}}%
    \put(0.23868726,0.05792392){\makebox(0,0)[lt]{\lineheight{1.25}\smash{\begin{tabular}[t]{l}$c$\end{tabular}}}}%
    \put(0.61228266,0.079064){\makebox(0,0)[lt]{\lineheight{1.25}\smash{\begin{tabular}[t]{l}$a$\end{tabular}}}}%
    \put(0.93816016,0.07486271){\makebox(0,0)[lt]{\lineheight{1.25}\smash{\begin{tabular}[t]{l}$b$\end{tabular}}}}%
    \put(0.745859,0.29739634){\makebox(0,0)[lt]{\lineheight{1.25}\smash{\begin{tabular}[t]{l}$c$\end{tabular}}}}%
  \end{picture}%
\endgroup%
\caption{Reidemeister III move, first and second case, respectively.}
\label{Figure:ReidIII}
\end{figure}

\subsection{Stable tame equivalence}\label{Sec:stabilizationI}
We define the stabilization of a second-order dga $(\cA,d,\{\cdot,\cdot\}_d)$ in degree $i\in\lZ$ and the notion of stable tame equivalence between second-order dg algebras. 

Let $(\cA,d,\{\cdot,\cdot\}_d)$ be a second-order dga. Let $q_1,\dots,q_n,t^+,t^-$ be the generators of $\cA$ as before. Stabilization of $(\cA,d,\{\cdot,\cdot\}_d)$ in degree $i\in\lZ$ is a second-order dga $(\cA^s,d^s,\{\cdot,\cdot\}_{d^s})$ defined as follows. Graded algebra $\widetilde\cA^s$ is the tensor algebra generated by $q_1,\dots,q_n,q_a,q_b,t^+,t^-$ with the grading induced by the grading on $\cA$ and $|q_a|=i,|q_b|=i-1$, and relations $t^+t^-=t^-t^+=1$. We take $\cA^s=\widetilde\cA^s\oplus\hbar\, (\widetilde\cA^s\otimes\widetilde\cA^{s,\cyc})$ as before. We can see $\cA$ as a vector subspace of $\cA^s$. Denote by $C$ the vector subspace of $\cA^s$ generated by words that contain at least one $q_a$ or $q_b$. Then we have $\cA^s\cong\cA\oplus C$.

We define a second-order derivation $d^s$ on $\cA^s$ by taking
\begin{align*}
&d^s(q_i)=\begin{cases}
d(q_i),&i\neq a,b,\\
0,&i=b\\
q_b,&i=a
\end{cases}\\
&d^s(t^\pm)=d(t^\pm),\\
&\{q_i,q_j\}_{d^s}=\begin{cases}
0, &\text{for }i\in\{a,b\}\text{ or }j\in\{a,b\}\\
\{q_i,q_j\}_{d}, &\text{otherwise}
\end{cases}\\
&\{t^\pm,q_i\}_{d^s}=\begin{cases}
0, &\text{for }i\in\{a,b\}\\
\{t^\pm,q_i\}_{d}, &\text{otherwise}
\end{cases}\\
&\{q_i,t^\pm\}_{d^s}=\begin{cases}
0, &\text{for }i\in\{a,b\}\\
\{q_i,t^\pm\}_{d}, &\text{otherwise}
\end{cases}\\
&\{t^\pm,t^\pm\}_{d^s}=\{t^\pm,t^\pm\}_{d}.
\end{align*}

It is easy to show that this defines a second-order dga structure on $\cA^s$. For length 1 words $x,y\neq q_a,q_b$, we have $d^s\circ d^s(x)=0$ and $(1\otimes d^s+d^s\otimes 1)\{x,y\}_{d^s}=\{d^s x,y\}_{d^s}+(-1)^{|x|}\{x,d^s y\}_{d^s}$ using the properties of $d$. Additionally, we trivially get $d^s\circ d^s(q_a)=d^s\circ d^s(q_b)=0$ and $(1\otimes d^s+d^s\otimes 1)\{x,y\}_{d^s}=0=\{d^s x,y\}_{d^s}+(-1)^{|x|}\{x,d^s y\}_{d^s}$ if $x\in\{q_a,q_b\}$ or $y\in\{q_a,q_b\}$.

\vspace{2.1mm}
Next, we define the notion of stable tame equivalence for second-order dg algebras.
\begin{defi}
An algebra automorphism of $\cA$ is \textit{elementary} if it is of the form $\phi(q_i)=\pm q_i$ or $\phi(q_i)=q_i t^\pm$ or $\phi(q_i)=t^\pm q_i$ for some $i\in\{1,\dots, n\}$ and $\phi(q_j)=q_j$ for $j\neq i$, or if it is of the form
\begin{align*}
\phi(q_j)=\begin{cases}
q_j, &j\neq i\\
q_i +\omega_i, &j=i
\end{cases}
\end{align*}
for some $\omega_i\in\cA$ that does not contain letter $q_i$. A second-order algebra automorphism of $\cA$ is \textit{elementary} if it is of the form
\begin{align*}
&\phi(q_j)=q_j,\\
&\{q_j,q_k\}_\phi=\begin{cases}
\omega_{\iota,\kappa}, &j=\iota,k=\kappa\\
0,&\text{otherwise} 
\end{cases}
\end{align*}
for some $\iota,\kappa\in\{1,\dots,n\}$ and $\omega_{\iota,\kappa}\in\widetilde\cA\otimes\widetilde\cA$.
\end{defi}

\begin{defi}
A second-order algebra automorphism of $\cA$ is \textit{tame} if it is a composition of elementary automorphisms. A second-order algebra isomorphism $\phi:\cA\to\cA'$ for $\cA,\cA'$ generated by $q_1,\dots,q_n$ and $q_1',\dots,q_n'$ is tame if it is a composition of a tame automorphism of $\cA$ and the algebra isomorphism that sends $q_i$ to $q_i',i\in\{1,\dots,n\}$.
\end{defi}

\begin{defi}
We say second-order dg algebras $(\cA,d,\{\cdot,\cdot\}_d),(\cA',d',\{\cdot,\cdot\}_{d'})$ are \textit{stable tame equivalent} if there exist second-order dg algebras $(\widetilde\cA,\widetilde d,\{\cdot,\cdot\}_{\widetilde d}),(\widetilde\cA',\widetilde d',\{\cdot,\cdot\}_{\widetilde d'})$ obtained by taking stabilizations of $(\cA,d,\{\cdot,\cdot\}_d)$, $(\cA',d',\{\cdot,\cdot\}_{d'})$ finitely many times that are isomorphic through a tame second-order dga isomorphism. 
\end{defi}

\subsection{Change of the base point}\label{Sec:invar_base_pt}

Let $\Lambda_0,\Lambda_1$ be Legendrian knots that differ only in the choice of the base point. We denote the second-order dga corresponding to $\Lambda_\iota,\iota\in\{0,1\}$ by $(\cA(\Lambda_\iota),d_\iota,\{\cdot,\cdot\}_{d_\iota})$. Denote the Reeb chords by $\gamma_1,\dots,\gamma_n$. The main goal of this section is to prove the following proposition.
\begin{prop}\label{Prop:eq_after_T}
Second-order dg algebras $(\cA(\Lambda_0),d_0,\{\cdot,\cdot\}_{d_0}),(\cA(\Lambda_1),d_1,\{\cdot,\cdot\}_{d_1})$ associated to a Legendrian knot with two choices of the base point are tame isomorphic.
\end{prop}

Change of the base point can be seen as a sequence of crossings of the base point over Reeb chord endpoints, see Figure \ref{Figure:change_basept} and Figure \ref{Figure:change_basept_IIcase}. We consider a crossing over a negative Reeb chord endpoint, as shown in Figure \ref{Figure:change_basept}. Crossing over a positive end goes similarly.
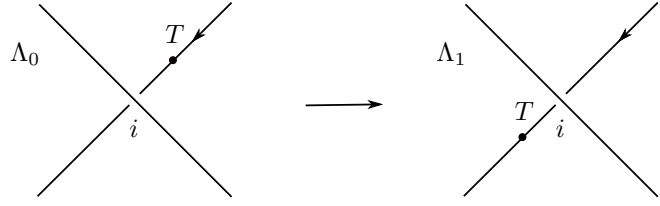
\begin{figure}
\def\svgwidth{86mm}
\begingroup%
  \makeatletter%
  \providecommand\rotatebox[2]{#2}%
  \newcommand*\fsize{\dimexpr\f@size pt\relax}%
  \newcommand*\lineheight[1]{\fontsize{\fsize}{#1\fsize}\selectfont}%
  \ifx\svgwidth\undefined%
    \setlength{\unitlength}{476.59500395bp}%
    \ifx\svgscale\undefined%
      \relax%
    \else%
      \setlength{\unitlength}{\unitlength * \real{\svgscale}}%
    \fi%
  \else%
    \setlength{\unitlength}{\svgwidth}%
  \fi%
  \global\let\svgwidth\undefined%
  \global\let\svgscale\undefined%
  \makeatother%
  \begin{picture}(1,0.30256519)%
    \lineheight{1}%
    \setlength\tabcolsep{0pt}%
    \put(0.18253885,0.09317762){\makebox(0,0)[lt]{\lineheight{1.25}\smash{\begin{tabular}[t]{l}$i$\end{tabular}}}}%
    \put(0.23814431,0.23955228){\makebox(0,0)[lt]{\lineheight{1.25}\smash{\begin{tabular}[t]{l}$T$\end{tabular}}}}%
    \put(-0.002115,0.21049125){\makebox(0,0)[lt]{\lineheight{1.25}\smash{\begin{tabular}[t]{l}$\Lambda_0$\end{tabular}}}}%
    \put(0.84148745,0.09334416){\makebox(0,0)[lt]{\lineheight{1.25}\smash{\begin{tabular}[t]{l}$i$\end{tabular}}}}%
    \put(0.77790204,0.12052775){\makebox(0,0)[lt]{\lineheight{1.25}\smash{\begin{tabular}[t]{l}$T$\end{tabular}}}}%
    \put(0.65683339,0.21065797){\makebox(0,0)[lt]{\lineheight{1.25}\smash{\begin{tabular}[t]{l}$\Lambda_1$\end{tabular}}}}%
    \put(0,0){\includegraphics[width=\unitlength,page=1]{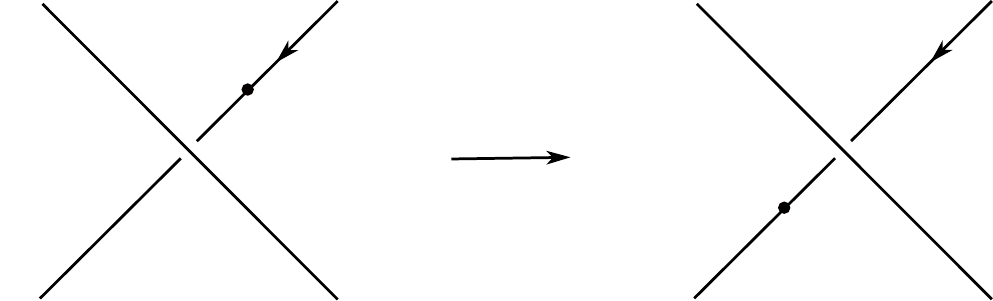}}%
  \end{picture}%
\endgroup%
\caption{Change of the base point, crossing over a negative Reeb chord endpoint.}
\label{Figure:change_basept}
\end{figure}

Consider the graded algebra isomorphism $\phi:\cA(\Lambda_0)\to\cA(\Lambda_1)$ given by
\begin{align*}
\phi(q_j)=\begin{cases}
q_j,&j\neq i\\
q_i t^+,&j=i
\end{cases}
\end{align*}
and $\phi(t^\pm)=t^\pm$.

\begin{lemma}\label{Lemma:basept1}
For all $s_1,s_2\in\{t^\pm,q_i\,|\,i=1,\dots,n\}$, we have
\begin{align*}
(\phi\otimes\phi)\{s_1,s_2\}_{d_0}=\{\phi s_1,\phi s_2\}_{d_1}.
\end{align*}
\end{lemma}
\begin{proof}
For $u$ a $J$-holomorphic curve on $\Lambda_0,\Lambda_1$ with no positive puncture at $\gamma_i$, the word corresponding to $u$ on $\Lambda_1$ is obtained by replacing all $q_i$'s in the word corresponding to $u$ on $\Lambda_0$ by $q_it^+$. The orientation signs do not change. For $s_1=q_j,s_2=q_k, j,k\neq i$, it follows that 
$$(\phi\otimes\phi)\{q_j,q_k\}_{d_0}=\{q_j,q_k\}_{d_1}=\{\phi q_j,\phi q_k\}_{d_1}.$$ 
The same holds for $s_1=t^\pm,s_2=q_j$ and $s_2=t^\pm,s_1=q_j$ for $j\neq i$.

Let now $s_1=q_i,s_2=q_j,j\neq i$. Similar as above, we have $(\phi\otimes\phi) d_{0,\lD}(q_i,q_j)=d_{1,\lD}(q_i,q_j)(1\otimes t^+)$. Additionally, we have $\delta_0(j^\pm,i^-)=1,\delta_1(j^\pm,i^-)=0,\delta_0(j^\pm,i^+)=\delta_1(j^\pm,i^+)$ and
\begin{align*}
&(\phi\otimes\phi) d_{0,f}(q_i,q_j)=\delta(j^+,i^+)q_j\otimes q_it^++(-1)^{|q_i||q_j|}q_it^+\otimes q_j-\\
&-q_it^+q_j\otimes 1-(-1)^{|q_i||q_j|}\delta(j^-,i^+)1\otimes q_jq_it^+=\\
&=d_{1,f}(q_i,q_j)(1\otimes t^+)+(q_i\otimes 1)\{t^+,q_j\}_{d_1}.
\end{align*}
Therefore,
\begin{align*}
&(\phi\otimes\phi)\{q_i,q_j\}_{d_0}=(\phi\otimes\phi) d_{0,\lD}(q_i,q_j)+(\phi\otimes\phi) d_{0,f}(q_i,q_j)=\\
&=\{q_i,q_j\}_{d_1}(1\otimes t^+)+(q_i\otimes 1)\{t^+,q_j\}_{d_1}=\\
&=\{q_i t^+,q_j\}_{d_1}=\{\phi q_i,\phi q_j\}_{d_1}.
\end{align*}
Similarly we get $(\phi\otimes\phi)\{q_j,q_i\}_{d_0}=\{\phi q_j,\phi q_i\}_{d_1}, (\phi\otimes\phi)\{t^\pm,q_i\}_{d_0}=\{\phi t^\pm,\phi q_i\}_{d_1}, (\phi\otimes\phi)\{q_i,t^\pm\}_{d_0}=\{\phi q_i,\phi t^\pm\}_{d_1}$.

Finally, for $s_1=s_2=q_i$ we have
\begin{align*}
(\phi\otimes\phi) d_{0,\lD}(q_i,q_i)=d_{1,\lD}(q_i,q_i)(t^+\otimes t^+),
\end{align*}
and
\begin{align*}
\delta_0(i^+,i^-)=1,\delta_0(i^-,i^+)=0,\\
\delta_1(i^+,i^-)=0,\delta_1(i^-,i^+)=1.
\end{align*}
Therefore,
\begin{align*}
&d_{1,f}(q_i,q_i)(t^+\otimes t^+)-q_it^+q_it^+\otimes 1+(-1)^{|q_i|}t^+\otimes q_iq_i t^+=\\
&=(\phi\otimes\phi)(-q_i q_i\otimes 1+\delta(i)q_i\otimes q_i)=(\phi\otimes\phi) d_{0,f}(q_i,q_i).
\end{align*}
From this we conclude
\begin{align*}
&(\phi\otimes\phi)\{q_i,q_i\}_{d_0}=\\
&=\{q_i,q_i\}_{d_1}(t^+\otimes t^+)-q_it^+q_it^+\otimes 1+(-1)^{|q_i|}t^+\otimes q_iq_i t^+=\\
&=\{q_i t^+,q_i t^+\}_{d_1}=\{\phi q_i,\phi q_i\}_{d_1}.
\end{align*}
\end{proof}

\begin{lemma}\label{Lemma:basept2}
For all $s\in\{t^\pm,q_i\,|\,i=1,\dots,n\}$, we have
\begin{align*}
\phi\circ d_0(s)=d_1\circ\phi(s).
\end{align*}
\end{lemma}
\begin{proof}
Let first $s=q_j,j\neq i$. As in the previous lemma, we get
\begin{align*}
&d_{1,\lD}\circ\phi (q_j)=\phi \circ d_{0,\lD}(q_j),\\
&d_{1,A}\circ\phi (q_j)=\phi\circ d_{0,A}(q_j).
\end{align*}
Since $j\neq i$, $d_{1,f}(\phi q_j)=d_{1,f}(q_j)=d_{0,f}(q_j)=\phi d_{0,f}(q_j)$ trivially holds. This implies $d_1\circ\phi (q_j)=\phi\circ d_0(q_j)$. Similarly we get $d_1\circ\phi (t^\pm)=\phi\circ d_0(t^\pm)$.

For $s=q_i$, we have as before
\begin{align*}
&d_{1,\lD}(q_i)t^+=\phi d_{0,\lD}(q_i),\\
&d_{1,A}(q_i)t^+=\phi d_{0,A}(q_i).
\end{align*}
Additionally, it is not difficult to show
\begin{align*}
d_{1,f}(q_it^+)=\phi d_{0,f}(q_i).
\end{align*}
If $|q_i|$ is odd, this follows from
\begin{align*}
d_{1,f}(q_it^+)=&d_{1,f}(q_i)t^++(-1)^{|q_i|}q_id_{1,f}(t^+)+d_{1,f}(q_i,t^+)=\\
=&-\frac{1}{2}\sum_{j\neq i}(-1)^{|q_j|}\left(\delta_1(j^+,i^+)+\delta_1(j^-,i^+)\right)\hbar(q_it^+\otimes 1)+\\
&+\operatorname{tb}(\Lambda)\hbar(q_i t^+\otimes 1)+\hbar(q_i t^+\otimes 1)=\\
=&\left(\operatorname{tb}(\Lambda)-\frac{1}{2}\sum_{j\neq i}(-1)^{|q_j|}\left(\delta_0(j^+,i^+)+\delta_0(j^-,i^+)\right)-(-1)^{|q_i|}\right)\hbar(q_it^+\otimes 1)=\\
=&\phi\left(\frac{1}{2}\sum_{j\neq i}(-1)^{|q_j|}\left(\delta_0(i^+,j^+)+\delta_0(i^+,j^-)\right)\hbar(q_i\otimes 1)\right)=\\
=&\phi d_{0,f}(q_i).
\end{align*}
Similarly when $|q_i|$ is even. This implies 
\begin{align*}
\phi d_{0}(q_i)=d_1\phi (q_i),
\end{align*}
which finishes the proof.
\end{proof}

\begin{cor}\label{Cor:basept3}
For all $s\in\cA(\Lambda_0),s_1,s_2\in\widetilde\cA(\Lambda_0)$, we have
\begin{align*}
&\phi\circ d_0(s)=d_1\circ \phi(s),\\
&(\phi\otimes\phi)\{s_1,s_2\}_{d_0}=\{\phi s_1,\phi s_2\}_{d_1}.
\end{align*}
\end{cor}
\paragraph{\textit{Proof of Proposition \ref{Prop:eq_after_T}:}} follows from Corollary \ref{Cor:basept3} when the base point crosses a negative Reeb chord endpoint. The proof in the case where the base point crosses $i^+$ follows similarly. Here we consider the graded algebra isomorphism $\phi$ given by $\phi(t^\pm)=t^\pm$ and 
\begin{align*}
\phi(q_j)=\begin{cases}
q_j,&j\neq i\\
t^+ q_i,&j=i
\end{cases}
\end{align*}

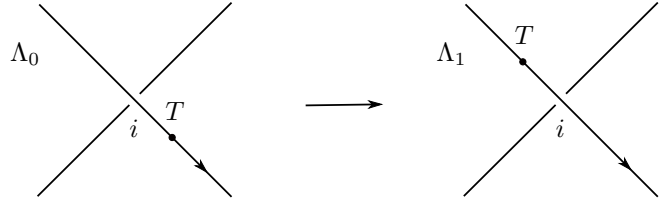
\begin{figure}
\def\svgwidth{86mm}
\begingroup%
  \makeatletter%
  \providecommand\rotatebox[2]{#2}%
  \newcommand*\fsize{\dimexpr\f@size pt\relax}%
  \newcommand*\lineheight[1]{\fontsize{\fsize}{#1\fsize}\selectfont}%
  \ifx\svgwidth\undefined%
    \setlength{\unitlength}{476.59508481bp}%
    \ifx\svgscale\undefined%
      \relax%
    \else%
      \setlength{\unitlength}{\unitlength * \real{\svgscale}}%
    \fi%
  \else%
    \setlength{\unitlength}{\svgwidth}%
  \fi%
  \global\let\svgwidth\undefined%
  \global\let\svgscale\undefined%
  \makeatother%
  \begin{picture}(1,0.30256495)%
    \lineheight{1}%
    \setlength\tabcolsep{0pt}%
    \put(0.18253882,0.09317724){\makebox(0,0)[lt]{\lineheight{1.25}\smash{\begin{tabular}[t]{l}$i$\end{tabular}}}}%
    \put(0.23814427,0.11995397){\makebox(0,0)[lt]{\lineheight{1.25}\smash{\begin{tabular}[t]{l}$T$\end{tabular}}}}%
    \put(-0.002115,0.21049085){\makebox(0,0)[lt]{\lineheight{1.25}\smash{\begin{tabular}[t]{l}$\Lambda_0$\end{tabular}}}}%
    \put(0.8414873,0.09334377){\makebox(0,0)[lt]{\lineheight{1.25}\smash{\begin{tabular}[t]{l}$i$\end{tabular}}}}%
    \put(0.77790191,0.23697789){\makebox(0,0)[lt]{\lineheight{1.25}\smash{\begin{tabular}[t]{l}$T$\end{tabular}}}}%
    \put(0.65683327,0.21065756){\makebox(0,0)[lt]{\lineheight{1.25}\smash{\begin{tabular}[t]{l}$\Lambda_1$\end{tabular}}}}%
    \put(0,0){\includegraphics[width=\unitlength,page=1]{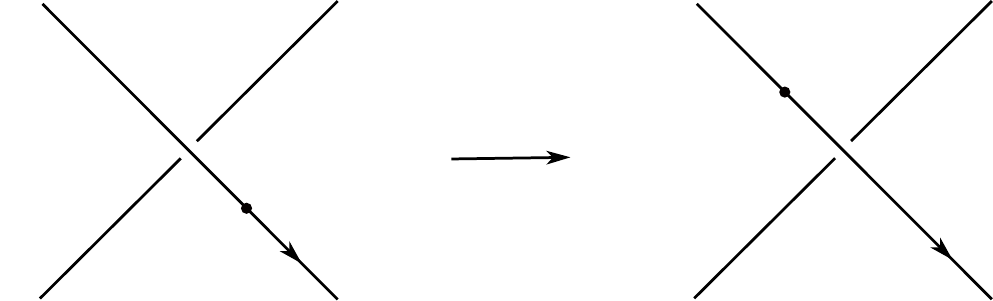}}%
  \end{picture}%
\endgroup%
\caption{Change of the base point, crossing over a positive Reeb chord endpoint.}
\label{Figure:change_basept_IIcase}
\end{figure}

\subsection{Type IV degeneration}\label{Section:Invar_IV_degeneration}

Let $\Lambda_s,s\in[0,1]$ be a generic Legendrian knot isotopy with a type IV degenerate knot at $s=\frac{1}{2}$ and $u_0\in\cM^\pi_{2,0}(\Lambda_\frac{1}{2})$ a rigid holomorphic annulus in the Lagrangian projection such that $\Omega_\frac{1}{2}(u_0)=0$. This annulus can be lifted to an index -1 $J$-holomorphic annulus on $\lR\times\Lambda_\frac{1}{2}$. Denote by $\gamma_1,\dots,\gamma_n$ the Reeb chords on $\Lambda_\iota$ and by $(\cA(\Lambda_\iota),d_\iota,\{\cdot,\cdot\}_{d_\iota})$ the second-order dga corresponding to $\Lambda_\iota,\iota\in\{0,1\}$. We find a graded algebra isomorphism $\phi:\cA(\Lambda_0)\to\cA(\Lambda_1)$ such that 
\begin{align*}
&d_1\circ\phi=\phi\circ d_0,\\
&\{\phi \cdot,\phi \cdot\}_{d_1}=(\phi\otimes\phi)\{\cdot,\cdot\}_{d_0},
\end{align*}
proving the following proposition.

\begin{prop}\label{Prop:eq_after_I}
Let $\Lambda_s,s\in[0,1]$ be a Legendrian knot isotopy with type IV degeneration as above, then the second-order dg algebras $(\cA(\Lambda_0),d_0,\{\cdot,\cdot\}_{d_0}),(\cA(\Lambda_1),d_1,\{\cdot,\cdot\}_{d_1})$ are tame isomorphic.
\end{prop}
Let $\gamma_a$ be the Reeb chord at the positive puncture of $u_0$. Fix a marked point $e_2$ on the inner boundary of $u_0$ and denote by $\widetilde \omega=\hbar\, \overline w(u_0,e_2)=\hbar(\overline w_2(u_0)\otimes\overline w_1(u_0,e_2))\in\hbar\,( \widetilde\cA(\Lambda_1)\otimes \widetilde\cA(\Lambda_1))$ the word pair obtained by looking at the negative punctures of $u_0$ at the two boundary components and crossings over the base point as before. We consider the graded algebra morphism $\phi:\cA(\Lambda_0)\to\cA(\Lambda_1)$ given by $\phi(t^\pm)=t^\pm$ and
\begin{align*}
\phi(q_i)=
\begin{cases}
q_a-\epsilon(u_0)\widetilde \omega,&i=a\\
q_i,&\text{otherwise}
\end{cases}
\end{align*}
where $\epsilon(u_0)=\epsilon(u_0,e_2)$ is the sign of $u_0$ on $\pi_{xy}(\Lambda_0)$ with respect to the marked point $e_2$. More precisely, 
$$\epsilon(u_0)=\epsilon\,\epsilon_1\,\epsilon_2\prod\epsilon_\bullet^{u_0},$$ 
where $\epsilon=1$ if $\Omega_{\Lambda_0}(u_0)>0$ and $-1$ otherwise, $\epsilon_1,\epsilon_2$ are the orientation signs at the marked points with respect to the orientation of $\Lambda$ (where the first marked point is taken to be right after the positive puncture), and $\epsilon_\bullet^{u_0}$ are the signs at the corners of $u_0$. Note that $\widetilde\omega$ does not contain letter $q_a$, moreover, $l(\widetilde\omega)<l(q_a)$.

We show that $\phi$ is a second-order dga isomorphism. The main idea has already been seen in Corollary \ref{Corollary:counting_annuli_from_O}.

\begin{prop}\label{typeIIIcompactness}
For all $i\in\{1,\dots,n\}$, we have
\begin{equation}
\label{Eq:lemma_degIV_change}
\begin{aligned}
d_{1,A}(q_i)-d_{0,A}(q_i)=\widetilde\phi\circ d_{0,\lD}(q_i)-d_{1,\lD}\circ\widetilde\phi(q_i),
\end{aligned}
\end{equation}
where $\widetilde\phi=\phi-\operatorname{id}$.
\end{prop}
\begin{proof}
Let $i\neq a$. As we have seen in Corollary \ref{Corollary:counting_annuli_from_O}, the difference between the count of annuli on $\Lambda_0$ and $\Lambda_1$ corresponds to the ways $u_0$ can be glued to some rigid disk on $\lR\times\Lambda_0$ with one positive puncture. This implies (\ref{Eq:lemma_degIV_change}) up to signs. Let for example $u=q_{j_1}\dots q_{j_n}p_i$ be an index zero disk with a positive puncture at $\gamma_i$ and a negative puncture $q_{j_k}=q_a$ at $\gamma_a$. The orientation of the string obtained by gluing $u_0$ to $u$ is given by
\begin{align*}
(-1)^{|q_{j_1}|+\dots+|q_{j_{k-1}}|}\epsilon\,\epsilon(u)\,\epsilon(u_0)\langle v\rangle.
\end{align*}
It is not difficult to see that the difference between the algebraic counts of zeros of the obstruction sections $\Omega_{\Lambda_1}$ and $\Omega_{\Lambda_0}$ on this family of annuli is given by
\begin{align*}
-(-1)^{|q_{j_1}|+\dots+|q_{j_{k-1}}|}\epsilon(u)\epsilon(u_0),
\end{align*}
i.e. $d_{1,A}(q_i)-d_{0,A}(q_i)$ contains a summand equal to 
\begin{align*}
&-(-1)^{|q_{j_1}|+\dots+|q_{j_{k-1}}|}\epsilon(u)\epsilon(u_0) \hbar(q_{j_1}\dots q_{j_{k-1}}\overline w_2(u_0) q_{j_{k+1}}\dots q_{j_n}\otimes \overline w_1(u_0,e_2))=\\
&=-\epsilon(u)\epsilon(u_0)q_{j_1}\dots q_{j_{k-1}}\cdot\widetilde\omega\cdot q_{j_{k+1}}\dots q_{j_n}.
\end{align*} 
The corresponding summand in $\widetilde \phi\circ d_{0,\lD}(q_i)$ with the same sign is obtained by first taking $\epsilon(u)q_{j_1}\dots q_{j_n}\in d_{0,\lD}(q_i)$ and then applying $\widetilde\phi$. This finishes the proof in the case $i\neq a$. The proof goes similarly for $i=a$.
\end{proof}

\begin{prop}\label{Prop:typeIV1}
The map $\phi:\cA(\Lambda_0)\to \cA(\Lambda_1)$ defined above satisfies 
\begin{align*}
&d_1\circ \phi(s)=\phi\circ d_0(s),\\
&\{\phi s_1,\phi s_2\}_{d_1}=(\phi\otimes\phi)\{s_1,s_2\}_{d_0},
\end{align*}
for all $s\in\cA(\Lambda_0),s_1,s_2\in\widetilde\cA(\Lambda_0)$.
\end{prop}

\begin{proof}
It is enough to show
\begin{align*}
&d_1\circ\phi(s)=\phi\circ d_0(s),\\
&\{\phi s_1,\phi s_2\}_{d_1}=(\phi\otimes\phi)\{s_1,s_2\}_{d_0},
\end{align*}
for all $s,s_1,s_2\in\{t^\pm,q_i\,|\,i=1,\dots,n\}$, or equivalently (since $\widetilde\omega\in\hbar\, (\widetilde\cA(\Lambda_1)\otimes\widetilde\cA^{\cyc}(\Lambda_1))$)
\begin{align*}
&d_1\circ\phi(s)=\phi\circ d_0(s),\\
&\{s_1,s_2\}_{d_1}=\{s_1,s_2\}_{d_0}.
\end{align*}
Disks with one or two positive punctures on $\Lambda_0$ and $\Lambda_1$ are the same. Additionally, $d_{0,f}(s)=d_{1,f}(s)$ and $d_{0,f}(s_1,s_2)=d_{1,f}(s_1,s_2)$ for all $s,s_1,s_2$, so the second equality follows trivially. 

Since $d_{1,\lD}=d_{0,\lD},d_{1,f}=d_{0,f}$ and $\widetilde\phi(\widetilde\cA)\subset\hbar\,(\widetilde\cA\otimes\widetilde\cA^{\cyc})$, 
$$d_{0,A}(q_i)-d_{1,A}(q_i)=d_{1,\lD}\widetilde\phi(q_i)-\widetilde\phi d_{0,\lD}(q_i)$$ 
is equivalent to
\begin{align*}
d_{0}(q_i)-d_{1}(q_i)=d_{1}\widetilde\phi(q_i)-\widetilde\phi d_{0}(q_i)=\\
=d_1\phi(q_i)-d_1(q_i)-\phi d_0(q_i)+d_0(q_i).
\end{align*}
This implies
$$d_1\circ \phi(q_i)=\phi\circ d_0(q_i)$$
for all $i$. Moreover, $d_1\phi(t^\pm)=d_{1,f}(t^\pm)=d_{0,f}(t^\pm)=\phi d_{0}(t^\pm)$.
\end{proof}

\paragraph{\textit{Proof of Proposition \ref{Prop:eq_after_I}:}} Follows from Proposition \ref{Prop:typeIV1}. The morphism $\phi$ is invertible and the inverse is given by the algebra morphism 
\begin{align*}
\phi^{-1}(q_i)=
\begin{cases}
q_a+\epsilon(u_0)\widetilde \omega,&i=a\\
q_i,&\text{otherwise}
\end{cases}
\end{align*}

\subsection{Reidemeister III move}
\label{Section:Invar_III_degeneration}

In this section, we show invariance under Reidemeister III move. Let $\Lambda_s,s\in[0,1]$ be a Legendrian knot isotopy with a Reidemeister III move at $s=\frac{1}{2}$ in the Lagrangian projection as shown in Figure \ref{Figure:ReidIII}. Denote the "small" triangle by $\omega_\Delta$ with punctures at the Reeb chords  $\gamma_a,\gamma_b,\gamma_c$. We distinguish two cases. First, when $\omega_\Delta$ has one positive puncture at $a$, and second, when it has two positive punctures at $b$ and $c$.  Denote by $\gamma_1,\dots,\gamma_n$ the Reeb chords on $\Lambda_\iota$ and by $(\cA(\Lambda_\iota),d_\iota,\{\cdot,\cdot\}_{d_\iota})$ the second-order dga corresponding to $\Lambda_\iota,\iota\in\{0,1\}$ as before. Note that there is a canonical isomorphism $\cA(\Lambda_0)\cong\cA(\Lambda_1)$ that identifies $q_i,i\in\{1,\dots,n\}$ on the two sides as shown in Figure \ref{Figure:ReidIII}. Our main goal is to find a second-order dga isomorphism $(\cA(\Lambda_0),d_0,\{\cdot,\cdot\}_{d_0})\to(\cA(\Lambda_1),d_1,\{\cdot,\cdot\}_{d_1})$, proving the following.

\begin{prop}\label{Prop:eq_after_III}
Let $\Lambda_s,s\in[0,1]$ be a Legendrian knot isotopy as above, then there exists a tame second-order dga isomorphism between $(\cA(\Lambda_0),d_0,\{\cdot,\cdot\}_{d_0})$ and $(\cA(\Lambda_1),d_1,\{\cdot,\cdot\}_{d_1})$.
\end{prop}

We first consider the Reidemeister III move where $\omega_\Delta$ has one positive puncture, i.e. $\omega_\Delta=q_bq_cp_a$ (see Figure \ref{Figure:ReidIII}, top). Let $\phi_1:\cA(\Lambda_0)\to\cA(\Lambda_1)$ be the second-order graded algebra morphism given by
\begin{align*}
&\phi_1(q_i)=q_i,\phi_1(t^\pm)=t^\pm,\\
&\{q_i,q_j\}_{\phi_1}=\begin{cases}
(-1)^{|q_b|}\epsilon_\Delta q_a\otimes 1,&i=b,j=c\\
-(-1)^{(|q_c|+1)(|q_a|+1)}\epsilon_\Delta 1\otimes q_a,&i=c,j=b\\
0,&\text{otherwise}
\end{cases}\\
&\{t^\pm,q_i\}_{\phi_1}=0,\{q_i,t^\pm\}_{\phi_1}=0,\{t^\pm,t^\pm\}_{\phi_1}=0,
\end{align*}
where $\epsilon_\Delta$ is the product of the signs at the corners of $\omega_\Delta$ on $\Lambda_0$.

\begin{figure}
\def\svgwidth{90mm}
\begingroup%
  \makeatletter%
  \providecommand\rotatebox[2]{#2}%
  \newcommand*\fsize{\dimexpr\f@size pt\relax}%
  \newcommand*\lineheight[1]{\fontsize{\fsize}{#1\fsize}\selectfont}%
  \ifx\svgwidth\undefined%
    \setlength{\unitlength}{283.11066692bp}%
    \ifx\svgscale\undefined%
      \relax%
    \else%
      \setlength{\unitlength}{\unitlength * \real{\svgscale}}%
    \fi%
  \else%
    \setlength{\unitlength}{\svgwidth}%
  \fi%
  \global\let\svgwidth\undefined%
  \global\let\svgscale\undefined%
  \makeatother%
  \begin{picture}(1,0.38175043)%
    \lineheight{1}%
    \setlength\tabcolsep{0pt}%
    \put(0,0){\includegraphics[width=\unitlength,page=1]{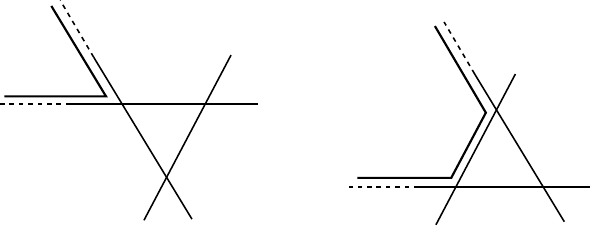}}%
  \end{picture}%
\endgroup%
\caption{Appearance of a disk during Reidemeister III move.}
\label{Figure:disk_reidIII}
\end{figure}

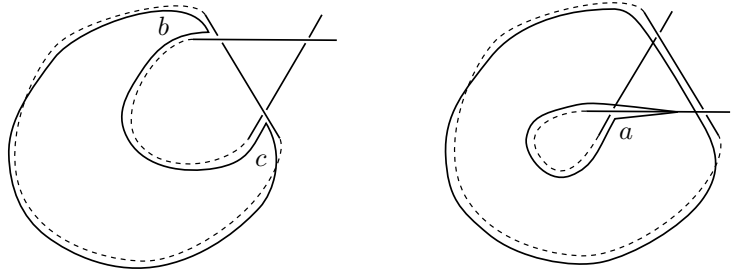
\begin{figure}
\def\svgwidth{96mm}
\begingroup%
  \makeatletter%
  \providecommand\rotatebox[2]{#2}%
  \newcommand*\fsize{\dimexpr\f@size pt\relax}%
  \newcommand*\lineheight[1]{\fontsize{\fsize}{#1\fsize}\selectfont}%
  \ifx\svgwidth\undefined%
    \setlength{\unitlength}{441.1280053bp}%
    \ifx\svgscale\undefined%
      \relax%
    \else%
      \setlength{\unitlength}{\unitlength * \real{\svgscale}}%
    \fi%
  \else%
    \setlength{\unitlength}{\svgwidth}%
  \fi%
  \global\let\svgwidth\undefined%
  \global\let\svgscale\undefined%
  \makeatother%
  \begin{picture}(1,0.3697912)%
    \lineheight{1}%
    \setlength\tabcolsep{0pt}%
    \put(0,0){\includegraphics[width=\unitlength,page=1]{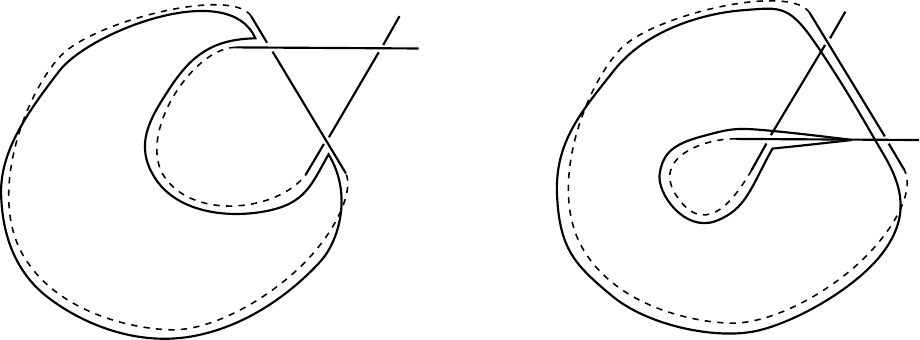}}%
    \put(0.84643616,0.17979672){\makebox(0,0)[lt]{\lineheight{1.25}\smash{\begin{tabular}[t]{l}$a$\end{tabular}}}}%
    \put(0.2076517,0.32431259){\makebox(0,0)[lt]{\lineheight{1.25}\smash{\begin{tabular}[t]{l}$b$\end{tabular}}}}%
    \put(0.34245219,0.1433642){\makebox(0,0)[lt]{\lineheight{1.25}\smash{\begin{tabular}[t]{l}$c$\end{tabular}}}}%
  \end{picture}%
\endgroup%
\caption{Appearance of an annulus during Reidemeister III move.}
\label{Figure:Annuli_reidIII}
\end{figure}

\begin{prop}\label{Prop:typeIII1}
The map $\phi_1:\cA(\Lambda_0)\to \cA(\Lambda_1)$ defined above satisfies 
$$d_1\circ \phi_1(s)=\phi_1\circ d_0(s)$$
for all $s\in\cA(\Lambda_0)$, and
\begin{align*}
&\{\phi_1 s_1,\phi_1 s_2\}_{d_1}-(d_1\otimes 1+1\otimes d_1)\{s_1,s_2\}_{\phi_1}=(\phi_1\otimes \phi_1)\{s_1,s_2\}_{d_0}+\{d_0 s_1,s_2\}_{\phi_1}+(-1)^{|s_1|}\{s_1,d_0 s_2\}_{\phi_1}
\end{align*}
for all $s_1,s_2\in\widetilde\cA(\Lambda_0)$.
\end{prop}
\begin{proof}
It is enough to show
\begin{align*}
&d_1(q_i)=\phi_1\circ d_0(q_i),\\
&\{q_i,q_j\}_{d_1}-(d_1\otimes 1+1\otimes d_1)\{q_i,q_j\}_{\phi_1}=\{q_i,q_j\}_{d_0}+\{d_0q_i,q_j\}_{\phi_1}+(-1)^{|q_i|}\{q_i,d_0q_j\}_{\phi_1},
\end{align*}
for all $i,j\in\{1,\dots,n\}$. To show the first equality, we notice that disks with one positive puncture are the same on $\Lambda_0$ and $\Lambda_1$. Moreover, every disappearing/appearing annulus is of the form $\hbar(w_1w_3p_i\otimes q_aw_2)$ or $\hbar(w_1q_aw_3p_i\otimes w_2)$, where $w_1q_cw_2q_bw_3p_i$ or $w_1q_bw_2q_cw_3p_i$ is an index zero disk other than $q_bq_cp_a$ with negative punctures at $b$ and $c$. This can be written as
\begin{align*}
d_{1,A}(q_i)-d_{0,A}(q_i)=\widetilde\phi_1 d_{0,\lD}(q_i),
\end{align*}
for $i\neq a$, where $\widetilde\phi_1=\phi_1-\operatorname{id}$. To see that the signs on the two sides are equal, take for example a disk $u'=w_1q_cw_2q_bw_3p_i$ of the form as shown in Figure \ref{Figure:Annuli_reidIII}, other cases go similarly. We have a 1-parameter family of annuli $w_1w_3p_i\otimes q_a w_2$ on $\pi_{xy}(\Lambda_1)$ that does not exist on $\Lambda_0$. The values of the obstruction section at the boundary are equal to $+\infty$ and $-\infty$. The count of zeros of the obstruction section is given by
\begin{align*}
-\epsilon_\Delta\epsilon(u'),
\end{align*}
and the corresponding summand in $\widetilde\phi_1\circ d_{0,\lD}(q_i)$ also comes with the sign (see (\ref{Eq:Second_order_morphism_formula}))
$$-\epsilon_\Delta\epsilon(u')(-1)^{(1+|w_1q_cw_2|)(|q_a|+|w_2|+1)}=-\epsilon_\Delta\epsilon(u'),$$ 
since here $|q_a|+|w_2|\equiv 1\pmod 2$. We obviously have $d_{0,f}(q_i)=d_{1,f}(q_i)$ for $i\neq a,b,c$. Additionally, it is not difficult to see $d_{0,f}(q_b)=d_{1,f}(q_b),d_{0,f}(q_c)=d_{1,f}(q_c)$. This shows $d_1(q_i)=\phi_1\circ d_0(q_i)$ for $i\neq a$. Moreover, we have
\begin{align*}
&d_{1,\lD}(q_a)=d_{0,\lD}(q_a),d_{1,A}(q_a)=d_{0,A}(q_a),\\
&d_{1,f}(q_a)-d_{0,f}(q_a)=\epsilon(a)\hbar(q_a\otimes 1),\\
&\widetilde\phi_1\circ d_{0,\lD}(q_a)=\widetilde\phi_1\left((-1)^{|q_b|}\epsilon(a)\epsilon_\Delta q_bq_c\right)=\epsilon(a)\hbar(q_a\otimes 1),
\end{align*}
where $\epsilon(a)$ is the sign of the arc $b^-c^+$ on the small triangle on $\Lambda_0$
This implies $d_1\circ{\phi_1}(q_a)={\phi_1}\circ d_0(q_a)$.

Next, we show 
\begin{equation}
\label{Eq:degIII_bracket}
\begin{aligned}
\{q_i,q_j\}_{d_1}-(d_1\otimes1 +1\otimes d_1)\{q_i,q_j\}_{\phi_1}=\{q_i,q_j\}_{d_0}+\{d_0q_i,q_j\}_{\phi_1}+(-1)^{|q_i|}\{q_i,d_0q_j\}_{\phi_1}. 
\end{aligned}
\end{equation}
If $i,j$ are both different from $b,c$, then
\begin{align*}
LHS=\{q_i,q_j\}_{d_1}=\{q_i,q_j\}_{d_0}=RHS,
\end{align*}
since the disks with positive punctures at $i$ and $j$ are the same on $\Lambda_0$ and $\Lambda_1$ and they come with the same orientation signs. Let now for example $i=b,j\neq a,b,c$ (similarly for $i=c,j\neq a,b,c$ and $i\neq a,b,c,j\in\{b,c\}$). Disks with positive punctures at $b$ and $j$ that appear on $\Lambda_1$ and not on $\Lambda_0$ are of the form $u'=w_1p_bq_aw_2p_j$, where $u''=w_1q_cw_2p_j$ is an index zero disk with one positive puncture at $j$ and a negative corner at $c$ that does not cover the small triangle on $\Lambda_0$. Similarly, disks with positive punctures at $b,j$ that appear on $\Lambda_0$ and not on $\Lambda_1$ are of the form $\overline u'=w_1p_bq_aw_2p_j$, where $\overline u''=w_1q_cw_2p_j$ is an index zero disk with one positive puncture at $j$ and a negative corner at $c$ that covers the small triangle on $\Lambda_0$. Additionally, $d_{1,f}(q_b,q_j)=d_{0,f}(q_b,q_j)$. From this we get
\begin{align*}
&\{q_b,q_j\}_{d_1}-\{q_b,q_j\}_{d_0}=(-1)^{|q_b|}\{q_b,d_0q_j\}_{\phi_1},
\end{align*}
i.e. (\ref{Eq:degIII_bracket}) holds for $i=b,j\neq a,b,c$.
To see that the signs match, we consider the case where we have a disk $u'=w_1p_bq_aw_2p_j$ on $\Lambda_1$ that does not appear on $\Lambda_0$, which corresponds to a disk $u''=w_1q_cw_2p_j$ on $\Lambda_0$ with a negative corner at $c$ that does not cover the small triangle. Let $\epsilon(u'')$ be the sign of $u''$ on $\Lambda_0$. Then the sign $\epsilon(u',\gamma_j^+)$ of the disk $u'$ is equal to
\begin{align*}
\epsilon(u',\gamma_j^+)=\epsilon_\Delta\epsilon(u''),
\end{align*}
and the corresponding summand $q_a w_2\otimes w_1$ appears on the LHS and on the RHS with the sign $(-1)^{(|q_b|+1)|w_1|}\epsilon(u'')\epsilon_\Delta$.

If $(i,j)=(b,a)$, similar as above we get
\begin{align*}
&d_{1,\lD}(q_b,q_a)=d_{1,\lD}(q_b,q_a)+(-1)^{|q_b|}\{q_b,\widetilde d_0q_a\}_{\phi_1},
\end{align*}
where
\begin{align*}
&\widetilde d_0(q_a)=\sum_{u\in\cM_1(\Lambda_0,J,\gamma_a), u\neq q_bq_cp_a}\epsilon(u)\widetilde w(u)=d_0(q_a)-(-1)^{|q_b|}\epsilon_\Delta\epsilon(a)q_bq_c.
\end{align*} 
Additionally, we have
\begin{align*}
&d_{1,f}(q_b,q_a)-d_{0,f}(q_b,q_a)=\left(\delta_1(a^+,b^+)-\delta_0(a^+,b^+)\right)q_a\otimes q_b=\\
&=(-1)^{|q_b|}\epsilon(a)q_a\otimes q_b=(-1)^{|q_b|}\{q_b,(-1)^{|q_b|}\epsilon_\Delta\epsilon(a)q_bq_c\}_{\phi_1}.
\end{align*}
This implies (\ref{Eq:degIII_bracket}) for $i=b,j=a$.

For $(i,j)=(b,c)$, we notice that for every disk with positive punctures at $b,c$ that does not exist on the other side, there is a corresponding disk with a positive corner at $a$ different from $q_bq_cp_a$. This implies
\begin{align*}
&d_{1,\lD}(q_b,q_c)-d_{0,\lD}(q_b,q_c)=(\widetilde d_0\otimes 1)\{q_b,q_c\}_{\phi_1}.
\end{align*}
Additionally, we have
\begin{align*}
&d_{1,f}(q_b,q_c)-d_{0,f}(q_b,q_c)=\left(-\delta_1(c^+,b^-)+\delta_0(c^+,b^-)\right)q_bq_c\otimes 1=\\
&=\epsilon(a)q_bq_c\otimes 1=\left((d_1-\widetilde d_0)\otimes 1\right)\{q_b,q_c\}_{\phi_1}.
\end{align*}
This implies (\ref{Eq:degIII_bracket}) for $i=b,j=c$. Other cases go similarly.
\end{proof}

Next, we consider the second Reidemeister III move where $\omega_\Delta=q_ap_bp_c$ (see Figure \ref{Figure:ReidIII}, bottom). Let $\phi_2:\cA(\Lambda_0)\to\cA(\Lambda_1)$ be the graded algebra morphism given by
\begin{align*}
&\phi_2(q_i)=
\begin{cases}
q_i,&i\neq a\\
q_a-\epsilon_\Delta q_cq_b,& i=a
\end{cases}\\
&\phi_2(t^\pm)=t^\pm,
\end{align*}
where $\epsilon_\Delta$ is the product of the signs at the corners of $\omega_\Delta$ on $\Lambda_0$.

\begin{prop}\label{Prop:typeIII2}
The map $\phi_2:\cA(\Lambda_0)\to \cA(\Lambda_1)$ defined above satisfies 
$$d_1\circ \phi_2(s)=\phi_2\circ d_0(s)$$
for all $s\in\cA(\Lambda_0)$, and 
\begin{align*}
&\{\phi_2 s_1,\phi_2 s_2\}_{d_1}=(\phi_2\otimes\phi_2)\{s_1,s_2\}_{d_0}
\end{align*}
for all $s_1,s_2\in\widetilde\cA(\Lambda_0)$.
\end{prop}
\begin{proof}
It is enough to show
\begin{align*}
&d_1\circ\phi_2(q_i)=\phi_2\circ d_0(q_i),\\
&\{\phi_2 q_i,\phi_2 q_j\}_{d_1}=(\phi_2\otimes\phi_2)\{q_i,q_j\}_{d_0},
\end{align*}
for all $i,j\in\{1,\dots,n\}$.

Let first $i\neq a$. Every index zero disk on $\Lambda_1$ with a positive puncture at $\gamma_i$ that does not exist on $\Lambda_0$ is of the form  $u'=w_1q_cq_bw_2p_i$, where $u=w_1q_aw_2p_i$ is the corresponding index zero disk on $\Lambda_0$ with a negative corner at $a$ that does not cover the small triangle (see Figure \ref{Figure:disk_reidIII}). Moreover, if the disk $u$ comes with a sign $\epsilon(u)$, then it is not difficult to see that $u'$ comes with a sign $-\epsilon_\Delta\epsilon(u)$. Similarly, every index zero disk on $\Lambda_0$ with a positive puncture at $\gamma_i$ that does not exist on $\Lambda_1$ is of the form  $\overline u'=w_1q_cq_bw_2p_i$, where $\overline u=w_1q_aw_2p_i$ is the corresponding index zero disk on $\Lambda_0$ with a negative corner at $a$ that covers the small triangle. If $\overline u$ comes with a sign $\epsilon(\overline u)$, then $\overline u'$ comes with a sign $\epsilon_\Delta\epsilon(\overline u)$. Analogous holds for annuli with one positive puncture at $i\neq a$ and disks with two positive punctures at $i,j\neq a$ different from $\omega_\Delta$. 

Additionally, it is not difficult to see $d_{1,f}(q_i)=d_{0,f}(q_i)$ for $i\neq a$, $d_{1,f}(q_i,q_j)=d_{0,f}(q_i,q_j)$ for $\{i,j\}\neq \{b,c\},i,j\neq a$ and 
\begin{align*}
&d_{1,f}(q_b,q_c)-d_{0,f}(q_b,q_c)=-(-1)^{|q_b||q_c|}\left(\delta_1(c^-,b^+)-\delta_0(c^-,b^+)\right)1\otimes q_cq_b=\\
&=(-1)^{|q_b||q_c|}\epsilon(a)1\otimes q_cq_b=(\phi_2-\operatorname{id})\{q_b,q_c\}_{\widetilde d},
\end{align*}
where $\{q_b,q_c\}_{\widetilde d}=-(-1)^{|q_b||q_c|}\epsilon(a)\epsilon_\Delta 1\otimes q_a$ is obtained by gluing disk $\omega_\Delta=q_ap_bp_c$. Similarly, 
\begin{align*}
&d_{1,f}(q_c,q_b)-d_{0,f}(q_c,q_b)=-\left(\delta_1(b^+,c^-)-\delta_0(b^+,c^-)\right)q_cq_b\otimes 1=\\
&=-\epsilon(a)q_cq_b\otimes 1=(\phi_2-\operatorname{id})\{q_c,q_b\}_{\widetilde d},
\end{align*} 
where $\{q_c,q_b\}_{\widetilde d}=\epsilon(a)\epsilon_\Delta q_a\otimes 1$. This proves $d_{1}\circ \phi_2(q_i)=\phi_2\circ d_0(q_i)$ and $\{\phi_2 q_i,\phi_2 q_j\}_{d_1}=(\phi_2\otimes\phi_2)\{q_i,q_j\}_{d_0}$ for $i,j\neq a$.

Let now $i=a$. Every disk $u'$ with a positive puncture at $a$ that appears on $\Lambda_0$ and not on $\Lambda_1$ is of the form $q_cwp_a$ or $wq_bp_a$, where $u=wp_b$, respectively $u=wp_c$ is an index zero disk on $\Lambda_0$ whose positive corner at $b$ or $c$ covers the small triangle. If $u= w p_c$ ($u= w p_b$) comes with a sign $\epsilon(u)$, then it is not difficult to see that $u'$ comes with a sign $-\epsilon_\Delta\epsilon(u)$ ($-\epsilon_\Delta\epsilon(u)(-1)^{|q_c|}$). Similarly, for every index zero disk $u=wp_b$ or $u=wp_c$ on $\Lambda_0$ whose positive corner at $b$ or $c$ does not cover the small triangle, we have a corresponding disk $u'=q_cwp_a$, respectively $u'=wq_bp_a$ on $\Lambda_1$ that does not appear on $\Lambda_0$. Moreover, if $u= wp_c$ ($u= wp_b$) comes with a sign $\epsilon(u)$, then $u'$ comes with a sign $\epsilon_\Delta\epsilon(u)$ ($\epsilon_\Delta\epsilon(u)(-1)^{|q_c|}$). Similar holds for disks with two positive punctures and annuli. Additionally, similar as in the previous proposition, for every disk $w_1p_bw_2p_c$ other than $\omega_\Delta$ with two positive punctures at $b$ and $c$, we have a disappearing/appearing annulus $\hbar(w_1p_a\otimes w_2)$ with a positive puncture at $a$. 
This can be written as
\begin{align*}
(d_{1,\lD}+d_{1,A})(q_a)-(d_{0,\lD}+d_{0,A})(q_a)=\epsilon_\Delta(d_{1,\lD}+d_{1,A})(q_cq_b)- \epsilon(a)\hbar(q_a\otimes 1),
\end{align*}
where $\epsilon(a)$ is the orientation sign of the arc $b^+c^-$ on the small triangle on $\Lambda_0$. 

Additionally, it is not difficult to check that
\begin{align*}
&d_{1,f}(q_a)-d_{0,f}(q_a)=\epsilon(a)\hbar(q_a\otimes 1),
\end{align*}
and
\begin{align*}
\epsilon_\Delta d_{1,f}(q_cq_b)=-(\phi_2-\operatorname{id}) d_{0,f}(q_a).
\end{align*}
This shows $d_1\circ \phi_2(q_a)=\phi_2\circ d_0(q_a)$. Proof of $\{\phi_2 q_i,\phi_2 q_j\}_{d_1}=(\phi_2\otimes\phi_2)\{q_i,q_j\}_{d_0}$ for $i$ or $j$ equal to $a$ goes similarly.
\end{proof}

\paragraph{\textit{Proof of Proposition \ref{Prop:eq_after_III}:}}
Follows from Proposition \ref{Prop:typeIII1}  and Proposition \ref{Prop:typeIII2}. The inverses of the morphisms above are given by
\begin{align*}
&\phi_1^{-1}(q_i)=q_i,\phi_1^{-1}(t^\pm)=t^\pm,\\
&\{q_i,q_j\}_{\phi_1^{-1}}=\begin{cases}
-(-1)^{|q_b|}\epsilon_\Delta q_a\otimes 1,&i=b,j=c\\
(-1)^{(|q_c|+1)(|q_a|+1)}\epsilon_\Delta 1\otimes q_a,&i=c,j=b\\
0,&\text{otherwise}
\end{cases}\\
&\{t^\pm,q_i\}_{\phi_1^{-1}}=0,\{q_i,t^\pm\}_{\phi_1^{-1}}=0,\{t^\pm,t^\pm\}_{\phi_1^{-1}}=0,
\end{align*} 
and 
\begin{align*}
&\phi_2^{-1}(q_i)=
\begin{cases}
q_i,&i\neq a\\
q_a+\epsilon_\Delta q_cq_b,& i=a
\end{cases}\\
&\phi_2^{-1}(t^\pm)=t^\pm.
\end{align*}

\subsection{Reidemeister II move}

Let $\Lambda_s,s\in[0,1]$ be a Legendrian knot isotopy with a Reidemeister II move at $s=\frac{1}{2}$ in the Lagrangian projection as shown in Figure \ref{Figure:ReidII}. Denote by $\gamma_1,\dots,\gamma_n$ the Reeb chords on $\Lambda_1$ and by $(\cA(\Lambda_\iota),d_\iota,\{\cdot,\cdot\}_{d_\iota})$ the second-order dga corresponding to $\Lambda_\iota,\iota\in\{0,1\}$. Denote the two disappearing chords on $\Lambda_0$ by $a$ and $b$, where the action of $a$ is larger than the action of $b$. This section is divided into two parts. First, we show the stabilization $(\cA^s,d^s,\{\cdot,\cdot\}_{d^s})$ of the second-order dga $(\cA(\Lambda_1),d_1,\{\cdot,\cdot\}_{d_1})$ in degree $|q_a|$ is quasi-isomorphic to $(\cA(\Lambda_1),d_1)$ as a chain complex. Then, we show that the stabilized second-order dga $(\cA^s,d^s,\{\cdot,\cdot\}_{d^s})$ is tame isomorphic to $(\cA(\Lambda_0),d_0,\{\cdot,\cdot\}_{d_0})$.

\subsubsection{Stabilizations}\label{Sec:stabilizationII}
In Section \ref{Sec:stabilizationI} we defined the notion of stabilization of a second-order dga. In this section we show a stabilization $(\cA,d^s,\{\cdot,\cdot\}_{d^s})$ of a second-order dga $(\cA,d,\{\cdot,\cdot\}_d)$ is quasi-isomorphic to $(\cA,d)$ as a chain complex. Additionally, we describe a sufficient condition for a stabilization of one second-order dga to be isomorphic to another second-order dga that will be used to show invariance under Reidemeister II move in the following section.

Denote by $C\subset\cA^s$ the subspace generated by words that contain at least one letter $q_a,q_b$.

\begin{lemma}\label{Lemma:Chain_homotopy}
There exists a linear map $h:\cA^s\to C\subset\cA^s$ such that
\begin{align*}
h\circ d^s+d^s\circ h=\operatorname{id}-\tau,
\end{align*}
where $\tau:\cA^s\cong\cA\oplus C\to\cA\subset \cA^s$ is the projection. In particular, the chain complexes $(\cA,d)$ and $(\cA^s,d^s)$ are quasi-isomorphic.
\end{lemma}

\begin{proof}
We define a (first order) derivation $\widehat h:\cA^s\to\cA^s$ by
\begin{align*}
&\widehat h(q_i)=
\begin{cases}
0,&i\neq b\\
q_a,&i=b
\end{cases}
\end{align*}
and a linear map
\begin{align*}
&h(w)=\begin{cases}
\frac{1}{n(w)}\widehat h(w), &n(w)\neq 0\\
0, &n(w)=0
\end{cases}
\end{align*}
where $n(w)$ is the number of appearances of $q_a$ and $q_b$ in word $w$. Obviously, we have $h(s)\in C$ for all $s\in\cA^s$.

We show that $h\circ d^s(s)+d^s\circ h(s)=s-\tau(s)$, or equivalently, $\widehat h\circ d^s(s)+d^s\circ \widehat h(s)=n(s)s$ for all words $s\in\cA^s$. It is easy to see that this holds for $s=q_a,q_b$
\begin{align*}
\widehat h\circ d^s(q_a)+d^s\circ \widehat h(q_a)=\widehat h(q_b)=q_a=n(q_a)q_a,\\
\widehat h\circ d^s(q_b)+d^s\circ \widehat h(q_b)=d^s(q_a)=q_b=n(q_b)q_b.
\end{align*}
We have $d^s(\cA)\subset\cA\subset\ker \widehat h$, and therefore
\begin{align*}
\widehat h\circ d^s(s)+d^s\circ \widehat h(s)=0=n(s)s
\end{align*}
for any word $s\in\cA$. It is not difficult to see that 
\begin{equation}
\label{Eq:homotopy_bracket}
\begin{aligned}
(\widehat h\otimes 1+1\otimes \widehat h)\{u,v\}_{d^s}=\{\widehat h u,v\}_{d^s}+(-1)^{|u|}\{u,\widehat h v\}_{d^s}
\end{aligned}
\end{equation} 
for all $u,v\in\cA^s$. First, we have $(\widehat h\otimes 1+1\otimes \widehat h)\{q_\iota,q_\kappa\}_{d^s}=0=\{\widehat h q_\iota,q_\kappa\}_{d^s}+(-1)^{|q_\iota|}\{q_\iota,\widehat h q_\kappa\}_{d^s}$ when $\iota$ or $\kappa$ is equal to $a$ or $b$, which follows trivially from $\{s_1,s_2\}_{d^s}=0$ for $s_1$ or $s_2$ equal to $q_a$ or $q_b$. Additionally, $(\widehat h\otimes 1+1\otimes\widehat h)\{s_1,s_2\}_{d^s}=0=\{\widehat h s_1,s_2\}_{d^s}+(-1)^{|s_1|}\{s_1,\widehat h s_2\}_{d^s}$ follows trivially for $s_1,s_2\in\cA$. This implies (\ref{Eq:homotopy_bracket}) for all $u,v\in\cA^s$ using the properties of the antibracket. Using (\ref{Eq:homotopy_bracket}), we get
\begin{align*}
&\widehat h\circ d^s(uv)+d^s\circ \widehat h(uv)=\left(\widehat hd^s(u)+d^s\widehat h(u)\right)v+u\left(\widehat hd^s(v)+d^s\widehat h(v)\right)
\end{align*}
for any $u,v\in\cA^s$, so the statement follows by induction on the length of the word. Moreover, for $\sigma=u\otimes v\in\widetilde\cA^s\otimes\widetilde\cA^{s}$ we have
\begin{align*}
&(\widehat h\otimes 1+1\otimes \widehat h)(d^s\otimes 1+1\otimes d^s)\sigma+(d^s\otimes 1+1\otimes d^s)(\widehat h\otimes 1+1\otimes \widehat h)\sigma=\\
=&\widehat hd^s(u)\otimes v+d^s\widehat h(u)\otimes v+u\otimes \widehat hd^s(v)+u\otimes d^s\widehat h(v)=\\
=&(n(u)+n(v))\sigma=n(\sigma)\sigma,
\end{align*}
which finishes the proof.
\end{proof}

Recall the notion of action on $\cA(\Lambda)$. Let $l(q_i)\in\lR_{>0}$ be the length of the Reeb chord $\gamma_i$. For any word $w=t^{\pm j_0}q_{i_1}t^{\pm j_1}\dots t^{\pm j_{k-1}}q_{i_k}t^{\pm j_k}\in\cA$, we define the action of $w$ as
$$l(w)=\sum_{j=1}^k l(q_{i_j}),$$
and similarly for $w=\hbar(w_1\otimes w_2)$
$$l\left(\hbar(w_1\otimes w_2)\right)=l(w_1)+l(w_2).$$
Additionally, we define 
$$l\left(\sum_{i=1}^k a_iw_i\right)=\max_{i=1,\dots, k}l(w_i)$$ 
for $w_i\in\cA$ generators and $a_i\in\lQ, a_i\neq 0,i=1,\dots,k$. For $(\cA^s,d^s,\{\cdot,\cdot\}_{d^s})$ a stabilization of $(\cA,d,\{\cdot,\cdot\}_d)$ and $l_a,l_b\in\lR_{>0}$, we extend the action on $\cA$ to $\cA^s$ as above by taking $l(q_a)=l_a,l(q_b)=l_b$. For a stabilization corresponding to a Reidemeister II move, $|l_a-l_b|$ is small.
\vspace{3.1mm}

Let $W\coloneq\{s=q_i,t^\pm q_i,q_i t^\pm,q_\iota q_\kappa\,|\,i=1,\dots,n;\iota,\kappa=1,\dots,n,a,b;l(s)>l(q_a)\}$. We order the words $s_1,\dots,s_k$ in $W$ by their action
\begin{align*}
&l(s_1)\leq l(s_2)\leq \dots\leq l(s_k),
\end{align*}
additionally requiring that $q_i$ comes before $t^\pm q_i,q_i t^\pm$. For $\Lambda$ generic, we can assume that the inequalities $l(s_i)\leq l(s_{i+1})$ are strict except for $s_i=q_\iota q_\kappa,s_{i+1}=q_\kappa q_\iota$ for some $\iota,\kappa$ and $\{s_i,s_{i+1}\}\subset\{q_j,q_jt^\pm,t^\pm q_j\}$ for some $j$. Denote by $L_i$ and $P_i,i\in\{0,\dots,k\}$ the subsets of $\{s_1,\dots,s_i\}$ consisting of words of length 1 and 2, respectively, such that 
$$L_i\sqcup P_i=\{s_1,\dots,s_i\}.$$

\begin{rmk} Since $0<l(q_a)-l(q_b)$ is small, the map $\pi_{\widetilde\cA^s}\circ d^s\circ\iota_{\widetilde\cA^s}$ decreases the action by at least $l(q_a)-l(q_b)$. The map $h$ from Lemma \ref{Lemma:Chain_homotopy} increases the action by at most $l(q_a)-l(q_b)$.
\end{rmk}

\begin{defi}
We say a linear map $f:\cA^s\to\cA'$ is \textit{weakly filtered} if 
\begin{align*}
&l(\pi_{\widetilde\cA}\circ f(s))\leq l(s),\\
&l(f(s))\leq l(s)+l(q_a)-l(q_b),
\end{align*}
and $f(\hbar\,(\widetilde\cA\otimes\widetilde\cA^{\cyc}))\subset\hbar(\widetilde\cA'\otimes\widetilde\cA'^{\cyc})$.
\end{defi}

The following lemma is the main ingredient in the proof of invariance under Reidemeister II move.

\begin{lemma}\label{Lemma:bootstrap}
Let $(\cA_0,d^0,\{\cdot,\cdot\}_{d^0}),(\cA_1,d^1,\{\cdot,\cdot\}_{d^1})$ be second-order dg algebras associated to Legendrian knots $\Lambda_0,\Lambda_1$ close enough to Reidemeister II degeneration as before, and let $(\cA^s,d^s,\{\cdot,\cdot\}_{d^s})$ be the stabilization of $(\cA_1,d^1,\{\cdot,\cdot\}_{d^1})$ in degree $|q_a|$. Assume there exists a weakly filtered tame second-order graded algebra isomorphism $\phi:\cA^s\to\cA_0$ with a weakly filtered inverse such that for the second-order differential
$$\widehat d=\phi^{-1}\circ d^0\circ \phi$$
and the corresponding antibracket $\{\cdot,\cdot\}_{\widehat d}$, we have
\begin{equation}\label{Eq:bootstrapping}
\begin{aligned}
&\tau\circ\widehat d(s)=\tau\circ d^s(s),\\
&(\tau\otimes \tau)\{s_1,s_2\}_{\widehat d}=(\tau\otimes \tau)\{s_1,s_2\}_{d^s},
\end{aligned}
\end{equation} 
for all $s\in\cA^s,s_1,s_2\in\widetilde\cA^s$, where $\tau:\cA^s\cong\cA_1\oplus  C\to\cA_1$ is the projection. Assume additionally that $\widehat d(s)=d^s(s),\{s_1,s_2\}_{\widehat d}=\{s_1,s_2\}_{d^s}$ whenever $l(s)\leq l(q_a),l(s_1s_2)\leq l(q_a)$. Then there exists a tame second-order dg algebra isomorphism $\Phi:(\cA^s,d^s,\{\cdot,\cdot\}_{d^s})\to(\cA_0,d^0,\{\cdot,\cdot\}_{d^0})$.
\end{lemma}

\begin{proof}
Let $s_i,i=1,\dots,k$ be words in $W$ ordered as above. We construct inductively a sequence of second-order graded algebra morphisms $\phi_i,i\in\{0,1,\dots,k\}$ and second-order differentials $d_i:\cA^s\to\cA^s,i\in\{-1,0,1,\dots,k\}$ (together with antibrackets $\{\cdot,\cdot\}_{d_i}$)
\begin{align*}
&d_{-1}=d^0,\\
&\phi_0=\phi,
\end{align*}
such that $d_i$ satisfies property (\ref{Eq:diff_action_property}), for $i\geq 0$
\begin{align*}
&d_i=\phi_i^{-1}\circ d_{i-1}\circ \phi_i,\\
&\tau\circ d_i=\tau\circ d^s,\\
&(\tau\otimes \tau)\{\cdot,\cdot\}_{d_i}=(\tau\otimes\tau)\{\cdot,\cdot\}_{d^s},\\
&d_i(s)=d^s(s)\text{ for }s\in L_i,\\
&\{s_1,s_2\}_{d_i}=\{s_1,s_2\}_{d^s}\text{ for }s_1s_2\in P_i,
\end{align*}
and such that $d_i(s)=d^s(s),\{s_1,s_2\}_{d_i}=\{s_1,s_2\}_{d^s}$ whenever $l(s)\leq l(q_a),l(s_1s_2)\leq l(q_a)$.

The proof goes by induction. By assumption, $d_0=\widehat d$ and $\phi_0=\phi$ satisfy the conditions above. The fact that $\phi,\phi^{-1}$ are weakly filtered implies that $\widehat d$ satisfies (\ref{Eq:diff_action_property}). Assume now we have constructed $\phi_j,d_j$ for $j<i$ such that the conditions above are satisfied.

First, assume $s_i=q_\iota$ ($\iota\neq a,b$). The construction is done in two steps. Consider the graded algebra morphism $\phi_i':\cA^s\to\cA^s$ given by
\begin{align*}
\phi_i'(q_j)=\begin{cases}
q_j,&j\neq\iota\\
q_\iota+h(d^sq_\iota-d_{i-1}q_\iota),&j=\iota
\end{cases}
\end{align*}
and $\phi_i'(t^\pm)=t^\pm$. Summands in $d^sq_\iota-d_{i-1}q_\iota$ are either of action $< l(q_\iota)-2l(q_a)+2l(q_b)$ or of the form $\hbar(q_\iota\otimes 1),\hbar(1\otimes q_\iota)$ (since $\iota\neq a,b$). Therefore, all the summands in $h(d^sq_\iota-d_{i-1}q_\iota)$ are of action smaller than $l(q_\iota)$, using $h(\hbar(q_\iota\otimes 1))=0=h(\hbar(1\otimes q_\iota))$. It follows that $\phi_i'$ is filtered. By Lemma \ref{Lemma:inverse_existence_I_order}, $\phi_i'$ is invertible with a filtered inverse. The second-order differential $d_{i}'=\phi_i'^{-1}\circ d_{i-1}\circ \phi_i'$ therefore satisfies (\ref{Eq:diff_action_property}). 
 
All summands in $d_{i-1}(s_j)$ for $j<i$ are of action $\leq l(s_j)<l(q_\iota)$, therefore, for all $j<i$ we have
\begin{align*}
d_i'(s_j)= \phi_i'^{-1}\circ d_{i-1}(s_j)=d_{i-1}(s_j)=d^s(s_j).
\end{align*}
Additionally, $\{s_1,s_2\}_{d_i'}=(\phi_i'^{-1}\otimes \phi_i'^{-1})\{\phi_i's_1,\phi_i's_2\}_{d_{i-1}}=\{s_1,s_2\}_{d^s}$ trivially follows for any $s_1s_2\in P_i=P_{i-1}$. Similarly, it is easy to see $d_i'(q_i)=d^s(q_i),\{q_j,q_k\}_{d_i'}=\{q_j,q_k\}_{d^s}$ whenever $l(q_i)\leq l(q_a),l(q_jq_k)\leq l(q_a)$.

For $w_1,w_2\in\cA^s$, we say $w_1\sim w_2$ if $\pi_{\widetilde\cA^s}(w_1-w_2)=0$. Next we show $d_i'(q_\iota)\sim d^s(q_\iota)$. First, notice that
\begin{align*}
d_{i-1}\left(h(d^sq_\iota-d_{i-1}q_\iota)\right)=d^s(h(d^sq_\iota-d_{i-1}q_\iota)),
\end{align*}
since $l(\sigma)<l(q_\iota)$ for all summands $\sigma$ in $h(d^sq_\iota-d_{i-1}q_\iota)$. Using Lemma \ref{Lemma:Chain_homotopy} and $\tau\circ d^s=\tau\circ d_{i-1}$ we get
\begin{align*}
&d^s(h(d^sq_\iota-d_{i-1}q_\iota))=d^sq_\iota-d_{i-1}q_\iota-h(d^s(d^sq_\iota-d_{i-1}q_\iota))=\\
=&d^sq_\iota-d_{i-1}q_\iota+h(d^s\circ d_{i-1}q_\iota).
\end{align*}
Since $d_{i-1}$ satisfies (\ref{Eq:diff_action_property}), we get $d^s(d_{i-1}q_\iota)\sim d_{i-1}^2 q_\iota= 0$, and therefore
\begin{align*}
d_{i-1}(h(d^sq_\iota-d_{i-1}q_\iota))\sim d^sq_\iota-d_{i-1}q_\iota.
\end{align*}
Now we have
\begin{align*}
d_i'(q_\iota)=&\phi_i'^{-1}\circ d_{i-1}\circ \phi_i'(q_\iota)=\\
=&\phi_i'^{-1}(d_{i-1}q_\iota+d_{i-1}(h(d^sq_\iota-d_{i-1}q_\iota)))\sim\\
\sim&\phi_i'^{-1}\circ d^s (q_\iota).
\end{align*}

Finally, we have $\phi_i'^{-1}\circ d^s(q_\iota)\sim d^s(q_\iota)$ since $d^s$ satisfies (\ref{Eq:diff_action_property}). More precisely, $d^sq_\iota= m\hbar(q_\iota\otimes 1)+n \hbar(1\otimes q_\iota) + \widetilde S$ for some $m,n\in\lZ$ and $\widetilde S\in\cA^s$ such that $l(\widetilde S)<l(q_\iota)$, for which we have $\phi_i'^{-1}(\widetilde S)=\widetilde S$.

Moreover, for all $s\in\cA^s,s_1,s_2\in\widetilde\cA^s$ we show
\begin{align*}
&\tau\circ d_{i}'(s)=\tau\circ d^s(s),\\
&(\tau\otimes \tau)\{s_1,s_2\}_{d_i'}=(\tau\otimes\tau)\{s_1,s_2\}_{d^s}.
\end{align*} 
It is enough to prove $\tau(d_i'(s)-d_{i-1}(s))=0$ and $(\tau\otimes\tau)(\{s_1,s_2\}_{d_i'}-\{s_1,s_2\}_{d_{i-1}})=0$. Since $d^s(C)\subset C$, using $\tau(d^s-d_{i-1})=0$ we get $d_{i-1}(C)\subset C$. It is easy to see from the definition that the images of $\phi_i'-\operatorname{id}$ and $\phi_i'^{-1}-\operatorname{id}$ are in $C$ since the image of $h$ is in $C$. Then $\tau(d_i'(s)-d_{i-1}(s))=0$ follows trivially. Similarly, for $x$ or $y$ in $C$ we have $\{x,y\}_{d^s}\in C$, which implies $\{x,y\}_{d_{i-1}}\in C$. Together with $\phi_i'(s)-s\in C$ and $\phi_i'^{-1}(s)-s\in C,\forall s\in\cA^s$, we get $(\tau\otimes\tau)\{s_1,s_2\}_{d_i'}=(\tau\circ \phi_i'^{-1}\otimes\tau\circ \phi_i'^{-1})\{\phi_i's_1,\phi_i's_2\}_{d_{i-1}}=(\tau\otimes \tau)\{s_1,s_2\}_{d_{i-1}}$.


In the second step, we consider the graded algebra morphism $\phi_i'':\cA^s\to\cA^s$ given by
\begin{align*}
\phi_i''(q_j)=\begin{cases}
q_j,&j\neq \iota\\
q_\iota+h(d^sq_\iota-d_{i}'q_\iota),&j=\iota
\end{cases}
\end{align*}
and define $\phi_i=\phi_i'\circ\phi_i''$. Similar as above, all the summands in $h(d^sq_\iota-d_i'q_\iota)$ are of action $<l(q_\iota)$. Then, $\phi_i''$ is filtered with a filtered inverse, and  $d_{i}\coloneq\phi_i''^{-1}\circ d_i'\circ \phi_i''=\phi_i^{-1}\circ d_{i-1}\circ \phi_i$ satisfies (\ref{Eq:diff_action_property}). Moreover, $\pi_{\widetilde\cA^s}\circ h(d^sq_\iota-d_i'q_\iota)=0$ since $d^s q_\iota\sim d_{i}'q_\iota$.

All summands in $d_i'(s_j),j<i$ are of action $\leq l(s_j)<l(q_\iota)$, therefore, for all $j<i$ we have
\begin{align*}
d_i(s_j)= \phi_i''^{-1}\circ d_{i}'(s_j)=d_{i}'(s_j)=d^s(s_j).
\end{align*}
Additionally, $\{s_1,s_2\}_{d_i}=(\phi_i''^{-1}\otimes\phi_i''^{-1})\{\phi_i''s_1,\phi_i''s_2\}_{d_{i}'}=\{s_1,s_2\}_{d^s}$ trivially follows for any $s_1s_2\in P_i=P_{i-1}$ since $l(s_1s_2)<l(q_\iota)$. Similarly, it is easy to see $d_i(s)=d^s(s),\{s_1,s_2\}_{d_i}=\{s_1,s_2\}_{d^s}$ whenever $l(s)\leq l(q_a),l(s_1s_2)\leq l(q_a)$.

Next, we show $d_i(q_\iota)=d^s(q_\iota)$. Since $l(h(d^sq_\iota-d_i'q_\iota))<l(q_\iota)$, we have $d_i'(h(d^sq_\iota-d_i'q_\iota))=d^s(h(d^sq_\iota-d_i'q_\iota))$.  Moreover, since $\pi_{\widetilde\cA^s}(d_i'(q_\iota)-d^s(q_\iota))=0$, we have $d^s(\hbar(1\otimes q_\iota))=d_i'(\hbar(1\otimes q_\iota)),d^s(\hbar(q_\iota\otimes 1))=d_i'(\hbar(q_\iota\otimes 1))$. This implies $d^s\circ d_i'(q_\iota)=d_i'\circ d_i'(q_\iota)=0$. Now, we get 
\begin{align*}
&d_{i}'(h(d^sq_\iota-d_{i}'q_\iota))=d^s(h(d^sq_\iota-d_{i}'q_\iota))=\\
=&d^sq_\iota-d_{i}'q_\iota
\end{align*} 
using Lemma \ref{Lemma:Chain_homotopy} and $\tau(d^s-d_i')=0$. Similarly, using $\pi_{\widetilde\cA^s}(\phi_i''(q_\iota)-q_\iota)=0$ together with the fact that $d^s$ satisfies (\ref{Eq:diff_action_property}), we show $\phi_i''^{-1}\circ d^s(q_\iota)=d^s(q_\iota)$. This implies,
\begin{align*}
d_i(q_\iota)=&\phi_i''^{-1}\circ d_{i}'\circ \phi_i''(q_\iota)=\\
=&\phi_i''^{-1}(d_{i}'(q_\iota)+d_{i}'(h(d^sq_\iota-d_{i}'q_\iota)))=\\
=&\phi_i''^{-1}\circ d^s(q_\iota)=\\
=&d^s(q_\iota).
\end{align*}

Moreover, we show
\begin{align*}
&\tau\circ d_{i}(s)=\tau\circ d^s(s),\\
&(\tau\otimes\tau)\{s_1,s_2\}_{d_i}=(\tau\otimes\tau)\{s_1,s_2\}_{d^s},
\end{align*}
same as before for all $s\in\cA^s,s_1,s_2\in\widetilde\cA^s$.

Next, we consider the case $s_i=q_\iota q_\kappa$ ($t^\pm q_j,q_j t^\pm$ goes similarly). Note that here we can have $\iota$ or $\kappa$ equal to $a,b$. We define a second-order graded algebra morphism $\phi_i$ as follows
\begin{align*}
&\phi_i(q_i)=q_i,i=1,\dots,n,a,b,\\
&\phi_i(t^\pm)=t^\pm,\\
&\{q_i, q_j\}_{\phi_i}=
\begin{cases}
(h\otimes 1+1\otimes h)(\{q_\iota, q_\kappa\}_{d_{i-1}}-\{q_\iota, q_\kappa\}_{d^s}),&i=\iota,j=\kappa\\
0,&\text{otherwise}
\end{cases}
\end{align*}
We have $d_{i-1}(s)=d^s(s)$ for $s=q_\iota,q_\kappa,q_a,q_b$, as well as for any $q_j$ such that $l(q_j)<l(q_\iota q_\kappa)$. This implies
\begin{align}
\label{Eq:higher_action_works}
(d^s\otimes 1+1\otimes d^s)\sigma=(d_{i-1}\otimes 1+1\otimes d_{i-1})\sigma
\end{align}
for all summands $\sigma$ in $\{q_\iota, q_\kappa\}_{d_{i-1}},\{q_\iota, q_\kappa\}_{d^s}$ and $\{q_\iota,q_\kappa\}_{\phi_i}\in\widetilde\cA\otimes\widetilde\cA$. Additionally, we have $\{s_1,s_2\}_{d^s}=\{s_1,s_2\}_{d_{i-1}}$ for all $s_1,s_2\in\cA^s$ such that $l(s_1s_2)<l(q_\iota q_\kappa)$. All summands $\sigma$ in $(h\otimes 1+1\otimes h)\left(\{q_\iota, q_\kappa\}_{d_{i-1}}-\{q_\iota, q_\kappa\}_{d^s}\right)$ have action smaller than or equal to $l(q_\iota q_\kappa)+l(q_a)-l(q_b)$. It follows that $\phi_i$ is weakly filtered. Moreover, by Lemma \ref{Lemma:inverse_existence}, it is invertible and the inverse is weakly filtered. 
 
We define second-order derivation $d_i=\phi_i^{-1}\circ d_{i-1}\circ \phi_i$ (with respect to the antibracket $\{v,w\}_{d_i}=-(\phi_i^{-1}\otimes \phi_i^{-1})(d_{i-1}\otimes 1+1\otimes d_{i-1})\{v,w\}_{\phi_{i}}+(\phi_i^{-1}\otimes \phi_i^{-1})\{\phi_i v,\phi_i w\}_{d_{i-1}}+\{d_{i-1}\phi_i v,\phi_i w\}_{\phi_i^{-1}}+(-1)^{|v|}\{\phi_i v,d_{i-1}\phi_i w\}_{\phi_i^{-1}}$, see Lemma \ref{Lemma:Second_order_properties}).

Maps $(\phi_i-\operatorname{id})=\pi_{\widetilde\cA\otimes\widetilde\cA^{\cyc}}\circ(\phi_i-\operatorname{id}),(\phi_i^{-1}-\operatorname{id})=(\phi_i^{-1}-\operatorname{id})\circ\pi_{\widetilde\cA^s}$ increase the action by at most $l(q_a)-l(q_b)$ and $\pi_{\widetilde\cA^s}\circ d_{i-1},d_{i-1}\circ\pi_{\widetilde\cA\otimes\widetilde\cA^{\cyc}}$ decrease it by at least $l(q_a)-l(q_b)$, therefore, $d_i=d_{i-1}+(\phi_i^{-1}-\operatorname{id})\circ d_{i-1}+d_{i-1}\circ(\phi_i-\operatorname{id})$ satisfies (\ref{Eq:diff_action_property}).

It is easy to see that $(\phi_i^{-1}\otimes \phi_i^{-1})(v\otimes w)=v\otimes w$ for all $v,w$ and $\{s_1,s_2\}_{\phi_i^{-1}}=0$ whenever $l(s_1s_2)<l(q_\iota q_\kappa)$. From this and the fact that $\pi_{\widetilde\cA}\circ d_{i-1}$ decreases the action, we conclude $\{q_\iota,q_\kappa\}_{d_i}=-(d^s\otimes 1+1\otimes d^s)\{q_\iota,q_\kappa\}_{\phi_{i}}+\{ q_\iota, q_\kappa\}_{d_{i-1}}$. Using Lemma \ref{Lemma:Chain_homotopy}, $(\tau\otimes\tau)(\{\cdot,\cdot\}_{d^s}-\{\cdot,\cdot\}_{d_{i-1}})=0$, and the observations above, we get
\begin{align*}
-&(d^s\otimes 1+1\otimes d^s)\{q_\iota,q_\kappa\}_{\phi_{i}}-\{q_\iota, q_\kappa\}_{d^s}+\{q_\iota, q_\kappa\}_{d_{i-1}}= \\
=&(d^s\otimes 1+1\otimes d^s)\circ (h\otimes 1+1\otimes h)\left(\{q_\iota, q_\kappa\}_{d^s}-\{q_\iota, q_\kappa\}_{d_{i-1}}\right)-\\
-&\{q_\iota, q_\kappa\}_{d^s}+\{q_\iota, q_\kappa\}_{d_{i-1}}=\\
=&-(h\otimes 1+1\otimes h)\circ (d^s\otimes 1+1\otimes d^s)(\{q_\iota, q_\kappa\}_{d^s}-\{q_\iota, q_\kappa\}_{d_{i-1}})=\\
=&-(h\otimes 1+1\otimes h)((d^s\otimes 1+1\otimes d^s)\{q_\iota, q_\kappa\}_{d^s}-(d_{i-1}\otimes 1+1\otimes d_{i-1})\{q_\iota, q_\kappa\}_{d_{i-1}})=\\
=&-(h\otimes 1+1\otimes h)(\{d^sq_\iota, q_\kappa\}_{d^s}+(-1)^{|q_\iota|}\{q_\iota, d^sq_\kappa\}_{d^s}-\{d^sq_\iota, q_\kappa\}_{d_{i-1}}-(-1)^{|q_\iota|}\{q_\iota, d^sq_\kappa\}_{d_{i-1}})=\\
=&0.
\end{align*}

It follows that $\{q_\iota,q_\kappa\}_{d_i}=\{q_\iota,q_\kappa\}_{d^s}$. The last equality follows from the fact that $\pi_{\widetilde\cA^s}\circ d^s$ decreases the action.

For all $s_1s_2\in P_{i-1}$, we have
\begin{align*}
\{s_1,s_2\}_{d_i}=&\{s_1,s_2\}_{d_{i-1}}+\{d_{i-1}s_1,s_2\}_{\phi_i^{-1}}+(-1)^{|s_1|}\{s_1,d_{i-1}s_2\}_{\phi_i^{-1}}=\\
=&\{s_1,s_2\}_{d_{i-1}}=\{s_1,s_2\}_{d^s}.
\end{align*}
Additionally, for all $s\in L_i=L_{i-1}$ we have
\begin{align*}
d_i(s)=\phi_i^{-1}\circ d_{i-1}(s)=d_{i-1}(s)=d^s(s).
\end{align*}
Similarly, $d_i(q_i)=d^s(q_i),\{q_j,q_k\}_{d_i}=\{q_j,q_k\}_{d^s}$ whenever $l(q_i)\leq l(q_a),l(q_jq_k)\leq l(q_a)$.

We show $\tau\circ d_{i}(s)=\tau\circ d^s(s),\forall s\in\cA^s$ and $(\tau\otimes\tau)\{s_1,s_2\}_{d_{i}}=(\tau\otimes\tau)\{s_1,s_2\}_{d^s},\forall s_1,s_2\in\widetilde\cA^s$ as before using $\operatorname{Im}(\phi_i-\operatorname{id}),\operatorname{Im}(\phi_i^{-1}-\operatorname{id})\subset C$.

Finally, we define a second-order graded algebra isomorphism $\Phi=\phi\circ\phi_1\circ\dots\circ\phi_k:\cA^s\to\cA_0$. Then by construction, $D\coloneq\Phi^{-1}\circ d^0\circ\Phi$ is a second-order differential such that $D(s)=d_k(s)=d^s(s)$ and $\{s_1,s_2\}_D=\{s_1,s_2\}_{d_k}=\{s_1,s_2\}_{d^s}$ for all length 1 words $s,s_1,s_2$. Using Lemma \ref{Lemma:Second_order_derivation}, we get 
$$\Phi^{-1}\circ d_0\circ\Phi(s)=d^s(s)$$
for all $s\in\cA^s$, which finishes the proof.
\end{proof}

\subsubsection{Invariance under Reidemeister II move}
\label{Section:Invar_II_degeneration}

Let $\Lambda_s,s\in[0,1]$ be a Legendrian knot isotopy with a Reidemeister II move as before. Denote the "small" disappearing bigon with a positive puncture at $a$ and a negative puncture at $b$ by $\widetilde\omega=q_bp_a$. 

\begin{defi}\label{Def:admissible_ReidII}
We say a Reidemeister II move is \textit{admissible} if there are no index zero disks on $\lR\times \Lambda_0$ with two positive punctures, where both of them are at $a$. 
\end{defi}

In this section we show invariance under admissible Reidemeister II moves. For any two knots obtained by taking front resolutions of Legendrian isotopic knots, there is an isotopy that does not pass through non-admissible Reidemeister II moves (see Remark \ref{Remark:typeII_extra_assumption} below). 

We denote by $(\cA^s,d^s,\{\cdot,\cdot\}_{d^s})$ the stabilization of $(\cA(\Lambda_1),d_1,\{\cdot,\cdot\}_{d_1})$ in degree $|q_a|$ as before. There is an obvious algebra isomorphism $\cA^s\cong\cA(\Lambda_0)$. The main goal of this section is to prove the following proposition.

\begin{prop}\label{Prop:eq_after_II}
Let $\Lambda_s,s\in[0,1]$ be an admissible Reidemeister II move as above, then the second-order dg algebras $(\cA(\Lambda_0),d_0,\{\cdot,\cdot\}_{d_0})$ and $(\cA^s,d^s,\{\cdot,\cdot\}_{d^s})$ are tame isomorphic. In particular, 
$$H_*(\cA(\Lambda_0),d_0)\cong H_*(\cA(\Lambda_1),d_1).$$
\end{prop}

We define a second-order algebra morphism $\phi_0:\cA^s\to\cA(\Lambda_0)\cong\cA^s$ by taking
\begin{align*}
&\phi_0(q_i)=
\begin{cases}
q_i,&i\neq b\\
\epsilon(-1)^{|q_a|} q_b+\omega,&i=b
\end{cases}\\
&\phi_0(t^\pm)=t^\pm
\end{align*}
where $d_0(q_a)=\epsilon(-1)^{|q_a|}q_b+\omega$,
\begin{align*}
&\{q_i,q_j\}_{\phi_0}=\begin{cases}
\{q_a,q_j\}_{d_0},&i=b,j\neq b\\
(-1)^{|q_i|}\{q_i,q_a\}_{d_0},&i\neq b,j=b\\
\frac{1}{2}\{q_a,d_0(q_a)\}_{d_0}-\frac{(-1)^{|q_a|}}{2}\{d_0(q_a),q_a\}_{d_0},&i=j=b
\end{cases}\\
&\{t^\pm,q_b\}_{\phi_0}=\{t^\pm,q_a\}_{d_0},\\
&\{q_b,t^\pm\}_{\phi_0}=\{q_a,t^\pm\}_{d_0},
\end{align*}
and zero otherwise. Note that the product of the signs at the corners of $\widetilde\omega$ is equal to $(-1)^{|q_a|}$, and $\epsilon$ is the orientation sign of the arc $a^+b^+$ with respect to the orientation on $\Lambda_0$. Summands in $\omega=d_0(q_a)-\epsilon(-1)^{|q_a|}q_b$ do not contain $q_b$.

Our goal is to show that the morphism $\phi_0$ satisfies the properties of Lemma \ref{Lemma:bootstrap}.

\begin{rmk}\label{Remark:pairing_t}
For any $i,j,k\in\{1,\dots,n\}$ such that $l(q_i),l(q_jq_k)<l(q_b)$, we have $d_0(q_i)=d_1(q_i),\{q_j,q_k\}_{d_0}=\{q_j,q_k\}_{d_1}$. Additionally, by definition we have
\begin{align*}
\phi_0\circ d^s(s)=d_0\circ \phi_0(s)
\end{align*}
for $s\in\{q_a,q_b,t^\pm\}$, and 
\begin{align*}
&\{d^ss_1,s_2\}_{\phi_0}+(-1)^{|s_1|}\{s_1,d^ss_2\}_{\phi_0}+(\phi_0\otimes\phi_0)\{s_1,s_2\}_{d^s}=\\
=&\{\phi_0s_1,\phi_0s_2\}_{d_0}-(d_0\otimes 1+1\otimes d_0)\{s_1,s_2\}_{\phi_0}
\end{align*} 
for $s_1=t^\pm, s_2\in\{q_a,q_b\}$ and $s_1\in\{q_a,q_b\},s_2=t^\pm$.
\end{rmk}

\begin{rmk}\label{Remark:typeII_extra_assumption}
We show that there exists an admissible isotopy between the front resolutions of any two Legendrian isotopic knots. Let $\Lambda_s,s\in[0,1]$ be a Legendrian knot isotopy that in the front projection has one Reidemeister II degeneration (see \cite{Ng03_front}) where a right or a left cusp crosses some arc. 
Denote the newly created chords on $\Lambda_1$ by $a$ and $b$, with $l(a)>l(b)$. We say an index zero disk after the move is \textit{inadmissible} if it has two positive punctures, both of them at the left quadrant at $a$ if we have a right cusp and at the right quadrant at $a$ if we have a left cusp.\\ 
Assume first that a right cusp crosses either over or under. We notice that any inadmissible disk $u$ would have to have at least one more positive puncture different from $a$, since the two positive punctures at $a$ are in the left quadrant. Therefore, in this case there exists no inadmissible disk with two positive punctures. It follows that the corresponding isotopy of the front resolutions is admissible.
\\
Next, assume that a left cusp crosses an arc and that there exists an inadmissible disk $u$ in the front resolution of $\Lambda_1$. Then we can change the isotopy into an admissible isotopy as follows. Since $u$ is of index zero, i.e. the rotation number of its boundary is one, and it has two positive corners at $a$ in the right quadrant, then there exists a point $s_0$ on $\partial u$ such that $u$ has a vertical tangency at $s_0$ and such that $u$ is on the right side of its boundary in a neighborhood of $s_0$. In other words, $\partial u$ passes through a right cusp from the right in the front projection. Fix a path on $u$ from $s_0$ to another boundary point that separates the two positive punctures on $u$. We change the original non-admissible isotopy in a way such that, first, we move the cusp $s_0$ along the chosen path, crossing the boundary of $u$. Since $s_0$ is a right cusp, this isotopy is admissible. We can repeat this step until all inadmissible disks in the original isotopy are gone. Afterwards, we cross with the original left cusp. Finally, we cross back with all the right cusps the same way as in the first step.\\
Note additionally that the Reidemeister I degeneration in the front projection gives us an admissible Reidemeister II move in the Lagrangian projection after resolution. Here, there are no index zero disks with two positive punctures at $a$.
\end{rmk}

\begin{lemma}
Map $\phi_0:\cA^s\to\cA(\Lambda_0)\cong\cA^s$ is invertible and weakly filtered.
\end{lemma}
\begin{proof}
It is not difficult to see that $\phi_0=\phi_2\circ \phi_1,\{\cdot,\cdot\}_{\phi_0}=\{\phi_1\cdot,\phi_1\cdot\}_{\phi_2}$, for $\phi_1:\cA^s\to\cA^s$ graded algebra morphism given by
\begin{align*}
&\phi_1(q_i)=\begin{cases}
q_i,&i\neq b\\
d_0(q_a),&i=b
\end{cases}\\
&\phi_1(t^\pm)=t^\pm,
\end{align*}
and  $\phi_2:\cA^s\to\cA^s$ second-order graded algebra morphism given by
\begin{align*}
&\phi_2(q_i)=q_i,\phi_2(t^\pm)=t^\pm,\\
&\{q_i,q_j\}_{\phi_2}=\begin{cases}
\{q_b,q_b\}_{\phi_0}-\{q_b,\omega\}_{\phi_0}-\{\omega,q_b\}_{\phi_0},&i=j=b\\
\epsilon(-1)^{|q_a|}\{q_b,q_j\}_{\phi_0},&i=b,j\neq b\\
\epsilon(-1)^{|q_a|}\{q_i,q_b\}_{\phi_0},&i\neq b,j= b
\end{cases}\\
&\{q_b,t^\pm\}_{\phi_2}=\epsilon(-1)^{|q_a|}\{q_b,t^\pm\}_{\phi_0},\\
&\{t^\pm,q_b\}_{\phi_2}=\epsilon(-1)^{|q_a|}\{t^\pm,q_b\}_{\phi_0},
\end{align*}
and zero otherwise. Map $\phi_2$ is invertible by Lemma \ref{Lemma:inverse_existence}. Similar as in Lemma \ref{Lemma:inverse_existence_I_order}, $\phi_1$ is invertible and the inverse is given by
\begin{align*}
\phi_1^{-1}(q_i)=\begin{cases}
q_i,&i\neq b\\
\epsilon(-1)^{|q_a|}(q_b-\omega),&i=b
\end{cases}
\end{align*}
and $\phi_1(t^\pm)=t^\pm$.
\end{proof}

Next, we define a second-order graded algebra morphism $\psi_0:\cA^s\to\cA(\Lambda_1)$ by
\begin{align*}
&\psi_0(q_i)=
\begin{cases}
-\epsilon(-1)^{|q_a|}\pi_{\cA(\Lambda_1)}\omega,&i=b\\
0,&i=a\\
q_i,&i\neq a,b
\end{cases}\\
&\psi_0(t^\pm)=t^\pm,
\end{align*}
which uniquely determines $\psi_0\otimes\psi_0$ on $\hbar\,(\widetilde\cA^s\otimes\widetilde\cA^{s})$, together with
\begin{align*}
&\{q_i,q_j\}_{\psi_0}=
\begin{cases}
-\epsilon(-1)^{|q_a|}(\psi_0\otimes\psi_0)\{q_a,q_j\}_{d_0},&i=b,j\neq b\\
-\epsilon(-1)^{|q_a|+|q_i|}(\psi_0\otimes\psi_0)\{q_i,q_a\}_{d_0},&i\neq b,j=b\\
\frac{1}{2}(\psi_0\otimes\psi_0)\left(-\{q_a,\epsilon(-1)^{|q_a|}q_b-\omega\}_{d_0}+(-1)^{|q_a|}\{\epsilon(-1)^{|q_a|}q_b-\omega,q_a\}_{d_0}\right),&i=j=b
\end{cases}
\end{align*}
and zero otherwise. Note that by assumption we have $\{q_b,q_a\}_{\psi_0}=-\epsilon(-1)^{|q_a|}(\psi_0\otimes\psi_0)\{q_a,q_a\}_{d_0}=0$ and similarly $\{q_a,q_b\}_{\psi_0}=0$. From this we can conclude $\{q_a,s\}_{\psi_0}=0,\{s,q_a\}_{\psi_0}=0$ for all $s$.

In the following lemma, we show that $\psi_0$ is the inverse of $\phi_0$ modulo words with $q_a,q_b$. Recall we defined $C\subset \cA^s$ as the subspace generated by words with at least one letter $q_a,q_b$.

\begin{lemma}\label{Lemma:phi_0_inverse_proj}
We have 
\begin{align*}
&\psi_0\circ\phi_0=\tau,\\
&\{\cdot,\cdot\}_{\psi_0\circ\phi_0}=\{\phi_0\cdot,\phi_0\cdot\}_{\psi_0}+(\psi_0\otimes\psi_0)\{\cdot,\cdot\}_{\phi_0}=0,
\end{align*}
where $\tau:\cA^s\cong\cA(\Lambda_1)\oplus C\to\cA(\Lambda_1)$ is the projection. 
\end{lemma}

\begin{proof}
For $i\neq a,b$ we have
\begin{align*}
\psi_0(\phi_0(q_i))=\psi_0(q_i)=q_i=\tau(q_i).
\end{align*}
Additionally,
\begin{align*}
&\psi_0(\phi_0(q_a))=\psi_0(q_a)=0,\\
&\psi_0(\phi_0(q_b))=\psi_0\left(\epsilon(-1)^{|q_a|}q_b+\omega\right)=0.
\end{align*}
For the second equality we use $\omega-\pi_{\cA(\Lambda_1)}\omega=A\hbar(q_a\otimes 1)+B\hbar(1\otimes q_a)\subset\ker\psi_0$ for some $A,B\in\lZ$ and $\psi_0(\pi_{\cA(\Lambda_1)}\omega)=\pi_{\cA(\Lambda_1)}\omega$. 

Next, we show that $\{\cdot,\cdot\}_{\psi_0\circ\phi_0}=\{\phi_0\cdot,\phi_0\cdot\}_{\psi_0}+(\psi_0\otimes\psi_0)\{\cdot,\cdot\}_{\phi_0}$ vanishes. For entries different from $q_b$ it follows trivially. For $q_i\neq q_b$ we have
\begin{align*}
\{q_b,q_i\}_{\psi_0\circ\phi_0}&=\{\phi_0q_b,\phi_0q_i\}_{\psi_0}+(\psi_0\otimes\psi_0)\{q_b,q_i\}_{\phi_0}=\\
&=\{\epsilon(-1)^{|q_a|}q_b+\omega,q_i\}_{\psi_0}+(\psi_0\otimes\psi_0)\{q_b,q_i\}_{\phi_0}=\\
&=-(\psi_0\otimes\psi_0)\{q_a,q_i\}_{d_0}+(\psi_0\otimes\psi_0)\{q_a,q_i\}_{d_0}=0.
\end{align*}
Similarly, $\{q_i,q_b\}_{\psi_0\circ\phi_0}=0$ for $q_i\neq q_b$ and $\{t^\pm,q_b\}_{\psi_0\circ\phi_0}=0=\{q_b,t^\pm\}_{\psi_0\circ\phi_0}$. 

Using the properties of the antibracket, we easily get
\begin{align*}
&\{s,q_b\}_{\psi_0}=-\epsilon(-1)^{|q_a|+|s|}(\psi_0\otimes\psi_0)\{s,q_a\}_{d_0},\\
&\{q_b,s\}_{\psi_0}=-\epsilon(-1)^{|q_a|}(\psi_0\otimes\psi_0)\{q_a,s\}_{d_0},
\end{align*}
for all words $s\in\widetilde\cA(\Lambda_1)$ by induction on the length of the word. This also gives us
\begin{align*}
&\{q_b,q_b\}_{\psi_0\circ\phi_0}=\{\phi_0q_b,\phi_0q_b\}_{\psi_0}+(\psi_0\otimes\psi_0)\{q_b,q_b\}_{\phi_0}=\\
&=\{q_b,q_b\}_{\psi_0}+\epsilon(-1)^{|q_a|}\{q_b,\omega\}_{\psi_0}+\epsilon(-1)^{|q_a|}\{\omega,q_b\}_{\psi_0}+(\psi_0\otimes\psi_0)\{q_b,q_b\}_{\phi_0}=\\
&=0.
\end{align*}
This finishes the proof.
\end{proof}

\begin{lemma}\label{Lemma:prop2a}
For all $s\in C$, we have 
\begin{align*}
&\psi_0\circ d_0\circ\phi_0(s)=0.
\end{align*}
Additionally,
\begin{equation}\label{Eq:ReidII_part1}
\begin{aligned}
&\{d_0\circ\phi_0(s_1),\phi_0(s_2)\}_{\psi_0}+(-1)^{|s_1|}\{\phi_0(s_1),d_0\circ\phi_0(s_2)\}_{\psi_0}+\\
&+(\psi_0\otimes\psi_0)\{\phi_0(s_1),\phi_0(s_2)\}_{d_0}-(\psi_0\otimes\psi_0)\circ (d_0\otimes 1+1\otimes d_0)\{s_1,s_2\}_{\phi_0}=0
\end{aligned}
\end{equation}
for all $s_1,s_2\in\cA^s$ such that at least one of them is in $C$.
\end{lemma}

\begin{proof}
We have
\begin{align*}
&\psi_0(d_0(\phi_0(q_b)))=\psi_0(d_0^2(q_a))=0,\\
&\psi_0(d_0(\phi_0(q_a)))=\psi_0(d_0(q_a))=
\psi_0(\phi_0(q_b))=\tau(q_b)=0.
\end{align*}
In order to prove (\ref{Eq:ReidII_part1}), it is enough to show
\begin{align*}
&R(s_1,s_2)\coloneq\{d_0\circ\phi_0(s_1),\phi_0(s_2)\}_{\psi_0}+(-1)^{|s_1|}\{\phi_0(s_1),d_0\circ\phi_0(s_2)\}_{\psi_0}+\\
&+(\psi_0\otimes\psi_0)\{\phi_0(s_1),\phi_0(s_2)\}_{d_0}-(\psi_0\otimes\psi_0)\circ (d_0\otimes 1+1\otimes d_0)\{s_1,s_2\}_{\phi_0}=0
\end{align*}
for $s_1,s_2\in\{t^\pm,q_i\,|\,i=1,\dots,n\}$ such that at least one of $s_1,s_2$ equal to $q_a$ or $q_b$. This also implies $\psi_0\circ d_0\circ\phi_0(s)=0$ for all $s\in C$.

First, if $s_1=q_a,s_2=q_i,i\neq b$, we have
\begin{align*}
&R(q_a,q_i)=\{\epsilon(-1)^{|q_a|}q_b,q_i\}_{\psi_0}+(-1)^{|q_a|}\{q_a,d_0q_i\}_{\psi_0}+\\
&+(\psi_0\otimes\psi_0)\{q_a,q_i\}_{d_0}-(\psi_0\otimes\psi_0)\circ (d_0\otimes 1+1\otimes d_0)\{q_a,q_i\}_{\phi_0}=\\
&=-(\psi_0\otimes\psi_0)\{q_a,q_i\}_{d_0}+(\psi_0\otimes\psi_0)\{q_a,q_i\}_{d_0}=0,
\end{align*}
and similarly $R(q_i,q_a)=0,i\neq b$. 

Let now $s_1=q_b$ and $s_2=q_i,i\neq b$. First, we notice that for $j\neq b$
\begin{align*}
&\{d_0q_a,q_j\}_{\psi_0}+(\psi_0\otimes\psi_0)\{q_a,q_j\}_{d_0}=-(\psi_0\otimes\psi_0)\{q_a,q_j\}_{d_0}+(\psi_0\otimes\psi_0)\{q_a,q_j\}_{d_0}=0,\\
&\{d_0q_a,t^\pm\}_{\psi_0}+(\psi_0\otimes\psi_0)\{q_a,t^\pm\}_{d_0}=0,
\end{align*}
and similarly for $j=b$
\begin{align*}
&\{d_0q_a,q_b\}_{\psi_0}+(\psi_0\otimes\psi_0)\{q_a,q_b\}_{d_0}=\epsilon(-1)^{|q_a|}\{q_b,q_b\}_{\psi_0}+\epsilon(\psi_0\otimes\psi_0)\{\omega,q_a\}_{d_0}+(\psi_0\otimes\psi_0)\{q_a,q_b\}_{d_0}=\\
&=\epsilon\frac{1}{2}(\psi_0\otimes\psi_0)(\{d_0q_a,q_a\}_{d_0}+(-1)^{|q_a|}\{q_a,d_0q_a\}_{d_0})=\\
&=\epsilon\frac{1}{2}(\psi_0\otimes\psi_0)(d_0\otimes 1+1\otimes d_0)\{q_a,q_a\}_{d_0}=0.
\end{align*}
From this we easily get
\begin{align*}
\{d_0(q_a),d_0(q_i)\}_{\psi_0}+(\psi_0\otimes\psi_0)\{q_a,d_0(q_i)\}_{d_0}=0
\end{align*}
using the properties of the antibracket, which implies $R(q_b,q_i)=0$ for $i\neq b$. Similarly we get $R(q_i,q_b)=0$.

Finally, for $s_1=s_2=q_b$ we have
\begin{align*}
&R(q_b,q_b)=(\psi_0\otimes\psi_0)\{\phi_0q_b,\phi_0q_b\}_{d_0}-(\psi_0\otimes\psi_0)\circ (d_0\otimes 1+1\otimes d_0)\{q_b,q_b\}_{\phi_0}=\\
&=(\psi_0\otimes\psi_0)\{d_0q_a,d_0q_a\}_{d_0}-\frac{1}{2}(\psi_0\otimes\psi_0)\circ (d_0\otimes 1+1\otimes d_0)\left(\{q_a,d_0q_a\}_{d_0}-(-1)^{|q_a|}\{d_0q_a,q_a\}_{d_0}\right)=\\
&=(\psi_0\otimes\psi_0)\{d_0q_a,d_0q_a\}_{d_0}-\frac{1}{2}(\psi_0\otimes\psi_0)\left(\{d_0q_a,d_0q_a\}_{d_0}+\{d_0q_a,d_0q_a\}_{d_0}\right)=0.
\end{align*}
This finishes the proof.
\end{proof}

\begin{lemma}\label{Lemma:prop2b}
For all $s\in \cA(\Lambda_1),s_1,s_2\in \widetilde\cA(\Lambda_1)$, we have 
\begin{align*}
&d_1(s)=\psi_0\circ d_0(s),\\
&\{s_1,s_2\}_{d_1}=\{d_0 s_1,s_2\}_{\psi_0}+(-1)^{|s_1|}\{s_1,d_0s_2\}_{\psi_0}+(\psi_0\otimes\psi_0)\{s_1,s_2\}_{d_0}.
\end{align*}
\end{lemma}
\begin{proof}
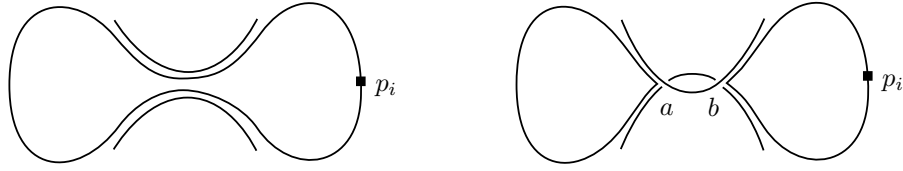
\begin{figure}
\def\svgwidth{135mm}
\begingroup%
  \makeatletter%
  \providecommand\rotatebox[2]{#2}%
  \newcommand*\fsize{\dimexpr\f@size pt\relax}%
  \newcommand*\lineheight[1]{\fontsize{\fsize}{#1\fsize}\selectfont}%
  \ifx\svgwidth\undefined%
    \setlength{\unitlength}{419.28634219bp}%
    \ifx\svgscale\undefined%
      \relax%
    \else%
      \setlength{\unitlength}{\unitlength * \real{\svgscale}}%
    \fi%
  \else%
    \setlength{\unitlength}{\svgwidth}%
  \fi%
  \global\let\svgwidth\undefined%
  \global\let\svgscale\undefined%
  \makeatother%
  \begin{picture}(1,0.15953323)%
    \lineheight{1}%
    \setlength\tabcolsep{0pt}%
    \put(0.93457948,0.07686235){\makebox(0,0)[lt]{\lineheight{1.25}\smash{\begin{tabular}[t]{l}$p_i$\end{tabular}}}}%
    \put(0.71549133,0.04849778){\makebox(0,0)[lt]{\lineheight{1.25}\smash{\begin{tabular}[t]{l}$a$\end{tabular}}}}%
    \put(0.76276565,0.0470387){\makebox(0,0)[lt]{\lineheight{1.25}\smash{\begin{tabular}[t]{l}$b$\end{tabular}}}}%
    \put(0,0){\includegraphics[width=\unitlength,page=1]{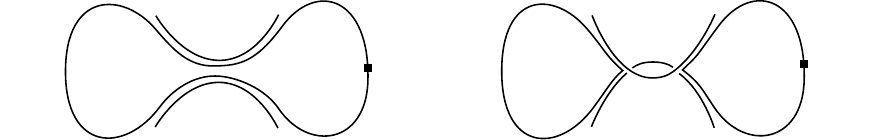}}%
    \put(0.43533781,0.07156491){\makebox(0,0)[lt]{\lineheight{1.25}\smash{\begin{tabular}[t]{l}$p_i$\end{tabular}}}}%
    \put(-0.00170527,0.07890665){\makebox(0,0)[lt]{\lineheight{1.25}\smash{\begin{tabular}[t]{l}\textcolor{white}{.}\end{tabular}}}}%
  \end{picture}%
\endgroup%
\caption{Disk pinching during Reidemeister II move.}
\label{Fig:disk_pinching}
\end{figure}

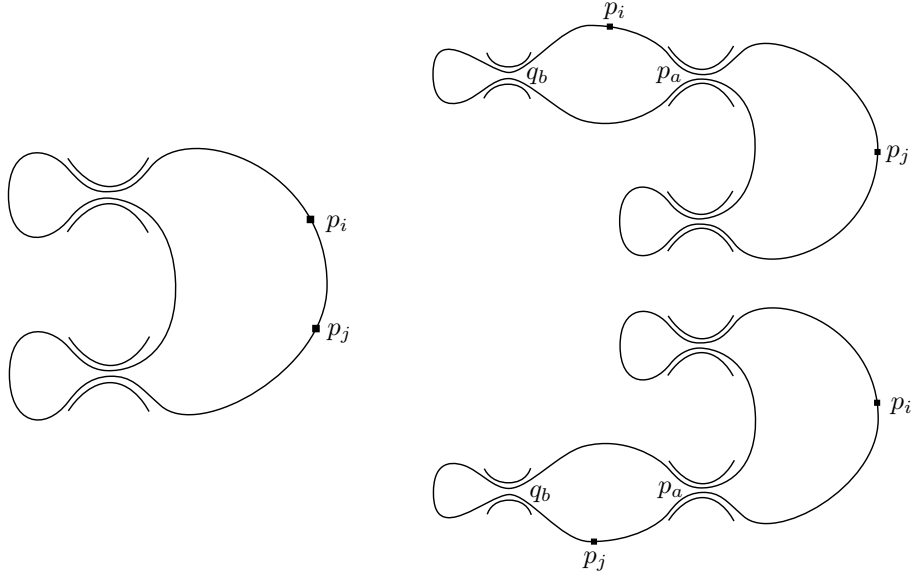
\begin{figure}
\def\svgwidth{130mm}
\begingroup%
  \makeatletter%
  \providecommand\rotatebox[2]{#2}%
  \newcommand*\fsize{\dimexpr\f@size pt\relax}%
  \newcommand*\lineheight[1]{\fontsize{\fsize}{#1\fsize}\selectfont}%
  \ifx\svgwidth\undefined%
    \setlength{\unitlength}{515.11381244bp}%
    \ifx\svgscale\undefined%
      \relax%
    \else%
      \setlength{\unitlength}{\unitlength * \real{\svgscale}}%
    \fi%
  \else%
    \setlength{\unitlength}{\svgwidth}%
  \fi%
  \global\let\svgwidth\undefined%
  \global\let\svgscale\undefined%
  \makeatother%
  \begin{picture}(1,0.58198055)%
    \lineheight{1}%
    \setlength\tabcolsep{0pt}%
    \put(0.37994901,0.34878642){\makebox(0,0)[lt]{\lineheight{1.25}\smash{\begin{tabular}[t]{l}$p_i$\end{tabular}}}}%
    \put(0.38078097,0.23813125){\makebox(0,0)[lt]{\lineheight{1.25}\smash{\begin{tabular}[t]{l}$p_j$\end{tabular}}}}%
    \put(0.95685792,0.16260892){\makebox(0,0)[lt]{\lineheight{1.25}\smash{\begin{tabular}[t]{l}$p_i$\end{tabular}}}}%
    \put(0.64265244,0.00419153){\makebox(0,0)[lt]{\lineheight{1.25}\smash{\begin{tabular}[t]{l}$p_j$\end{tabular}}}}%
    \put(0.71905273,0.0765081){\makebox(0,0)[lt]{\lineheight{1.25}\smash{\begin{tabular}[t]{l}$p_a$\end{tabular}}}}%
    \put(0.58738627,0.07267763){\makebox(0,0)[lt]{\lineheight{1.25}\smash{\begin{tabular}[t]{l}$q_b$\end{tabular}}}}%
    \put(0.66383867,0.56664568){\makebox(0,0)[lt]{\lineheight{1.25}\smash{\begin{tabular}[t]{l}$p_i$\end{tabular}}}}%
    \put(0.95264934,0.42002409){\makebox(0,0)[lt]{\lineheight{1.25}\smash{\begin{tabular}[t]{l}$p_j$\end{tabular}}}}%
    \put(0,0){\includegraphics[width=\unitlength,page=1]{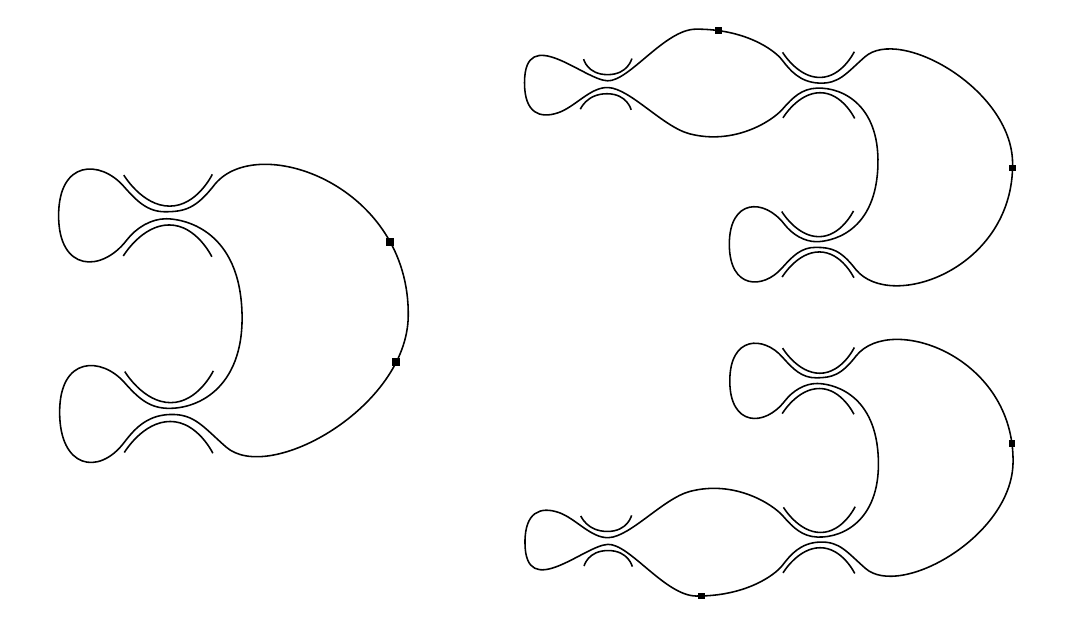}}%
    \put(0.71776583,0.50092726){\makebox(0,0)[lt]{\lineheight{1.25}\smash{\begin{tabular}[t]{l}$p_a$\end{tabular}}}}%
    \put(0.58410886,0.49990591){\makebox(0,0)[lt]{\lineheight{1.25}\smash{\begin{tabular}[t]{l}$q_b$\end{tabular}}}}%
    \put(-0.00138804,0.47138619){\makebox(0,0)[lt]{\lineheight{1.25}\smash{\begin{tabular}[t]{l}\textcolor{white}{.}\end{tabular}}}}%
  \end{picture}%
\endgroup%
\caption{Pinching of disks with two positive punctures during Reidemeister II move.}
\label{Fig:disk_pinchingII}
\end{figure}

First, it is easy to see that $d_{1,f}(q_i)=d_{0,f}(q_i)=\psi_0\circ d_{0,f}(q_i)$ and $d_{1,f}(q_i,q_j)=d_{0,f}(q_i,q_j)=(\psi_0\otimes\psi_0)\circ d_{0,f}(q_i,q_j)$ for all $i,j\neq a,b$. Therefore, $d_{1,f}=\psi_0\circ d_{0,f}$ on $\cA(\Lambda_1)$. Now, it is enough to show 
\begin{align*}
&\pi_{\widetilde\cA}\circ \psi_0\circ d_{0,\lD}(q_i)=d_{1,\lD}(q_i),\\
&(\psi_0\otimes\psi_0)\circ d_{0,\lD}(q_i,q_j)+\{d_{0,\lD}q_i,q_j\}_{\psi_0}+(-1)^{|q_i|}\{q_i,d_{0,\lD}q_j\}_{\psi_0}=d_{1,\lD}(q_i,q_j),\\
&\pi_{\hbar\;\widetilde\cA\otimes\widetilde\cA^{\cyc}}\circ \psi_0\circ d_{0,\lD}(q_i)+\psi_0\circ d_{0,A}(q_i)=d_{1,A}(q_i),
\end{align*}
for all $i,j\in\{1,\dots,n\}$. 

The first equality follows from \cite[Lemma 8.2.]{Chekanov02}. Let $u$ be an index zero disk on $\Lambda_1$ with one positive puncture at $i$. During the isotopy, $u$ is pinched as shown in Figure \ref{Fig:disk_pinching} at $k\in\lN_0$ places. It is decomposed into an index zero disk on $\Lambda_0$ that contains the original positive puncture and has $k$ negative punctures at $b$ (and no negative punctures at $a$), and $k$ index zero disks with one positive puncture at $a$ different from the bigon $q_bp_a$. Moreover, for every such collection of disks, there is a unique disk on $\Lambda_1$ that decomposes into it. It is not difficult to check that the signs on the two sides also match. This shows $\pi_{\widetilde\cA}\circ\psi_0\circ d_{0,\lD}(q_i)=d_{1,\lD}(q_i)$ for all $i\neq a,b$.

The proof of the second equality
\begin{align*}
(\psi_0\otimes\psi_0)\circ d_{0,\lD}(q_i,q_j)+\{d_{0,\lD}q_i,q_j\}_{\psi_0}+(-1)^{|q_i|}\{q_i,d_{0,\lD}q_j\}_{\psi_0}=d_{1,\lD}(q_i,q_j)
\end{align*}
follows the same idea. Pinching of disks with arbitrarily many positive punctures also appears in \cite{Ng_rLSFT}. Let $v$ be an index zero disk on $\Lambda_1$ with two positive punctures at $\gamma_i$ and $\gamma_j$. The disk is pinched during the isotopy as before, however, disks that have both a positive puncture at $a$ and a negative puncture at $b$ (and one more positive puncture) can appear. The possible disk configurations are depicted in Figure \ref{Fig:disk_pinchingII} and correspond to summands in $(\psi_0\otimes\psi_0)\circ d_{0,\lD}(q_i,q_j),(-1)^{|q_i|}\{q_i,d_{0,\lD}q_j\}_{\psi_0},\{d_{0,\lD}q_i,q_j\}_{\psi_0}$, respectively. Since we do not have index zero disks with two positive punctures at $a$ on $\Lambda_0$ by assumption, pinching shown in Figure \ref{Fig:disk_pinching_inadmissible} does not appear. It is not difficult to check that the signs on the two sides also match.

\begin{figure}
\def\svgwidth{55mm}
\begingroup%
  \makeatletter%
  \providecommand\rotatebox[2]{#2}%
  \newcommand*\fsize{\dimexpr\f@size pt\relax}%
  \newcommand*\lineheight[1]{\fontsize{\fsize}{#1\fsize}\selectfont}%
  \ifx\svgwidth\undefined%
    \setlength{\unitlength}{235.4292342bp}%
    \ifx\svgscale\undefined%
      \relax%
    \else%
      \setlength{\unitlength}{\unitlength * \real{\svgscale}}%
    \fi%
  \else%
    \setlength{\unitlength}{\svgwidth}%
  \fi%
  \global\let\svgwidth\undefined%
  \global\let\svgscale\undefined%
  \makeatother%
  \begin{picture}(1,0.61110343)%
    \lineheight{1}%
    \setlength\tabcolsep{0pt}%
    \put(0,0){\includegraphics[width=\unitlength,page=1]{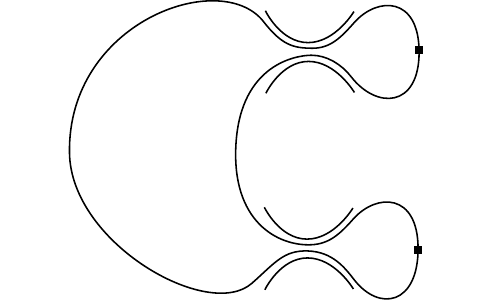}}%
    \put(0.88348969,0.49344279){\makebox(0,0)[lt]{\lineheight{1.25}\smash{\begin{tabular}[t]{l}$p_i$\end{tabular}}}}%
    \put(0.87802841,0.08749737){\makebox(0,0)[lt]{\lineheight{1.25}\smash{\begin{tabular}[t]{l}$p_j$\end{tabular}}}}%
    \put(0.49392759,0.50163474){\makebox(0,0)[lt]{\lineheight{1.25}\smash{\begin{tabular}[t]{l}$p_a$\end{tabular}}}}%
    \put(0.47799946,0.08795263){\makebox(0,0)[lt]{\lineheight{1.25}\smash{\begin{tabular}[t]{l}$p_a$\end{tabular}}}}%
    \put(-0.003037,0.3086741){\makebox(0,0)[lt]{\lineheight{1.25}\smash{\begin{tabular}[t]{l}\textcolor{white}{.}\end{tabular}}}}%
  \end{picture}%
\endgroup%
\caption{Disk pinching excluded for admissible Reidemeister II move.}
\label{Fig:disk_pinching_inadmissible}
\end{figure}
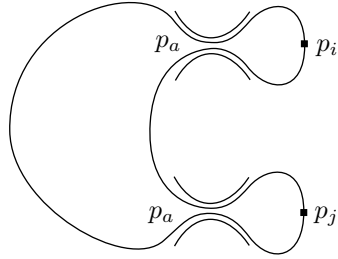

Finally, we show
\begin{align}
\label{Eq:annulus_pinching}
\pi_{\hbar\;\widetilde\cA\otimes\widetilde\cA^{\cyc}}\circ \psi_0\circ d_{0,\lD}(q_i)+\psi_0\circ d_{0,A}(q_i)=d_{1,A}(q_i).
\end{align}
For annuli, the pinching can be on one boundary component (Figure \ref{Fig:annulus_pinching}) or between two boundary components (Figure \ref{Figure:signs_annuli_II_final}), which we call \textit{split} and \textit{non-split pinching}. 
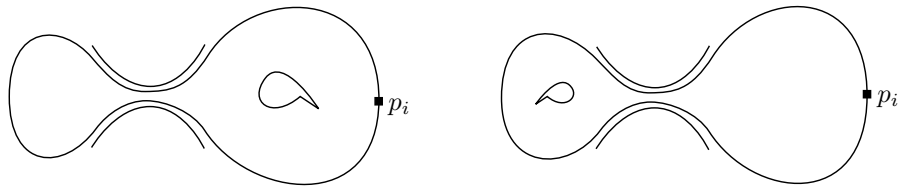
\begin{figure}
\def\svgwidth{130mm}
\begingroup%
  \makeatletter%
  \providecommand\rotatebox[2]{#2}%
  \newcommand*\fsize{\dimexpr\f@size pt\relax}%
  \newcommand*\lineheight[1]{\fontsize{\fsize}{#1\fsize}\selectfont}%
  \ifx\svgwidth\undefined%
    \setlength{\unitlength}{511.59734734bp}%
    \ifx\svgscale\undefined%
      \relax%
    \else%
      \setlength{\unitlength}{\unitlength * \real{\svgscale}}%
    \fi%
  \else%
    \setlength{\unitlength}{\svgwidth}%
  \fi%
  \global\let\svgwidth\undefined%
  \global\let\svgscale\undefined%
  \makeatother%
  \begin{picture}(1,0.18349228)%
    \lineheight{1}%
    \setlength\tabcolsep{0pt}%
    \put(0,0){\includegraphics[width=\unitlength,page=1]{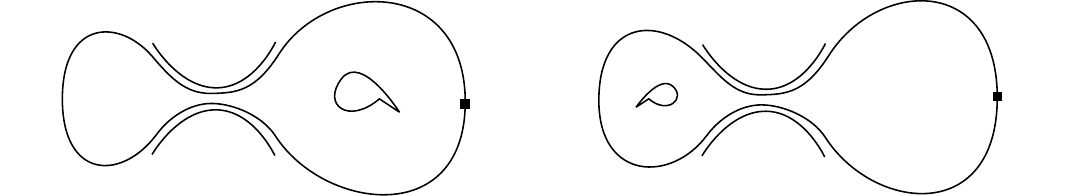}}%
    \put(0.44635964,0.07699727){\makebox(0,0)[lt]{\lineheight{1.25}\smash{\begin{tabular}[t]{l}$p_i$\end{tabular}}}}%
    \put(0.94638375,0.08501218){\makebox(0,0)[lt]{\lineheight{1.25}\smash{\begin{tabular}[t]{l}$p_i$\end{tabular}}}}%
    \put(-0.00139758,0.09186665){\makebox(0,0)[lt]{\lineheight{1.25}\smash{\begin{tabular}[t]{l}\textcolor{white}{.}\end{tabular}}}}%
  \end{picture}%
\endgroup%
\caption{Split annulus pinching during Reidemeister II move.}
\label{Fig:annulus_pinching}
\end{figure}
Consider a connected component $\cM^\pi$ of the 1-dimensional moduli space $\cM^\pi_{2,1}(\Lambda_1)$ of holomorphic annuli with boundary on $\pi_{xy}(\Lambda_1)$. Assume first that the pinching is away from the path of the branch point. We distinguish three cases.
\newline
First, assume both boundary points of $\cM^\pi$ are split. If we have a non-split pinching, by Lemma \ref{Lemma:OSection_ReidII} the obstruction section at both boundary points of $\cM^\pi$ goes either to $+\infty$ or $-\infty$. Therefore, the count of annuli on $\lR\times\Lambda_1$ coming from $\cM^\pi$ is zero. Correspondingly, there are no summands on the LHS of (\ref{Eq:annulus_pinching}) coming from the pinched configuration of curves. If we have a split pinching, either the branch point is not separated from the annular part, or it is. In the first case, we have a corresponding 1-dimensional family $\cM^{\pi}_t,t\in[0,1/2]$ of holomorphic annuli on $\pi_{xy}(\Lambda_{t})$. By Lemma \ref{Lemma:OSection_ReidIIb}, the values of the obstruction section at $\partial\cM^{\pi}$ are close to the values of the obstruction section at $\partial\cM^{\pi}_\frac{1}{2}$. Therefore, the count of annuli on $\lR\times\Lambda_1$ coming from $\cM^\pi$ is equal to the count of annuli on $\lR\times\Lambda_0$ coming from $\cM^\pi_{0}$. In the second case, the values of the obstruction section at $\partial\cM^\pi$ are close to the same (generically) non-zero value, so the corresponding count of annuli is zero. Note that by assumption we cannot have non-split pinching at two places in this case (see Figure \ref{Fig:annulus_pinching_inadmissible}). This shows the desired correspondence of annuli in the first case.
\newline
Second, assume exactly one boundary point of $\cM^\pi$ is split. We first consider the case of non-split pinching. The non-split boundary point consists of a disk $\overline v$ glued to itself at a crossing $\overline\gamma$. Disk $\overline v$ gets pinched in a way that separates the positive and the negative puncture at $\overline\gamma$ (see Figure \ref{Figure:signs_annuli_II_final}). Denote the two disks that appear after pinching by $v_1,v_2$, where $v_1$ is the disk that contains the positive puncture at $a$. Here we simplify and assume that there are no other pinching points (which have to be split), the general case goes analogously. If $v_1$ contains the positive puncture at $\overline\gamma$, one value of the obstruction section at $\partial\cM^\pi$ converges to $+\infty$ and the other to $-\infty$ by Lemma \ref{Lemma:OSection_ReidII} and Lemma \ref{Lemma:OSection_hyperbolic}. Therefore, the count of annuli in $\cM^\pi$ that lift to $\Lambda_1$ is equal to $\pm 1$. Correspondingly, we have a summand in $\psi_0\circ d_\lD(q_i)$ obtained by first gluing $v_2$ to $q_i$ and then $v_1$ to $v_2$ at $q_b$ and $\overline\gamma$. On the other hand, if $v_1$ contains the negative puncture at $\overline\gamma$, the values of the obstruction section at $\partial\cM^\pi$ converge either both to $+\infty$ or both to $-\infty$. Then the count of annuli that lift to $\lR\times\Lambda_1$ is zero. Additionally, there are no corresponding summands on the LHS of (\ref{Eq:annulus_pinching}). In the case of split pinching, as before, we have a corresponding family $\cM^{\pi}_t,t\in[0,1/2]$ of annuli on $\pi_{xy}(\Lambda_{t})$ and the values of the obstruction section at $\partial\cM^{\pi}$ are close to the values of the obstruction section at $\partial\cM^{\pi}_\frac{1}{2}$. This shows the correspondence of annuli in the second case.
\newline
Third, assume both boundary points of $\cM^\pi$ are non-split. We first consider the case of non-split pinching (assume for simplicity there is no other split pinching). As before, the first (second) point in $\partial\cM^\pi$ consists of a disk $v$ ($\overline v$) glued to itself at $j$ ($\overline j$), and the pinching gives us two disks $v_1,v_2$ ($\overline v_1,\overline v_2$), where we denote by $v_1$ ($\overline v_1$) the disk that contains the positive puncture at $a$ as before. The values of the obstruction section at $\partial\cM^\pi$ are $\pm\infty$, depending on whether $v_1,\overline v_1$ contain the positive or the negative puncture at $j,\overline j$. If the values of the obstruction section are equal, then the count of annuli that lift to $\lR\times\Lambda_1$ is zero and there are zero or two corresponding summands (with canceling signs, see below) on the LHS of (\ref{Eq:annulus_pinching}). If the two values are different, then the count of annuli that lift is $\pm 1$ and there is one corresponding summand on the LHS of (\ref{Eq:annulus_pinching}). Next, in this case we can also have two non-split pinching points (see Figure \ref{Figure:signs_annuli_I_final}). After pinching, we get an index zero disk on $\Lambda_0$ with two negative punctures at $b$ and the original positive puncture at $i$, and two pairs of disks, one for each non-split boundary point, each with one positive puncture at $a$, that can be glued together at $j,\overline j$. For each of the two curve configurations, we get a corresponding summand in (\ref{Eq:annulus_pinching}) with weight $\pm\frac{1}{2}$. They come with opposite signs if the values of the obstruction section at $\partial\cM^\pi$ are the same, and with the same sign otherwise (see below). On the other hand, we have zero annuli that lift to $\lR\times\Lambda_1$ in the first case and one in the second case. Finally, in the case of split pinching, we have a corresponding family $\cM^{\pi}_t,t\in[0,1/2]$ of annuli on $\pi_{xy}(\Lambda_t)$ as before, and the values $\Omega(\partial\cM^\pi)$ converge to $\Omega(\partial\cM^{\pi}_\frac{1}{2})$. This shows the desired correspondence of annuli in the third case.
\begin{figure}
\def\svgwidth{70mm}
\begingroup%
  \makeatletter%
  \providecommand\rotatebox[2]{#2}%
  \newcommand*\fsize{\dimexpr\f@size pt\relax}%
  \newcommand*\lineheight[1]{\fontsize{\fsize}{#1\fsize}\selectfont}%
  \ifx\svgwidth\undefined%
    \setlength{\unitlength}{253.85780838bp}%
    \ifx\svgscale\undefined%
      \relax%
    \else%
      \setlength{\unitlength}{\unitlength * \real{\svgscale}}%
    \fi%
  \else%
    \setlength{\unitlength}{\svgwidth}%
  \fi%
  \global\let\svgwidth\undefined%
  \global\let\svgscale\undefined%
  \makeatother%
  \begin{picture}(1,0.55638115)%
    \lineheight{1}%
    \setlength\tabcolsep{0pt}%
    \put(0,0){\includegraphics[width=\unitlength,page=1]{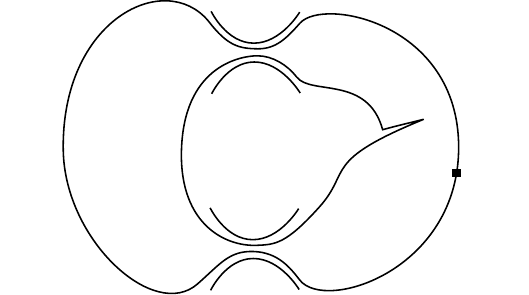}}%
    \put(0.89194765,0.21407749){\makebox(0,0)[lt]{\lineheight{1.25}\smash{\begin{tabular}[t]{l}$p_i$\end{tabular}}}}%
    \put(0.36732872,0.45380698){\makebox(0,0)[lt]{\lineheight{1.25}\smash{\begin{tabular}[t]{l}$p_a$\end{tabular}}}}%
    \put(0.35846578,0.07606456){\makebox(0,0)[lt]{\lineheight{1.25}\smash{\begin{tabular}[t]{l}$p_a$\end{tabular}}}}%
    \put(-0.00281653,0.26978942){\makebox(0,0)[lt]{\lineheight{1.25}\smash{\begin{tabular}[t]{l}\textcolor{white}{.}\end{tabular}}}}%
  \end{picture}%
\endgroup%
\caption{Annulus pinching excluded for admissible Reidemeister II move.}
\label{Fig:annulus_pinching_inadmissible}
\end{figure}
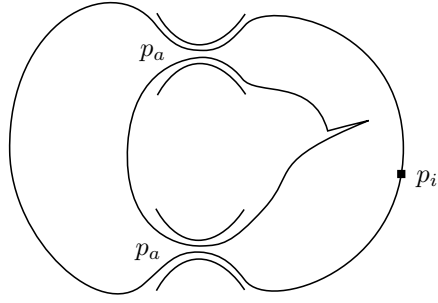

We additionally check that the signs on the two sides match. For example, consider the case of a connected component of $\cM^\pi_{2,1}$ with two non-split boundary points, such that we have two non-split pinching points and the values of the obstruction section $\Omega$ at the boundary are $+\infty$ and $-\infty$, as shown in Figure \ref{Figure:signs_annuli_I_final}. Let the disks obtained after pinching be $u=q_{j_1}\dots q_{j_K} p_{j_0}$, $v_1=q_{i_1}\dots q_{i_s}p_I q_{k_j}\dots q_{k_m} p_a,v_2= q_{k_1}\dots q_{k_{j-1}}q_Iq_{i_{s+1}}\dots q_{i_n}p_a$ and $\widetilde v_1=q_{i_1}\dots q_{i_t}q_J q_{k_{j+1}}\dots q_{k_m} p_a,\widetilde v_2=q_{k_1}\dots  q_{k_{j}}p_Jq_{i_{t+1}}\dots q_{i_n}p_a$ as shown in Figure \ref{Figure:signs_annuli_I_final}, where $u$ has two negative punctures at $b$ ($j_k=j_l=b,1\leq k< l\leq K$). We have two corresponding summands in $\psi_0(\epsilon(u) q_{j_1}\dots  q_{j_K})$ with coefficients
\begin{align*}
\frac{1}{2}\epsilon(u)\epsilon_2\epsilon_3(-1)^{|q_{k_1}|+\dots+|q_{k_{j-1}}|}\prod\epsilon_\bullet^{v_1}\prod\epsilon_\bullet^{v_2}
\end{align*}
and
\begin{align*}
\frac{1}{2}\epsilon(u)\epsilon_2\epsilon_3(-1)^{|q_{k_1}|+\dots+|q_{k_{j-1}}|}\prod\epsilon_\bullet^{\widetilde v_1}\prod\epsilon_\bullet^{\widetilde v_2}.
\end{align*}
Additionally, the count of annuli that lift to $\lR\times\Lambda$ on this connected component is equal to
\begin{align*}
&\epsilon(u) \epsilon_2\epsilon_3(-1)^{|q_{k_1}|+\dots+|q_{k_{j-1}}|}\prod \epsilon_\bullet^{v_1}\prod\epsilon_\bullet^{v_2}=\\
=&\epsilon(u)\epsilon_2\epsilon_3(-1)^{|q_{k_1}|+\dots+|q_{k_{j-1}}|}\prod \epsilon_\bullet^{\widetilde v_1}\prod\epsilon_\bullet^{\widetilde v_2}.
\end{align*} 
Similarly, if we have $\Omega=-\infty$ or $\Omega=+\infty$ for both boundary points, then the two summands in $\psi_0(\epsilon(u)q_{j_1}\dots q_{j_K})$ come with opposite signs and the count of annuli that lift is zero.

\begin{figure}
\def\svgwidth{158mm}
\begingroup%
  \makeatletter%
  \providecommand\rotatebox[2]{#2}%
  \newcommand*\fsize{\dimexpr\f@size pt\relax}%
  \newcommand*\lineheight[1]{\fontsize{\fsize}{#1\fsize}\selectfont}%
  \ifx\svgwidth\undefined%
    \setlength{\unitlength}{886.19733431bp}%
    \ifx\svgscale\undefined%
      \relax%
    \else%
      \setlength{\unitlength}{\unitlength * \real{\svgscale}}%
    \fi%
  \else%
    \setlength{\unitlength}{\svgwidth}%
  \fi%
  \global\let\svgwidth\undefined%
  \global\let\svgscale\undefined%
  \makeatother%
  \begin{picture}(1,0.29288742)%
    \lineheight{1}%
    \setlength\tabcolsep{0pt}%
    \put(0.26402617,0.05021638){\makebox(0,0)[lt]{\lineheight{1.25}\smash{\begin{tabular}[t]{l}$j_l=b$\end{tabular}}}}%
    \put(0.25742912,0.23320184){\makebox(0,0)[lt]{\lineheight{1.25}\smash{\begin{tabular}[t]{l}$j_k=b$\end{tabular}}}}%
    \put(0.33201106,0.14396774){\makebox(0,0)[lt]{\lineheight{1.25}\smash{\begin{tabular}[t]{l}$u$\end{tabular}}}}%
    \put(0.10836502,0.21988364){\makebox(0,0)[lt]{\lineheight{1.25}\smash{\begin{tabular}[t]{l}$v_1$\end{tabular}}}}%
    \put(0.09383162,0.0771145){\makebox(0,0)[lt]{\lineheight{1.25}\smash{\begin{tabular}[t]{l}$v_2$\end{tabular}}}}%
    \put(0.08652179,0.1797463){\makebox(0,0)[lt]{\lineheight{1.25}\smash{\begin{tabular}[t]{l}$\epsilon_3$\end{tabular}}}}%
    \put(0.18712687,0.0580505){\makebox(0,0)[lt]{\lineheight{1.25}\smash{\begin{tabular}[t]{l}$\epsilon_2$\end{tabular}}}}%
    \put(0.40435824,0.18717549){\makebox(0,0)[lt]{\lineheight{1.25}\smash{\begin{tabular}[t]{l}$\epsilon_1$\end{tabular}}}}%
    \put(0.15234952,0.28153153){\makebox(0,0)[lt]{\lineheight{1.25}\smash{\begin{tabular}[t]{l}$i_1$\end{tabular}}}}%
    \put(0.03243389,0.21410456){\makebox(0,0)[lt]{\lineheight{1.25}\smash{\begin{tabular}[t]{l}$i_s$\end{tabular}}}}%
    \put(-0.00082149,0.1431474){\makebox(0,0)[lt]{\lineheight{1.25}\smash{\begin{tabular}[t]{l}$i_{s+1}$\end{tabular}}}}%
    \put(0.14351999,0.00205371){\makebox(0,0)[lt]{\lineheight{1.25}\smash{\begin{tabular}[t]{l}$i_n$\end{tabular}}}}%
    \put(0.16655086,0.10893405){\makebox(0,0)[lt]{\lineheight{1.25}\smash{\begin{tabular}[t]{l}$k_1$\end{tabular}}}}%
    \put(0.14168975,0.14065147){\makebox(0,0)[lt]{\lineheight{1.25}\smash{\begin{tabular}[t]{l}$k_{j-1}$\end{tabular}}}}%
    \put(0.12134399,0.16281029){\makebox(0,0)[lt]{\lineheight{1.25}\smash{\begin{tabular}[t]{l}$k_j$\end{tabular}}}}%
    \put(0.158891,0.21489075){\makebox(0,0)[lt]{\lineheight{1.25}\smash{\begin{tabular}[t]{l}$k_m$\end{tabular}}}}%
    \put(0.7821806,0.05057004){\makebox(0,0)[lt]{\lineheight{1.25}\smash{\begin{tabular}[t]{l}$j_l=b$\end{tabular}}}}%
    \put(0.77558356,0.2335557){\makebox(0,0)[lt]{\lineheight{1.25}\smash{\begin{tabular}[t]{l}$j_k=b$\end{tabular}}}}%
    \put(0.85016701,0.14432162){\makebox(0,0)[lt]{\lineheight{1.25}\smash{\begin{tabular}[t]{l}$u$\end{tabular}}}}%
    \put(0.62651721,0.22023749){\makebox(0,0)[lt]{\lineheight{1.25}\smash{\begin{tabular}[t]{l}$\widetilde v_1$\end{tabular}}}}%
    \put(0.6238321,0.07746816){\makebox(0,0)[lt]{\lineheight{1.25}\smash{\begin{tabular}[t]{l}$\widetilde v_2$\end{tabular}}}}%
    \put(0.70528006,0.05840416){\makebox(0,0)[lt]{\lineheight{1.25}\smash{\begin{tabular}[t]{l}$\epsilon_2$\end{tabular}}}}%
    \put(0.92251633,0.1875294){\makebox(0,0)[lt]{\lineheight{1.25}\smash{\begin{tabular}[t]{l}$\epsilon_1$\end{tabular}}}}%
    \put(0.67050232,0.28188538){\makebox(0,0)[lt]{\lineheight{1.25}\smash{\begin{tabular}[t]{l}$i_1$\end{tabular}}}}%
    \put(0.53607843,0.11048287){\makebox(0,0)[lt]{\lineheight{1.25}\smash{\begin{tabular}[t]{l}$i_t$\end{tabular}}}}%
    \put(0.54211682,0.06975113){\makebox(0,0)[lt]{\lineheight{1.25}\smash{\begin{tabular}[t]{l}$i_{t+1}$\end{tabular}}}}%
    \put(0.66167272,0.00240737){\makebox(0,0)[lt]{\lineheight{1.25}\smash{\begin{tabular}[t]{l}$i_n$\end{tabular}}}}%
    \put(0.6847038,0.10928771){\makebox(0,0)[lt]{\lineheight{1.25}\smash{\begin{tabular}[t]{l}$k_1$\end{tabular}}}}%
    \put(0.65984243,0.14100535){\makebox(0,0)[lt]{\lineheight{1.25}\smash{\begin{tabular}[t]{l}$k_{j-1}$\end{tabular}}}}%
    \put(0.63163772,0.1631641){\makebox(0,0)[lt]{\lineheight{1.25}\smash{\begin{tabular}[t]{l}$k_j$\end{tabular}}}}%
    \put(0.67704373,0.21524451){\makebox(0,0)[lt]{\lineheight{1.25}\smash{\begin{tabular}[t]{l}$k_m$\end{tabular}}}}%
    \put(0.58992851,0.0858051){\makebox(0,0)[lt]{\lineheight{1.25}\smash{\begin{tabular}[t]{l}$+$\end{tabular}}}}%
    \put(0.0617323,0.1922634){\makebox(0,0)[lt]{\lineheight{1.25}\smash{\begin{tabular}[t]{l}$+$\end{tabular}}}}%
    \put(0.39522797,0.22167083){\makebox(0,0)[lt]{\lineheight{1.25}\smash{\begin{tabular}[t]{l}$j_1$\end{tabular}}}}%
    \put(0.40421196,0.10417452){\makebox(0,0)[lt]{\lineheight{1.25}\smash{\begin{tabular}[t]{l}$j_K$\end{tabular}}}}%
    \put(0.91344796,0.22216457){\makebox(0,0)[lt]{\lineheight{1.25}\smash{\begin{tabular}[t]{l}$j_1$\end{tabular}}}}%
    \put(0,0){\includegraphics[width=\unitlength,page=1]{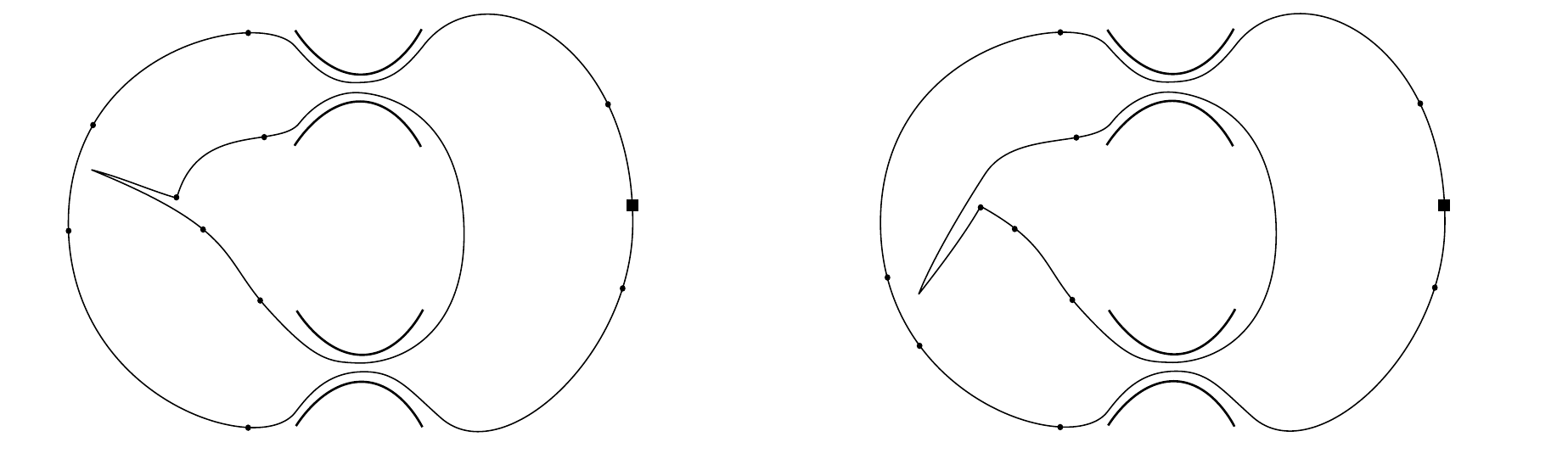}}%
    \put(0.92243199,0.10466826){\makebox(0,0)[lt]{\lineheight{1.25}\smash{\begin{tabular}[t]{l}$j_K$\end{tabular}}}}%
  \end{picture}%
\endgroup%
\caption{Two non-split pinching points on an annulus.}
\label{Figure:signs_annuli_I_final}
\end{figure}
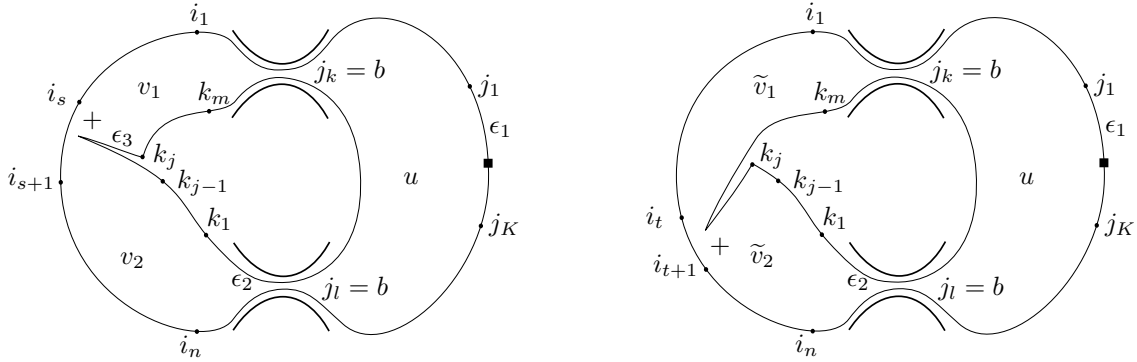

Second, consider for example the case shown in Figure \ref{Figure:signs_annuli_II_final} with one non-split pinching, where we have a connected component of $\cM^\pi_{2,1}$ with one non-split boundary point with $\Omega=-\infty$, and the other boundary is split or non-split with $\Omega=+\infty$. The count of annuli that lift to $\lR\times\Lambda$ is given by
\begin{align*}
\epsilon(u)(-1)^{|q_a|}\epsilon\,\epsilon_2\prod\epsilon_\bullet^v,
\end{align*}
which is also the sign of the corresponding summand in $\pi_{\hbar\;\widetilde\cA\otimes\widetilde\cA^{\cyc}}\circ \psi_0\circ d_{0,\lD}(q_i)$. If the second boundary point is also non-split with $\Omega=-\infty$, then no annuli lift and we have two corresponding summands with signs
\begin{align*}
\epsilon(u)(-1)^{|q_a|}\epsilon\,\epsilon_2\prod\epsilon_\bullet^v
\end{align*}
and 
\begin{align*}
\epsilon(\widetilde u)(-1)^{|q_a|}\epsilon\,\epsilon_2\prod\epsilon_\bullet^{\widetilde v}=-\epsilon(u)(-1)^{|q_a|}\epsilon\,\epsilon_2\prod\epsilon_\bullet^v
\end{align*}
that cancel out. Other cases go similarly.

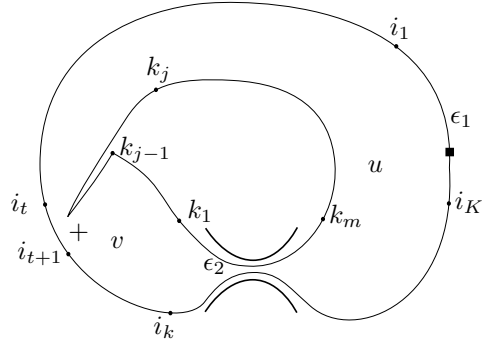
\begin{figure}
\def\svgwidth{70mm}
\begingroup%
  \makeatletter%
  \providecommand\rotatebox[2]{#2}%
  \newcommand*\fsize{\dimexpr\f@size pt\relax}%
  \newcommand*\lineheight[1]{\fontsize{\fsize}{#1\fsize}\selectfont}%
  \ifx\svgwidth\undefined%
    \setlength{\unitlength}{410.39801887bp}%
    \ifx\svgscale\undefined%
      \relax%
    \else%
      \setlength{\unitlength}{\unitlength * \real{\svgscale}}%
    \fi%
  \else%
    \setlength{\unitlength}{\svgwidth}%
  \fi%
  \global\let\svgwidth\undefined%
  \global\let\svgscale\undefined%
  \makeatother%
  \begin{picture}(1,0.63638963)%
    \lineheight{1}%
    \setlength\tabcolsep{0pt}%
    \put(0.67645632,0.31294861){\makebox(0,0)[lt]{\lineheight{1.25}\smash{\begin{tabular}[t]{l}$u$\end{tabular}}}}%
    \put(0.18771784,0.16858795){\makebox(0,0)[lt]{\lineheight{1.25}\smash{\begin{tabular}[t]{l}$v$\end{tabular}}}}%
    \put(0.36359336,0.12604051){\makebox(0,0)[lt]{\lineheight{1.25}\smash{\begin{tabular}[t]{l}$\epsilon_2$\end{tabular}}}}%
    \put(0.83268482,0.40624985){\makebox(0,0)[lt]{\lineheight{1.25}\smash{\begin{tabular}[t]{l}$\epsilon_1$\end{tabular}}}}%
    \put(0.71795689,0.57344864){\makebox(0,0)[lt]{\lineheight{1.25}\smash{\begin{tabular}[t]{l}$i_1$\end{tabular}}}}%
    \put(-0.00177388,0.2398786){\makebox(0,0)[lt]{\lineheight{1.25}\smash{\begin{tabular}[t]{l}$i_t$\end{tabular}}}}%
    \put(0.01126518,0.14826915){\makebox(0,0)[lt]{\lineheight{1.25}\smash{\begin{tabular}[t]{l}$i_{t+1}$\end{tabular}}}}%
    \put(0.26942937,0.00443471){\makebox(0,0)[lt]{\lineheight{1.25}\smash{\begin{tabular}[t]{l}$i_k$\end{tabular}}}}%
    \put(0.33378198,0.22633302){\makebox(0,0)[lt]{\lineheight{1.25}\smash{\begin{tabular}[t]{l}$k_1$\end{tabular}}}}%
    \put(0.20412549,0.34494822){\makebox(0,0)[lt]{\lineheight{1.25}\smash{\begin{tabular}[t]{l}$k_{j-1}$\end{tabular}}}}%
    \put(0.2554816,0.49774732){\makebox(0,0)[lt]{\lineheight{1.25}\smash{\begin{tabular}[t]{l}$k_j$\end{tabular}}}}%
    \put(0.60285207,0.21546895){\makebox(0,0)[lt]{\lineheight{1.25}\smash{\begin{tabular}[t]{l}$k_m$\end{tabular}}}}%
    \put(0,0){\includegraphics[width=\unitlength,page=1]{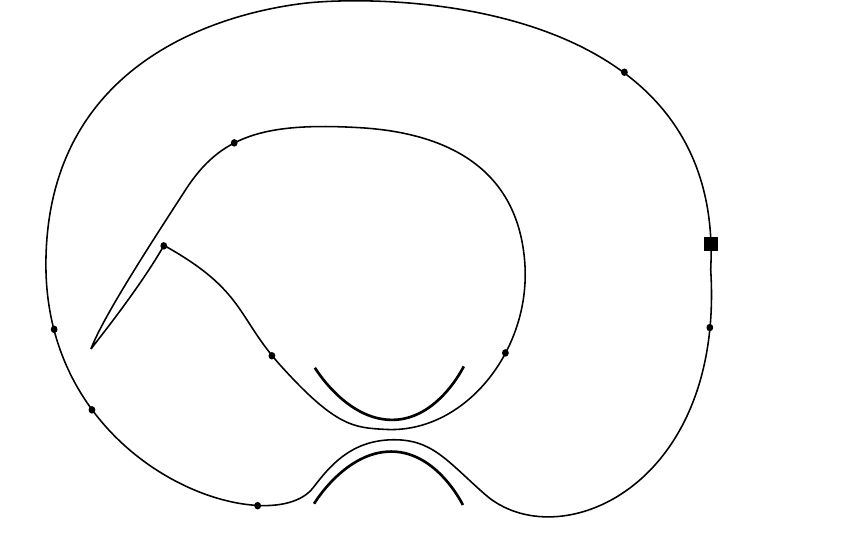}}%
    \put(0.84136207,0.24325446){\makebox(0,0)[lt]{\lineheight{1.25}\smash{\begin{tabular}[t]{l}$i_K$\end{tabular}}}}%
    \put(0.10772021,0.18528501){\makebox(0,0)[lt]{\lineheight{1.25}\smash{\begin{tabular}[t]{l}$+$\end{tabular}}}}%
  \end{picture}%
\endgroup%
\caption{Non-split annulus pinching during Reidemeister II move.}
\label{Figure:signs_annuli_II_final}
\end{figure}

Next, we discuss the case where we have pinching along the path of the branch point. In this case, we have more pinching points on one boundary of $\cM^\pi$ compared to the other. We consider some of the more complicated cases, other cases go similarly. For example, in the case shown in Figure \ref{Figure:1-2_pinching}, we have an annulus $u_0$ on $\Lambda_1$ (no annuli on $\Lambda_0$) and three corresponding configurations of curves that contribute to $\psi_0\circ d_0$ as shown in Figure \ref{Figure:1-2_pinching}. It is not difficult to see that the summand corresponding to the left figure comes with a sign equal to the orientation sign of $u_0$ (and coefficient $1$). The summands corresponding to the second and the third figure come with opposite signs and coefficient $\frac{1}{2}$, and they cancel out. Next, consider the case shown in Figure \ref{Figure:1-2_pinchingIII}, where we have one split boundary point. Here there is no annulus on $\Lambda_1$ by Lemma \ref{Lemma:OSection_hyperbolic} and Lemma \ref{Lemma:OSection_ReidII}, and we have two summands in $\psi_0\circ d_0$ that cancel out as before. Similarly in the case shown in Figure \ref{Figure:1-2_pinchingII}. Here we have an annulus $u_0$ on $\Lambda_1$ and two summands in $\psi_0\circ d_0$ with equal signs and coefficients $\frac{1}{2}$. It is not difficult to check that their sign agrees with the orientation sign of $u_0$.

\begin{figure}
\def\svgwidth{163mm}
\begingroup%
  \makeatletter%
  \providecommand\rotatebox[2]{#2}%
  \newcommand*\fsize{\dimexpr\f@size pt\relax}%
  \newcommand*\lineheight[1]{\fontsize{\fsize}{#1\fsize}\selectfont}%
  \ifx\svgwidth\undefined%
    \setlength{\unitlength}{1148.63652806bp}%
    \ifx\svgscale\undefined%
      \relax%
    \else%
      \setlength{\unitlength}{\unitlength * \real{\svgscale}}%
    \fi%
  \else%
    \setlength{\unitlength}{\svgwidth}%
  \fi%
  \global\let\svgwidth\undefined%
  \global\let\svgscale\undefined%
  \makeatother%
  \begin{picture}(1,0.22208364)%
    \lineheight{1}%
    \setlength\tabcolsep{0pt}%
    \put(0.02033246,0.0618795){\makebox(0,0)[lt]{\lineheight{1.25}\smash{\begin{tabular}[t]{l}$+$\end{tabular}}}}%
    \put(0.24809893,0.131498){\makebox(0,0)[lt]{\lineheight{1.25}\smash{\begin{tabular}[t]{l}$+$\end{tabular}}}}%
    \put(0.08448875,0.18709184){\makebox(0,0)[lt]{\lineheight{1.25}\smash{\begin{tabular}[t]{l}$+$\end{tabular}}}}%
    \put(0.16787956,0.18755825){\makebox(0,0)[lt]{\lineheight{1.25}\smash{\begin{tabular}[t]{l}$-$\end{tabular}}}}%
    \put(0.42933843,0.19152258){\makebox(0,0)[lt]{\lineheight{1.25}\smash{\begin{tabular}[t]{l}$+$\end{tabular}}}}%
    \put(0.51534081,0.19460075){\makebox(0,0)[lt]{\lineheight{1.25}\smash{\begin{tabular}[t]{l}$-$\end{tabular}}}}%
    \put(0.79452265,0.18965702){\makebox(0,0)[lt]{\lineheight{1.25}\smash{\begin{tabular}[t]{l}$+$\end{tabular}}}}%
    \put(0.88052552,0.19273519){\makebox(0,0)[lt]{\lineheight{1.25}\smash{\begin{tabular}[t]{l}$-$\end{tabular}}}}%
    \put(0.79760089,0.02651327){\makebox(0,0)[lt]{\lineheight{1.25}\smash{\begin{tabular}[t]{l}$+$\end{tabular}}}}%
    \put(0.88099193,0.02959144){\makebox(0,0)[lt]{\lineheight{1.25}\smash{\begin{tabular}[t]{l}$-$\end{tabular}}}}%
    \put(0.48446572,0.0475941){\makebox(0,0)[lt]{\lineheight{1.25}\smash{\begin{tabular}[t]{l}$+$\end{tabular}}}}%
    \put(0.5145013,0.03005781){\makebox(0,0)[lt]{\lineheight{1.25}\smash{\begin{tabular}[t]{l}$-$\end{tabular}}}}%
    \put(0,0){\includegraphics[width=\unitlength,page=1]{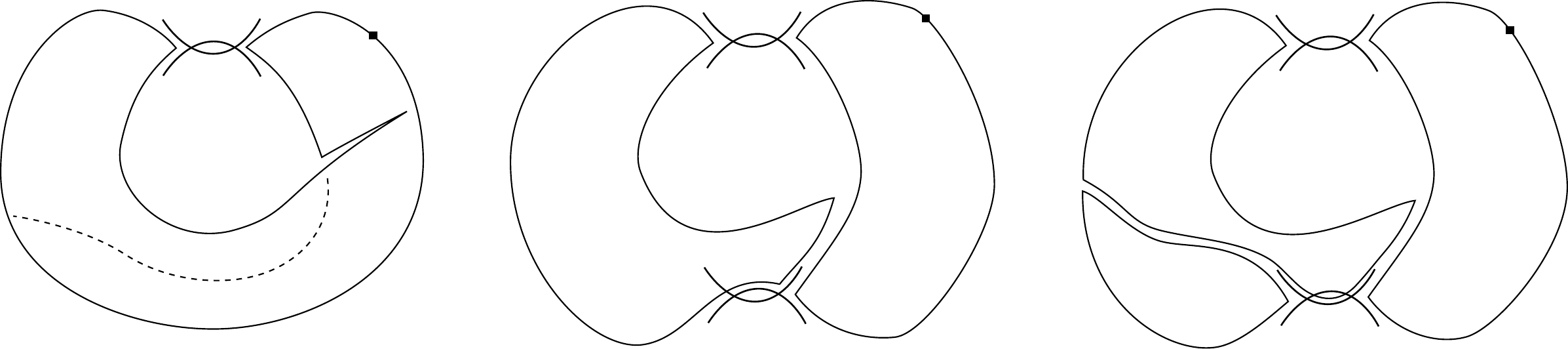}}%
    \put(0.69518182,0.07319914){\makebox(0,0)[lt]{\lineheight{1.25}\smash{\begin{tabular}[t]{l}$+$\end{tabular}}}}%
  \end{picture}%
\endgroup%
\caption{Corners marked with $+$ ($-$) are positive (negative).}
\label{Figure:1-2_pinching}
\end{figure}

\begin{figure}
\def\svgwidth{163mm}
\begingroup%
  \makeatletter%
  \providecommand\rotatebox[2]{#2}%
  \newcommand*\fsize{\dimexpr\f@size pt\relax}%
  \newcommand*\lineheight[1]{\fontsize{\fsize}{#1\fsize}\selectfont}%
  \ifx\svgwidth\undefined%
    \setlength{\unitlength}{1148.63652806bp}%
    \ifx\svgscale\undefined%
      \relax%
    \else%
      \setlength{\unitlength}{\unitlength * \real{\svgscale}}%
    \fi%
  \else%
    \setlength{\unitlength}{\svgwidth}%
  \fi%
  \global\let\svgwidth\undefined%
  \global\let\svgscale\undefined%
  \makeatother%
  \begin{picture}(1,0.22208364)%
    \lineheight{1}%
    \setlength\tabcolsep{0pt}%
    \put(0.02033246,0.0618795){\makebox(0,0)[lt]{\lineheight{1.25}\smash{\begin{tabular}[t]{l}$+$\end{tabular}}}}%
    \put(0.08448875,0.18709184){\makebox(0,0)[lt]{\lineheight{1.25}\smash{\begin{tabular}[t]{l}$+$\end{tabular}}}}%
    \put(0.16787956,0.18755825){\makebox(0,0)[lt]{\lineheight{1.25}\smash{\begin{tabular}[t]{l}$-$\end{tabular}}}}%
    \put(0.42933843,0.19152258){\makebox(0,0)[lt]{\lineheight{1.25}\smash{\begin{tabular}[t]{l}$+$\end{tabular}}}}%
    \put(0.51534081,0.19460075){\makebox(0,0)[lt]{\lineheight{1.25}\smash{\begin{tabular}[t]{l}$-$\end{tabular}}}}%
    \put(0.79452265,0.18965702){\makebox(0,0)[lt]{\lineheight{1.25}\smash{\begin{tabular}[t]{l}$+$\end{tabular}}}}%
    \put(0.88052552,0.19273519){\makebox(0,0)[lt]{\lineheight{1.25}\smash{\begin{tabular}[t]{l}$-$\end{tabular}}}}%
    \put(0.79760089,0.02651327){\makebox(0,0)[lt]{\lineheight{1.25}\smash{\begin{tabular}[t]{l}$+$\end{tabular}}}}%
    \put(0.88099193,0.02959144){\makebox(0,0)[lt]{\lineheight{1.25}\smash{\begin{tabular}[t]{l}$-$\end{tabular}}}}%
    \put(0.48446572,0.0475941){\makebox(0,0)[lt]{\lineheight{1.25}\smash{\begin{tabular}[t]{l}$+$\end{tabular}}}}%
    \put(0.5145013,0.03005781){\makebox(0,0)[lt]{\lineheight{1.25}\smash{\begin{tabular}[t]{l}$-$\end{tabular}}}}%
    \put(0,0){\includegraphics[width=\unitlength,page=1]{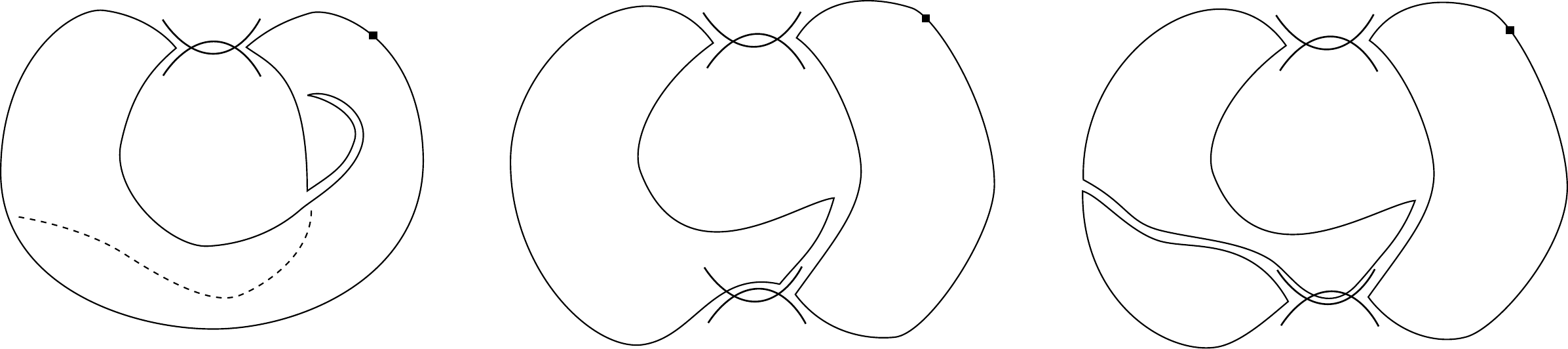}}%
    \put(0.69518182,0.07319914){\makebox(0,0)[lt]{\lineheight{1.25}\smash{\begin{tabular}[t]{l}$+$\end{tabular}}}}%
  \end{picture}%
\endgroup%
\caption{Corners marked with $+$ ($-$) are positive (negative).}
\label{Figure:1-2_pinchingIII}
\end{figure}

\begin{figure}
\def\svgwidth{110mm}
\begingroup%
  \makeatletter%
  \providecommand\rotatebox[2]{#2}%
  \newcommand*\fsize{\dimexpr\f@size pt\relax}%
  \newcommand*\lineheight[1]{\fontsize{\fsize}{#1\fsize}\selectfont}%
  \ifx\svgwidth\undefined%
    \setlength{\unitlength}{775.17429449bp}%
    \ifx\svgscale\undefined%
      \relax%
    \else%
      \setlength{\unitlength}{\unitlength * \real{\svgscale}}%
    \fi%
  \else%
    \setlength{\unitlength}{\svgwidth}%
  \fi%
  \global\let\svgwidth\undefined%
  \global\let\svgscale\undefined%
  \makeatother%
  \begin{picture}(1,0.32907874)%
    \lineheight{1}%
    \setlength\tabcolsep{0pt}%
    \put(0.15440601,0.28379402){\makebox(0,0)[lt]{\lineheight{1.25}\smash{\begin{tabular}[t]{l}$+$\end{tabular}}}}%
    \put(0.28184248,0.28835519){\makebox(0,0)[lt]{\lineheight{1.25}\smash{\begin{tabular}[t]{l}$-$\end{tabular}}}}%
    \put(0.6955281,0.28102967){\makebox(0,0)[lt]{\lineheight{1.25}\smash{\begin{tabular}[t]{l}$+$\end{tabular}}}}%
    \put(0.82296529,0.28559084){\makebox(0,0)[lt]{\lineheight{1.25}\smash{\begin{tabular}[t]{l}$-$\end{tabular}}}}%
    \put(0.70008937,0.03928679){\makebox(0,0)[lt]{\lineheight{1.25}\smash{\begin{tabular}[t]{l}$+$\end{tabular}}}}%
    \put(0.82365641,0.04384795){\makebox(0,0)[lt]{\lineheight{1.25}\smash{\begin{tabular}[t]{l}$-$\end{tabular}}}}%
    \put(0.23609244,0.07052391){\makebox(0,0)[lt]{\lineheight{1.25}\smash{\begin{tabular}[t]{l}$+$\end{tabular}}}}%
    \put(0.28059852,0.04453901){\makebox(0,0)[lt]{\lineheight{1.25}\smash{\begin{tabular}[t]{l}$-$\end{tabular}}}}%
    \put(0.548327,0.1084649){\makebox(0,0)[lt]{\lineheight{1.25}\smash{\begin{tabular}[t]{l}$-$\end{tabular}}}}%
    \put(0.42185712,0.18425815){\makebox(0,0)[lt]{\lineheight{1.25}\smash{\begin{tabular}[t]{l}$-$\end{tabular}}}}%
    \put(0,0){\includegraphics[width=\unitlength,page=1]{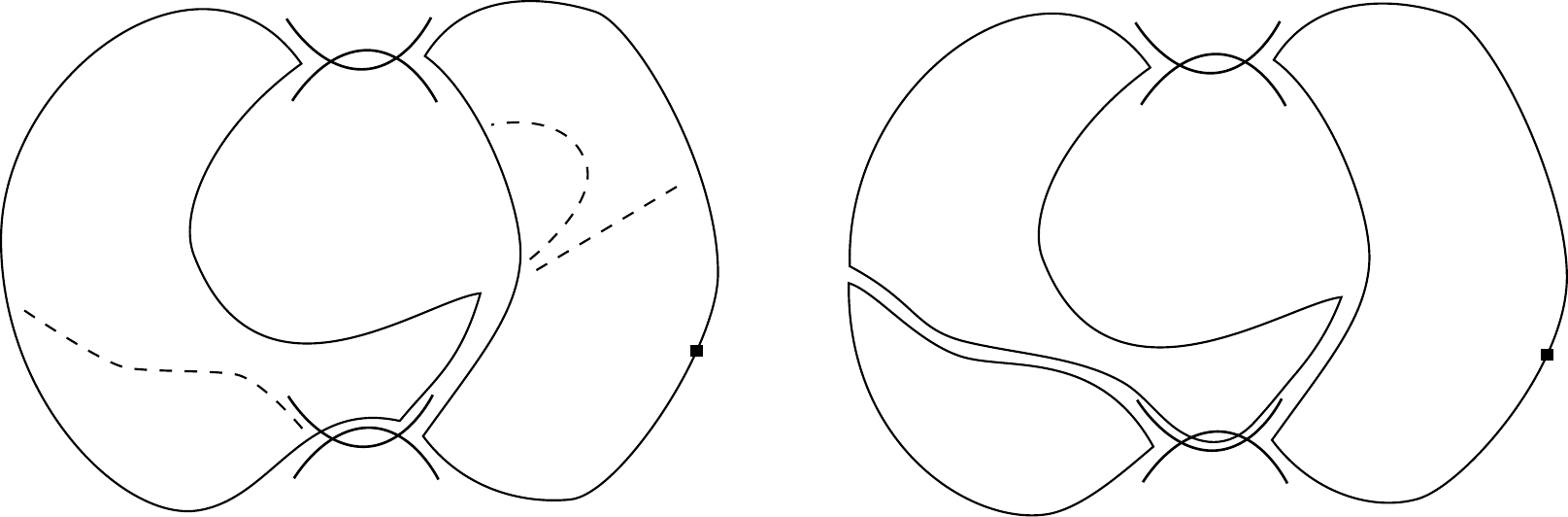}}%
    \put(0.0243428,0.09137583){\makebox(0,0)[lt]{\lineheight{1.25}\smash{\begin{tabular}[t]{l}$-$\end{tabular}}}}%
  \end{picture}%
\endgroup%
\caption{Corners marked with $+$ ($-$) are positive (negative).}
\label{Figure:1-2_pinchingII}
\end{figure}

Finally, this shows
\begin{align*}
\pi_{\hbar\;\widetilde\cA\otimes\widetilde\cA^{\cyc}}\circ \psi_0\circ d_{0,\lD}(q_i)+\psi_0\circ d_{0,A}(q_i)=d_{1,A}(q_i),
\end{align*} 
which finishes the proof.
\end{proof}
From Lemma \ref{Lemma:prop2a} and Lemma \ref{Lemma:prop2b} we now get the following corollary.
\begin{cor}\label{Corollary:bootstrap_property}
For all $s\in\cA^s,s_1,s_2\in\widetilde\cA^s$ we have
\begin{align*}
&\tau\circ\phi_0^{-1}\circ d_0\circ\phi_0(s)=\tau\circ d^s(s),\\
&(\tau\otimes\tau)\{s_1,s_2\}_{\phi_0^{-1}\circ d_0\circ\phi_0}=(\tau\otimes\tau)\{s_1,s_2\}_{d^s},
\end{align*}
where $\tau:\cA^s\to\cA(\Lambda_1)$ is the projection.
\end{cor}

\paragraph{\textit{Proof of Proposition \ref{Prop:eq_after_II}:}}
The morphism $\phi_0:\cA^s\to\cA(\Lambda_0)$ satisfies (\ref{Eq:bootstrapping}) by Corollary \ref{Corollary:bootstrap_property}. By Lemma \ref{Lemma:bootstrap} and Remark \ref{Remark:pairing_t}, there exists a tame second-order dg algebra isomorphism $\Phi:(\cA^s,d^s,\{\cdot,\cdot\}_{d^s})\to(\cA(\Lambda_0),d_0,\{\cdot,\cdot\}_{d_0})$. Note additionally that the chain complexes $(\cA(\Lambda_1),d_1)$ and $(\cA^s,d^s)$ are quasi-isomorphic by Lemma \ref{Lemma:Chain_homotopy}.

\section{Augmentations}\label{Sec:Augmentations}
\subsection{Second-order augmentations}
In this section, we define the notion of a second-order augmentation and construct a second-order augmentation from an exact Lagrangian filling. More specifically, we consider exact Lagrangian disk fillings for Legendrian knots $\Lambda$ with $\operatorname{tb}(\Lambda)=-1$. The more general case goes similarly but requires working with homological coefficients on the filling.
\begin{defi}
A \textit{second-order augmentation} is a second-order dg algebra morphism $(\varepsilon,\{\cdot,\cdot\}_\varepsilon):(\cA,d,\{\cdot,\cdot\}_d)\to(\lQ\oplus\hbar\,\lQ,0,0)$. In other words, a second-order graded algebra morphism $\varepsilon:\cA\to\lQ\oplus\hbar\,\lQ$ with respect to $\varepsilon_0$-antibracket $\{\cdot,\cdot\}_\varepsilon:\widetilde\cA\times\widetilde\cA\to\lQ\cong\lQ\otimes\hbar\lQ$ is a second-order augmentation if 
\begin{align*}
&\varepsilon\circ d(u)=0,\\
&(\varepsilon\otimes\varepsilon)\{u,v\}_d+\{du,v\}_\varepsilon+(-1)^{|u|}\{u,dv\}_\varepsilon=0,
\end{align*}
for all $u,v\in\cA$.
\end{defi}

Let $L\subset\lR^4$ be an exact Lagrangian disk filling of a Legendrian knot $\Lambda$. Denote by $\cM_1(L,\gamma_i)$ the moduli space of  $J$-holomorphic disks and by $\cM_2(L,\gamma_i)$ the moduli space of $J$-holomorphic annuli in $\lR^4$ with boundary on $L$ and one positive puncture at $\gamma_i$. We define
\begin{align*}
\varepsilon_L(q_i)=\sum_{\substack{u\in\cM_1(L,\gamma_i)\\\ind(u)=0}}\epsilon(u)+\hbar\sum_{\substack{v\in\cM_2(L,\gamma_i)\\\ind(v)=0}}\epsilon(v),
\end{align*}
where $\epsilon(\cdot)\in\{+1,-1\}$ are the corresponding orientation signs, and $\varepsilon_L(t^\pm)=1$. Additionally, we define
\begin{align*}
\{q_i,q_j\}_{\varepsilon_L}=\sum_{\substack{u\in\cM_1(L,\gamma_i,\gamma_j)\\ \ind(u)=0}}\epsilon(u,\gamma_j^+),
\end{align*}
where $\cM_1(L,\gamma_i,\gamma_j)$ is the moduli space of $J$-holomorphic disks on $L$ with positive punctures at $\gamma_i,\gamma_j$ and $\varepsilon(u,\gamma_j^+)$ are the corresponding orientation signs. We additionally define $\{t^\pm,s\}_{\varepsilon_L}=\{s,t^\pm\}_{\varepsilon_L}=0$. The index of a $J$-holomorphic curve $u$ on $L$ is given by
\begin{align*}
\operatorname{ind}(u)\coloneqq k_++\mu_L([u])+\sum_{i=1}^{k}\epsilon_i\mu_{CZ}(\gamma_{i})-1,
\end{align*}
where $\mu_L$ is the Maslov class, $k^+$ the number of positive punctures and $\gamma_i$ the Reeb chords at the punctures (positive if $\epsilon_i=1$ and negative if $\epsilon_i=-1$). Similar as in the symplectization, the moduli space of index zero curves on $L$ is a compact manifold of dimension zero, therefore, $(\varepsilon,\{\cdot,\cdot\}_\varepsilon)$ is well defined. 

\begin{prop}
For $L$ an exact Lagrangian disk filling of a Legendrian knot $\Lambda$, the second-order graded algebra morphism $(\varepsilon_L,\{\cdot,\cdot\}_{\varepsilon_L})$ defined above is a second-order augmentation of $(\cA(\Lambda),d,\{\cdot,\cdot\}_d)$.
\end{prop}
\begin{rmk}
The following proof relies on the construction of virtual perturbations for the moduli space of index 1 disks on $L$ with two positive punctures asymptotic to the same Reeb chord. The proposition should therefore be seen as conjectural. 
\end{rmk}
\begin{proof}
The moduli space of regular index one curves with boundary on $L$ satisfies analogous versions of Proposition \ref{Prop:compactness_disk} and Proposition \ref{Prop:compactness_annulus}. It is a manifold of dimension 1 with a compactification consisting of 2-buildings (with the bottom level on $L$ and the top level in the symplectization) and nodal disks and annuli on $L$.

By looking at the boundary of the moduli space of index one disks on $L$ with one positive puncture, we get
\begin{align*}
\varepsilon_L\circ d(q_i)|_{\hbar=0}=0.
\end{align*}

We say Reeb chords $\gamma_i,\gamma_j,i\neq j$ are linked if there is exactly one point $i^+,i^-$ between $j^+,j^-$ on $\Lambda$. For $i,j\in\{1,\dots,n\},i\neq j$, let
\begin{align*}
E(q_i,q_j)= \begin{cases}
\pm\varepsilon_L(q_i)\varepsilon_L(q_j),&i,j\text{ linked}\\
0,&\text{ otherwise}
\end{cases}
\end{align*}
The boundary of the moduli space of index 1 disks on $L$ with two positive punctures at $\gamma_i,\gamma_j$  contains disk buildings consisting of a level on $\lR\times\Lambda$ and a level on $L$, and nodal curves consisting of two index zero disks on $L$ together with their boundary intersection. From this we get
\begin{align*}
\{d_0q_i,q_j\}_{\varepsilon_L}+(-1)^{|q_i|}\{q_i,d_0 q_j\}_{\varepsilon_L}+(\varepsilon_L\otimes\varepsilon_L) d_{\lD}(q_i,q_j)+E(q_i,q_j)=0.
\end{align*}
Moreover, it is not difficult to see from the definition that
\begin{align*}
(\varepsilon_L\otimes\varepsilon_L) d_f(q_i,q_j)=E(q_i,q_j),
\end{align*}
which implies
\begin{align*}
(\varepsilon_L\otimes\varepsilon_L)\{q_i,q_j\}_d+\{d_0q_i,q_j\}_{\varepsilon_L}+(-1)^{|q_i|}\{q_i,d_0q_j\}_{\varepsilon_L}=0.
\end{align*}
The proof for $i=j$ goes similarly but relies on the construction of virtual perturbations for the moduli space of index 1 disks on $L$ with two positive punctures at $\gamma_i$.

Finally, the boundary of the moduli space of index 1 annuli on $L$ with one positive puncture consists of nodal annuli and annulus buildings with one level on $L$ and one level on $\lR\times\Lambda$. Obviously, the summands in $\pi_{\hbar\lQ}\circ\varepsilon_L\circ(d_\lD+d_A)$ correspond to the annulus buildings. Nodal annuli can be seen as index zero disks with a boundary self-intersection or an interior intersection with $L$. Since all loops on $L$ are contractible, it is not difficult to see that the summands in $\varepsilon_L\circ d_f(q_i)$ correspond to the contributions of nodal annuli with a positive puncture at $\gamma_i$. From this we get 
\begin{align*}
\pi_{\hbar\lQ}\circ\varepsilon_L\circ d(q_i)=0.
\end{align*}
This shows that $\varepsilon_L$ is a second-order augmentation of $(\cA(\Lambda),d,\{\cdot,\cdot\}_d)$.
\end{proof}

\subsection{$\hbar$-linearization}

We introduce the notion of $\hbar$-linearization of a second-order dga $(\cA,d,\{\cdot,\cdot\}_{d})$ with respect to a second-order augmentation. We define an invariant of Legendrian knots that consists of finite dimensional homology groups and is potentially easier to work with. This is a generalization of the linearized Legendrian knot invariant defined in \cite[Section 5]{Chekanov02}.

Let $(\cA,d,\{\cdot,\cdot\}_d)$ be a second-order dga, where $\cA$ is generated by $q_1,\dots,q_n$ as before, and $(\varepsilon,\{\cdot,\cdot\}_\varepsilon)$ a second-order augmentation of $(\cA,d,\{\cdot,\cdot\}_d)$. As an example, we should think of the second-order dga corresponding to a slice Legendrian knot. Denote by $\psi_\varepsilon$ the second-order algebra automorphism of $\cA$ given by 
\begin{equation}
\label{Eq:augment_conjug}
\begin{aligned}
&\psi_\varepsilon(q_i)=q_i-\varepsilon(q_i),\\
&\{q_i,q_j\}_{\psi_\varepsilon}=-\{q_i,q_j\}_\varepsilon,
\end{aligned}
\end{equation}
with the inverse given by
\begin{align*}
&\psi_\varepsilon^{-1}(q_i)=q_i+\varepsilon(q_i),\\
&\{q_i,q_j\}_{\psi_\varepsilon}=\{q_i,q_j\}_\varepsilon.
\end{align*}

Denote 
$$d_\varepsilon=\psi_\varepsilon^{-1}\circ d\circ \psi_\varepsilon.$$ 
Let $\cL^Q_\cA=\cL^Q\subset\cA/_{t^\pm=1,x\otimes y=(-1)^{|x||y|}y\otimes x}$ be the vector subspace generated by $q_i,\hbar(q_i\otimes 1),q_iq_j$ for $i,j\in\{1,\dots,n\}$. Since $\varepsilon$ is a second-order augmentation, we have a well-defined map $d_\varepsilon^Q:\cL^Q\to\cL^Q$ given by 
\begin{align*}
d_\varepsilon^Q(s)=\pi\circ d_\varepsilon(s),
\end{align*}
where $\pi:\cA\to\cL^Q$ is the projection, such that 
$$d_\varepsilon^Q\circ d_\varepsilon^Q=0.$$

\begin{thm}
The set
\begin{align*}
P(\Lambda)\coloneq P\left(\cA(\Lambda),d,\{\cdot,\cdot\}_d\right)\coloneq\left\{ H_*(\cL^Q_{\cA(\Lambda)},d_\varepsilon^Q)\,|\,\varepsilon\text{ II ord. augmentation}\right\}
\end{align*}
is an invariant of the Legendrian knot $\Lambda$ up to Legendrian knot isotopy.
\end{thm}
The proof of the theorem follows from the lemmas below.

We say a second-order dga $(\cA,d,\{\cdot,\cdot\}_d)$ is \textit{augmented} if the images of $d$ and $\{\cdot,\cdot\}_d$ do not contain any constant terms. If $(\cA,d,\{\cdot,\cdot\}_d)$ is augmented, $(\cL^Q_\cA,\pi\circ d)$ is a well-defined chain complex. For a second-order algebra morphism $(\varepsilon,\{\cdot,\cdot\}_\varepsilon):\cA\to\lQ\oplus\hbar\lQ$ and the second-order algebra isomorphism $\psi_\varepsilon:\cA\to\cA$ given by (\ref{Eq:augment_conjug}), the second-order dga $(\cA,\psi^{-1} d \psi,\{\cdot,\cdot\}_{\psi^{-1} d \psi})$ is augmented if and only if $(\varepsilon,\{\cdot,\cdot\}_\varepsilon)$ is a second-order augmentation.  

\begin{lemma}
Let $(\cA,d,\{\cdot,\cdot\}_d)$ and $(\cA,d',\{\cdot,\cdot\}_{d'})$ be tame isomorphic second-order dg algebras. Then
\begin{align*}
P(\cA,d,\{\cdot,\cdot\}_d)=P(\cA,d',\{\cdot,\cdot\}_{d'}).
\end{align*}
\end{lemma}
\begin{proof}
Let $e:(\cA,d,\{\cdot,\cdot\}_d)\to(\cA,d',\{\cdot,\cdot\}_{d'})$ be a tame isomorphism and $\eta:\cA\to\cA$
\begin{align*}
\eta(q_i)=q_i+\varepsilon_i,\\
\{q_i,q_j\}_\eta=\varepsilon_{ij}
\end{align*}
a second-order algebra morphism with $\varepsilon_i,\varepsilon_{ij},i,j\in\{1,\dots n\}$ constants terms such that $\eta^{-1}d\eta$ is augmented. Without loss of generality, we can assume $e$ is elementary. We show there exists an automorphism $\widetilde\eta:\cA\to\cA$ of the form
\begin{align*}
&\widetilde\eta(q_i)=q_i+\widetilde\varepsilon_i,\\
&\{q_i,q_j\}_{\widetilde\eta}=\widetilde\varepsilon_{ij},
\end{align*}
for $\widetilde\varepsilon_i,\widetilde\varepsilon_{ij}$ constant terms, and an automorphism $r:\cA\to\cA$ such that $r|_{\cL^Q}$ does not contain any constant terms, such that
\begin{equation}
\begin{aligned}
\label{Eq:tame_iso_invar}
&e\circ\eta=\widetilde\eta\circ r,\\
&\{\cdot,\cdot\}_{e\circ\eta}=\{\cdot,\cdot\}_{\widetilde\eta\circ r}.
\end{aligned}
\end{equation}

Assume first $e$ is of the form
\begin{align*}
&e(q_i)=q_i,\\
&\{q_j,q_k\}_e=
\omega_{jk},
\end{align*}
for some $\omega_{jk}\in\cA\otimes\cA$. Take
\begin{align*}
&\widetilde\varepsilon_i=\varepsilon_i,\\
&r_i=0,\\
&\widetilde\varepsilon_{ij}=(\varepsilon\otimes\varepsilon)\omega_{ij}+\varepsilon_{ij},\\
&r_{ij}=(\eta^{-1}\otimes\eta^{-1})(\omega_{ij}+\varepsilon_{ij}-\widetilde\varepsilon_{ij}),
\end{align*}
where $\varepsilon=\pi'\circ\eta^{-1}$ for $\pi'$ the projection to constant terms, and define an automorphism $r$ by
\begin{align*}
&r(q_i)=q_i+r_i,\\
&\{q_i,q_j\}_r=r_{ij}.
\end{align*}
It is not difficult to check that $r|_{\cL^Q}$ does not contain any constant terms and that (\ref{Eq:tame_iso_invar}) holds.

Similarly if $e$ is of the form
\begin{align*}
e(q_j)=\begin{cases}
q_j,&j\neq i\\
q_i+\omega_i,&j=i
\end{cases}
\end{align*}
for some $i\in\{1,\dots,n\}$, where $\omega_i\in\widetilde A$ does not contain letter $q_i$. Then for
\begin{align*}
&\widetilde\varepsilon_j=
\begin{cases}
\varepsilon_j,&j\neq i\\
\varepsilon_i+\varepsilon(\omega_i),&j=i
\end{cases}
\\
&r_j=\begin{cases}
0,&j\neq i\\
\eta^{-1}(\omega_i+\varepsilon_i-\widetilde\varepsilon_i),&j=i
\end{cases}\\
&\widetilde\varepsilon_{jk}=\varepsilon_{jk},r_{jk}=0\text{ for }j,k\neq i,
\end{align*}
we get $\{q_j,q_k\}_{e\circ\eta}=\{q_j,q_k\}_{\widetilde\eta\circ r}$ for $j,k\neq i$. Additionally, for any $j\in\{1,\dots,n\}$, $e\circ\eta(q_j)=\widetilde\eta\circ r(q_j)$ and $r(q_j)$ does not contain constant terms. Similarly we find $\widetilde\varepsilon_{jk},r_{jk}$ for $j=i$ or $k=i$.

From (\ref{Eq:tame_iso_invar}) we then have
\begin{align*}
H_*(\cL^Q,\pi\eta^{-1}d\eta)=H_*(\cL^Q,\pi r^{-1}\widetilde\eta^{-1}d' \widetilde \eta r)\overset{(\pi r)_*}{\cong}H_*(\cL^Q,\pi \widetilde\eta^{-1}d' \widetilde\eta).
\end{align*}
We additionally notice that $\eta^{-1}d \eta=r^{-1}(\widetilde\eta^{-1}d' \widetilde \eta) r$ is augmented if and only if $\widetilde\eta^{-1}d'\widetilde\eta$ is augmented. This finishes the proof of the lemma.
\end{proof}

\begin{lemma}
Let $(\cA^s,d^s,\{\cdot,\cdot\}_{d^s})$ be a stabilization of a second-order dga $(\cA,d,\{\cdot,\cdot\}_d)$. Then
\begin{align*}
P(\cA,d,\{\cdot,\cdot\}_d)=P(\cA^s,d_s,\{\cdot,\cdot\}_{d_s}).
\end{align*}
\end{lemma}
\begin{proof}
Let $(\varepsilon,\{\cdot,\cdot\}_{\varepsilon})$ be a second-order augmentation of $(\cA^s,d^s,\{\cdot,\cdot\}_{d^s})$. Then the restriction of $(\varepsilon,\{\cdot,\cdot\}_\varepsilon)$ to $\cA$ is a second-order augmentation of $(\cA,d,\{\cdot,\cdot\}_d)$. We show that
\begin{align*}
H_*(\cL^Q_{\cA^s},d^{s,Q}_\varepsilon)\cong H_*(\cL^Q_{\cA},d^{Q}_\varepsilon).
\end{align*}

Define a linear map $\Phi:\cL_{\cA^s}^Q\to\cL_\cA^Q$ by
$$\Phi\left(q_aA+\widetilde A q_a+q_b B+\widetilde B q_b+C+D\right)=D-\pi\{q_a,B\}_{d^{s}_\varepsilon}-(-1)^{|\widetilde B|}\pi\{\widetilde B,q_a\}_{d^s_\varepsilon},$$
where $A,\widetilde A,B,\widetilde B$ are linear combinations of $q_i,i\neq a,b$, $C\in\cL_{\cA^s}^Q$ contains only letters $q_a,q_b$ and $D\in\cL_\cA^Q$. Note that $\pi\{q_a,B\}_{d^{s}_\varepsilon},\pi\{\widetilde B,q_a\}_{d^s_\varepsilon}\in\cL_\cA^Q$. Using
$$(d_\varepsilon^s\otimes 1+1\otimes d_\varepsilon^s)\left(\{q_a,B\}_{d^s_\varepsilon}+(-1)^{|\widetilde B|}\{\widetilde B,q_a\}_{d^s_\varepsilon}\right)=\{q_b,B\}_{d^s_\varepsilon}+(-1)^{|q_a|}\{q_a,d_\varepsilon^s B\}_{d^s_\varepsilon}+(-1)^{|\widetilde B|}\{d^s_\varepsilon\widetilde B,q_a\}_{d^s_\varepsilon}+\{\widetilde B,q_b\}_{d^s_\varepsilon},$$
we get that
$$\Phi:(\cL_{\cA^s}^Q,d_\varepsilon^{s,Q})\to(\cL_\cA^Q,d^Q_\varepsilon)$$
is a chain map. Map $\Phi$ is obviously surjective. 

For $S=q_aA+\widetilde A q_a+q_b B+\widetilde B q_b+C+D$, we have $d_\varepsilon^{s,Q}(S)=0$ if and only if
\begin{align*}
&A=(-1)^{|q_a|}\pi_L d_\varepsilon^{s,Q}(B),\\
&\widetilde A=(-1)^{|\widetilde B|}\pi_L d_\varepsilon^{s,Q}(\widetilde B),\\
&d_\varepsilon^{s,Q}(C)=0,\\
&d_\varepsilon^{Q}\circ\Phi(S)=0,
\end{align*}
where $\pi_{L}$ is the projection to linear terms. Additionally, it is not difficult to see that $d_\varepsilon^{s,Q}(C)=0$ implies $C=d_\varepsilon^{s,Q} (C')$ for some $C'\in\cL_{\cA^s}^Q$ that contains only letters $q_a,q_b$. If for $S\in\ker d_\varepsilon^{s,Q}$ we have $\Phi(S)=d_\varepsilon^Q(D')$ for some $D'\in\cL_\cA^Q$, this implies
\begin{align*}
S=d_\varepsilon^{s,Q}\left(q_aB+(-1)^{|\widetilde B|}\widetilde B q_a+C'+D'\right).
\end{align*}
This shows that
\begin{align*}
\Phi_*:H_*(\cL^Q_{\cA^s},d^{s,Q}_\varepsilon)\to H_*(\cL^Q_{\cA},d^{Q}_\varepsilon)
\end{align*}
is an isomorphism, which finishes the proof.
\end{proof}

\section{Examples}
\label{Sec:Examples}
\noindent In this section, we compute the invariant for some simple examples.

\begin{example}
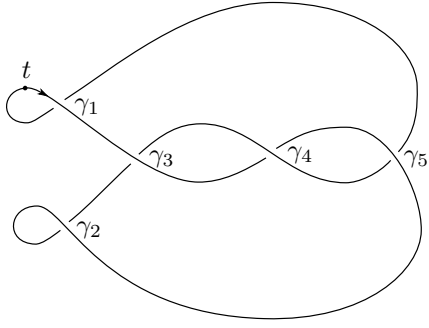
\begin{figure}
\def\svgwidth{80mm}
\begingroup%
  \makeatletter%
  \providecommand\rotatebox[2]{#2}%
  \newcommand*\fsize{\dimexpr\f@size pt\relax}%
  \newcommand*\lineheight[1]{\fontsize{\fsize}{#1\fsize}\selectfont}%
  \ifx\svgwidth\undefined%
    \setlength{\unitlength}{377.00268716bp}%
    \ifx\svgscale\undefined%
      \relax%
    \else%
      \setlength{\unitlength}{\unitlength * \real{\svgscale}}%
    \fi%
  \else%
    \setlength{\unitlength}{\svgwidth}%
  \fi%
  \global\let\svgwidth\undefined%
  \global\let\svgscale\undefined%
  \makeatother%
  \begin{picture}(1,0.52794642)%
    \lineheight{1}%
    \setlength\tabcolsep{0pt}%
    \put(0,0){\includegraphics[width=\unitlength,page=1]{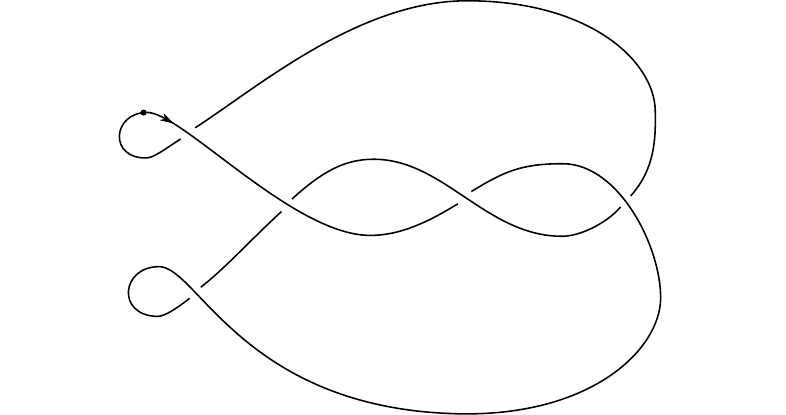}}%
    \put(0.26485429,0.34790636){\makebox(0,0)[lt]{\lineheight{1.25}\smash{\begin{tabular}[t]{l}$\gamma_1$\end{tabular}}}}%
    \put(0.26811019,0.14319014){\makebox(0,0)[lt]{\lineheight{1.25}\smash{\begin{tabular}[t]{l}$\gamma_2$\end{tabular}}}}%
    \put(0.38889372,0.26113164){\makebox(0,0)[lt]{\lineheight{1.25}\smash{\begin{tabular}[t]{l}$\gamma_3$\end{tabular}}}}%
    \put(0.61909293,0.2682366){\makebox(0,0)[lt]{\lineheight{1.25}\smash{\begin{tabular}[t]{l}$\gamma_4$\end{tabular}}}}%
    \put(0.8123466,0.26113164){\makebox(0,0)[lt]{\lineheight{1.25}\smash{\begin{tabular}[t]{l}$\gamma_5$\end{tabular}}}}%
    \put(0.17738259,0.40080931){\makebox(0,0)[lt]{\lineheight{1.25}\smash{\begin{tabular}[t]{l}$t$\end{tabular}}}}%
    \put(-0.00189653,0.27540994){\makebox(0,0)[lt]{\lineheight{1.25}\smash{\begin{tabular}[t]{l}\textcolor{white}{.}\end{tabular}}}}%
  \end{picture}%
\endgroup%
\caption{Right-handed trefoil.}
\label{Fig:right_trefoilb}
\end{figure}
Consider the Legendrian right-handed trefoil shown in Figure \ref{Fig:right_trefoilb}. The algebra is generated by $t^\pm,q_i,i\in\{1,\dots,5\}$, with grading $|q_1|=|q_2|=1,|q_3|=|q_4|=|q_5|=0,|t^\pm|=0$. The second-order dga structure is given by
\begin{align*}
&d(q_1) = t^--q_3q_4q_5-q_5-q_3-\hbar(q_1\otimes 1)-\hbar(1\otimes q_1),\\
&d(q_2) = 1+q_5q_4q_3+q_5+q_3+\hbar(q_2\otimes 1)-\hbar(1\otimes q_2),\\
&d(q_3) = -\hbar(1\otimes q_3),\\
&d(q_4) = \hbar(q_4\otimes 1),\\
&d(q_5) = -\hbar(1\otimes q_5),\\
&d(t^+)=2 \hbar(t^+\otimes 1),\\
&d(t^-)=- \hbar(t^-\otimes 1)-\hbar(1\otimes t^-),
\end{align*}
and the antibracket
\begin{align*}
&\{q_1, q_1\}_d = 1\otimes q_1q_1,\\
&\{q_1, q_2\}_d = q_2\otimes q_1+1\otimes q_2q_1,\\
&\{q_1, q_3\}_d = q_3\otimes q_1-1\otimes q_3q_1,\\
&\{q_1, q_4\}_d = q_4\otimes q_1-1\otimes q_4q_1,\\
&\{q_1, q_5\}_d = q_5\otimes q_1-1\otimes q_5q_1,\\
&\{q_2, q_1\}_d = -q_2\otimes q_1+1\otimes q_1q_2,\\
&\{q_2, q_2\}_d = 1\otimes q_2q_2,\\
&\{q_2, q_3\}_d = q_2\otimes q_3-1\otimes q_3q_2,\\
&\{q_2, q_4\}_d = q_4\otimes q_2-q_2q_4\otimes 1,\\
&\{q_2, q_5\}_d = q_2\otimes q_5-1\otimes q_5q_2,\\
&\{q_3, q_1\}_d = q_3\otimes q_1-1\otimes q_1q_3,\\
&\{q_3, q_2\}_d = q_2\otimes q_3-1\otimes q_2q_3,\\
&\{q_3, q_3\}_d = -1\otimes q_3q_3+q_3\otimes q_3,\\
&\{q_3, q_4\}_d = -1\otimes 1+q_4\otimes q_3-q_3q_4\otimes 1-1\otimes q_4q_3,\\
&\{q_3, q_5\}_d = q_5\otimes q_3+q_3\otimes q_5-1\otimes q_5q_3,\\
&\{q_4, q_1\}_d = q_4\otimes q_1-1\otimes q_1q_4,\\
&\{q_4, q_2\}_d = q_4\otimes q_2-q_4q_2\otimes 1,\\
&\{q_4, q_3\}_d = 1\otimes 1+q_4\otimes q_3,\\
&\{q_4, q_4\}_d = -q_4q_4\otimes 1+q_4\otimes q_4,\\
&\{q_4, q_5\}_d = -1\otimes 1+q_4\otimes q_5-q_4q_5\otimes 1-1\otimes q_5q_4,\\
&\{q_5, q_1\}_d = q_5\otimes q_1-1\otimes q_1q_5,\\
&\{q_5, q_2\}_d = q_2\otimes q_5-1\otimes q_2q_5,\\
&\{q_5, q_3\}_d = -1\otimes q_3q_5,\\
&\{q_5, q_4\}_d = 1\otimes 1+q_4\otimes q_5,\\
&\{q_5, q_5\}_d = -1\otimes q_5q_5+q_5\otimes q_5,
\end{align*}
together with (\ref{Eq:bracket_for_t}). The differential from (\ref{Eq:P2}) is not trivial. For example, we have
\begin{align*}
D[q_3q_4]=[-\hbar (1\otimes 1)-\hbar(1\otimes q_4q_3)]\neq 0.
\end{align*}
The degree zero homology of the Chekanov--Eliashberg dga of the right-handed trefoil is isomorphic to
\begin{align*}
\lQ[q_3,q_4,q_5,t^\pm]/(t^--q_3q_4q_5-q_5-q_3,1+q_5q_4q_3+q_5+q_3)
\end{align*}
(see also \cite[p. 290]{MurCas19}). From this, we conclude $D[q_3q_4]\neq 0$.
\end{example}

\begin{example}
\begin{figure}
\def\svgwidth{60mm}
\begingroup%
  \makeatletter%
  \providecommand\rotatebox[2]{#2}%
  \newcommand*\fsize{\dimexpr\f@size pt\relax}%
  \newcommand*\lineheight[1]{\fontsize{\fsize}{#1\fsize}\selectfont}%
  \ifx\svgwidth\undefined%
    \setlength{\unitlength}{247.97865315bp}%
    \ifx\svgscale\undefined%
      \relax%
    \else%
      \setlength{\unitlength}{\unitlength * \real{\svgscale}}%
    \fi%
  \else%
    \setlength{\unitlength}{\svgwidth}%
  \fi%
  \global\let\svgwidth\undefined%
  \global\let\svgscale\undefined%
  \makeatother%
  \begin{picture}(1,0.74133467)%
    \lineheight{1}%
    \setlength\tabcolsep{0pt}%
    \put(0.17324927,0.36248436){\makebox(0,0)[lt]{\lineheight{1.25}\smash{\begin{tabular}[t]{l}$\gamma_1$\end{tabular}}}}%
    \put(0.49649039,0.52867155){\makebox(0,0)[lt]{\lineheight{1.25}\smash{\begin{tabular}[t]{l}$\gamma_2$\end{tabular}}}}%
    \put(0.51593319,0.22406585){\makebox(0,0)[lt]{\lineheight{1.25}\smash{\begin{tabular}[t]{l}$\gamma_3$\end{tabular}}}}%
    \put(0.64987299,0.66477198){\makebox(0,0)[lt]{\lineheight{1.25}\smash{\begin{tabular}[t]{l}$\gamma_4$\end{tabular}}}}%
    \put(0.68011764,0.37096788){\makebox(0,0)[lt]{\lineheight{1.25}\smash{\begin{tabular}[t]{l}$\gamma_5$\end{tabular}}}}%
    \put(0.69091898,0.07284314){\makebox(0,0)[lt]{\lineheight{1.25}\smash{\begin{tabular}[t]{l}$\gamma_6$\end{tabular}}}}%
    \put(0,0){\includegraphics[width=\unitlength,page=1]{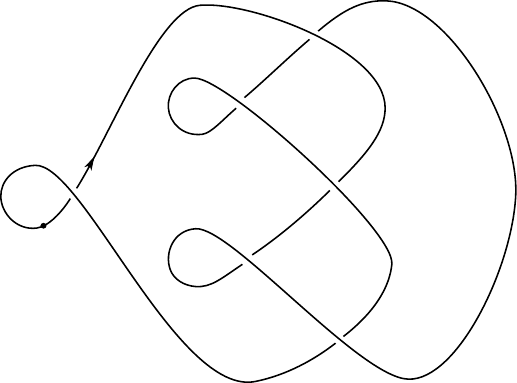}}%
    \put(0.07273103,0.25249568){\makebox(0,0)[lt]{\lineheight{1.25}\smash{\begin{tabular}[t]{l}$t$\end{tabular}}}}%
  \end{picture}%
\endgroup%
\caption{Left-handed trefoil.}
\label{Fig:left_trefoil}
\end{figure}
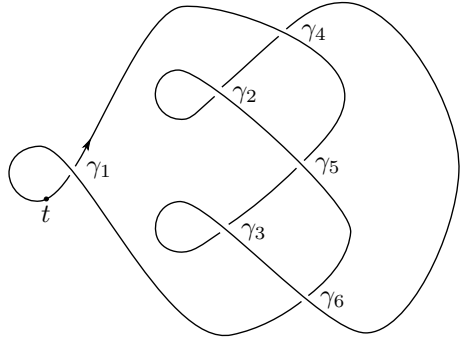
Next, we consider the Legendrian left-handed trefoil shown in Figure \ref{Fig:left_trefoil}. The algebra is generated by $t^\pm,q_i,i\in\{1,\dots,6\}$, with grading $|q_1|=|q_2|=|q_3|=|q_5|=1,|q_4|=|q_6|=-1,|t^\pm|=0$, and the second-order differential is given by
\begin{align*}
&d(q_1) = t^+ +q_5q_6q_2q_5+q_5q_3q_4q_5+q_5q_5+q_2q_5+q_5q_3-5\hbar(q_1\otimes 1)+\hbar(q_2q_5\otimes q_3)+\hbar(q_5q_3\otimes q_2),\\
&d(q_2)= 1+q_5q_4,\\
&d(q_3)=1+q_6q_5,\\
&d(q_4)=3\hbar(q_4\otimes 1)-\hbar(1\otimes q_4),\\
&d(q_5)=-3\hbar(q_5\otimes 1),\\
&d(q_6)=3\hbar(q_6\otimes 1)-\hbar(1\otimes q_6),\\
&d(t^+)=-5\hbar(t^+\otimes 1),\\
&d(t^-)=6\hbar(t^-\otimes 1)-\hbar(1\otimes t^-),
\end{align*}
and the antibracket
\begin{align*}
&\{q_1,q_1\}_d=-q_1q_1\otimes 1,\\
&\{q_1,q_2\}_d=-q_1\otimes q_2-q_1q_2\otimes 1,\\
&\{q_1,q_3\}_d=-q_1\otimes q_3-q_1q_3\otimes 1,\\
&\{q_1,q_4\}_d=-q_5q_6q_2q_2\otimes 1-q_5q_3q_4q_2\otimes 1-q_5q_2\otimes 1-q_2q_2\otimes 1-q_1\otimes q_4-q_1q_4\otimes 1,\\
&\{q_1,q_5\}_d=-q_5q_3\otimes q_2q_5-q_1\otimes q_5-q_1q_5\otimes 1,\\
&\{q_1,q_6\}_d=-1\otimes q_3q_6q_2q_5-1\otimes q_3q_3q_4q_5-1\otimes q_3q_5-1\otimes q_3q_3-q_1\otimes q_6-q_1q_6\otimes 1,\\
&\{q_2,q_1\}_d=q_1\otimes q_2-q_2q_1\otimes 1,\\
&\{q_2,q_2\}_d=-q_2q_2\otimes 1,\\
&\{q_2,q_3\}_d=0,\\
&\{q_2,q_4\}_d=0,\\
&\{q_2,q_5\}_d=q_5\otimes q_2-q_2q_5\otimes 1,\\
&\{q_2,q_6\}_d=1\otimes 1-q_2\otimes q_6+1\otimes q_6q_2,\\
&\{q_3,q_1\}_d=q_1\otimes q_3-q_3q_1\otimes 1,\\
&\{q_3,q_2\}_d=q_2\otimes q_3-q_3\otimes q_2-q_3q_2\otimes 1+1\otimes q_2q_3,\\
&\{q_3,q_3\}_d=-q_3q_3\otimes 1,\\
&\{q_3,q_4\}_d=1\otimes 1-q_3\otimes q_4+1\otimes q_4q_3,\\
&\{q_3,q_5\}_d=q_5\otimes q_3-q_3q_5\otimes 1,\\
&\{q_3,q_6\}_d=q_6\otimes q_3-q_3\otimes q_6-q_3q_6\otimes 1+1\otimes q_6q_3,\\
&\{q_4,q_1\}_d=-1\otimes q_5q_6q_2q_2-1\otimes q_5q_3q_4q_2-1\otimes q_5q_2-1\otimes q_2q_2+q_1\otimes q_4-q_4q_1\otimes 1,\\
&\{q_4,q_2\}_d=q_2\otimes q_4-q_4\otimes q_2-q_4q_2\otimes 1+1\otimes q_2q_4,\\
&\{q_4,q_3\}_d=1\otimes 1+q_3\otimes q_4+1\otimes q_3q_4,\\
&\{q_4,q_4\}_d=1\otimes q_4q_4,\\
&\{q_4,q_5\}_d=q_5\otimes q_4-q_4q_5\otimes 1+1\otimes q_5q_4,\\
&\{q_4,q_6\}_d=q_6\otimes q_4-q_4\otimes q_6+1\otimes q_6q_4,\\
&\{q_5,q_1\}_d=-q_2q_5\otimes q_5q_3+q_1\otimes q_5-q_5q_1\otimes 1,\\
&\{q_5,q_2\}_d=-q_5\otimes q_2-q_5q_2\otimes 1,\\
&\{q_5,q_3\}_d=-q_5\otimes q_3-q_5q_3\otimes 1,\\
&\{q_5,q_4\}_d=-q_5\otimes q_4,\\
&\{q_5,q_5\}_d=-q_5q_5\otimes 1,\\
&\{q_5,q_6\}_d=-q_5\otimes q_6-q_5q_6\otimes 1+1\otimes q_6q_5,\\
&\{q_6,q_1\}_d=-q_3q_6q_2q_5\otimes 1-q_3q_3q_4q_5\otimes 1-q_3q_5\otimes 1-q_3q_3\otimes 1+q_1\otimes q_6-q_6q_1\otimes 1,\\
&\{q_6,q_2\}_d=1\otimes 1+q_2\otimes q_6+1\otimes q_2q_6,\\
&\{q_6,q_3\}_d=0,\\
&\{q_6,q_4\}_d=1\otimes q_4q_6,\\
&\{q_6,q_5\}_d=q_5\otimes q_6,\\
&\{q_6,q_6\}_d=1\otimes q_6q_6,
\end{align*}
together with (\ref{Eq:bracket_for_t}). Here we have two $J$-holomorphic annuli $q_2q_5p_1\otimes q_3$ and $q_5q_3p_1\otimes q_2$ (counted with signs).
\end{example}

\begin{example}
For any Legendrian knot we have $H_*(\cA(\Lambda),d_\Lambda)\cong H^\hbar_*(\Lambda)$, see (\ref{Eq:P2}). Then, for every stabilized Legendrian knot $\Lambda_{\operatorname{stab}}$ (see \cite{Chekanov02},\cite[Appendix B]{Ng_rLSFT}) the homology group  $H_*(\cA(\Lambda_{\operatorname{stab}}),d_{\Lambda_{\operatorname{stab}}})$ vanishes.
\end{example}

\bibliography{main.bib}{}
\bibliographystyle{amsplain}
\end{document}